\documentclass[english,reqno,11pt,empty]{amsart}
\usepackage{amsmath, amsthm, amsfonts}
\usepackage{hyperref}
\hypersetup{
	colorlinks=true,
		citecolor=blue!60!black,
		linkcolor=red!60!black,
		urlcolor=green!40!black,
		filecolor=yellow!50!black,
	breaklinks=true,
	pdfpagemode=UseNone,
	bookmarksopen=false,
}
\usepackage{amsmath,amsthm,amssymb}
\usepackage{amscd,indentfirst,epsfig}
\usepackage{latexsym}
\usepackage{times}
\usepackage{enumerate}
\usepackage{mathrsfs}
\usepackage{stmaryrd}
\usepackage{amsopn}
\usepackage{amsmath}
\usepackage{amssymb,dsfont,mathtools}
\usepackage{amsfonts,bm}
\usepackage{amsbsy,amsmath}
\usepackage{amscd}
\usepackage{xcolor}
\usepackage{mathtools}
\usepackage{cancel}
\usepackage{hyperref} 
\usepackage{mathrsfs}
\usepackage{amsmath,amsthm,amssymb}
\usepackage{latexsym}
\usepackage{latexsym}
\usepackage{bm}
\usepackage{enumerate}
\usepackage{mathrsfs}
\usepackage{stmaryrd}
\usepackage{amsopn}
\usepackage{amsmath}
\usepackage{amssymb}
\usepackage{amsfonts}
\usepackage{amsbsy}
\usepackage{amscd,indentfirst}
\usepackage{hyperref}
\usepackage{amsfonts,amsmath,latexsym,amssymb,verbatim,amsbsy}
\usepackage{amsthm}
\usepackage{colordvi}
\usepackage{dsfont}
\usepackage{hyperref}
\hypersetup{
	colorlinks=true,
		citecolor=blue!60!black,
		linkcolor=red!60!black,
		urlcolor=green!40!black,
		filecolor=yellow!50!black,
	breaklinks=true,
	pdfpagemode=UseNone,
	bookmarksopen=false,
}
\usepackage{amsmath,amsthm,amssymb}
\usepackage{amscd,indentfirst,epsfig}
\usepackage{latexsym}
\usepackage{times}
\usepackage{enumerate}
\usepackage{mathrsfs}
\usepackage{stmaryrd}
\usepackage{amsopn}
\usepackage{amsmath}
\usepackage{amssymb,dsfont,mathtools}
\usepackage{amsfonts,bm}
\usepackage{amsbsy,amsmath}
\usepackage{amscd}
\usepackage{xcolor}

\usepackage[mono=false]{libertine}
\usepackage[T1]{fontenc}
\usepackage{amsthm}
\usepackage[cal=euler, scr=boondoxo]{}
\usepackage{microtype}

\usepackage{numprint}

\makeatletter
\def \leq {\leqslant}
\def \le {\leq}
\def \geq {\geqslant}

\def \ge {\geq}

\def\R{\mathbb R}

\def\N{\mathbb N}
\def\C{\mathscr{C}}
\def\D{\mathcal D}

\def \M {\mathcal{M}}

\def\g{\gamma}
\def \ds {\displaystyle}
\def\Gg  {\bm{G}_{\g}}

\def \d {\mathrm{d}}

\def \Q {\mathcal{Q}}

 \def \w {\bm{w}}

\def \X {\mathbb{X}}

\def \ind {\mathbf{1}}
\numberwithin{equation}{section}

\newcommand{\emptylabel}[1]{}

\newcommand{\abs}[1]{\left\vert#1\right\vert}

\def\:{\colon}

\def\C{\mathbb{C}}
\def\N{\mathbb{N}}
\def\P{\mathbb{P}}
\def\R{\mathbb{R}}

\def\D{\mathcal{D}}

\DeclareMathOperator{\sgn}{sgn}

\def\d{\,\mathrm{d}}
\def\dx{\d x}

\def\dy{\d y}
\def\dz{\d z}

\def\M{\mathcal M} 
\def\P{\mathcal P}
\def\p{\partial}

\def\ir{\int_{\R}}

\numberwithin{equation}{section}

\usepackage{natbib}
\newtheorem{theo}{Theorem}[section]
\newtheorem{cor}[theo]{Corollary}
\newtheorem{rmk}[theo]{Remark}
\newtheorem{lem}[theo]{Lemma}
\newtheorem{prp}[theo]{Proposition}

\newtheorem{rem}[theo]{Remark}
\newtheorem{defi}[theo]{Definition}
\theoremstyle{example}

\newcommand{\vertiii}[1]{{\left\vert\kern-0.25ex\left\vert\kern-0.25ex\left\vert #1  
    \right\vert\kern-0.25ex\right\vert\kern-0.25ex\right\vert}}                      
\newcommand{\verti}[1]{{\left\vert\kern-0.25ex\left\vert\kern-0.25ex\left\vert #1    
    \right\vert\kern-0.25ex\right\vert\kern-0.25ex\right\vert}}						 

\title[One-dimensional inelastic Boltzmann equation]{One-dimensional
  inelastic Boltzmann equation: Regularity \& uniqueness of
  self-similar profiles for moderately hard potentials}

\def\theauthor{R. Alonso, V. Bagland, J. A. Ca\~{n}izo, B. Lods, S. Throm}

\hypersetup{pdfauthor={\theauthor}}

\vspace{1cm}

\author{R. Alonso}
\address{Division of Arts \& Sciences, Texas A\&M University at Qatar, Education City, Doha, Qatar.}
\email{ricardo.alonso@qatar.tamu.edu}

\author{V. Bagland}
\address{Universit\'{e} Clermont Auvergne, LMBP, UMR 6620 - CNRS,  Campus des C\'ezeaux, 3, place Vasarely, TSA 60026, CS 60026, F-63178 Aubi\`ere Cedex, France.}
\email{Veronique.Bagland@uca.fr}

\author{J. A. Ca\~{n}izo}
\address{Departamento de Matem\'{a}tica Aplicada \& IMAG, Universidad de Granada, Avenida de Fuentenueva S/N, 18071 Granada, Spain.}
\email{canizo@ugr.es}

\author{B. Lods}
\address{Universit\`{a} degli Studi di Torino \& Collegio Carlo Alberto, Department of Economics, Social Sciences, Applied Mathematics and Statistics ``ESOMAS'', Corso Unione Sovietica, 218/bis, 10134 Torino, Italy.}
\email{bertrand.lods@unito.it}

\author{S. Throm}
\address{{Ume\aa} University,  Department of Mathematics and Mathematical Statistics,  901 87 Ume\aa, Sweden }
\email{sebastian.throm@umu.se}

\date{}

\begin{document}

\maketitle

\begin{abstract}
  We prove uniqueness of self-similar profiles for the one-dimensional
  inelastic Boltzmann equation with moderately hard potentials, that
  is with collision kernel of the form $|\cdot|^{\g}$ for $\g >0$
  small enough (explicitly quantified). Our result provides the first
  uniqueness statement for self-similar profiles of inelastic
  Boltzmann models allowing for strong inelasticity besides the
  explicitly solvable case of Maxwell interactions (corresponding to
  $\g=0$). Our approach relies on a perturbation argument from the
  corresponding Maxwell model through a careful study of the
  associated linearised operator. In particular, a part of the paper
  is devoted to the trend to equilibrium for the Maxwell model in
  suitable weighted Sobolev spaces, an extension of results which are
  known to hold in weaker topologies. Our results can be seen as a
  first step towards a full proof, in the one-dimensional setting, of
  a conjecture in \cite{ernst} regarding the determination of the
  long-time behaviour of solutions to inelastic Boltzmann equation.
\end{abstract}

\tableofcontents

\section{Introduction}
\label{sec:intro}
 
We treat the one-dimensional Boltzmann equation for moderately hard
potentials, proving regularity and uniqueness of equilibrium
self-similar profiles.  In the process we also contribute to
generalising results related to the asymptotic convergence for the
time-dependent Maxwell model as well by better describing such
convergence in standard Lebesgue and Sobolev spaces.

\subsection{One-dimensional inelastic Boltzmann equation} Inelastic
models for granular matter are ubiquitous in nature and rapid granular
flows are usually described by a suitable modification of the
Boltzmann equation, see \cite{garzo,vill}. Inelastic interactions are
characterised, at the microscopic level, by the continuous dissipation
of the kinetic energy for the system. Typically, in the usual $3D$
physical situation, two particles with velocities
$(v,v_{\star}) \in \R^{3}\times \R^{3}$ interact and, due to inelastic
collision, their respective velocities $v'$ and $v_{\star}'$ after
collision are such that momentum is conserved
$$v+v_{\star}=v'+v_{\star}'$$
but kinetic energy is dissipated at the moment of impact:
$$|v'|^{2}+|v_{\star}'|^{2}  \leq |v|^{2}+|v_{\star}|^{2}.$$
Often the dissipation of kinetic energy is measured in terms of a
single parameter, usually called the \emph{restitution coefficient},
which is the ratio between the magnitude of the normal component of
the relative velocity after and before collision. This coefficient
$e \in [0,1]$ may depend on the relative velocity and encode all the
physical features. It holds then
$$\langle v'-v'_{\ast},n\rangle=-e\,\langle v-v_{\star},n\rangle\,$$
where $n \in \mathbb{S}^{2}$ stands for the unit vector that points
from the $v$-particle center to the $v_{\star}$-particle center at the
moment of impact. Here above, $\langle\cdot,\cdot\rangle$ denotes the
Euclidean inner product in $\R^{3}$.

For one-dimensional interactions, we will rather denote by $x,y$ the
velocities before collision and $x',y'$ those after collision and the
collision mechanism is then described more easily as
$$x'=ax+(1-a)y, \qquad y'=(1-a)x+ay, \qquad a \in [\tfrac{1}{2},1]$$
where  the parameter $a$ describes now the intensity of inelasticity and one checks indeed that $x'+y'=x+y$ whereas
\begin{equation}
  \label{eq:enxx}
  |x'|^{2}+|y'|^{2}-|x|^{2}-|y|^{2} = - 2ab |x-y|^{2} \leq 0,
  \qquad b=1-a,
\end{equation}
where we used that $a^{2}+b^{2}-1=-2ab$ if $b=1-a$. In this case, the
inelastic Boltzmann equation is given by the following, as proposed
in \cite{BK}:
\begin{equation}\label{Intro-e1}
\partial_{s}f(s,x) = \Q_{\g}(f,f)(s,x),\qquad\qquad (s,x)\in (0,\infty)\times\R\,,
\end{equation}
with given initial condition $f(0,x)=f_0(x)\geq0$.  The interaction operator is defined as
\begin{equation}\label{Intro-e1.5}
\Q_{\g}(f,f)(x) = \int_{\R}f(x-ay)f(x+(1-a)y)|y|^{\gamma}\d y - f(x)\int_{\R}f(x-y)|y|^{\gamma}\d y
\end{equation}
for fixed $\gamma\geq0$ and $a\in[\frac{1}{2},1]$.  Notice that the model \eqref{Intro-e1} conserves mass and momentum
$$\ds\int_{\R}\Q_\g(f,f)(x)\dx =\int_{\R}\Q_{\g}(f,f)x\dx= 0\,, 
$$
but dissipates energy since
\begin{multline}\label{eq:ener}\int_{\R}\Q_{\g}(f,f)(x)|x|^{2}\d x\\
=\frac{1}{2}\int_{\R^{2}}f(u)f(v)|u'-v'|^{\g}\left[|u'|^{2}+|v'|^{2}-|u|^{2}-|v|^{2}\right]\d u\d v		\\
=-ab\int_{\R \times \R}f(u)f(v)|u-v|^{\g+2}\d u\d v, \qquad \qquad b=1-a,
\end{multline}
where we used a change of variable $u=x-ay,v=x+by$ and a symmetry argument to get the first identity while we used \eqref{eq:enxx} to establish the second one. This implies that, for any nonnegative initial datum $f_{0}$ and  any solution $f(s,z)$ to \eqref{Intro-e1}, it holds
$$\dfrac{\d}{\d s}\int_{\R}f(s,z)\left(\begin{array}{cc}1 \\ z \end{array}\right)\d z=\left(\begin{array}{cc}0 \\ 0\end{array}\right)$$
while
\begin{equation}\label{eq:dissiab}
\dfrac{\d}{\d s}\int_{\R}f(s,z)z^{2}\d z=-ab\int_{\R\times \R}f(s,u)f(s,v)|u-v|^{\g+2}\d u\d v \leq 0.\end{equation}
One sees therefore that the single parameter $a$ (through the product $ab=a(1-a)$) measures the strength of energy dissipation. The case $a=1$ represents a purely elastic interaction which, in one dimension, is described by no interaction at all, or simply $\Q_\gamma(f,f)=0$.  The other case $a=\frac{1}{2}$ is the case of extreme inelasticity or the sticky particle case; that is, after interaction the particles remain attached yet considered distinct so that no global mass is lost. From now, in all this manuscript, we consider this latter case
$$a=\frac{1}{2}$$ 
but wish to point out that the general case $a\in\left(\frac{1}{2},1\right)$ can be treated in a similar fashion.\\

The above dissipation of kinetic energy (together with mass and
momentum conservation) leads to a natural equilibrium given by the
distribution that accumulates all the initial mass, say $m_0$, at the
initial system's bulk momentum $z_0$:
$$\Q_{\g}(F,F)= 0 \implies \exists \ m_{0} \geq0, z_{0} \in \R \text{ such that } F=m_0\,\delta_{z_0}.$$ 
Such a degenerate solution is of course expected to attract all solutions to \eqref{Intro-e1} but, as for the multi-dimensional model, one expects that, before reaching the degenerate state, solutions behave according to some universal profile as an intermediate asymptotics. 

More precisely, it is believed that a more accurate description can be derived introducing a rescaling of the form 
\begin{equation*}
V_{\g}(s)\,g(t(s),x) = f(s,z)\,,\qquad \qquad x=V_{\g}(s)\,z\,,
\end{equation*}      
where the rescaling functions are given by 
\begin{equation*}
V_{\g}(s) =\begin{cases}
(1+c \,\gamma\, s)^{\frac1\gamma}\,\qquad &\text{ if } \g \in (0,1),\\
e^{cs}\,\qquad &\text{ if } \g=0\,,\end{cases} \quad \text{and} \quad  t(s)=t_{\g}(s)=\begin{cases} \frac{1}{c\gamma}\log(1+c\,\gamma\, s)\,, \qquad &\text{ if } \g \in (0,1),\\
s \,,&\text{ if } \g=0\,.
\end{cases}
\end{equation*}
We refer to \cite{bobcerc1} for a study of the Maxwell model $(\g=0)$
and \cite{ABCL} for the hard potential model $(\g \in (0,1))$. For
$\gamma >0$, this rescaling is useful for any $c > 0$ and we are free
to choose this parameter as we see fit; while for $\gamma = 0$, the
only useful choice is $c=\frac{1}{4}$. We will come back to this
later. Straightforward computations show then that, if $f(s,z)$ is a
solution to \eqref{Intro-e1}, the solution $g(s,z)$ satisfies
\begin{equation}\label{Intro-e2}
\partial_{t}g(t,x) + c\,\partial_{x}\big(x\,g(t,x)\big) = \Q_{\gamma}(g,g)(t,x),\qquad\qquad (t,x)\in (0,\infty)\times\R\,,
\end{equation} 
with $g(0,x) = f_0(x)$ so that
\begin{equation}\label{eq:conse}
  \int_{\R}g(t,x)\d x=\int_{\R}f_{0}(x)\d x=1, \qquad \int_{\R}xg(t,x)\d x=\int_{\R}xf_{0}(x)\d x=0, \qquad \forall t \geq0\end{equation}
due to the conservation of mass and momentum induced by both the drift term and the collision operator $\Q_{\g}$. 
Since, formally,
$$ \int_{\R}\partial_{x}\big(x\,g(t,x)\big) |x|^{2}\d x = -2\int_{\R}g(t,x)\,|x|^{2}\d x\,,$$    
one can interpret the rescaling as an artificial way to add energy into the system, the bigger the $c>0$ the more energy per time unit is added.  Thus, the rescaling has the same effect of adding a background linear ``anti-friction'' with constant $c>0$.  However, unless in the special case $\g=0$, evolution of the kinetic energy along solutions to \eqref{Intro-e2} is not given in closed form. Namely, if 
$$M_{2}(g(t))=\int_{\R}g(t,x)x^{2}\dx$$
one sees from \eqref{Intro-e2} and \eqref{eq:ener} that
\begin{equation}\label{eq:evolenerg}\begin{split}
\dfrac{\d}{\d t}M_{2}(g(t))-2cM_{2}(g(t))&=\int_{\R}\Q_{\g}(g,g)(t,x)x^{2}\d x\\
&=-\frac{1}{4}\int_{\R\times\R}g(t,x)g(t,y)|x-y|^{\g+2}\d x\d y\end{split}\end{equation}
so that the evolution of the second moment of $g(t)$ depends on the evolution of moments of order $\g+2$. The situation is very different in the case $\g=0$ and this basic observation will play a crucial role in our analysis.\medskip

It is important to observe that problems \eqref{Intro-e1} and
\eqref{Intro-e2} are related by a simple rescaling, so that knowledge
of properties of one of them is transferable to the other. Equation
\eqref{Intro-e2} is referred to as the \emph{self-similar
  equation}. For $\gamma > 0$ and any $c > 0$, it has at least one
non-trivial equilibrium with positive energy \citep{ABCL}, satisfying the
equation
\begin{equation}\label{Intro-e3}
c\,\partial_{x}\big(x\,\bm{G}(x)\big) = \Q_{\gamma}(\bm{G},\bm{G})(x)\,.
\end{equation} 
For $\gamma = 0$, there is a non-trivial equilibrium with positive
energy only if $c = \frac{1}{4}$. The equilibria are known as
\emph{self-similar profiles}. Of course, $\bm{G}$ depends on the
choice of $c>0$; however, they are all related by a simple
rescaling. Moreover, the fact the $\bm{G}$ is a regular (smooth)
function is helpful for the technical analysis, for example to have a
standard linearisation referent.  Indeed, in this document we prove
regularity properties for $\bm{G}:=\bm{G}_{\gamma}$ and answer the
uniqueness question for the problem \eqref{Intro-e3}, at least in the
context of moderately hard potentials, that is, our results will be
valid for relatively small positive $\gamma$.  \medskip

The question of uniqueness for self-similar profiles of the model
\eqref{Intro-e3} is notoriously difficult for $\gamma > 0$. Since the
model conserves mass and momentum, a uniqueness result has to take
into account this fact. In other words, steady states should be unique
in a space with fixed mass and momentum. In the case of Maxwell
interactions ($\gamma = 0$ and $c=\frac{1}{4}$), energy is additionally
conserved, and it is known that self similar profiles are unique when
mass, momentum and energy are fixed \citep{bobcerc1}. This case is
less technical, and somehow critical, since the self-similar rescaling
is uniquely determined (by the initial mass and energy) as opposed to
the case $\gamma>0$ where one can choose any $c>0$ to perform the
rescaling.  For the Maxwell case the rescaling leads to the
conservation of energy which is an important help in the analysis,
together with a computable spectral gap for the linearised equation.
We refer to \cite{MR2355628} for a good account of the theory of the
Maxwell model in one and multiple dimensions.  \medskip

There is another type of uniqueness result.  In the context of
$3D$-dissipative particles it is possible to define a weakly inelastic
regime.  A big difference between the one-dimensional problem and the
three-dimensional problem is that in the latter the elastic limit of
the model is the classical Boltzmann equation whereas in the
one-dimensional problem the elastic limit $a=1$ is simply $\partial_t
f =0$.  This is the reason one can study weakly inelastic systems as a
perturbation of the classical Boltzmann equation in several dimensions
with powerful tools such as entropy-entropy dissipation methods
leading to a uniqueness result in this context, see
\cite{MM09,CMS,AL}.  And yet, the same strategy completely fails in
the $1D$-dissipative model where such tools are not available.
\medskip

Our analysis for small positive $\gamma$ will be also perturbative
taking as reference the one-dimensional Maxwell sticky particle model;
that is, our result covers the most extreme case of inelasticity
providing a strong indication that the steady inelastic self-similar
profiles should be unique in full generality, for all collision
kernels and degrees of inelasticity.  This perturbation is highly
singular in two respects: first, the Maxwell model conserves energy
{in self-similar variables} which is not the case for $\gamma>0$.
This is a major difficulty since the spectral gap of the linearised
Maxwell model depends crucially on this conservation law.  Second, the
tail of the self-similar profiles are completely different, for
Maxwell models the profile enjoys some few statistical moments only,
whereas for hard potentials the profile has exponential tails.
Fortunately, steady states will enjoy regularity for all
$\gamma\geq0$, a property that will be also proved in this paper.

 \subsection{The problem at stake}
 
The main concern of the present paper is, as said, the uniqueness of the steady solutions $\Gg$ to the equation
\begin{equation}
\label{eq:steadyg}
c\partial_{x}(x\Gg ) = \Q_{\g}(\Gg ,\Gg )\,,
\end{equation}
with unit mass and zero momentum where, for $\g \in [0,1]$ ,
$\Q_{\g}(f,g)$ reads in its weak form as 
\begin{equation}\label{eq:weakgamma}
\int_{\R}\Q_{\g}(f,g)(x)\varphi(x)\d x=\frac{1}{2}\int_{\R^{2}}f(x)g(y)\Delta\varphi(x,y)|x-y|^{\g}\d x\d y\end{equation}
with
$$\Delta \varphi(x,y)=2\varphi\left(\frac{x+y}{2}\right)-\varphi(x)-\varphi(y)$$
for any suitable test function $\varphi$.  {We can split $\Q_{\g}(f,g)$ into positive and negative parts
$$\Q_{\g}(f,g)=\Q^{+}_{\g}(f,g)-\Q^{-}_{\g}(f,g)$$
 where, in weak form,
$$\int_{\R}\Q_{\g}^{+}(f,g)(x)\varphi(x)\d x=\int_{\R^{2}}f(x)g(y)\varphi\left(\frac{x+y}{2}\right)|x-y|^{\g}\d x\d y$$
and
$$\int_{\R}\Q^{-}_{\g}(f,g)\varphi(x)\d x=\frac{1}{2}\int_{\R^{2}}f(x)g(y)\left(\varphi(x)+\varphi(y)\right)|x-y|^{\g}\d x\d y.$$}
The existence of a suitable solution $\Gg$ to \eqref{eq:steadyg} with
finite moments up to third order has been obtained in \cite{ABCL}. {We
  will always assume here that $\Gg \in L^1_3(\R)$ is nonnegative and
  satisfies
\begin{equation}\label{eq:mom}
\int_{\R}\Gg(x)\dx=1, \qquad \int_{\R}\Gg(x)x \dx=0, \qquad \g\in [0,1].\end{equation}
Notice that the energy 
$$\int_{\R}x^{2}\Gg(x)\d x=M_{2}(\Gg)$$
is not known \emph{a priori} since the non-conservation of kinetic
energy precludes any simple selection mechanism to determine it at
equilibrium.

The crucial point of our analysis lies in the fact that this problem
has a very well-understood answer in the degenerate case in which
$\g=0$. Indeed, in such a case, many computations are explicit and,
for instance, the evolution of kinetic energy for equation
\eqref{Intro-e2} is given in closed form as, according to
\eqref{eq:evolenerg}
\begin{equation*}
\dfrac{\d}{\d t}M_{2}(g(t))-2cM_{2}(g(t))=-\frac{1}{4}\int_{\R\times\R}g(t,x)g(t,y)|x-y|^{2}\d x\d y=-\frac{1}{2}M_{2}(g(t))
\end{equation*}
where we used \eqref{eq:conse} to compute the contribution of the collision operator. In particular, for $\g=0$, energy is conserved if and only if $c=\frac{1}{4}$ and, in such a case, we can prescribe the energy of the kinetic energy of the steady state $\bm{G}_{0}.$ \medskip

For this reason, in the sequel we will always assume that
$$c=\frac{1}{4}.$$

\medskip

Another important property of the Maxwell molecules case is that, due to explicit computations in Fourier variables, solutions to \eqref{eq:steadyg} are actually explicit in this case (and in this case only). More precisely, one has the following
\begin{theo}[\cite{bobcerc1}]\label{theo:bob} Let $\M_{2}(\R)$ denote the set of real Borel measures on $\R$ with finite second order moment. Any $\mu \in \M_{2}(\R)$ such that 
\begin{equation}\label{eq:weakmes}
-\frac{1}{4}\int_{\R}x\varphi'(x)\mu(\d x)=\frac{1}{2}\int_{\R\times\R}\Delta \varphi(x,y)\mu(\dx)\mu(\dy) \qquad \forall \varphi \in \mathcal{C}_{b}^{1}(\R)\end{equation}
and satisfying 
$$\int_{\R}\mu(\dx)=1,\qquad \int_{\R}xg(x)\mu(\dx)=0, \qquad \int_{\R}x^{2}\mu(\d x)=\frac{1}{\lambda^{2}} >0$$
is of the form
$$\mu(\dx)=H_{\lambda}(x)\d x=\lambda\bm{H}(\lambda x)\d x$$
where
\begin{equation*}
\bm{H}(x) = \frac{2}{\pi (1+x^2)^2}\,\qquad x \in \R.
\end{equation*}
\end{theo}
Here above of course \eqref{eq:weakmes} is a measure-valued version of the steady equation \eqref{eq:steadyg} and in particular one sees that $\bm{H}$ is the unique solution to 
$$\frac{1}{4}\partial_{x}\left(x\bm{H}(x)\right)=\Q_{0}(\bm{H},\bm{H})(x)$$
with unit mass and energy and zero momentum. The existence and uniqueness has been obtained, through Fourier transform methods, in \cite{bobcerc1} and extended to measure solutions in \cite{MR2355628}. 

We introduce the following set of equilibrium solutions
$$\mathscr{E}_{\g}=\left\{\Gg \in L^{1}_{3}(\R)\;;\;\Gg \text{ satisfying \eqref{eq:steadyg} and \eqref{eq:mom}}\,\right\}$$
for any $\g \in [0,1)$. The above Theorem ensures that elements of $\mathscr{E}_{0}$ are entirely described by their kinetic energy, i.e., given $E >0$, 
$$\left\{g \in \mathscr{E}_{0} \text{ and } M_{2}(g)=E\right\} \text{ reduces to a singleton}.$$
The main objective of the present contribution is to prove that, for moderately hard potentials, the situation is similar and more precisely, our main result can be summarized as
\begin{theo}\label{theo:mainUnique} There exists some explicit $\g^{\dagger} \in (0,1)$ such that, if $\g \in (0,\g^{\dagger})$ then $\mathscr{E}_{\g}$ reduces to a singleton.
\end{theo} 
Notice here the contrast between the case $\g >0$ where
$\mathscr{E}_{\g}$ is a singleton whereas, for $\g=0$,
$\mathscr{E}_{0}$ is an infinite one-dimensional set parametrised by the energy
of the steady solution. As we will see, it happens that, in the limit
$\g \to 0$, the steady equation \eqref{eq:steadyg} (with
$c=\frac{1}{4}$) somehow \emph{selects} the energy.

Before describing in details the main steps behind the proof of Theorem \ref{theo:mainUnique}, we need to introduce the notations that will be used in all the sequel.

\subsection{Notations}
\label{sec:notation}

For $s \in \R $ and $ p\geq 1$, we define the Lebesgue space
$L^{p}_{s}(\R)$ through the norm
$$\displaystyle \|f\|_{L^p_{s}} := \left(\int_{\R} \big|f(x)\big|^p \, 
(1+|x|)^{sp} \, \d x\right)^{1/p}, \qquad L^{p}_{s}(\R) :=\Big\{f :\R \to \R\;;\,\|f\|_{L^{p}_{s}} < \infty\Big\}\, .$$
 More generally, for any weight function $\varpi : \R \to \R^{+}$, we define, for any $p \geq 1$,
$$L^{p}(\varpi) :=\Big\{f : \R \to \R\;;\,\|f\|_{L^{p}(\varpi)}^{p}:=\int_{\R}\big|f\big|^{p}\,\varpi^p\,\d x  < \infty\Big\}\,.$$
A frequent choice will be the weight function
\begin{equation}\label{eq:weight}
\w_{s}(x)=\left(1+|x|\right)^{s}, \qquad s\geq 0, \qquad x \in \R. 
\end{equation}
With this notation, one can write for example $L^{p}_{s}(\R)=L^{p}(\w_s)$, for $p \geq 1,\,s \geq 0$.
\smallskip
\noindent
We define the weighted Sobolev spaces by 
$$W^{k,p}(\varpi) :=\Big\{f \in L^{p}(\varpi)\;;\;\partial_{x}^{\ell}f \in L^{p}(\varpi) \;\forall\, 0\le \ell \leq k\Big\}\,,\quad\text{with}\quad k \in \N\,,$$
with the standard norm
$$ \|f\|_{W^{k,p}(\varpi)} := \bigg( \sum_{\ell=0}^{k} 
\int_{\R} \big| \partial^{\ell}_x f(x)\big|^p \, \varpi(x)^p\, 
\d x\bigg)^{\frac{1}{p}}.$$
For $p=2$, we will simply write $H^{k}(\varpi)=W^{k,2}(\varpi)$, $k \in \N$. 
For general $s\ge 0$, the Sobolev space $H^s(\R)$ is defined thanks to the Fourier transform. 
$$H^s(\R):=\Big\{f \in L^{2}(\R)\;;\; \int_{\R} (1+|\xi|^2)^s |\hat{f}(\xi)|^2\d \xi<\infty\Big\}\,, $$
where 
$$ \hat{f}(\xi) =\int_{\R} f(x) e^{-ix\xi}\d x. $$
This space is endowed with the standard norm 
$$\|f\|_{H^s(\R)}=\left(\int_{\R} (1+|\xi|^2)^s |\hat{f}(\xi)|^2\d \xi\right)^{\frac{1}{2}}.$$
We shall also use an important shorthand for the moments of order $s \in \R$ of  measurable mapping $f : \R\to {\R}$, 
$$M_{s}(f) :=\int_{\R}f(x)\,  |x|^s \d x. $$
For $k \geq 0$ we define the norm (in Fourier variables)
\begin{equation}\label{eq:fourier:norm:1}
 \vertiii{\psi}_{k} := \sup_{\xi \in \R \setminus \{0\}} \frac{|\psi(\xi)|}{|\xi|^k}
\end{equation}
on the vector space of continuous functions
$\psi \: \R \to \C$ such that $\xi \mapsto \psi(\xi) / |\xi|^k$ is a
bounded continuous function (with a limit at $\xi = 0$). This norm
makes this vector space a Banach space. {Moreover, for $k\geq 0$ and $p\in(1,\infty)$, we define
\begin{equation}\label{def:normkp}
 \vertiii{\psi}_{k,p}^p := \ir \frac{|\psi(\xi)|^p}{|\xi|^{kp}} \d \xi,
\end{equation}
which makes sense if $|\psi(\xi)| \leq \min\{1 ,C|\xi|^3\}$ for some $C>0$  and
$\frac{1}{p} < k < 3 + \frac{1}{p}$.}

For $k>2$, we also define the following space of measures
{\begin{equation}\label{eq:X0:measure}
X_{0}:=  \left\{ \mu\in {\mathcal M}_{k}(\R) \;\Bigg| \Bigg. \begin{array}{l}
  \displaystyle \int_{\R} \mu(\d x) = \int_{\R} x\, \mu(\d x) = \int_{\R} x^2 \,\mu(\d x) = 0
\end{array}\right\}. 
\end{equation}
  Here above, ${\mathcal M}_{k}(\R)$ denotes the set of real Borel measures on $\R$ with finite total variation of order $k$ that are satisfying
  $$\int_{\R}  \w_{k}(x) \,|\mu|(\dx)<\infty.$$}
Then, by abuse of notation, for $\mu\in X_0$ and $0<k<3$, we define the norm 
\begin{equation}\label{eq:fourier:norm:2}
\vertiii{\mu}_{k} :=  \sup_{\xi\in\R\setminus \{0\}} \frac{|\hat{\mu}(\xi)|}{|\xi|^{k}}. 
\end{equation} 
{By \cite[Proposition 2.6]{MR2355628}, for any $0< k< 3$, if $\mu\in X_0$ then  {$\vertiii{\mu}_{k}$} is finite. Endowed with the norm  {$\vertiii{\cdot}_{k}$} with $2< k<3,$
the space $X_{0}$ is a Banach space, see Proposition 2.7 in \cite{MR2355628}.}

\subsection{Strategy and main intermediate
  results}\label{Sec:strategy} The idea to prove our main result
Theorem \ref{theo:mainUnique} is to adopt a \emph{perturbative
  approach} and to fully exploit the knowledge of the limiting case
$\g=0.$ This explains in particular why our result is valid for
\emph{moderately hard potentials} $\g \simeq 0.$ More precisely,
inspired by similar ideas developed in the context of the Smoluchowski
equation (see \cite{CanizoThrom} for a recent account and a source of
inspiration for the present work), the first step in our proof is to
show that, in some weak sense to be determined,
$$\lim_{\g \to 0}\Gg=\bm{G}_{0}$$
where $\bm{G}_{0}$ is a steady solution in the Maxwell case, i.e. a
solution to \eqref{eq:weakmes}, and $\bm{G}_{\gamma}$ is any steady
solution in $\mathscr{E}_{\gamma}$. Of course, such a limiting process
is very singular as for instance can be understood from the following
fundamental property of steady solutions.
\begin{prp}\label{prop:allmoments} For any $\g \in (0,1)$, one has
$$M_{k}(\Gg):=\int_{\R}\Gg(x)|x|^{k}\d x < \infty \qquad \text{ for any } k \geq0$$
whereas
$$M_{k}(\bm{G}_{0}) < \infty \iff k \in (-1,3).$$
\end{prp}
Of course, because we are interested here in the behaviour of $\Gg$ for $\g \to 0$, the above Proposition will play only a marginal role for our analysis (we refer the reader to Appendix \ref{app:C3} for a full proof) but it highlights the fact that the limiting process we are interested in is highly singular. In particular, since steady solutions to \eqref{eq:steadyg} exist with any energy, a first important step is to derive the correct limiting temperature
$$\lim_{\g\to0}M_{2}(\Gg)=E_{0}=M_{2}(\bm{G}_{0})$$
since that single parameter, thanks to Theorem \ref{theo:bob}, would allow to completely characterize $\bm{G}_{0}$.
This first step in our proof can be summarized in the following which
characterises the limit temperature for $\g\to 0$ and provides
additional moment and $L^2$ bounds. Its proof will be given in
Section~\ref{Sec:limit:temp}. We recall that we are always setting
$c=\frac{1}{4}$ in Eq. \eqref{eq:steadyg}: \label{Intro-e3}.
 \begin{theo}\phantomsection\label{theo:Unique}
For any $\varphi \in \mathcal{C}_{b}(\R)$ and any choice of steady states
$\Gg \in \mathscr{E}_{\g}$ for $\gamma > 0$ one has 
$$\lim_{\g\to0^{+}}\int_{\R}\Gg(x)\varphi(x)\dx=\int_{\R}\bm{G}_{0}(x)\varphi(x)\dx, \qquad \bm{G}_{0}(x)=\frac{2\lambda_{0}}{\pi\left(1+\lambda_{0}^{2}x^{2}\right)^{2}}, \qquad x \in \R$$
with 
\begin{equation*}
 \lambda_{0}:=\exp\left(A_{0}\right)\qquad \text{and}\qquad A_{0}:=\frac{1}{2}\int_{\R}\int_{\R}\bm{H}(x)\bm{H}(y)|x-y|^{2}\log|x-y|\dx\dy >0.
\end{equation*}
 Moreover, for any $\delta \in (0,3)$ there exist $C_{0} >0$ and $\g_{\star}  \in (0,1)$, both depending on $\delta$ such that
\begin{equation}\label{eq:estimGg}
\|\Gg\|_{L^2} \leq C_{0}, \qquad \qquad M_{k}(\Gg) \leq C_{0},  \qquad \forall \, \g \in [0,\g_{\star} ), \qquad k \in (0,3-\delta).\end{equation}
 \end{theo}
In fact it will be shown in Lemma~\ref{rmk:kern} that $A_0=\log2+\frac{1}{2}$ and thus $\lambda_0=2\sqrt{e}$.

Notice that, as documented in \cite{ABCL}, the derivation of $L^{2}$
bounds for solutions to the $1D$-Boltzmann equation is not an easy
task. This comes from the lack of regularizing effects induced by the
collision operator $\Q_{\g}$ in dimension $d=1$. Recall indeed that a
celebrated result in \cite{lions} asserts that, for very smooth
collision kernels, the Boltzmann collision operator (for elastic
interactions and in dimension $d$) maps, roughly speaking,
$L^{2}(\R^{d}) \times L^{2}(\R^{d})$ in the Sobolev space
$H^{\frac{d-1}{2}}(\R^{d})$, i.e. the collision operator induces a
gain of $\frac{d-1}{2}$ (fractional) derivative. One sees therefore
that no regularisation effect is expected in dimension $d=1$ whereas
gain of regularity is the fundamental tool for the derivation of
$L^{p}$-estimates for solutions to the Boltzmann equation (see
\cite{MoVi}). This simple heuristic consideration is also confirmed in
the related case of the Smoluchowski equation for which derivation of
suitable $L^{p}$-estimates $p >1$ is a notoriously difficult problem
(see \cite{bana, CanizoThrom}).

In the present paper, the derivation of $L^{2}$-bounds (uniformly with
respect to $\g > 0$) is deduced from quite technical arguments,
specific to the study of equilibrium solutions, and crucially exploits
the convergence of $\Gg$ towards $\bm{G}_{0}$ together with the fact
that $\bm{G}_{0}$ is completly explicit. In particular, our argument
does not seem to work for general solutions to \eqref{Intro-e2}.

A second step in our proof is then to be able to quantify the above convergence of $\Gg$ towards $\bm{G}_{0}$ and, in particular, to exploit the fact that $\bm{G}_{0}$ is a \emph{stable equilibrium solution} to \eqref{eq:steadyg} for $\g=0.$ To prove this, we need first to revisit several of the known results concerning the Maxwell molecules case and in particular the long time behaviour of the solutions to the evolution problem
\begin{equation}\label{introMax}
\partial_{t}g(t,x) + \frac{1}{4}\partial_{x}\left(xg(t,x)\right)=\Q_{0}(g,g)(t,x), \qquad x \in \R, t \geq 0\end{equation}
and show that any (suitably normalized) solution to this equation converges exponential fast towards $\bm{G}_{0}$ as $t \to \infty$ in Fourier norms $\vertiii{\cdot}_{k}$ and $\vertiii{\cdot}_{k,p}$ (see Theorem \ref{k-norm-cvgce}). Thanks to new regularity results regarding the above equation, we can also extend such an exponential convergence to more tractable Sobolev spaces. This careful analysis of the Maxwell equation together with new regularity bounds for the self-similar profile $\Gg$ allows to derive  the following stability estimate for self-similar profiles which is also a main ingredient in the proof of Theorem~\ref{theo:mainUnique} and whose proof will be given in Section~\ref{sec:upgrade}.
\begin{theo}[\textit{\textbf{Stability of profiles}}]\phantomsection\label{theo:sta} Let $2<a<3$. There exist $\g_\star\in(0,1)$ and an \emph{explicit} function $\overline{\eta}=\overline{\eta}(\gamma)$ depending on $a$, with $\lim_{\g\to0^+}\overline{\eta}(\gamma)= 0$, such that, for any $\g\in(0,\g_\star)$, for any $\Gg\in\mathscr{E}_{\g}$,   
$$ \|\Gg-\bm{G}_{0}\|_{L^1(\w_{a})} \le \overline{\eta}( \gamma).$$
\end{theo} 
We point out here that some suitable smoothness estimates for $\Gg$ in Sobolev spaces (uniformly with respect to $\g$) are required for the proof of Theorem \ref{theo:sta} (see Lemma \ref{lem:stabil}) and, as explained already  for $L^{2}$ estimates, the lack of regularization effect for the operator $\Q_{\g}$ induces severe technical obstacles in the proof of such Sobolev estimates for $\Gg$.

A final important tool for the proof of Theorem \ref{theo:mainUnique} is the \emph{quantitative stability} of the steady state $\bm{G}_{0}$ of \eqref{introMax} in the space $L^{1}(\w_{a})$. More precisely, let us introduce the following   linearisation $\mathscr{L}_{0}$ on $L^{1}(\w_{a})$.
\begin{defi}\label{def:spaces} We introduce, for $2 < a <3$ the functional spaces 
\begin{multline*}
\mathbb{X}_{a}=L^{1}(\w_{a}), \qquad {\mathbb{Y}}_{a}=\left\{f \in \mathbb{X}_{a}\;;\;\int_{\R}f(x)\d x=\int_{\R}f(x) x\d x=0\right\}\\
\text{ and }\qquad {\mathbb{Y}}_{a}^{0}=\left\{f \in {\mathbb{Y}_{a}}\;;\;\int_{\R}f(x) x^{2}\d x=0\right\}
\end{multline*}
with $\|\cdot\|_{\X_{a}}$ denoting the norm in $\X_{a}$.
We define then
$$\mathscr{L}_{0}:\mathscr{D}(\mathscr{L}_{0}) \subset \mathbb{X}_{a}\to \mathbb{X}_{a}$$
by
$$\mathscr{L}_{0}(h)=\Q_{0}(h,\bm{G}_{0})+\Q_{0}(\bm{G}_{0},h)-\frac{1}{4}\partial_{x}(xh), \qquad \forall h \in \mathscr{D}(\mathscr{L}_{0})$$
with 
$$\mathscr{D}(\mathscr{L}_{0})=\left\{f \in \mathbb{X}_{a}\;;\;\partial_{x}(x f) \in \mathbb{X}_{a}\right\}$$
and $\bm{G}_{0}$ given in Theorem~\ref{theo:Unique}.
\end{defi}

The stability of the profile $\bm{G}_{0}$ for the Maxwell molecules case is then established through the following result which provides a spectral gap for the linear operator $\mathscr{L}_{0}$ in the space $\mathbb{Y}_{a}^{0}$. The proof will be given in Section~\ref{sec:sg-small}.

\begin{prp}\label{restrict0}
  Let $2<a<3$. The operator $\left(\mathscr{L}_{0},\mathscr{D}(\mathscr{L}_{0})\right)$ on $\mathbb{X}_{a}$ is such that, for any $ {\nu \in(0,1-\frac{a}{4} -2^{1-a})}$, there exists $C(\nu)>0$ such that
\begin{equation}\label{eq:invert0} 
 \|\mathscr{L}_{0}h\|_{\X_{a}}\ge \frac{\nu}{C(\nu)} \|h\|_{\X_{a}}, \qquad \forall h \in \mathscr{D}(\mathscr{L}_{0}) \cap \mathbb{Y}_{a}^{0}.\end{equation}
In particular, the restriction $\widetilde{\mathscr{L}}_{0}$ of $\mathscr{L}_{0}$ to the space $\mathbb{Y}_{a}^{0}$  
is invertible with
\begin{equation}\label{eq:invert}\left\|\widetilde{\mathscr{L}}_{0}^{-1}g\right\|_{\X_{a}} \leq \frac{C(\nu)}{\nu}\|g\|_{\X_{a}}, \qquad \forall g \in \mathbb{Y}_{a}^{0}.\end{equation}
\end{prp}

The existence of a spectral gap for $\mathscr{L}_{0}$ in the space $X_{0}$ endowed with the Fourier norm $\vertiii{\cdot}_{k}$ is essentially well-known (see e.g. \cite{MR2355628}) but we revisit the arguments in Section \ref{sec:sg-small}. To derive a similar spectral gap estimate in the more tractable space $\mathbb{Y}_{a}^{0}$ (for some $2<a <3$), we rely on recent results from \cite{CanizoThrom} and  \cite{MM} using a suitable splitting
$$\mathscr{L}_{0}=A+B,$$ with $A: X_0 \to \mathbb{Y}_{a}^{0}$ bounded
and $B$ enjoying some dissipative properties (we refer to Section \ref{sec:sg-small} for more details, in particular, we point out already that the results of Section \ref{sec:sg-small} are given for the linearised operator around $\bm{H}$ but that the results therein translate to $\mathscr{L}_{0}$ by simple scaling arguments). As a consequence, we will finally deduce that $\mathscr{L}_{0}|_{\mathbb{Y}_{a}^{0}}$ has a spectral gap as well. Notice that, according to  \cite[Lemma 2.5 and Proposition 2.6]{MR2355628} for $\mu\in X_0$, we have 
$$\vertiii{\mu}_{k}\leq C \int_{\R} \w_{k}(x) \,|\mu|(dx) \qquad \mbox{ for any  }\; 2< k<3,$$ 
with $\w_{s}(x)=(1+|x|)^{s}$ for $s \geq0.$
Hence, for $a\geq k$, we have 
$$\mathbb{Y}_{a}^{0}\subset X_0.$$ 
Therefore, our scope here is to deduce the spectral property of the linearised operator on a \emph{small space} from those well-established on a \emph{large space}: it is a \emph{shrinkage} argument (see \cite{MM} for pioneering work) in contrast with the \emph{enlargement} techniques introduced in \cite{GMM}.  

Combining Proposition \ref{restrict0} together with Theorem \ref{theo:sta} yields then, in a non straightforward way, a full proof of Theorem \ref{theo:mainUnique}. Roughly speaking, the idea is to apply the above quantitative stability estimate to the difference $g_{\g}=\Gg^{1}-\Gg^{2}$ of two elements of $\mathscr{E}_{\g}$. Let us explain our main strategy in the simplified situation in which both profile share the same energy. In this case, if  $\Gg^{1},\Gg^{2} \in \X_{a}$ with $2 < a <3$ are two solutions of \eqref{eq:steadyg} and $g_{\g}=\Gg^{1}-\Gg^{2}$ then,  
$$M_{2}(\Gg^{1})=M_{2}(\Gg^{2}) \implies g_{\g} \in \mathbb{Y}_{a}^{0}.$$
Moreover, it can also be shown that there exists a mapping $\tilde{\eta}\;:\;[0,1] \to \R^{+}$ with 
$$\lim_{\g\to0^{+}}\tilde{\eta}(\g)=0$$
and such that
\begin{equation*}\label{eq:L0dif-intro}
\|\mathscr{L}_{0}\left(\Gg^{1}-\Gg^{2}\right)\|_{\X_{a}} \leq \tilde{\eta}(\g)\left\|\Gg^{1}-\Gg^{2}\right\|_{\X_{a}}, \qquad \g >0.\end{equation*}
Combining this with \eqref{eq:invert0}  one gets
$$\frac{\nu}{C(\nu)} \|g_{\g}\|_{\X_{a}} \leq \|\mathscr{L}_{0}g_{\g}\|_{\X_{a}} \leq \tilde{\eta}(\g)\left\|g_{\g}\right\|_{\X_{a}}\,.$$
Since $\lim_{{\g\to0}}\tilde{\eta}(\g)=0$, one can choose $\g^{\dagger} \in (0,1)$ such that
$$\frac{C(\nu)}{\nu}\tilde{\eta}(\g) < 1, \qquad \forall  \g \in (0,\g^{\dagger}),$$
from which
$$\|g_{\g}\|_{\X_{a}} < \|g_{\g}\|_{\X_{a}} \qquad \forall \g \in (0,\g^{\dagger}).$$
This shows that $g_{\g}=0$ for all $\g \in (0,\g^{\dagger})$ and gives a simplified version of Theorem \ref{theo:mainUnique} in the special case in which $\Gg^{1}$ and $\Gg^{2}$ share the same energy. To prove the uniqueness result (without any restriction on the energy), we need therefore, in some rough sense, to be able to control the fluctuation of kinetic energy introducing a kind of selection principle which allows to compensate the discrepancy of energies to apply a variant of \eqref{eq:invert0}. This is done in Section \ref{sec:uniqueness} to which we refer for technical details regarding such a procedure. 
 
\subsection{Main features of our contribution} 

A first important novelty and main interest of the present contribution is that, to our knowledge, it presents the \emph{first and only} uniqueness result for self-similar profiles associated to an inelastic Boltzmann equation for hard potentials in a regime of \emph{large inelasticity}. Indeed, the only uniqueness result available in the literature is the one in the $3D$ case obtained in \cite{MM} in a \emph{weakly inelastic regime} corresponding to a restitution coefficient $e \simeq 1$. Our analysis here is the first one dealing with highly inelastic interactions (the most inelastic one actually since, as said, $a=\frac{1}{2}$ corresponds to sticky particles) and we strongly believe that our approach can be adapted to the study of $3D$ models for arbitrary restitution coefficient $e \in (0,1)$ (of course, still in a regime of moderately hard potentials). 

Second, one of the main interests of our analysis is that it provides
a first step towards the equivalent of the so-called scaling
hypothesis which, in the study of Smoluchowski's equation, asserts
that self-similar profiles are unique and attract all solutions to the
associated evolution equation (see \cite{CanizoThrom} for a first
proof of the scaling hypothesis for non-explicitly solvable
kernels). In the present contribution, we proved, as in
\cite{CanizoThrom}, that, for singular perturbation of the explicitly
solvable case of Maxwell molecules (i.e. for $\g \simeq 0$), the
self-similar profile $\Gg$ is unique. Some additional work should be
undertaken to prove now that such a unique solution attracts all
solution to \eqref{Intro-e2} with some explicit (exponential) rate. We
strongly believe that the perturbative framework introduced in the
present contribution is also the right approach to the study of the
long-time behavior of solutions to \eqref{Intro-e2}, exploiting now
the fact that convergence in Fourier norm $\vertiii{\cdot}_{k}$ is
indeed exponential in the limit case $\g=0$ (see Theorem
\ref{k-norm-cvgce}). Combining this with a careful spectral analysis
of the linearization of $\Q_{\g}$ around the self-similar profile
$\Gg$ should provide important insights on this important question and
pave the way to a full mathematical justification of a conjecture in
\cite{ernst} about the long-time behaviour of granular gases, allowing
in particular for strong inelasticity.

 \subsection{Organisation of the paper} After  this Introduction, the rest of the paper is organised as follows. Section \ref{sec:apost} derives the main a posteriori estimates on the self-similar profile $\Gg \in \mathscr{E}_{\g}$, focusing mainly on estimates which are uniform with respect to the parameter $\g \simeq 0$. We also establish in this Section the proof of Theorem \ref{theo:Unique}. The two main results, Theorem \ref{theo:sta} and also our main result Theorem \ref{theo:mainUnique}, are proven in Section \ref{sec:sec3} in which we take for granted many of the results regarding the Maxwell equation \eqref{introMax}. The final Section \ref{sec:exp} is devoted to a comprehensive study of the special case $\g=0$, i.e. a careful analysis of solutions to \eqref{introMax}. We revisit the exponential convergence to equilibrium in Fourier norm $\vertiii{\cdot}_{k}$ and extend it to more tractable Sobolev spaces by a carefuly study of the regularising effects of \eqref{introMax}. We also establish the proof of the stability estimate Proposition \ref{restrict0}. The paper ends with three Appendices containing various technical results of independent interest. In Appendix \ref{app:fourier}, we recall some properties of the Fourier norm $\vertiii{\cdot}_{k}$ as well as some useful interpolation inequalities. Appendix \ref{app:QgQ0}
 is devoted to some functional estimates of the collision operator $\Q_{\g}$ and its linearised counterpart. Finally, Appendix \ref{app:rigor}  provides rigorous justifications of several results whose proofs given in the text are only formal. Indeed, we believe that the core text should contain the main technical ideas underlying some of the results and decided to postpone their rigorous justifications  to Appendix \ref{app:rigor} which also contains the full proof of the above Proposition \ref{prop:allmoments}.

\subsection*{Acknowledgments}

R.~Alonso gratefully acknowledges the support from O Conselho Nacional de Desenvolvimento Científico e Tecnológico, Bolsa de Produtividade em Pesquisa - CNPq (303325/2019-4). J.~Cañizo acknowledges support from grant
PID2020-117846GB-I00, the research network RED2018-102650-T, and the
María de Maeztu grant CEX2020-001105-M from the Spanish
government. B.~Lods gratefully acknowledges the financial support from the Italian
Ministry of Education, University and Research (MIUR), ``Dipartimenti
di Eccellenza'' grant 2018-2022 as well as the support from the
\textit{de Castro Statistics Initiative}, Collegio Carlo Alberto
(Torino).   The authours would like to acknowledge the support of the
Hausdorff Institute for Mathematics where this work started during
their stay at the 2019 Junior Trimester Program on Kinetic Theory.
\emph{Data sharing not applicable to this article as no datasets were generated or analysed during the current study.}

\section{Uniform a posteriori estimates in the limiting process}
\label{sec:apost}

As explained in the Introduction, our proof of the uniqueness of
solutions to \eqref{eq:steadyg} is based upon a perturbative approach
around the pivot case $\g=0$ corresponding to Maxwell molecules
interactions. To undertake this perturbative approach we need first
to establish \emph{uniform} estimates for the self-similar profile
$\Gg$ to \eqref{eq:steadyg} for $\g \in (0,1)$ small enough. We begin
with the control of the energy $M_{2}(\Gg)$.

\subsection{Uniform energy control}

As far as the energy is concerned, one has the following estimate:
\begin{lem}\phantomsection\label{lem:energy}
  The following universal bound holds true: for any $\g \in (0,1)$ and
  any $\Gg \in \mathscr{E}_{\g}$ one has
  \begin{equation*}
    M_2(\Gg)  \leq \frac{1}{2}\,\,.
  \end{equation*}
  As a consequence, there exists some universal constant $\tilde{C} >
  0$ such that, for any $\g \in (0,1)$ and
  any $\Gg \in \mathscr{E}_{\g}$,
  \begin{equation}\label{eq:tildeC}
    \left\|\Gg\right\|_{L^{1}_{s}} \leq \tilde{C}
    \qquad \forall s \in [0,2].
  \end{equation}
\end{lem}

\begin{proof}
  Multiplying the equation \eqref{eq:steadyg} by $|x|^{2}$ and
  integrating in $x$, one obtains formally
  \begin{equation}\label{eq1}
    \frac{1}{4}\int_{\R}\int_{\R}\Gg(x)\Gg(y)|x-y|^{2}\big(|x-y|^{\gamma}-1\big)\dx\dy =0\,.
  \end{equation}
  To prove this rigorously we note that $\Gg$ satisfies \eqref{eq:steadyg} in a weak sense, that is
  \begin{multline}\label{weak_form}
    -\frac14 \int_\R x \phi'(x) \Gg(x) \d x =\\ \frac12 \int_{\R^2}|x-y|^\gamma \left(2 \phi\left(\frac{x+y}2\right) -\phi(x)-\phi(y)\right) \Gg(x)  \Gg(y) \d x\d y
  \end{multline} 
for any $\phi\in\mathcal{C}_{b}^{1}(\R)$. Since $x \mapsto x^{2}$ does not belong to $\mathcal{C}_{b}^{1}(\R)$, one cannot take $\phi(x)=x^2$ but one considers a sequence of approximating functions $\left\{\phi_{\ell}\right\}_{\ell\geq0}\subset \mathcal{C}_{b}^{1}(\R)$ satisfying 
  $$\phi_\ell(x)=\left\{\begin{array}{lcl}x^2 & \mbox{ for }&  |x|\le \ell,\\
  \ell^2+1 & \mbox{ for } & |x|\ge \ell+1 \end{array}\right.\qquad \mbox{ and }\qquad  |\phi'_\ell(x)|\le 2\ell \quad\mbox{ for } x\in\R. $$
  Plugging $\phi_\ell$ in \eqref{weak_form}, one has for any $\ell >0$,
\begin{multline}\label{eq:weakL}
-\frac{1}{2}\int_{-\ell}^{\ell}x^{2}\Gg(x)\dx +\frac{1}{4}\int_{[-\ell,\ell]^{2}}|x-y|^{\g+2}\Gg(x)\Gg(y)\dx\dy\\
=\int_{\R^2\setminus [-\ell,\ell]^{2}}|x-y|^{\g}\left(\phi_{\ell}\left(\frac{x+y}2\right) -\frac{1}{2}\phi_{\ell}(x)-\frac{1}{2}\phi_{\ell}(y)\right) \Gg(x)  \Gg(y) \dx\dy\\
+\frac{1}{4} \int_{\ell < |x| \leq \ell+1}x\phi'_{\ell}(x)\Gg(x)\dx
\end{multline}
where we used that $|\frac{x+y}{2}|^{2} - \frac{1}{2}|x|^{2} - \frac{1}{2}|y|^{2} = -\frac{1}{4}|x-y|^{2}$ in the particle-particle collisional term for $(x,y) \in [-\ell,\ell]^{2}$.
Rewriting, we have moreover
\begin{multline*}
 -\int_{-\ell}^{\ell}x^2\Gg(x)\dx=-\int_{[-\ell,\ell]^2}x^2\Gg(x)\Gg(y)\dx\dy\\*
 +\int_{[-\ell,\ell]^2}x^2\Gg(x)\Gg(y)\dx\dy-\int_{-\ell}^{\ell}x^2\Gg(x)\dx\\*
 =-\frac{1}{2}\int_{[-\ell,\ell]^2}|x-y|^2\Gg(x)\Gg(y)\dx\dy-\int_{[-\ell,\ell]^2}xy\Gg(x)\Gg(y)\dx\dy\\*
 +\int_{-\ell}^{\ell}x^2\Gg(x)\dx\biggl(\int_{-\ell}^{\ell}\Gg(y)\dy-1\biggr).
\end{multline*}
Thus, we get from \eqref{eq:weakL}  
\begin{equation}\label{eq:weakL2}
\frac{1}{4}\int_{[-\ell,\ell]^{2}}|x-y|^2\Gg(x)\Gg(y)\bigl(|x-y|^{\g}-1\bigr)\dx\dy=R_{\ell}\end{equation}
with
\begin{multline}\label{eq:weakL-2}
 R_\ell:=\int_{\R^2\setminus [-\ell,\ell]^{2}}|x-y|^{\g}\left(\phi_{\ell}\left(\frac{x+y}2\right) -\frac{1}{2}\phi_{\ell}(x)-\frac{1}{2}\phi_{\ell}(y)\right) \Gg(x)  \Gg(y) \dx\dy\\
+\frac{1}{4} \int_{\ell < |x| \leq \ell+1}x\phi'_{\ell}(x)\Gg(x)\dx+ {\frac{1}{2}}\int_{[-\ell,\ell]^2}xy\Gg(x)\Gg(y)\dx\dy\\*
 - {\frac{1}{2}}\int_{-\ell}^{\ell}x^2\Gg(x)\dx\biggl(\int_{-\ell}^{\ell}\Gg(y)\dy-1\biggr).\end{multline}
Letting $\ell \to \infty$ one deduces easily that $R_\ell$ is converging to zero.  {This justifies \eqref{eq1}}. Now, using the elementary inequality $u-1\geq \log u,$ $(u >0)$ with $u=|x-y|^{\g}$ one deduces from \eqref{eq1} that
\begin{multline*}
0\geq\frac{1}{4}\int_{\R^{2}}\Gg(x)\Gg(y)|x-y|^{2}\log|x-y|\dx\dy\\
=\frac{1}{8}\int_{\R^{2}}\Gg(x)\Gg(y)|x-y|^{2}\log|x-y|^{2}\dx\dy.\end{multline*}
Applying the elementary inequality $r\log r \geq r-1$ $(r >0)$ with $r=|x-y|^{2}$ we deduce that
$$0 \geq \int_{\R^{2}}\Gg(x)\Gg(y)\left(|x-y|^{2}-1\right)\dx\dy.$$
Since  
$$\int_{\R}\int_{\R}\Gg(x)\Gg(y)|x-y|^{2} \dx \dy=2M_{2}(\Gg), \qquad \int_{\R}\int_{\R}\Gg(x)\Gg(y)\dx\dy=1$$
we deduce the result.
\end{proof} 

\subsection{Weak convergence}

A first consequence of the above energy estimate \eqref{eq:tildeC} is
that, for any choice of equilibria $\Gg \in \mathscr{E}_{\g}$,
$$\left\{\Gg(x)\d x\right\}_{\gamma\in(0,1)} \quad \text{ is a tight set of probability measures}.$$ 
From Prokhorov's compactness Theorem (see \cite[Theorem 1.7.6, p. 41]{kolo}),  there exist some probability measure $\mu(\d x)$ and some sequence $(\gamma_n)_{n\in\N}$ tending to $0$ such that $\left\{\bm{G}_{\gamma_n}(x)\dx\right\}_{n\in\N}$ converges narrowly to $\mu(\dx)$, that is
  $$\int_{\R} \varphi(x) \,\bm{G}_{\gamma_n}(x)\dx \underset{n\to\infty}{\longrightarrow} \int_{\R}\varphi(x)  \mu(\dx)\, \qquad \forall \varphi \in \mathcal{C}_{b}(\R).$$
  Let $\phi\in\mathcal{C}_{b}^{1}(\R)$. Set $\psi_n(x,y)=|x-y|^{\gamma_n} \left(2\phi\left(\frac{x+y}{2}\right)-\phi(x)-\phi(y)\right)$ and $\psi(x,y)=2\phi\left(\frac{x+y}{2}\right)-\phi(x)-\phi(y)$.  On the first hand, as already observed, one has 
$$\left|  |x-y|^{\gamma_n}-1\right|\le \gamma_n |\log(|x-y|)| \quad \mbox{ for } \quad |x-y|\le 1$$  and 
$$  \left| 2\phi\left(\frac{x+y}{2}\right)-\phi(x)-\phi(y) \right|\le |x-y|\, \|\phi'\|_{L^\infty}.$$ 
Therefore, for  $|x-y|\le 1$,
$$\left| \psi_n(x,y)-\psi(x,y)\right|\le \gamma_n |\log(|x-y|)|\,  |x-y|\, \|\phi'\|_{L^\infty} \le \frac{1}{e}  \gamma_n \|\phi'\|_{L^\infty}.$$
On the other hand, one has for any $R >1$,
$$\left|  |x-y|^{\gamma_n}-1\right|\le \gamma_n\,R \log R\quad \mbox{ for } \quad |x-y|\ge 1 \mbox{ and } |x|+|y|\le R$$ and
$$ \left| 2\phi\left(\frac{x+y}{2}\right)-\phi(x)-\phi(y) \right|\le 4 \|\phi\|_{L^\infty}.$$
Consequently, for  $|x-y|\ge 1$ and $|x|+|y|\le R$, $$\left| \psi_n(x,y)-\psi(x,y)\right|\le 4 \gamma_n R\log R  \|\phi\|_{L^\infty}.$$
We may thus conclude that 
$$\lim_{n\to\infty}\sup_{|x|+|y|\le R} |\psi_n(x,y)-\psi(x,y)|=0. $$
Now, for any $n\in\N$, $|\psi_n(x,y)|\le 4 (2+|x|+|y|) \|\phi\|_{L^\infty}$, which implies
$$\lim_{|x|+|y|\to \infty} \sup_{n\ge 1 }\frac{|\psi_n(x,y)|}{2+|x|^2+|y|^2}=0. $$
The uniform convergence in compact sets and the above control of the
tails of $\psi_n$ imply that 
$$\lim_{n\to \infty} \int_{\R^2} \psi_n(x,y) \bm{G}_{\gamma_n}(x) \bm{G}_{\gamma_n}(y)\d x \d y =  \int_{\R^2} \psi(x,y) \mu(\dx) \mu(\dy).$$
We refer to \cite[Proposition 2.2]{LM} for details on the argument
leading to this. We thereby obtain that $\mu(\d x)$ is a steady
solution to \eqref{eq:IB-selfsim}.
 
Notice that the energy of $\mu(\dx)$ is not explicit but, from Theorem \ref{theo:bob}, there exists $\lambda >0$ such that
$$\mu(\dx)=H_{\lambda}(x)\dx=\lambda\bm{H}(\lambda x)\dx$$
satisfying
$$\int_{\R} H_\lambda(x)\d x =1, \qquad \int_{\R} H_{\lambda}(x)x\d x =0, \qquad \int_{\R}H_\lambda(x) x^2\d x =\frac{1}{\lambda^2}.$$
We need here to identify the possible value(s) of the parameter $\lambda.$ Thus, \emph{any limiting point} (as $\g \to0$) of the family $\{\Gg(x)\dx\}_{\g\in (0,1)}$ is a steady solution to \eqref{eq:IB-selfsim}. If we are able to identify a \emph{unique} possible limiting positive energy, we would have a \emph{unique} possible limiting point and the whole net $\{\Gg(x)\dx\}_{\g}$ would converge to it.

\subsection{Pointwise control}

A second observation is the following uniform \emph{pointwise} upper
bound.
\begin{lem}\phantomsection There exists $C >0$ \emph{independent of $\g$} such that
\begin{equation}\label{eq:pointX}
\Gg(x) \leq \frac{C}{|x|} \qquad  {\mbox{ for  a.e. } x \in \R},\end{equation}
holds true for any $\g\in[0,1]$ and any $\Gg \in \mathscr{E}_{\g}$.
\end{lem}
\begin{proof} Formally,  integrating equation for $\Gg$ in $(0,x)$, with $x>0$, one has
\begin{equation*}
x\,\Gg(x) = 4\int^{x}_{0} \Q_{\g}(\Gg,\Gg) \dy \leq 4\| \Q^{+}_{\g}(\Gg,\Gg) \|_{L^1}\leq C\| \Gg \|_{L^{1}_{\g}} \| \Gg \|_{L^1}\,
\end{equation*}
and the uniform control provided by \eqref{eq:tildeC}  yields the result.
 
To justify rigorously this inequality, for $x >0$, one considers some nonnegative mollifying sequence $(\varrho_n)_{n\in\N}$ and define
$$\phi_n(y)=\int_{-\infty}^y\varrho_n(x-z)\d z, \qquad y \in \R, \qquad n \in \N.$$
Choosing the test-function $\phi_{n}$ in \eqref{weak_form}, one obtains
\begin{multline*}
-\frac14 \int_{\R} y\varrho_n(x-y)\Gg(y)\d y =\\
 \frac12 \int_{\R^2}|y-z|^\gamma \Gg(y)\Gg(z)  \left( 2\phi_n\left(\frac{y+z}{2}\right)-\phi_n(y)-\phi_n(z) \right) \d y \d z,\end{multline*}
  for any $n \in \N$. Since $0\le \phi_n(y)\le \int_{\R}\varrho_n(z)\d z=1$, we get
  \begin{align*}
    -\frac14 \int_{\R} y\varrho_n(x-y)\Gg(y)\d y & \ge - \frac12 \int_{\R^2}|y-z|^\gamma \Gg(y)\Gg(z)  \left(\phi_n(y)+\phi_n(z) \right) \d y \d z
    \\
    & \ge - 2 \|\Gg |\cdot|^\gamma\|_{L^1} \|\Gg\|_{L^1} \ge -2\tilde{C}
    \end{align*}
    where $\tilde{C}$ is defined in \eqref{eq:tildeC}. Letting $n\to\infty$, we get $x \Gg(x)\leq 8\tilde{C}$, for a.e. $x>0$. For $x<0$, we bound from above the above integral in order to obtain in the end  $-x \Gg(x)\leq 8\tilde{C}$, for a.e. $x<0$. This proves the result.
\end{proof}

\subsection{$L^{2}$-estimates on the profile}
\label{sec:L2est}

We deduce from this the following technical estimate regarding the control of $L^{2}$ norms.
\begin{lem}\phantomsection \label{lem:boundL2L} There exists some
  universal numerical constant $C_0 >0$ such that the inequality
  \begin{equation}
    \label{eq:boundL2L}
    \|\Gg\|_{L^2}^{2} \leq
    C_{0} \|\Gg\|_{L^2}^{2}\int_{-\ell}^{\ell}|x|^{\g}\Gg(x)\dx
    + C_0\,\ell^{\g-1}
  \end{equation}
  holds true for any $\g >0$, $\Gg \in \mathscr{E}_{\g}$ and any
  $\ell >0$.
\end{lem}
 
\begin{proof}
  We provide here a formal proof which provides the main ideas
  underlying the result. We refer to Appendix \ref{sec:justif-L2} for
  a rigorous justification of the formal argument that follows.  For
  any generic solution $\Gg$ to \eqref{eq:steadyg}, after multiplying
  \eqref{eq:steadyg} with $\Gg$ and integrating over $\R$ one sees
  that
\begin{equation}\label{eq:L2Gg}
\frac{1}{8}\|\Gg \|^{2}_{L^2} \leq \int_{\R}\Q^{+}_{\g}(\Gg,\Gg)\,\Gg\dx.\end{equation}
One sees that
\begin{multline*}
\int_{\R}\Q^{+}_{\g}(\Gg,\Gg)\,\Gg\dx=\int_{\R}\int_{\R}|x-y|^{\g}\Gg(x)\Gg(y)\Gg\left(\frac{x+y}{2}\right)\dx\dy\\
\leq \int_{\R}\int_{\R}\left(|x|^{\g}+|y|^{\g}\right)\Gg(x)\Gg(y)\Gg\left(\frac{x+y}{2}\right)\dx\dy\\
=2\int_{\R}|x|^{\g}\Gg(x)\dx\int_{\R}\Gg(y)\Gg\left(\frac{x+y}{2}\right)\dy.
\end{multline*}
Notice that such an inequality actually means that
\begin{equation}\label{eq:QgQ0}
\int_{\R}\Q^{+}_{\g}(\Gg,\Gg)\,\Gg\dx \leq 2\int_{\R}\Q^{+}_{0}(|\cdot|^{\g}\Gg,\Gg)\Gg\d x.\end{equation}
Given $\ell >0$, splitting the integral with respect to $x$ according to $|x| >\ell$ and $|x| \leq \ell$, one has
\begin{equation*}\begin{split}
\int_{\R}\Q^{+}_{\g}(\Gg,\Gg)\,\Gg\dx &\leq 2\int_{-\ell}^{\ell}|x|^{\g}\Gg(x)\dx\int_{\R}\Gg(y)\Gg\left(\frac{x+y}{2}\right)\dy\\
&\phantom{+++} +2\int_{|x|>\ell}|x|^{\g}\Gg(x)\dx\int_{\R}\Gg(y)\Gg\left(\frac{x+y}{2}\right)\dy\\
&\leq  2\int_{-\ell}^{\ell}|x|^{\g}\Gg(x)\dx\int_{\R}\Gg(y)\Gg\left(\frac{x+y}{2}\right)\dy\\
&\phantom{++++} + 2C\ell^{\g-1}\int_{\R}\dx\int_{\R}\Gg(y)\Gg\left(\frac{x+y}{2}\right)\dy\,,
\end{split}\end{equation*}
where we used \eqref{eq:pointX} in the last step. Clearly, the last integral can be estimated as
$$\int_{\R}\dx\int_{\R}\Gg(y)\Gg\left(\frac{x+y}{2}\right)\dy=2\|\Gg\|_{L^1}^2=2$$
whereas, for any given $x \in [-\ell,\ell]$, one has from Cauchy-Schwarz inequality
$$\int_{\R}\Gg(y)\Gg\left(\frac{x+y}{2}\right)\dy \leq \|\Gg\|_{L^2}\,\left\|\Gg\left(\frac{x+\cdot}{2}\right)\right\|_{L^2}=\sqrt{2}\|\Gg\|_{L^2}^{2}.$$
Combining these estimates, we deduce that
$$\int_{\R}\Q^{+}_{\g}(\Gg,\Gg)\,\Gg\dx\leq 2\sqrt{2}\|\Gg\|_{L^2}^{2}\,\int_{-\ell}^{\ell}\,|x|^{\g}\Gg(x)\dx+4C\ell^{\g-1}.$$
This gives the desired result thanks to \eqref{eq:L2Gg}.\end{proof}

A trivial bound for the integral is the following
$$\int_{-\ell}^{\ell}|x|^{\g}\Gg(x)\dx \leq \ell^\g\,\|\Gg\|_{L^1}=\ell^{\g}$$
which gives a bound like
$$\|\Gg\|_{L^2}^2 \leq C_0 \ell^{\g}\|\Gg\|_{L^2}^{2} + C_0 \ell^{\g-1}$$
and cannot provide a bound on $\|\Gg\|_{L^2}$ uniform with respect to $\g$. If one assumes say 
\begin{equation}\label{eq:controlL}
  C_0\int_{-\ell}^{\ell}|x|^{\g}\Gg(x)\dx
  \leq
  C_1\ell^{\g+1}
\end{equation}
for some universal (independent of $\g$) constant $C_1$, then picking
$\ell$ small enough would yield a uniform bound on $\|\Gg\|_{L^{2}}$
\emph{uniform with respect to $\g$ small enough}.  We are actually not
able to establish the bound \eqref{eq:controlL} for any
$\Gg \in \mathscr{E}_\g$ but will provide a similar estimate for any
sequence $\{\bm{G}_{\g_n}\}$ converging weakly-$\star$. Recall that
such a sequence always exists.  One has then the following
\begin{lem}\phantomsection\label{lem:L2bound}
  Let $\left(\g_{n}\right)_{n}$ be a sequence going to zero,
  $\bm{G}_{\g_n} \in \mathscr{E}_{\g_n}$ an equilibrium for each $n$,
  and $\lambda >0$ such that
  $$\lim_{n\to\infty}\int_{\R}\bm{G}_{\g_{n}}(x)\varphi(x)\dx=\int_{\R}H_{\lambda}(x)\varphi(x)\dx, \qquad \forall \varphi \in \mathcal{C}_{b}(\R).$$

  Then there exists $C=C(\lambda)$ depending only on $\lambda$ and
  $N \geq 1$ such that
  $$\sup_{n \geq N}\|\bm{G}_{\g_{n}}\|_{L^2} \leq C.$$
\end{lem}

\begin{proof}
  From the weak-$\star$ convergence, for any $\ell >0$, one can choose
  a smooth cutoff function $\varphi_{L}\geq0$ equal to one in
  $[-\ell,\ell]$ and vanishing on $\R \setminus [-2\ell,2\ell]$ to
  deduce that there exists $N >1$ such that
$$\int_{-\ell}^{\ell}\bm{G}_{\g_{n}}(x)\dx \leq 2\int_{-2\ell}^{2\ell}H_{\lambda}(x)\dx=2\int_{-2\lambda\,\ell}^{2\lambda\,\ell}\bm{H}(x)\dx \qquad \forall n \geq N.$$
Direct computations show that
$$\int_{-2\lambda\,\ell}^{2\lambda\,\ell}\bm{H}(x)\dx=\frac{2}{\pi}\left[\arctan\left(2\lambda\,\ell\right)+\frac{{2}\lambda\,\ell}{1+4\lambda^{2}\ell^{2}}\right] \leq \frac{8}{\pi}\lambda\,\ell \qquad \forall \ell >0, \lambda >0.$$
Thus, for any $n \geq N$, one has
$$\int_{-\ell}^{\ell}|x|^{\g_n}\bm{G}_{\g_{n}}(x)\dx \leq \ell^{\g_n}\,\int_{-\ell}^{\ell}\bm{G}_{\g_{n}}(x)\dx \leq \frac{16\lambda}{\pi}\ell^{\g_n+1}.$$
Arguing as described previously, plugging this into \eqref{eq:boundL2L} we get
$$\|\bm{G}_{\g_n}\|_{L^2}^{2} \leq \frac{16\,C_0\,\lambda}{\pi}\ell^{\g_n+1}\|\bm{G}_{\g_n}\|_{L^2}^{2} + C_0\,\ell^{\g_n-1} \qquad \forall n \geq N, \qquad \forall \ell >0.$$
Picking then $\ell\leq 1$ (depending on $\lambda$) such that
$$ \frac{16\,C_0\,\lambda}{\pi}\ell^{\g_n+1}\leq \frac{16\,C_0\,\lambda}{\pi}\ell \leq\frac{1}{2}, \qquad \text{ i.e. } \qquad \ell=\min\left\{\left(\frac{\pi}{{32}\lambda\,C_0}\right),1\right\}$$
we deduce that
$$\|\bm{G}_{\g_n}\|_{L^2}^{2} \leq 2C_0\,\ell^{\g_n-1}$$
and, since $\ell \leq 1$ and $\g_{n}\ge0$, $\ell^{\g_{n}} \leq 1$ so that
$$\|\bm{G}_{\g_{n}}\|_{L^{2}}^{2} \leq 2\frac{C_{0}}{\ell}=2C_{0}\max\left\{\frac{32}{\pi}C_{0}\lambda,1\right\}$$
which gives the result.\end{proof}

\subsection{Lower control of the collision frequency} We introduce the collision frequency
\begin{equation}\label{eq:coll:freq}
\Sigma_{\g}(y)=\int_{\R}|x-y|^{\g}\Gg(x)\dx 
\end{equation}
and recall also the notation $\w_{s}$ introduced in \eqref{eq:weight}. One has then the following
\begin{lem}\phantomsection\phantomsection\label{lem:Sigmag} Given $\g \in (0,1)$ and $\Gg \in \mathscr{E}_{\g}$, there exists $\kappa_{\g} >0$ such that the following holds
\begin{equation}\label{eq:sigma_g}\Sigma_{\g}(y) \geq \kappa_{\g}\,\w_{\g}(y) - (1-\tilde{\delta}^{\g})-\sqrt{2\tilde{\delta}}\|\Gg\|_{L^2}, \qquad \qquad \forall \tilde{\delta} \in (0,1).\end{equation}
Moreover, $\lim_{\g\to0}\kappa_{\g}=1.$
\end{lem}

\begin{proof} Let $\g \in (0,1)$ and $\Gg \in \mathscr{E}_{\g}$ be given. First, notice that, for any $y \in \R$ and any $\tilde{\delta} \in (0,1)$,
\begin{multline*}
\Sigma_{\g}(y)= \int_{\R}\left(|x-y|^{\g}+\ind_{|x-y|<\tilde{\delta}}\right)\,\Gg (x)\,\dx - \int_{\R}\ind_{|x-y|<\tilde{\delta}}\,\Gg (x)\,\dx\\
\geq \int_{\mathbb{R}}\left(|x-y|^{\g}+\ind_{|x-y|<\tilde{\delta}}\right)\,\Gg (x)\,\dx  - \sqrt{2\tilde{\delta}}\| \Gg  \|_{L^2}=:\Sigma_{\g}^{(\tilde{\delta})}(y)-\sqrt{2\tilde{\delta}}\|\Gg\|_{L^2}
\end{multline*}
thanks to Cauchy-Schwarz inequality. We need only to estimate the first term. To do so, for any $\eta >1$, we introduce the set 
$$I=I(y,\,\eta)=\left\{x \in \R\;;\;\w_{1}(x) \leq \eta^{-1}\w_{1}(y)\right\},$$ and write
$$
\Sigma_{\g}^{(\tilde{\delta})}(y)
= \int_{I}\left(|x-y|^{\g}+\ind_{|x-y|<\tilde{\delta}}\right)\,\Gg (x)\,\dx + \int_{I^{c}}\left(|x-y|^{\g}+\ind_{|x-y|<\tilde{\delta}}\right)\,\Gg (x)\,\dx .
$$
On the set $I$, one has  
$$|x-y|^{\g}\geq \big( (1+|y|) - (1+ |x|) \big)^{\g}\geq   \big( 1-\eta^{-1} \big)^{\g}\w_{\g}(y).$$
Therefore,
\begin{equation*}\begin{split}
\int_{I}\left(|x-y|^{\g}+\ind_{|x-y|<\tilde{\delta}}\right)\,\Gg (x)\,\dx  &\geq   \left(\frac{\eta-1}{\eta}\right)^{\g}\w_{\g}(y)\int_{I}\,\Gg (x)\dx \\
&\geq  \left(\frac{\eta-1}{\eta}\right)^{\g}\w_{\g}(y)\int_{I}\,\frac{\Gg (x)}{\w_{\g}(x)}\dx  \,.
\end{split}\end{equation*}
Now, observing that $|x-y|^{\g}+\ind_{|x-y|<\tilde{\delta}}\geq \tilde{\delta}^{\g}\,,$ $(\tilde{\delta} <1)$,
one has
\begin{equation*}\begin{split}
\int_{I^{c}}\left(|x-y|^{\g}+\ind_{|x-y|<\tilde{\delta}}\right)\,&\Gg (x)\,\dx 
\geq \tilde{\delta}^{\g}\int_{I^{c}}\Gg(x)\dx 
\geq \int_{I^{c}}\Gg (x)\,\dx- (1-\tilde{\delta}^{\gamma})\|\Gg \|_{L^1}\\
&\geq  \frac{\w_{\g}(y)}{\eta^{\gamma}}\int_{I^{c}}\frac{\Gg (x)}{\w_{\g}(x)}\dx - (1-\tilde{\delta}^{\gamma})\,,
\end{split}\end{equation*}
since $1 \geq \frac{1+|y|}{\eta(1+|x|)}$ for any $x \notin I$.   Choosing  then $\eta=2$ one sees that
$$\Sigma_{\g}^{(\tilde{\delta})}(y)\geq \frac{\w_{\g}(y)}{{2}^{\g}}\int_{\R}\frac{\Gg(x)}{\w_{\g}(x)}\dx - (1-\tilde{\delta}^{\g})\,,$$
which gives \eqref{eq:sigma_g} with
$$\kappa_{\g}:=\frac{1}{2^{\g}}\int_{\R}\frac{\Gg(x)}{\w_{\g}(x)}\dx, \qquad \g \in (0,1).$$
Let us now prove that $\lim_{\g\to0^{+}}\kappa_{\g}=1.$ Obviously, since $\w_{\g}(x) \geq 1$ and $\Gg$ has unit mass, one has
$$0 \leq \kappa_{\g} \leq 2^{-\g}.$$
One just needs to bound $\kappa_{\g}$ from below. For any $\g \in (0,1)$ and $r >0$
\begin{multline*}
\int_{\R}\frac{\Gg(x)}{\w_{\g}(x)}\dx\geq \int_{|x|\leq r}\frac{\Gg(x)}{\w_{\g}(x)}\dx \geq \frac{1}{\left(1+r\right)^{\g}}\int_{|x|\leq r}\Gg(x)\dx\\
\geq \frac{1}{\left(1+r\right)^{\g}}\left(1-\int_{|x|>r}\Gg(x)\dx\right) \geq \frac{1}{\left(1+r\right)^{\g}}\left(1-\frac{M_{2}(\Gg)}{r^{2}}\right) \geq \frac{1}{\left(1+r\right)^{\g}}\left(1-\frac{C}{r^{2}}\right)\end{multline*}
where we used that $\sup_{\g\in (0,1)}M_{2}(\Gg) \leq  {C}$. For any $\varepsilon >0$, one can first pick $r >1$ independent of $\g$ and large enough so that 
$$\int_{\R}\frac{\Gg(x)}{\w_{\g}(x)}\dx \geq \frac{1-\varepsilon}{\left(1+r\right)^{\g}}$$
so that
$$\frac{1-\varepsilon}{2^{\g}(1+r)^{\g}}\leq \kappa_{\g} \leq 2^{-\g}$$
and the result then follows letting $\g\to 0$. \end{proof}

\subsection{Uniform estimates for higher moments} We investigate here
some uniform estimates for higher moments $M_{k+\g}(\Gg)$. Of course,
since one expects $\Gg \to \bm{G}_{0}$ where $\bm{G}_{0}$ is a steady
state to \eqref{introMax} with $\bm{G}_{0} \in L^{1}_{3}(\R) \setminus
L^{1}_{4}(\R)$, for $k >3,$ it should hold that
$$\limsup_{\g\to0}M_{k}(\Gg)=\infty$$
however, one can expect, for $2 < k < 3$,
$$\limsup_{\g\to0}M_{k+\g}(\Gg) < \infty.$$
This is the object of the following
\begin{lem}\phantomsection\phantomsection\label{lem:momentsk}
  Let $\left(\g_{n}\right)_{n}$ be a sequence going to zero,
  $\bm{G}_{\g_n} \in \mathscr{E}_{\g_n}$ an equilibrium for each $n$,
  and $\lambda >0$ such that
  $$\lim_{n\to\infty}\int_{\R}\bm{G}_{\g_{n}}(x)\varphi(x)\dx=\int_{\R}H_{\lambda}(x)\varphi(x)\dx, \qquad \forall \varphi \in \mathcal{C}_{b}(\R).$$
  For any $\delta\in(0,1/2)$, there exists $C >0$ and $\bar{N} >1$
  such that
  $$M_{k+\g_{n}}(\bm{G}_{\g_{n}}) \leq C, \qquad \forall 2+\delta < k < 3-\delta, \quad\forall n \geq \bar{N}.$$
\end{lem}

\begin{proof} Formally, for any $k \geq0$ and any solution $\Gg$ to \eqref{eq:steadyg}, 
$$-\frac{k}{4}\int_{\R}\Gg(x)|x|^{k}\dx=\int_{\R}\Q_{\g}(\Gg,\Gg)(x)|x|^{k}\dx$$
with
\begin{equation*}\begin{split}
\int_{\R}\Q_{\g}(\Gg,\Gg)|x|^{k}\dx&=\int_{\R^{2}}\Gg(x)\Gg(y)|x-y|^{\g}\,\left(\left|\frac{x+y}{2}\right|^{k}-|y|^{k}\right)\dx\dy\\
&=\int_{\R^{2}}\Gg(x)\Gg(y)|x-y|^{\g}\,\left|\frac{x+y}{2}\right|^{k}\dx\dy - \int_{\R}\Gg(y)|y|^{k}\Sigma_{\g}(y)\dy\,,\end{split}\end{equation*}
where $\Sigma_{\g}(y)$ is the collision frequency defined in \eqref{eq:coll:freq}.
The above identity  holds formally and can be proved rigorously along the lines of the proof of \eqref{eq1}. Notice that, for $k<3$ it holds that
\begin{equation}\begin{split}
\Big| \frac{x+y}{2} \Big|^{k}&= 2^{-k}| x^3 + 3x^{2}y + 3xy^2 + y^3 |^{\frac{k}{3}}\\
&\leq 2^{-k}\big( | x |^k + 3|x|^{\frac{2k}{3}}|y|^{\frac{k}{3}} + 3|x|^{\frac{k}{3}}|y|^{\frac{2k}{3}} + |y|^k\big) \,\qquad \forall (x,y) \in \R^{2}. \label{ineqxplusy}
\end{split}\end{equation}
Then, with this inequality and a simple symmetry argument, one deduces that
\begin{equation*}\begin{split}
\int_{\R^{2}}\Gg(x)&\Gg(y)|x-y|^{\g}\,\left|\frac{x+y}{2}\right|^{k}\dx\dy\\
&\leq 3\cdot 2^{-k} \int_{\R^{2}}\Gg(x)\Gg(y)\left(|x|^{\g}+|y|^{\g}\right)\,\left(|x|^{\frac{2k}{3}}|y|^{\frac{k}{3}} + |x|^{\frac{k}{3}}|y|^{\frac{2k}{3}}\right)\dx\dy\\
&\phantom{++++} +2^{1-k}\int_{\R}\Gg(y)\Sigma_{\g}(y)|y|^{k}\dy\\
&\leq 6\cdot 2^{-k}\left[M_{\frac{2k}{3}+\g}(\Gg)M_{\frac{k}{3}}(\Gg)+M_{\frac{2k}{3}}(\Gg)M_{\frac{k}{3}+\g}(\Gg)\right]\\
&\phantom{+++++} +2^{1-k}\int_{\R}\Gg(y)\Sigma_{\g}(y)|y|^{k}\dy\,
\end{split}\end{equation*}
from which we obtain
\begin{multline*}
-\frac{k}{4}M_{k}(\Gg) +\left(1-2^{1-k}\right)\int_{\R}\Gg(y)\Sigma_{\g}(y)|y|^{k}\dy\\
\leq 6\cdot 2^{-k}\left[M_{\frac{2k}{3}+\g}(\Gg)M_{\frac{k}{3}}(\Gg)+M_{\frac{2k}{3}}(\Gg)M_{\frac{k}{3}+\g}(\Gg)\right]\,.
\end{multline*}
Notice that, with the condition $2+\delta<k<3-\delta$, one has that
\begin{equation*}
\max\left\{ M_{\frac{k}{3}}(\Gg), M_{\frac{2k}{3}}(\Gg), M_{\frac{k}{3}+\g}(\Gg)\right\}\leq M_0(\Gg)+M_2(\Gg) \leq \frac{3}{2}\,,
\end{equation*}
where we used that $\sup_{\g\in (0,1)}M_{2}(\Gg) \leq \frac{1}{2}.$ Moreover, using Young's inequality, one sees that, for any $\eta >0$,
\begin{equation*}
M_{\frac{2k}{3}+\g }(\Gg)\leq  {\frac{2k+3\g}{3k+3\g} \eta M_{k+\g }(\Gg) + \eta^{-\frac{2k+3\g}{k}}\frac{k}{3k+3\g}M_0(\Gg)}\leq\eta M_{k+\g}(\Gg)+\eta^{-2-\frac{3\g}{k}}\,.
\end{equation*}
With this, we deduce that there is $C>0$ (independent of $\g$ and $k$) such that
\begin{multline*}
-\frac{k}{4}M_{k}(\Gg) +\left(1-2^{1-k}\right)\int_{\R}\Gg(y)\Sigma_{\g}(y)|y|^{k}\dy\\
\leq C \left(1+\eta^{-2-\frac{3\g}{k}}+\eta\,M_{k+\g}(\Gg)\right), \qquad \quad \forall \eta >0\,.
\end{multline*}
We use now Lemma \ref{lem:Sigmag} to deal with the term involving the collision frequency. Precisely, considering now a converging sequence $\{\bm{G}_{\g_{n}}\}_{n}$ towards $H_{\lambda}$, we deduce from Lemma \ref{lem:L2bound}  that $\|\bm{G}_{\g_{n}}\|_{L^2} \leq \tilde{C}$ for $n \geq N$ and therefore  for $n \geq N$,
\begin{multline*}
\kappa_{\g_{n}}\left(1-2^{1-k}\right)\int_{\R}\bm{G}_{\g_{n}}(y)|y|^{k}\,\w_{\g_{n}}(y)\dy 
- \left(\frac{k}{4} + \left(1-\tilde{\delta}^{\g_{n}}\right)+\tilde{C}\sqrt{2\tilde{\delta}}\right)M_{k}(\bm{G}_{\g_{n}})\\
\leq C\left(1+\eta^{-2-\frac{3\g_{n}}{k}}+\eta\,M_{k+\g_{n}}(\bm{G}_{\g_{n}})\right), \qquad \quad \forall \eta >0\,.
\end{multline*}
Let $\varepsilon >0$ be fixed. Picking $\tilde{\delta} >0$ such that $\tilde{C}\sqrt{2\tilde{\delta}}=\frac{\varepsilon}{2}$, one can find $N'\geq N$ large enough so that
$$\left(1-\tilde{\delta}^{\g_{n}}\right)+\tilde{C}\sqrt{2\tilde{\delta}} \leq \varepsilon, \qquad \forall n \geq N'$$
and we deduce that
\begin{multline*}
\kappa_{\g_{n}}\left(1-2^{1-k}-C\frac{\eta}{\kappa_{\g_{n}}}\right)\int_{\R}\bm{G}_{\g_{n}}(y)|y|^{k}\,\w_{\g_{n}}(y)\dy 
- \left(\frac{k}{4} + \varepsilon\right)M_{k}(\bm{G}_{\g_{n}})\\
\leq  C\left(1+\eta^{-2-\frac{3\g_{n}}{k}}\right)\,.\end{multline*}
for any $n \geq N'$. Then
\begin{equation*}
\kappa_{\g_{n}}\left(1-2^{1-k}-C\frac{\eta}{\kappa_{\g_{n}}}-\frac{k}{4\kappa_{\g_{n}}}-\frac{\varepsilon}{\kappa_{\g_{n}}}\right)\int_{\R}\bm{G}_{\g_{n}}(y)|y|^{k}\,\w_{\g_{n}}(y)\dy
\leq C\left(1+\eta^{-2-\frac{3\g_{n}}{k}}\right)\,.\end{equation*}
There exists $\bar{\sigma}$ (independent of $k$) such that, for any $2+\delta < k < 3-\delta$,  ${\sigma_{k}:=1-2^{1-k}-\frac{k}{4}>\bar{\sigma} >0}$. One has
\begin{eqnarray*}1-2^{1-k}-C\frac{\eta}{\kappa_{\g_{n}}}-\frac{k}{4\kappa_{\g_{n}}}-\frac{\varepsilon}{\kappa_{\g_{n}}}& =&  \sigma_k - \left(\frac1{\kappa_{\g_{n}}}-1\right)\frac{k}4 -C\frac{\eta}{\kappa_{\g_{n}}}-\frac{\varepsilon}{\kappa_{\g_{n}}}\\
& \ge &  \bar{\sigma} - \left(\frac1{\kappa_{\g_{n}}}-1\right)\frac{k}4 -C\frac{\eta}{\kappa_{\g_{n}}}-\frac{\varepsilon}{\kappa_{\g_{n}}}.
\end{eqnarray*}
Recalling that $\lim_{n\to\infty}\kappa_{\g_n}=1$, one can choose $\bar{N}\geq N'$ large enough such that, for $n\ge \bar{N}$ 
$$\kappa_{\g_{n}}\ge \frac34 \qquad \mbox { and } \qquad \frac{k}{4}\left| \frac1{\kappa_{\g_{n}}}-1\right| \le \frac34\left| \frac1{\kappa_{\g_{n}}}-1\right|\le \frac{\bar{\sigma}}{9}.$$
One then chooses $\varepsilon,\eta$ small enough so that 
$$C\frac{\eta}{\kappa_{\g_{n}}} \le C\frac{4\eta}{3}\le  \frac{\bar{\sigma}}{9}\qquad \mbox { and }\qquad\frac{\varepsilon}{\kappa_{\g_{n}}} \le \frac{4\varepsilon}{3}\le \frac{\bar{\sigma}}{9}$$
for any $n \geq \bar{N}$ and gets that
$$\kappa_{\g_{n}}\left(1-2^{1-k}-C\frac{\eta}{\kappa_{\g_{n}}}-\frac{k}{4\kappa_{\g_{n}}}-\frac{\varepsilon}{\kappa_{\g_{n}}}\right) \geq \frac{\bar{\sigma}}{2} >0 \qquad \forall n \geq \bar{N}.$$
We then deduce the result noticing that, for $\eta\in(0,1)$, $\eta^{-2-\frac{3\g_{n}}{k}} \le \eta^{-2-\frac{3\g_{n}}{2}}$ and $\lim_{n\to\infty}\eta^{-2-\frac{3\g_{n}}{2}}=\eta^{-2}$.
\end{proof}
  
\subsection{Limiting temperature and proof of Theorem~\ref{theo:Unique}}
\label{Sec:limit:temp}
Recall that any converging sequence $\{\bm{G}_{\g_{n}}\}_{n}$ (with
$\lim_{n\to \infty}\g_{n}=0$) admits as a weak limit a function of the
form
$$H_{\lambda}(x)=\lambda\,\bm{H}(\lambda\,x), \qquad \lambda >0.$$
We prove here that $\lambda$ is actually uniquely determined, yielding the uniqueness of the possible limit point. Namely, we prove the following
\begin{lem}\phantomsection\phantomsection\label{lem:unique}
Let $\left(\g_{n}\right)_{n}$ be a sequence going to zero and $\lambda >0$ such that
$$\lim_{n\to\infty}\int_{\R}\bm{G}_{\g_{n}}(x)\varphi(x)\dx=\int_{\R}H_{\lambda}(x)\varphi(x)\dx, \qquad \forall \varphi \in \mathcal{C}_{b}(\R).$$
Then, 
$$\lambda=\lambda_{0}:=\exp\left(A_{0}\right)$$
where
\begin{equation}\label{eq:A0}
 A_{0}:=\frac{1}{2}\int_{\R}\int_{\R}\bm{H}(x)\bm{H}(y)|x-y|^{2}\log|x-y|\dx\dy >0.
\end{equation}
\end{lem} 
{\begin{rmk} We will see later on that $A_{0}$ can be made explicit and, according to Lemma \ref{rmk:kern}
$A_{0}=\log2+\frac{1}{2}$ from which
$\lambda_{0}=2\sqrt{e}.$\end{rmk}
}
\begin{proof} We consider a sequence $\{\bm{G}_{\g_{n}}\}_{n}$ and $\lambda >0$ such that
$$\lim_{n\to\infty}\int_{\R}\bm{G}_{\g_{n}}(x)\varphi(x)\dx=\int_{\R}H_{\lambda}(x)\varphi(x)\dx, \qquad \forall \varphi \in \mathcal{C}_{b}(\R).$$
Let $\delta>0$. Let us fix $k\in(2,3)$ and $s>0$ small enough such that $k+s\in(2+\delta,3-\delta)$. We consider $N\in\N$ large enough such that the conclusions of Lemma \ref{lem:L2bound}  and Lemma \ref{lem:momentsk} hold true and  $\g_{n}<k-2$ for any  $n\ge N$. Then, Lemma \ref{lem:momentsk} and Young's inequality imply that
\begin{equation}\label{Mk+s} 
M_{k+s}(\bm{G}_{\g_{n}}) \le M_{k+s+\gamma_n}(\bm{G}_{\g_{n}})+M_{0}(\bm{G}_{\g_{n}}) \le C+1=:\bar{C},
\end {equation}
 for any $n \geq N$. Introducing 
$$\Lambda_{\g}(r)=\frac{r^{\g}-1}{\g}, \qquad \forall r >0, \qquad \g >0$$
we recall from \eqref{eq1} that
\begin{equation}\label{eq1-gn}
\int_{\R}\int_{\R}\bm{G}_{\g_{n}}(x)\bm{G}_{\g_{n}}(y)|x-y|^{2}\Lambda_{\g_{n}}(|x-y|)\dx\dy=0\,\qquad \qquad \forall n\geq N.
\end{equation}
On the one hand, let $\tilde{\delta} \in \left(0,\frac{1}{e}\right)$ to be determined later. Using the elementary inequality {$|\Lambda_{\g}(r)| \leq -\log r$} for any $r \in (0,1)$, $\g >0$, we have
\begin{multline*}
\int_{\R}\int_{|x-y|\leq \tilde{\delta}}\bm{G}_{\g_{n}}(x)\bm{G}_{\g_{n}}(y)|x-y|^{2} {|\Lambda_{\g_{n}}(|x-y|)|}\dx\dy\\
\leq -\int_{\R}\int_{|x-y|\leq\tilde{\delta}}\bm{G}_{\g_{n}}(x)\bm{G}_{\g_{n}}(y)|x-y|^{2}\log(|x-y|)\dx\dy \leq -\tilde{\delta}^{2}\log \tilde{\delta}\int_{\R}\int_{\R}\bm{G}_{\g_{n}}(x)\bm{G}_{\g_{n}}(y)\dy\dx
\end{multline*}
from which
\begin{equation}\label{eq:bound-small}
\sup_{n\geq N}\int_{\R}\int_{|x-y|\leq \tilde{\delta}}\bm{G}_{\g_{n}}(x)\bm{G}_{\g_{n}}(y)|x-y|^{2} {|\Lambda_{\g_{n}}(|x-y|)|}\dx\dy \leq -\tilde{\delta}^{2}\log\tilde{\delta}, \qquad \forall \tilde{\delta} \in \left(0,\frac{1}{e}\right).\end{equation}
On the other hand, for $R >1$ to be determined later,  {since $\g_{n} < k-2$ for any $n \geq N$ one has}
\begin{multline*}
\int_{\R}\int_{|x-y|>R}\bm{G}_{\g_{n}}(x)\bm{G}_{\g_{n}}(y)|x-y|^{2}\Lambda_{\g_{n}}(|x-y|)\dx\dy\\
\leq \int_{\R}\int_{|x-y|>R}\bm{G}_{\g_{n}}(x)\bm{G}_{\g_{n}}(y)|x-y|^{2}\Lambda_{k-2}(|x-y|)\dx\dy\end{multline*}
where we used that the mapping $\g \mapsto \Lambda_{\g}(r)$ is non-decreasing for any $r>1$. Then, 
$$|x-y|^{2}\Lambda_{k-2}(|x-y|) \leq \frac{1}{k-2}|x-y|^{k} \leq \frac{1}{(k-2)R^{s}}|x-y|^{k+s}, \qquad |x-y| >R,\,\,s >0$$
where we recall that $s >0$ has been chosen small enough so that $k+s \in (2+\delta,3-\delta)$. Therefore,
\begin{equation*}\begin{split}
\int_{\R}\int_{|x-y|>R}&\bm{G}_{\g_{n}}(x)\bm{G}_{\g_{n}}(y)|x-y|^{2}\Lambda_{\g_{n}}(|x-y|)\dx\dy\\
&\leq \frac{2^{k+s-1}}{(k-2)R^{s}}\int_{\R}\int_{|x-y|>R}\bm{G}_{\g_{n}}(x)\bm{G}_{\g_{n}}(y)\left(|x|^{k+s}+|y|^{k+s}\right)\dx\dy\\
&\leq \frac{2^{k+s}}{(k-2)R^{s}}M_{k+s}(\bm{G}_{\g_{n}})\,.\end{split}\end{equation*}
{We  then deduce from \eqref{Mk+s} that}
\begin{equation}\label{eq:bound-large}
\sup_{n\geq N}\int_{\R}\int_{|x-y|>R}\bm{G}_{\g_{n}}(x)\bm{G}_{\g_{n}}(y)|x-y|^{2}\Lambda_{\g_{n}}(|x-y|)\dx\dy \leq \bar{C}R^{-s} \qquad \forall R >1.\end{equation}
Since, for any fixed $\tilde{\delta} >0,R>1$, $\left(\Lambda_{\g_{n}}(r)\right)_{n}$ converges to $\log r$ \emph{uniformly on the set} $\{\tilde{\delta} \leq r \leq R\}$, we deduce that
\begin{multline}\label{eq:limLog}
\lim_{n\to\infty}\int_{\R}\int_{\tilde{\delta} \leq |x-y|\leq R}\bm{G}_{\g_{n}}(x)\bm{G}_{\g_{n}}(y)|x-y|^{2}\Lambda_{\g_{n}}(|x-y|)\dx\dy\\
=\int_{\R}\int_{\tilde{\delta}\leq|x-y|\leq R}H_{\lambda}(x)H_{\lambda}(y)|x-y|^{2}\log|x-y|\dx\dy.\end{multline}
Combining \eqref{eq1-gn} with \eqref{eq:bound-small}--\eqref{eq:bound-large} and \eqref{eq:limLog}, for any $\varepsilon >0$, picking $\tilde{\delta} >0$ small enough so that $-\tilde{\delta}^{2}\log\tilde{\delta}\leq \varepsilon,$ and $R >1$ large enough so that $\bar{C}R^{-s} \leq \varepsilon,$ one can take $N >1$ large enough so that
\begin{equation*}
\left|\int_{\R}\int_{\tilde{\delta} \leq |x-y|\leq R}H_{\lambda}(x)H_{\lambda}(y)|x-y|^{2}\log|x-y|\dx\dy \right|\leq {3\varepsilon}\end{equation*}
from which we deduce easily that
\begin{equation}\label{eq1Hl}
\int_{\R}\int_{\R}H_{\lambda}(x)H_{\lambda}(y)|x-y|^{2}\log|x-y|\dx\dy=0.\end{equation}
Now, recalling that $H_{\lambda}(x)=\lambda\bm{H}(\lambda x)$ for any $x \in \R$, with the change of variables $u=\lambda\,x$, $v=\lambda\,y$, \eqref{eq1Hl} becomes
$$\frac{1}{\lambda^{2}}\int_{\R}\int_{\R}\bm{H}(u)\bm{H}(v)|u-v|^{2}\log\left(\frac{|u-v|}{\lambda}\right)\d u\d v=0$$
from which
$$\log \lambda \int_{\R}\int_{\R}\bm{H}(u)\bm{H}(v)|u-v|^{2}\d u\d v=\int_{\R}\int_{\R}\bm{H}(u)\bm{H}(v)|u-v|^{2}\log |u-v|\d u\d v=2A_{0}.$$
Since 
$$\int_{\R}\int_{\R}\bm{H}(u)\bm{H}(v)|u-v|^{2}\d u\d v=2\int_{\R}|u|^{2}\bm{H}(u)\bm{H}(v)\d u\d v=2$$
we deduce the result. 
 \end{proof}

 \begin{proof}[Proof of Theorem~\ref{theo:Unique}]
  Lemma~\ref{lem:unique} proves that the weakly-$\star$ compact family $\left\{\Gg\right\}_{\g\in (0,1)}$ admits a \emph{unique} possible limit (as $\g \to0$) given by
$$\bm{G}_{0}(x):=\lambda_{0}\bm{H}(\lambda_{0}x), \qquad \lambda_{0}=\exp\left(A_{0}\right)$$
with $A_{0}$ defined {in \eqref{eq:A0}}. In particular, the whole net $\left\{\Gg\right\}_{\g \in (0,1)}$ is converging (in the weak-$\star$ topology) towards $\bm{G}_{0}$. We can then resume the arguments of  
Lemma \ref{lem:L2bound} and Lemma \ref{lem:momentsk} to deduce  \eqref{eq:estimGg}. \end{proof}

We can complement the estimates \eqref{eq:estimGg} in
Theorem~\ref{theo:Unique} with $L^{2}$-moments estimates.
\begin{cor}\label{L2-weighted}
  For any $\delta>0$ there exists $\g_{\star} \in (0,1)$ and $C >0$ such that
  \begin{equation}\label{eq:L2-weighted}
    \|\Gg \|_{L^2(\w_{k})} \leq C
  \end{equation}
  for all $\g \in [0,\g_{\star})$ with $k+\g \in (0,3-\delta)$ and all $\Gg \in \mathscr{E}_\gamma$.
\end{cor}

\begin{proof}
  We give a formal proof here which presents the argument to obtain a
  uniform bound. A complete justification can be found in
  Appendix~\ref{sec:justif-L2}. Let $\g_\star\in(0,1)$ be such that
  the conclusion of Theorem \ref{theo:Unique} holds true. For any
  $k>0$, setting
$$G_{k}(x)=\Gg(x)\,|x|^{k}$$
one notices that
\begin{equation}\label{eq:Gk}
\int_{\R}\Q_{\g}(\Gg,\Gg)\Gg\,|x|^{2k}\dx=\frac{1}{4}\int_{\R}\partial_x (x\Gg)\, \Gg\,|x|^{2k} \dx = \Big( \frac{1}{8} - \frac{k}{4}\Big)\|G_{k} \|^{2}_{L^2}\,.
\end{equation}
Also, thanks to Lemma \ref{lem:Sigmag} (with $\tilde{\delta} = \gamma^2$) and Theorem \ref{theo:Unique}
\begin{align*}
\int_{\R}\Q^{-}_{\gamma}(\Gg,\Gg)\,&\Gg\,|x|^{2k} \dx =  \int_{\R}G^{2}_{k}(x)\Sigma_{\g}(x)\,\dx \\
& \geq  \kappa_{\gamma}\|G_{k}\w_{\frac{\g}{2}}\|^2_{L^2} - {\gamma|\log\g|}\,C\,\|G_{k}\|^2_{L^2}\geq \big(\kappa_\gamma - {\gamma|\log\g|}\,C\big)\|G_{k}\|^2_{L^2}\,,
\end{align*} 
where we used that, for $\tilde{\delta}=\g^{2}$, $-(1-\tilde{\delta}^{\g}) \simeq 2\g\log\g$. For the positive part, using that  $|x+y|^{k}\leq 2^{k-1}\big(|x|^{k} + |y|^{k}\big)$ while $(|x|^{k}+|y|^{k})|x-y|^{\g} \leq 2\big(|x|^{k+\g} + |y|^{k+\g}\big)$, we can argue as in the derivation of \eqref{eq:QgQ0} to conclude that
\begin{align*}
\int_{\R}\Q^{+}_{\gamma}(\Gg,\Gg)\,\Gg\,|x|^{2k} \dx &\leq 2\int_{\R}\Q^{+}_{0}(\Gg |x|^{k+\gamma},\Gg)\, G_k \dx \\
& \leq  {2\sqrt{2}}M_{k+\gamma}(\Gg)\| \Gg\|_{L^2} \|G_{k}\|_{L^2}\,.
\end{align*}
Therefore,  one deduces from Theorem \ref{theo:Unique}  that there exists some positive constant $C_{0}$ depending neither on $k$, nor on $\g$ such that for $k+\g\in(0,3-\delta)$, 
$$ \int_{\R}\Q^{+}_{\gamma}(\Gg,\Gg)\,\Gg\,|x|^{2k} \dx \leq C_{0}\|G_{k}\|_{L^2}.$$
Gathering these estimates with \eqref{eq:Gk}, one deduces that
\begin{equation*}
\Big(\kappa_\gamma  + \frac{1}{8} - \frac{k}{4} - {\gamma|\log\g|}\,C\Big)\| G_{k} \|^{2}_{L^2} \leq C_{0}\| G_{k}\|_{L^2}\,.
\end{equation*}
Since $\kappa_\gamma\rightarrow1$ as $\gamma \to 0^{+}$, one easily concludes that for some explicit $\gamma_{\star} >0$ (independent of $k$), it holds $\kappa_\gamma  + \frac{1}{8} - \frac{k}{4} - \gamma\,{|\log \g|}\,C \geq \kappa_{\g} +\frac{1}{8}-\frac{3}{4}-\g|\log\g|C \geq \frac{1}{8}$ for any $\gamma\in[0,\gamma_{\star})$ which proves the result since then $\|G_{k}\|_{L^2} \leq 8C_{0}.$
\end{proof}

\subsection{Higher regularity}\label{sec:weighted}
In this section we prove Sobolev regularity of $\Gg$ uniformly with respect to $\g$. Parts of the arguments are formal while a full justification is given in Appendix~\ref{sec:justification:sobolev}.
 From equation \eqref{eq:steadyg} we write
\begin{equation*}
x\partial_{x}\Gg = 4\Q_{\g}(\Gg ,\Gg ) - \Gg\,.
\end{equation*}
Consequently, taking the $L^{2}(\w_{k})$ norm,   one has
\begin{align*}
\begin{split}
\| x\partial_{x}\Gg \|_{L^{2}(\w_{k})}  &\leq  4\| \Q_{\g}(\Gg ,\Gg )\|_{ L^{2}(\w_{k}) } + \| \Gg \|_{L^{2}(\w_{k})}\\
&\leq  {C\|\Gg\w_{k+\g}\|_{L^2}\,(\|\Gg\w_{k}\|_{L^1}+ \|\Gg\w_{\g}\|_{L^1})} + \| \Gg \|_{L^{2}(\w_{k})}\,,
\end{split}
\end{align*}
thanks to Proposition \ref{prop:QgL2}. Let $\delta>0$. Using now the uniform estimates obtained in Theorem \ref{theo:Unique} and Corollary \ref{L2-weighted}, we see that
\begin{equation}\label{eq:init-gradient-regularity}
  \sup_{\g \in (0,\g_{\star})}\|x\partial_{x}\Gg\|_{L^{2}(\w_{k})}=C_{1} < \infty, \qquad \forall   k+2\g <3-\delta.\end{equation}
In order to deduce from this some $L^{2}$-estimate for $\partial_{x}\Gg$, we need to handle the small values of $x$. Introducing now
$$\Gg'(x)=\partial_{x}\Gg(x)$$
one differentiate \eqref{eq:steadyg} to obtain that
\begin{equation}\label{gradient-equation}
\frac{1}{4}\partial_x(x\Gg') + \frac{1}{4}\Gg' = \Q_{\g}(\Gg,\Gg') + \Q_{\g}(\Gg',\Gg).
\end{equation}
Let us estimate each of the four terms in the right side in the following lemmata.
\begin{lem}[\textit{\textbf{Gain part estimate}}]\phantomsection\label{collision+}
 Let $\delta>0$ and $\g_\star\in(0,1)$ given by Corollary \ref{L2-weighted}. For any $\tilde{\delta}>0$, $\g \in(0,\g_\star)$ and  {$0\le k<3-\frac{5\gamma}{2}-\delta$} it holds that
\begin{equation}\label{eq:L2-weighted-q+}
\int_{\R}\Big(\Q^{+}_{\g}(\Gg,|\Gg'|) + \Q^{+}_{\g}(|\Gg'|,\Gg)\Big) \,|\Gg'|\,\w_{2k} \dx \leq C\sqrt{\tilde{\delta}}\| \Gg'\w_{k+\frac{\g}{2}}\|^2_{L^2} + \frac{C}{\tilde{\delta}^{\frac{5}{2}}}\,,
\end{equation}
for some explicit $C>0$.
\end{lem}
\begin{proof}
Both terms are estimated similarly, so we only focus on the first. As in the proof of \eqref{eq:QgQ0}, one first observes that
\begin{multline*}
\int_{\R}\Q^{+}_{\g}(\Gg,|\Gg'|) \,|\Gg'|\,\w_{2k} \dx\\
\leq \int_{\R^{2}}\left(|x|^{\g}+|y|^{\g}\right)\Gg(x)|\Gg'(y)|\,\left|\Gg'\left(\frac{x+y}{2}\right)\right|\w_{2k}\left(\frac{x+y}{2}\right)\dx\dy\\
\leq \int_{\R^{2}}\left(|x|^{\g}+|y|^{\g}\right)\Gg(x)|\Gg'(y)|\,\left|\Gg'\left(\frac{x+y}{2}\right)\right|\w_{k}\left(\frac{x+y}{2}\right) \w_{k}(x)\w_{k}(y) \dx\dy
\end{multline*}
where we used first 
 that $\w_{2k}(\cdot) {=} \w_{k}(\cdot)^{2}$ and then that 
$\w_{k}\left(\frac{x+y}{2}\right)\leq \w_{k}(x)\w_{k}(y).$ Consequently, using the fact that {$r \mapsto r^{\g}$} is concave, one checks that 
$$|x|^{\g}+|y|^{\g} \leq 2\w_{\frac{\g}{2}}(x)\w_{\frac{\g}{2}}(y)\w_{{\frac{\g}{2}}}\left(\frac{x+y}{2}\right)$$
from which we deduce that
\begin{multline*}
\int_{\R}\Q^{+}_{\g}(\Gg,|\Gg'|) \,|\Gg'|\,\w_{2k} \dx\\
\leq 2\int_{\R}\left(\w_{k+\frac{\g}{2}}(x)\Gg(x)\right)\left(\w_{k+\frac{\g}{2}}(y)|\Gg'(y)|\right)\w_{k+\frac{\g}{2}}\left(\frac{x+y}{2}\right)\left|\Gg'\left(\frac{x+y}{2}\right)\right|\dx\dy\\
={2}\int_{\R}\Q_{0}^{+}\left(\w_{k+\frac{\g}{2}}\Gg,\w_{k+\frac{\g}{2}}|\Gg'|\right)\w_{k+\frac{\g}{2}}|\Gg'|\dx.
\end{multline*}
Therefore, for any $\tilde{\delta} >0$, one has
\begin{align*}
\int_{\R}\Q^{+}_{\g}(\Gg,|\Gg'|) &\,|\Gg'|\,\w_{2k} \dx\\
&\leq {2}\int_{\R}\Q^{+}_{0}(\Gg\w_{k+\frac{\g}{2}},\big[\ind_{[-\tilde{\delta},\tilde{\delta}]} + \ind_{|x|>\tilde{\delta}}\big]|\Gg'|\w_{k+\frac{\g}{2}}) \,|\Gg'|\,\w_{k+\frac{\g}{2}} \dx \\
&\leq {4} \Big(\| \Gg\w_{k+\frac{\g}{2}}\|_{L^2}\|\ind_{[-\tilde{\delta},\tilde{\delta}]} \Gg'\w_{k+\frac{\g}{2}} \|_{L^1} \\
&\qquad\qquad + \| \Gg\w_{k+\frac{\g}{2}}\|_{L^1}\|\ind_{|x| > \tilde{\delta}} \Gg'\w_{k+\frac{\g}{2}} \|_{L^2} \Big)\| \Gg'\,\w_{k+\frac{\g}{2}}\|_{L^2}\,\,
\end{align*}
where we used the known estimates for $\Q^{+}_{0}$ (see Lemma \ref{lem:Q+0}). Consequently, using again Theorem \ref{theo:Unique} and Corollary \ref{L2-weighted} one deduces that there exists $C >0$ such that
$$\int_{\R}\Q^{+}_{\g}(\Gg,|\Gg'|) \,|\Gg'|\,\w_{2k} \dx \leq C\| \Gg'\,\w_{k+\frac{\g}{2}}\|_{L^2}\left(\|\ind_{[-\tilde{\delta},\tilde{\delta}]} \Gg'\w_{k+\frac{\g}{2}} \|_{L^1}  +\|\ind_{|x| > \tilde{\delta}} \Gg'\w_{k+\frac{\g}{2}} \|_{L^2}\right)$$
as soon as $\g \in (0,\g_{\star})$ and $k+\frac{3\g}{2} < 3-\delta$ where we applied Corollary \ref{L2-weighted} to $\| \Gg\w_{k+\frac{\g}{2}}\|_{L^2}$. Now, one has
\begin{equation*}
\|\ind_{[-\tilde{\delta},\tilde{\delta}]} \Gg'\w_{k+\frac{\g}{2}} \|_{L^1} \leq \sqrt{\tilde{\delta}}\,\| \Gg'\w_{k+\frac{\g}{2}}\|_{L^2}\,,
\end{equation*}
whereas, thanks to \eqref{eq:init-gradient-regularity}
\begin{equation*}
  \|\ind_{|x|>\tilde{\delta}} \Gg'\w_{k+\frac{\g}{2}} \|_{L^2} \leq \frac{1}{\tilde{\delta}}\,\| x\Gg'\|_{L^{2}(\w_{k+\frac{\g}{2}})} \leq \frac{C_{1}}{\tilde{\delta}},\qquad  0 \le k<3-\frac{5\gamma}{2}-\delta\,.
\end{equation*}
Thus
\begin{equation*}
\int_{\R}\Q^{+}_{\g}(\Gg,|\Gg'|) \,|\Gg'|\,\w_{2k} \dx\leq C\sqrt{\tilde{\delta}}\| \Gg'\w_{k+\frac{\g}{2}}\|^2_{L^2} + \frac{C_{1}}{\tilde{\delta}}\| \Gg'\w_{k+\frac{\g}{2}}\|_{L^2}\,.
\end{equation*}
The result follows from here using Young's inequality.
\end{proof}
The loss operator is estimated in the following
\begin{lem}[\textit{\textbf{Loss part estimate}}]\phantomsection\label{collision-} Let $\delta>0$ and $\g_\star\in(0,1)$ given by Corollary \ref{L2-weighted}. There exists some positive constant $C >0$, such that, for any $\tilde{\delta}>0$, $\g \in(0,\g_{\star})$ and  $0\le k <3-\gamma-\delta$ it holds that
\begin{multline}\label{eq:L2-weighted-q-}
\int_{\R}\left[\Q^{-}_{\g}(\Gg',\Gg)+\Q_{\g}^{-}(\Gg,\Gg')\right]\,\Gg'\,\w_{2k}\dx \\
\geq \left(\kappa_{\g}-C\sqrt{\tilde{\delta}}\right)\|\Gg'\w_{k+\frac{\g}{2}}\|_{L^2}^{2}-\frac{C}{\tilde{\delta}^{\frac{5}{2}}}- {C\g |\log\g|}\|\Gg'\w_{k}\|_{L^2}^{2}\,\end{multline}
where $\kappa_{\g}$ has been defined in Lemma \ref{lem:Sigmag}.
\end{lem}
\begin{proof} One has
\begin{multline*}
\mathcal{J}:=\int_{\R}\left[\Q^{-}_{\g}(\Gg',\Gg)+\Q_{\g}^{-}(\Gg,\Gg')\right]\,\Gg'\,\w_{2k}\dx=\int_{\R}\left(\Gg'(x)\right)^{2}\Sigma_{\g}(x)\w_{2k}(x)\dx\\
+\int_{\R^{2}}\Gg'(y)\Gg(x)|x-y|^{\g}\Gg'(x)\w_{2k}(x)\dx\dy=\mathcal{J}_{1}+\mathcal{J}_{2}.\end{multline*}
The first term $\mathcal{J}_{1}$ is easily estimated using Lemma \ref{lem:Sigmag} {(with $\tilde{\delta}=\gamma^2$)} and Theorem  \ref{theo:Unique}, as in the proof of Corollary \ref{L2-weighted}.  For $\g_{\star}$ given by Corollary \ref{L2-weighted}, one has for any $\g\in(0,\g_\star)$,
$$\mathcal{J}_{1} \geq  \kappa_{\gamma}\| \Gg' \w_{k+\frac{\g}{2}}\|^2_{L^2} - {C\g |\log \g|} \| \Gg'\w_{k}\|^2_{L^2}\,$$
for some positive constant $C>0$ independent of $\g$ and $k$. To estimate $\mathcal{J}_{2}$, we introduce a smooth cutoff function $0 \leq \chi(x) \leq 1$ with support in the unitary interval $[-1,1]$ and set $\chi_{\tilde{\delta}}(x)=\chi(\tilde{\delta}^{-1}x)$. For any $x \in \R,$ one has then
\begin{multline*}
\int_{\R}\Gg'(y) |x-y|^{\gamma}\dy = \int_{\R}\Gg'(y) \chi_{\tilde{\delta}}(x-y)|x-y|^{\gamma}\dy +   \int_{\R}\Gg'(y) (1-\chi_{\tilde{\delta}}(x-y))|x-y|^{\gamma}\dy\\
=\int_{\R}\Gg'(y) \chi_{\tilde{\delta}}(x-y)|x-y|^{\gamma}\dy  \\
-  \int_{\R} \Gg(y) \partial_y\big[ (1-\chi_{\tilde{\delta}}(x-y))|x-y|^{\gamma} \big]\dy.
\end{multline*}
Notice that, for $\tilde{\delta} \in (0,1)$,
$$\left|\int_{\R}\Gg'(y) \chi_{\tilde{\delta}}(x-y)|x-y|^{\gamma}\dy\right| \leq \tilde{\delta}^{\g}\int_{|x-y|<\tilde{\delta}}|\Gg'(y)|\dy \leq {\sqrt{2\tilde{\delta}}}\|\Gg'\|_{L^2}\,,$$
while
$$\left|\int_{\R} \Gg(y) \partial_y\big[ (1-\chi_{\tilde{\delta}}(x-y))|x-y|^{\gamma} \big]\dy\right| \leq \frac{C}{\tilde{\delta}}\|\Gg\|_{L^1}=\frac{C}{\tilde{\delta}}$$
where we used the fact that
\begin{align*}
\left|\partial_y\big[ (1-\chi_{\tilde{\delta}}(x-y))|x-y|^{\gamma}\right|&=\left|\chi'_{\tilde{\delta}}(x-y)|x-y|^{\g} {-}\g\left(1-\chi_{\tilde{\delta}}(x-y)\right)(x-y)|x-y|^{\g-2}\right|\\
&\leq \|\chi_{\tilde{\delta}}'\|_{\infty}|x-y|^{\g}{\ind_{|x-y|\leq \tilde{\delta}}}+\g\left(1-\chi_{\tilde{\delta}}(x-y)\right)|x-y|^{\g-1}\\
&\leq \left(\|\chi'\|_{\infty}+\g\right)\tilde{\delta}^{\g-1}, \qquad \tilde{\delta} \in (0,1).\end{align*}
Consequently,
$$\left|\int_{\R}\Gg'(y) |x-y|^{\gamma}\dy\right| \leq  \sqrt{2\tilde{\delta}}\|\Gg'\w_{k}\|_{L^2}+\frac{C}{\tilde{\delta}}\,,$$
and
\begin{multline*}
|\mathcal{J}_{2}| \leq \left( {\sqrt{2\tilde{\delta}}\|\Gg'\w_{k}\|_{L^2}}+\frac{C}{\tilde{\delta}}\right)\int_{\R}\Gg(x)|\Gg'(x)|\w_{2k}(x)\dx \\
\leq \left( {\sqrt{2\tilde{\delta}}}\|\Gg'\w_{k}\|_{L^2}+\frac{C}{\tilde{\delta}}\right)\|\Gg \w_{k}\|_{L^2}\,\|\Gg'\w_{k}\|_{L^2}.\end{multline*}
Using again Corollary \ref{L2-weighted} to estimate $\|\Gg\w_{k}\|_{L^2}$ for $k+\g<3-\delta$, we deduce using Young's inequality that there exists $C >0$ such that
$$|\mathcal{J}_{2}| \leq C\sqrt{\tilde{\delta}}\| \Gg'\w_{k}\|^2_{L^2} + \frac{C}{\tilde{\delta}^{\frac{5}{2}}}.$$
Since $\|\Gg'\w_{k}\|_{L^2}^{2} \leq \| \Gg'\w_{k+\frac{\g}{2}}\|^2_{L^2}$, we deduce then the Lemma from the bound on $\mathcal{J}_{1}$ and $|\mathcal{J}_{2}|$.
\end{proof} 
We have all in hands, starting from \eqref{gradient-equation} to deduce the following
\begin{theo}\label{theo:gradient}
Let $\delta>0$. There exists $\g_\star\in(0,1)$ such that, for any  $\g \in(0,\g_\star)$ and $0\leq k<3-\frac{5\gamma}{2}-\delta$ it holds that
\begin{align}\label{eq:gradient}
\| \Gg\|_{H^{1}(\w_k)} + \| \Gg\|_{W^{1,1}}\leq C\,,
\end{align}
for some explicit $C>0$ depending on $\g_{\star}$ but not $\g$ and $k$.  In particular, as a consequence of the $W^{1,1}$ control of $\Gg$, it holds that
\begin{equation*}
\big| \widehat\Gg(\xi)\big|\leq \frac{C}{1+|\xi|}\,.
\end{equation*}
\end{theo}
\begin{proof} Let us fix  $\g \in(0,\g_\star)$ and $0\leq k<3-\frac{5\gamma}{2}-\delta$, where $\g_\star\in(0,1)$ is given by Corollary \ref{L2-weighted}.
Multiply equation \eqref{gradient-equation} by $\Gg'\w_{2k}$ and integrate to obtain 
\begin{equation*}
 {\frac38\|\Gg'\w_{k}\|^2_{L^2}-\frac k4 \int_\R|x| \w_{2k-1}(x)(\Gg'(x))^2\dx}
  =\int_{\R}\big( \Q_{\g}(\Gg,\Gg') + \Q_{\g}(\Gg',\Gg) \big) \Gg' w_{2k}\dx\,,
\end{equation*}
where we used integration by parts and the fact  that $x\partial_x\w_{2k}(x)=2k|x|\w_{2k-1}(x)$ to show that
$$\int_{\R}\partial_x\left(x\Gg'(x)\right) {\Gg'(x)}\w_{2k}(x)\dx= {\frac{1}{2}\int_{\R}\left[\Gg'(x)\right]^{2}\w_{2k}(x)\dx -k\int_{\R}\left[\Gg'(x)\right]^{2}|x|\w_{2k-1}(x)\dx}.$$
Consequently, using Lemmata \ref{collision+} and \ref{collision-}, there is $C >0$ such that, for any $\tilde{\delta} >0$,  it holds that
\begin{equation*}
\left( \frac38-\frac k4 - C {\g |\log\g| }\right) \big\| \Gg'\w_{k} \|^{2}_{L^2} \leq -\big(\kappa_\gamma - 2C\sqrt{\tilde{\delta}}\big)\| \Gg'\w_{k+\frac{\g}{2}}\|^2_{L^2} + \frac{2C}{\tilde{\delta}^{\frac{5}{2}}}\,.
\end{equation*}
Recalling that $\lim_{\g\to0}\kappa_{\g}=1$, we can fix $\g_{\star}$ small enough (up to reducing our previous $\g_{\star}$) and $\tilde{\delta}>0$ such that $\kappa_\gamma - 2C\sqrt{\tilde{\delta}}>\frac{3}{4}$ for any $\g\in (0,\g_{\star})$ and conclude that
  $$\left( \frac98-\frac k4 - C\g |\log\g| \right) \big\| \Gg'\w_{k} \|^{2}_{L^2} \leq \frac{2C}{\tilde{\delta}^{\frac{5}{2}}}\,. $$
Finally, since $\frac98-\frac k4>\frac98-\frac34>0$, up to taking $\g_{\star}$ still smaller, one has $\frac98-\frac k4 - C \g |\log\g|>\frac98-\frac 34 - C \g |\log\g|>0$ for any $\g\in (0,\g_{\star})$ and thus 
$$\sup_{\g\in [0,\g_{\star})}\|{\Gg'}\|_{L^{2}(\w_k)}\leq \bar{C}\,, $$
  for some constant $\bar{C}$ independent of $k$.
  
For the $L^{1}$ estimate on the gradient, return to equation \eqref{gradient-equation}, multiply it by $\mathrm{sign}(\Gg')$ and integrate to obtain that
\begin{equation*}
\frac{1}{4}\| \Gg'\|_{L^1} = \int_{\R}\big( \Q_{\g}(\Gg,\Gg') + \Q_{\g}(\Gg',\Gg) \big) \mathrm{sign}(\Gg')\dx\,.
\end{equation*}
To estimate the right-hand side, one simply notices from the weak-form  \eqref{eq:weakgamma} that
$$\int_{\R}\big( \Q_{\g}(\Gg,\Gg') + \Q_{\g}(\Gg',\Gg) \big) \mathrm{sign}(\Gg')\dx 
\leq 3\int_{\R}\Sigma_{\g}(y)\,|\Gg'(y)|\dy.$$
According to Jensen's inequality
\begin{equation}\label{eq:frequency}
\Sigma_{\g}(y)=\int_{\R}|x-y|^{\g}\Gg(x)\dx \leq \left|y-\int_{\R}x\Gg(x)\dx\right|^{\g}\leq |y|^{\g}, \qquad \forall y \in \R, 
\end{equation}
from which
\begin{equation*}\begin{split}
\int_{\R}\big( \Q_{\g}(\Gg,\Gg') + \Q_{\g}(\Gg',\Gg) \big) \mathrm{sign}(\Gg')\dx &\leq 3\int_{\R}|y|^{\g}|\Gg'(y)|\dy\\
&\leq 3\| \Gg' \|_{L^{2}(\w_1)}\left(\int_{\R}\frac{|y|^{2\g}}{\left(1+|y|\right)^{2}}\dy\right)^{\frac{1}{2}}\,.
\end{split}\end{equation*}
The last integral is finite for  {$\g \in [0,\frac{1}{2})$} and can be estimated uniformly with respect to $\g$ for, say,  {$\g \in [0,\frac{1}{3})$}. Then, from the first part of the proof, since $\sup_{\g \in (0,\g_{\star})}\|\Gg'\|_{L^{2}(\w_{1})} < \infty$, we deduce that
$$\frac{1}{4}\|\Gg'\|_{L^1} \leq C_0, \qquad \forall \g \in (0,\g_{\star})$$
which proves the $W^{1,1}$ estimate.\end{proof}

Since, according to Theorem \ref{theo:gradient}, the family
$\{\Gg\}_{\g\in (0,\g_{\star})}$ is bounded in $H^1(\R)$, we get immediately the following corollary.

\begin{cor}\label{cor:hoelder}
 Under the assumption of Theorem~\ref{theo:gradient} there exists some positive constant $C >0$ such that
\begin{equation}\label{eq:Holder}
\sup_{\g \in (0,\g_{\star})}\|\Gg\|_{L^\infty} \leq C\, \qquad \quad \left|\Gg(x)-\Gg(y)\right| \leq C\,|x-y|^{\frac{1}{2}}, \qquad \forall x,y \in \R.\end{equation}
\end{cor}

One has the following estimate for differences of two equilibrium solutions.
\begin{lem}\phantomsection\label{lem:Diffmom} Let $\delta>0$ and $\g_\star\in(0,1)$ given by Corollary \ref{L2-weighted}. Let $\g \in (0,\g_{\star})$ and $\Gg^{1},\Gg^{2}\in \mathscr{E}_{\g}$ be given. For any $2 < k<3-\gamma-\delta$, there exists $\g_{\star}(k) >0$ and  $C_{k} >0$ such that
$$\|\Gg^{1}-\Gg^{2}\|_{L^{1}(\w_{k+\g})} \leq C_{k}\|\Gg^{1}-\Gg^{2}\|_{L^{1}(\w_{\g+\frac{2k}{3}})} \qquad \forall \g \in (0,\g_{\star}(k)).$$
\end{lem}
\begin{proof} The proof follows the argument of the proof of Lemma \ref{lem:momentsk}. Let $\g_\star\in(0,1)$ given in Theorem \ref{theo:Unique} and $\g\in(0,\g_\star)$. We introduce $g_{\g}=\Gg^{2}-\Gg^{1}$ and observes that
\begin{equation}\label{eq:diffe}
\frac{1}{4}\partial_{x}\left(xg_\g(x)\right)=\Q_{\g}(g_\g,\Gg^{2})+\Q_{\g}(\Gg^{1},g_{\g}).\end{equation}
We multiply then 
\eqref{eq:diffe} by $\mathrm{sign}(g_{\g})|x|^{k}$ and integrate over $\R$ to deduce
\begin{equation*}\begin{split}
-\frac{k}{4}M_{k}\left(|g_{\g}|\right)&=\int_{\R}\left[\Q_{\g}(g_{\g},\Gg^{2})+\Q_{\g}(\Gg^{1},g_{\g})\right]\mathrm{sign}(g_{\g}(x))|x|^{k}\dx\\
&=\int_{\R^{2}}g_{\g}(x)\bm{S}_{\g}(y)|x-y|^{\g}
\left[{2}\,\mathrm{sign}\left(g_{\g}\left(\frac{x+y}{2}\right)\right)\left|\frac{x+y}{2}\right|^{k}\right.\\
&\phantom{+++++}\left.-\mathrm{sign}(g_{\g}(x))|x|^{k}-\mathrm{sign}(g_{\g}(y))|y|^{k}\right]\dx\dy\\
&\leq -\int_{\R}\sigma_{\g}(x)|g_{\g}(x)|\,|x|^{k}\dx + \int_{\R^{2}}|g_{\g}(x)|\,\bm{S}_{\g}(y)|x-y|^{\g}|y|^{k}\dy\dx\\
&\phantom{+++}+ {2}\int_{\R^{2}}|g_{\g}(x)|\,\bm{S}_{\g}(y)|x-y|^{\g}\left|\frac{x+y}{2}\right|^{k}\dx\dy
\end{split}\end{equation*}
where
$$\bm{S}_{\g}=\frac{1}{2}\left(\Gg^{2}+\Gg^{1}\right), \qquad \sigma_{\g}(x)=\int_{\R}\bm{S}_{\g}(y)|x-y|^{\g}\dy, \qquad x \in \R.$$
Arguing exactly as in Lemma \ref{lem:momentsk}, one deduces {without} difficulty that
\begin{multline*}
-\frac{k}{4}M_{k}\left(|g_{\g}|\right) \leq  {-(1-2^{1-k})}\int_{\R}\sigma_{\g}(x)|g_{\g}(x)|\,|x|^{k}\dx \\
+\left(1+ {2^{1-k}}\right)\,\left[M_{\g}(|g_{\g}|)M_{k}(\bm{S}_{\g})+M_{0}(|g_{\g}|)M_{k+\g}(\bm{S}_{\g})\right]\\
+ {6}\left[M_{\g+\frac{2k}{3}}(|g_{\g}|)M_{\frac{k}{3}}(\bm{S}_{\g})+M_{\frac{2k}{3}}(|g_{\g}|)M_{\frac{k}{3}+\g}(\bm{S}_{\g})\right.\\
\phantom{++++}\left.+M_{\frac{k}{3}+\g}(|g_{\g}|)M_{\frac{2k}{3}}(\bm{S}_{\g}) + M_{\frac{k}{3}}(|g_{\g}|)M_{\frac{2k}{3}+\g}(\bm{S}_{\g})\right].
\end{multline*}
One deduces from Theorem \ref{theo:Unique} that there exists $C>0$ independent of $k$ such that  $\|\bm{S}_{\g}\|_{L^{1}(\w_{k+\g})} \le C$ for any $\g\in (0,\g_{\star})$ and $k+\gamma\in(0,3-\delta)$. Using this bound and estimating every moment of $|g_{\g}|$ by $\|g_{\g}\|_{L^{1}(\w_{\frac{2k}{3}+\g})}$, yields
$$-\frac{k}{4}M_{k}\left(|g_{\g}|\right)
 + {(1-2^{1-k})}\int_{\R}\sigma_{\g}(x)|g_{\g}(x)|\,|x|^{k}\dx  \leq C'\|g_{\g}\|_{L^{1}(\w_{\g+\frac{2k}{3}})}$$
 for some suitable $C' >0$ depending neither on $k$ nor on $\g$. Of course, one checks easily that $\sigma_{\g}$ satisfies a bound as in Lemma \ref{lem:Sigmag}, i.e.
\begin{equation*}\sigma_{\g}(y) \geq \overline{\kappa}_{\g}\,\w_{\g}(y) - (1-\tilde{\delta}^{\g})-\sqrt{2\tilde{\delta}}\|\bm{S}_{\g}\|_{L^2}, \qquad \qquad \forall \tilde{\delta} \in (0,1)\end{equation*}
for some explicit $\overline{\kappa}_{\g}$ with 
$\lim_{\g\to0}\overline{\kappa}_{\g}=1.$ Of course, according to Theorem \ref{theo:Unique}, $\sup_{\g\in (0,\g_{\star})}\|\bm{S}_{\g}\|_{L^{2}} \leq C.$ We can then, as in Lemma \ref{lem:momentsk}, fix $\varepsilon >0$ and choose $\g$ small enough and $\tilde{\delta}$ small enough so that $1-\tilde{\delta}^{\g}+\sqrt{2\tilde{\delta}}\|\bm{S}_{\g}\|_{L^{2}}\leq \varepsilon$ and then, for a suitable choice of $\g_{\star}(k)$ such that 
$$\overline{\kappa}_{\g}\left(1-{2^{1-k}}-\frac{k}{4\overline{\kappa}_{\g}}-\frac{\varepsilon}{\overline{\kappa}_{\g}}\right) \geq \frac{\sigma_{k}}{2} \qquad \forall \g \in (0,\g_{\star}(k))$$ 
with ${\sigma_{k}:=1-2^{1-k}-\frac{k}{4} >0}.$ This gives then, as in Lemma \ref{lem:momentsk}, 
$$\|g_{\g}\|_{L^{1}(\w_{\g+k})} \leq \frac{2C'_{k}}{\sigma_{k}}\|g_{\g}\|_{L^{1}(\w_{\g+\frac{2k}{3}})} \qquad \qquad \forall \g \in (0,\g_{\star}(k))$$
which is the desired estimate with $C_{k}=\frac{2C'_{k}}{\sigma_{k}}.$
\end{proof}

\section{Stability and uniqueness}\label{sec:sec3}

We are now in position to quantify first the stability of the profile $\bm{G}_{0}$ in the limit $\g \to 0^{+}$ and deduce from this the uniqueness of the steady profile $\Gg \in \mathscr{E}_{\g}$ for $\g$ small enough. 

\subsection{Stability of the profile -- upgrading the convergence}\label{sec:upgrade}

The results of Section \ref{sec:apost} ensure the convergence (in a weak-$\star$ sense) of $\bm{G}_{\g}$ towards $\bm{G}_{0}$ as $\g \to 0.$ We upgrade here the convergence to the (strong) $L^{1}(\w_{a})$ topology and, more importantly, provide also a \emph{quantitative} estimate of $\|\Gg-\bm{G}_{0}\|_{L^{1}(\w_{a})}.$ To do so, we {will resort to a comparison between the collision operator $\Q_{\g}$ and the operator $\Q_{0}$ (corresponding to Maxwellian interactions) given in Proposition \ref{diff_Q} in  Appendix \ref{app:QgQ0}.} 

Let us denote by ${\mathcal N}_0(f)$ the self-similar operator associated to the Maxwellian case $\gamma=0$, that is
$$  {\mathcal N}_0(f) = -\frac14 \partial_x(xf)+\Q_{0}(f,f).$$
Let us denote by ${\mathcal N}_\gamma(f)$ the self-similar operator associated to the general case $\gamma>0$, that is
$$  {\mathcal N}_\gamma(f) = -\frac{1}{4} \partial_x(xf)+\Q_\gamma(f,f).$$

\begin{lem}\phantomsection\label{N0_Ggamma}
  Let $2<a<3$ and $\delta>0$ such that $a<3-\delta$. Let $\g_{\star} \in (0,1)$ be defined in Corollary \ref{L2-weighted} (notice $\g_{\star}$ depends on $\delta$ and thus on $a$). For any $\gamma \in (0,\g_{\star}),s>0$ satisfying  $s+\gamma+a<3-\delta$, there exists $C_{0} >0$ depending only on $s$  such that, for any profile $\Gg \in \mathscr{E}_{\g}$,
 $$  \|{\mathcal N}_0(\Gg)\|_{L^1(\w_{a})}\le  C_{0}\g^{\frac{s}{s+1}}\left(1+|\log\g|\right).$$
 \end{lem}
\begin{proof}
  Since ${\mathcal N}_\gamma(\Gg)=0$, one has
  \begin{align*}
    \|{\mathcal N}_0(\Gg)\|_{L^1(\w_{a})} & = \|{\mathcal N}_0(\Gg)-{\mathcal N}_\gamma(\Gg)\|_{L^1(\w_{a})} \\
    & \le   \|\Q_{0}(\Gg,\Gg)-\Q_\gamma(\Gg,\Gg)\|_{L^1(\w_{a})}
\end{align*}
Noticing that, according to \eqref{eq:estimGg} and \eqref{eq:L2-weighted}, there exists $C>0$ such that, for any $\gamma \in (0,\g_{\star}),s>0$ satisfying  $s+\gamma+a<3-\delta$,
  $$\max\left(\|\Gg\|_{L^1(\w_{a+s+\g})},\|\Gg\|_{L^1(\w_{a})}\right) \le C,\qquad \text{ and } \quad \|\Gg\|_{L^2(\w_{a})}\le C,$$
the result then follows from  Proposition \ref{diff_Q}.\end{proof}

We introduce here the following steady state of $\mathcal{N}_{0}$ with the same mass, momentum and energy of $\Gg$, namely
$${\bm{h}_{\g}}(x)=H_{\lambda_{\g}}(x)=\lambda_{\g} {\bm{H}}(\lambda_{\g}x), \qquad \qquad \lambda_{\g}=\frac{1}{\sqrt{M_{2}(\Gg)}}, \qquad \g \in (0,1).$$
Since
$$\lim_{\g \to0^{+}}M_{2}(\Gg)=\int_{\R}\bm{G}_0(x)|x|^{2}\dx$$
we have
$$\lim_{\g \to 0^{+}}\lambda_{\g}=\lambda_{0}$$
and, noticing that
$$\left|\bm{h}_{\g}(x)-\bm{G}_{0}(x)\right| \leq C\left|\lambda_{\g}-\lambda_{0}\right|\bm{G}_{0}(x),$$
for some $C$ that can be made independent on $\g$, we have
\begin{equation}\label{stab} \|\bm{h}_{\g}-\bm{G}_{0}\|_{L^{1}(\w_{a})} \leq C_{a}\left|\lambda_{\g}-\lambda_{0}\right| \qquad \forall a \in (0,3).
\end{equation}
To compare then $\Gg$ to $\bm{G}_{0}$, it is enough to compare $\Gg$ to $\bm{h}_{\g}.$ This is the object of the following
\begin{prp}\phantomsection\label{lem:stabil}
Let $2<a<3$. There exist $\g_\star\in(0,1)$  and  an \emph{explicit} function $\eta=\eta(\gamma)$ depending on $a$, with $\lim_{\g\to0^+}\eta(\gamma)= 0$, such that, for any $\g\in(0,\g_\star)$, any $\Gg\in\mathscr{E}_{\g}$, $$ \|\Gg-\bm{h}_{\g}\|_{L^1(\w_{a})} \le \eta( \gamma).$$
 \end{prp}
 
\begin{proof} Let us denote by $g(t,x)$ the solution to \eqref{eq:IB-selfsim} with initial condition $\Gg$. Then, for every $t\ge 0$, 
\begin{equation}\label{diff_G}
  \|\Gg-\bm{h}_{\g}\|_{L^1(\w_{a})} \le \|\Gg-g(t)\|_{L^1(\w_{a})}+\|g(t)-\bm{h}_{\g}\|_{L^1(\w_{a})}.
  \end{equation}
In order to obtain a bound for $\|g(t)-\bm{h}_{\g}\|_{L^1(\w_{a})}$, we shall use the convergence of $g(t)$ towards $\bm{h}_{\g}$ as $t\to \infty$ given in Fourier norm by Theorem \ref{k-norm-cvgce} (see also Remark \ref{rem:scal}). Choosing 
$$a_{*} > a, \qquad 0 < \alpha < \frac{2(a_*-a)}{2a_*+1}\,,$$
it follows from Lemmas \ref{L1-L2} and  \ref{L2-knorm} that, for any $\beta >0$ and $0 < r <1$, 
\begin{align}\phantomsection
\|g(t)-\bm{h}_{\g}\|_{L^1(\w_{a})} &\le  C \|g(t)-\bm{h}_{\g}\|^\alpha_{L^2} \left(\|g(t)\|^{1-\alpha}_{L^1(\w_{a_*})} + \|\bm{h}_{\g}\|^{1-\alpha}_{L^1(\w_{a_*})} \right) \nonumber\\
& \le  C_{r,\beta,\alpha} \vertiii{\widehat{g}(t)-\widehat{\bm{h}_{\g}}}_{{a}}^{\alpha(1-r)}\left(\|g(t)\|^{1-\alpha}_{L^1(\w_{a_*})} + \|\bm{h}_{\g}\|^{1-\alpha}_{L^1(\w_{a_*})} \right) \nonumber\\
& \times \left(\|g(t)\|_{H^M}^{r} +\|\bm{h}_{\g}\|_{H^M}^{r} +\|g(t)\|_{H^N}^{r}+\|\bm{h}_{\g}\|_{H^N}^{r}\right)^\alpha \label{eq:alignL2HM}
\end{align}
for some explicit constant $C_{r,\beta,\alpha}$ depending on $\alpha,r,\beta$ and where $M=a\frac{(1-r)}{r},$ $N=M+\frac{(1-r)(\beta+1)}{2r} \geq M$. {Notice that, choosing $r$ as close as desired from $1$, we can assume $N \leq 1$.}
Observing  that
$$\|\bm{h}_{\g}\|_{H^{m}}^{2}=\int_{\R}\left(1+|\xi|^{2}\right)^{m}\left(1+\lambda_{\g}^{-1}|\xi|\right)^{2}\exp\left(-\frac{2}{\lambda_{\g}}|\xi|\right)\d\xi$$
where $\lambda_{\g}$ is bounded from below for $\g$ small enough (recall that $\lim_{\g\to0^{+}}\lambda_{\g}=\lambda_{0}$), one easily checks that
$$\sup_{\g \in (0,\g_{\star})}\|\bm{h}_{\g}\|_{H^{m}} < \infty$$
for any $m \in \R^{+}$ whereas, for $a_{*} < 3$, $\sup_{\g\in (0,\g_{\star})}\|\bm{h}_{\g}\|_{L^{1}(\w_{a_{*}})} < \infty$. So that there is $C_{0} >0$ (depending on $r,\alpha,\beta$ and $a_{*}$ but not on $\g$) such that
\begin{equation}\label{eq:gtHg}
\|g(t)-\bm{h}_{\g}\|_{L^{1}(\w_{a})} \leq C_{0} \vertiii{\widehat{g}(t)-\widehat{\bm{h}_{\g}}}_{{a}}^{\alpha(1-r)}\left(1+\|g(t)\|_{L^{1}(\w_{a_{*}})}^{1-\alpha}\right)\left(1+\|g(t)\|_{H^{N}}^{ r}\right)\end{equation}
where we recall that $N \geq M$. Let $\delta>0$ such that $a_{*}<3-\delta$. Let $\g_\star$ be such that the results of Theorem \ref{theo:Unique}, Corollary \ref{L2-weighted} and Theorem \ref{theo:gradient} hold and such that $a+\g_\star<3-\delta$. Now, by virtue of Theorem \ref{theo:gradient}, the initial datum $g(0)=\Gg$ is such that (recall that $N \leq 1$),
\begin{equation*}\label{eq:GgSobolev}
\sup_{\g \in (0,\g_\star)}\|\Gg\|_{H^N} < \infty \qquad \text{ as well as } \qquad    \left|\widehat{\Gg}(\xi)\right | \leq (1+|\xi|)^{-1} \qquad \xi \in \R\end{equation*}
which, according to Theorem \ref{theo:Sobolev} implies that  {there exists  $C >0$ such that
\begin{equation*}
  \|g(t)\|_{H^{N}} \leq C, \qquad \qquad \forall t\geq0, \qquad \g \in (0,\g_{\star})\,.\end{equation*}
Let us now show that we also have a uniform bound with respect to $t$ and $\g$ for  $\|g(t)\|_{L^1(\w_{a_{*}})}$ .
First, for $2<k<3$, one has
$$  \frac{\d}{\d t}\int_{\R}|x|^k g(t,x)\dx + \left(1-\frac{k}{4}\right) \int_{\R}|x|^k g(t,x)\dx =  \int_{\R^2}g(t,x)g(t,y) \,\left|\frac{x+y}{2}\right|^k \dx \dy.$$
We then deduce from \eqref{ineqxplusy} that
\begin{multline*} \frac{\d}{\d t}\int_{\R}|x|^k g(t,x)\dx + \left(1-\frac{k}{4}-2^{1-k}\right) \int_{\R}|x|^k g(t,x)\dx \\
  \le  3 \times 2^{1-k} \int_{\R}g(t,x) |x|^{\frac{2k}3}\dx  \int_{\R} g(t,y) |y|^{\frac{k}{3}}\ \dy 
    \le 3 \times 2^{1-k} M_2(\Gg)^2,
\end{multline*}
since $k<3$. It thus follows that, for any $t\ge 0$, 
$$\int_{\R}|x|^k g(t,x)\dx \le \min\left( \int_{\R}|x|^k \Gg(x)\dx, \frac{3 \times 2^{1-k} M_2(\Gg)^2}{1-\frac{k}{4}-2^{1-k}}\right).$$
Consequently, it follows from Theorem \ref{theo:Unique} that there exists  $C >0$ such that, for $2<a<a_{*}<3-\delta$
\begin{equation*}\label{eq:Cgt}
  \|g(t)\|_{L^1(\w_{a_{*}})} \leq C, \qquad \qquad \forall t\geq0,\qquad \g \in (0,\g_{\star})\,.\end{equation*}}
We deduce then from \eqref{eq:gtHg} and  {Theorem} \ref{k-norm-cvgce}, that, for any $t\geq0$
\begin{multline}\label{L1a-conv}
\|g(t)-\bm{h}_{\g}\|_{L^1(\w_{a})} \le C_{1} e^{- \alpha\sigma(1-r)t} \vertiii{\widehat{\Gg}-\widehat{\bm{h}_{\g}}}_{a}^{ \alpha(1-r)} \\
\leq  \tilde{C}_{1} e^{- \alpha\sigma(1-r)t} \|\Gg-\bm{h}_{\g}\|_{L^1(\w_{a})}^{ \alpha(1-r)}\,,\end{multline}
 for some positive constants $C_{1},\tilde{C}_{1}$ independent of $\g \in (0,\g_{\star})$ and where we used Lemma \ref{lem:mukvertk} to bound $\vertiii{\widehat{\Gg}-\widehat{\bm{h}_{\g}}}_{a}$ by $\|\Gg-\bm{h}_{\g}\|_{L^1(\w_{a})}.$

Let us now look for a bound of $\|\Gg-g(t)\|_{L^1(\w_{a})}$. We deduce from \eqref{eq:IB-selfsim}   and \eqref{eq:steadyg} that 
$$
\partial_t(\Gg- g)  {+}\frac14 \partial_x(x(\Gg-g)) 
=\Q_\gamma(\Gg,\Gg)-\Q_0(g,g).$$
Multiplying the above equation with $\sgn(\Gg- g)\,\w_{a}$ and integrating over $\R$ we obtain 
\begin{multline*}
  \frac{\d }{\d t} \|\Gg- g\|_{L^1(\w_{a})} 
    -\frac{a}{4} \int_\R |x| \w_{a-1}(x) |\Gg(x)-g(t,x)| \d x \\
\le \|\Q_\gamma(\Gg,\Gg)- \Q_0(\Gg,\Gg)\|_{L^1(\w_{a})}+ \|\Q_{0}(\Gg,\Gg)-\Q_{0}(g,g)\|_{L^1(\w_{a})}.
\end{multline*}
Now, 
\begin{align*}
  \|\Q_{0}(\Gg,\Gg)-\Q_{0}(g,g)\|_{L^1(\w_{a})} & =
  \|\Q_{0}(\Gg-g,\Gg+g)\|_{L^1(\w_{a})} \\
  &\le 2 \|\Gg-g\|_{L^1(\w_{a})} \|\Gg+g\|_{L^1(\w_{a})},
\end{align*}
and with Proposition \ref{diff_Q} together with Theorem \ref{theo:Unique} and Corollary \ref{L2-weighted}, it implies that, for $s>0$ such that $s+\g_\star+a<3-\delta$,
$$\frac{\d }{\d t} \|\Gg- g\|_{L^1(\w_{a})}  \le C_1  \|\Gg-g\|_{L^1(\w_{a})}  {+ C_{2}\gamma^{\frac{s}{s+1}}(1+|\log\gamma|)}, \qquad \g\in(0,\g_\star),$$
with $C_1>0$  and $C_2>0$.
Finally, the Gronwall inequality would lead to
\begin{equation}\label{toto}\|\Gg- g(t)\|_{L^1(\w_{a})}  \le \frac{ {C_{2}\gamma^{\frac{s}{s+1}}(1+|\log\gamma|)}}{C_1} e^{C_1t},\qquad t\ge 0.
\end{equation}
Finally, \eqref{diff_G} together with \eqref{L1a-conv} and \eqref{toto} gives, for any $t\ge 0$,
\begin{equation}\label{eq:difGhg}
\|\Gg-\bm{h}_{\g}\|_{L^1(\w_{a})}\le \frac{{C_{2}\gamma^{\frac{s}{s+1}}(1+|\log\gamma|)}}{C_1} e^{C_1t} +  \tilde{C}_1 e^{- \alpha\sigma(1-r)t} {\|\Gg-\bm{h}_{\g}\|_{L^1(\w_{a})}^{ \alpha(1-r)}} .\end{equation}
Theorem \ref{theo:Unique} further implies
 \begin{equation*}
\|\Gg-\bm{h}_{\g}\|_{L^1(\w_{a})}\le \frac{{C_{2}\gamma^{\frac{s}{s+1}}(1+|\log\gamma|)}}{C_1} e^{C_1t} +  C_{a,\alpha,r} e^{-\alpha\sigma(1-r)t}.\end{equation*}
Choosing $t=(C_1+ \alpha\sigma(1-r))^{-1}\log\left(\frac{C_{a,\alpha,r}C_1}{C_2\gamma^{\frac{s}{s+1}}}\right)$ we get
 \begin{equation*}
\|\Gg-\bm{h}_{\g}\|_{L^1(\w_{a})}\le C_{a,\alpha,r}\left(\frac{C_2 \gamma^{\frac{s}{s+1}}}{C_{a,\alpha,r} C_1}\right)^{\frac{\alpha \sigma (1-r)}{ {C_1}+\alpha \sigma(1-r)}}(2+|\log\gamma|):=\eta(\gamma)\end{equation*}
which proves the result.
\end{proof}  
 
We have now everything in hands to give the proof of Theorem~\ref{theo:sta}.
\begin{proof}[Proof of Theorem~\ref{theo:sta}]  Let $\delta$ and $\g_\star$ be defined as in the proof of Proposition \ref{lem:stabil}. For such $\delta$ and $\g_\star$, the results of Theorem \ref{theo:Unique} and Corollary \ref{L2-weighted} hold and $a+\g_\star<3-\delta$. From the estimate $\|\Gg-\bm{G}_{0}\|_{L^{1}(\w_{a})} \leq \|\Gg-\bm{h}_{\g}\|_{L^{1}(\w_{a})}+\|\bm{h}_{\g}-\bm{G}_{0}\|_{L^{1}(\w_{a})}$ and using Proposition \ref{lem:stabil}  {and \eqref{stab}}, we deduce that
\begin{equation}\label{eq:Ca}
\|\Gg-\bm{G}_{0}\|_{L^{1}(\w_{a})} \leq \eta(\g) + C_{a}\left|\lambda_{\g}-\lambda_{0}\right|\,,\end{equation}
where we recall that $\bm{G}_{0}(x)=\lambda_{0}\bm{H}(\lambda_{0}x),$ $\bm{h}_{\g}(x)=\lambda_{\g}\bm{H}(\lambda_{\g}x)$ where $\lambda_{\g}=M_{2}(\Gg)^{-\frac{1}{2}}$ is such that $\int_{\R}x^{2}\bm{h}_{\g}(x)\dx=M_{2}(\Gg)$. It is therefore enough to quantify the rate of convergence of $\lambda_{\g}$ to $\lambda_{0}.$ Resuming the computations of Lemma \ref{lem:unique}, we see that
\begin{equation*}
\mathscr{I}_{0}(\bm{h}_{\g},\bm{h}_{\g})=\frac{1}{\lambda_{\g}^{2}}\int_{\R^{2}}\bm{H}(x)\bm{H}(y)|x-y|^{2}\log\frac{|x-y|}{\lambda_{\g}}\dx\dy
=\frac{2}{\lambda_{\g}^{2}} {\log\frac{\lambda_{0}}{\lambda_{\g}}}\end{equation*}
where we introduced the notation
\begin{equation}\label{eq:mathIo}
\mathscr{I}_{0}(f,g)=\int_{\R^{2}}f(x)g(y)|x-y|^{2}\log|x-y|\dx\dy, \qquad f,g \in L^{1}(\w_{s}), \qquad s>2,\end{equation}
and used that $\mathscr{I}_{0}(\bm{H},\bm{H})=2\log \lambda_{0}$ and $\int_{\R^{2}}\bm{H}(x)\bm{H}(y)|x-y|^{2}\dx\dy=2$ as established in Lemma \ref{lem:unique}. We introduce also the notation
$$\mathscr{I}_{\g}(f,g)=\g^{-1}\int_{\R^{2}}f(x)g(y)|x-y|^{2}\left(|x-y|^{\g}-1\right)\dx\dy, \qquad \qquad f,g \in L^{1}(\w_{2+\g})$$ 
and recall (see \eqref{eq1}) that $\mathscr{I}_{\g}(\Gg,\Gg)=0$. One has then the following
\begin{equation*}\begin{split}
\frac{2}{\lambda_{\g}^{2}} {\log\frac{\lambda_{0}}{\lambda_{\g}}}&=\mathscr{I}_{0}(\bm{h}_{\g},\bm{h}_{\g})\\
&=\mathscr{I}_{0}(\bm{h}_{\g}-\Gg,\bm{h}_{\g}-\Gg)+2\mathscr{I}_{0}(\bm{h}_{\g}-\Gg,\Gg) {-} \mathscr{I}_{0}(\Gg,\Gg).
\end{split}\end{equation*}
{ Hence, 
$$\left|\frac{2}{\lambda_{\g}^{2}}\log\frac{\lambda_{0}}{\lambda_{\g}} \right|\leq C_{a}\|\bm{h}_{\g}-\Gg\|_{L^{1}(\w_{a})}\left(2\|\Gg\|_{L^{1}(\w_{a})}+\|\bm{h}_{\g}-\Gg\|_{L^{1}(\w_{a})}\right) + {|\mathscr{I}_{0}(\Gg,\Gg)|}\,,$$}
for $2 < a <3$ and where we used Lemma \ref{lem:I0fg}. We deduce then from Proposition \ref{lem:stabil} and the fact that $\sup_{\g\in(0,\g_{\star})}\|\Gg\|_{L^{1}(\w_{a})} {<\infty} $ that
$$\left|\lambda_{\g}^{-2}\log\left(\frac{\lambda_{\g}}{\lambda_{0}}\right)^{2}\right| \leq \tilde{\eta}(\g) + \left|\mathscr{I}_{0}(\Gg,\Gg)\right|$$
where $\tilde{\eta}(\g)=C_{a}\eta(\g)\left(2\sup_{\g}\|\Gg\|_{L^{1}(\w_{a})}+\eta(\g)\right) \to 0$ as $\g \to 0^{+}$ is an explicit function. Since $\mathscr{I}_{\g}(\Gg,\Gg)=0$,
$$\left|\lambda_{\g}^{-2}\log\left(\frac{\lambda_{\g}}{\lambda_{0}}\right)^{2}\right| \leq \tilde{\eta}(\g) + \left|\mathscr{I}_{0}(\Gg,\Gg)-\mathscr{I}_{\g}(\Gg,\Gg)\right|$$
and, using Lemma \ref{lem:IgfgI0} together with the estimates in Theorem \ref{theo:Unique} and Corollary \ref{L2-weighted}, we obtain that, for $2<a<3-\delta$ and $s>0$ such that $a+s+\g_\star<3-\delta$,
$$\left|\lambda_{\g}^{-2}\log\left(\frac{\lambda_{\g}}{\lambda_{0}}\right)^{2}\right| \leq \tilde{\eta}(\g) + \bm{C}\g^{\frac{s}{s+1}}|\log\g|, \qquad \forall \g \in (0,\g_{\star})\,,$$
for some positive constant $\bm{C}$ depending on $a,s$.  {Noticing that $\lambda_{\g} \to \lambda_{0}$ as $\g \to 0$, it is bounded both from above and below for $\g$ small enough, we get that there is $\bm{C}_{0}$} such that
$$\left|\log \frac{\lambda_{\g}}{\lambda_{0}} \right| \leq \bm{C}_{0}\left(\tilde{\eta}(\g)+ \g^{\frac{s}{s+1}}|\log\g|\right).$$
Since $|\log x| \geq \frac{|1-x|}{\max(1,x)}$, there exists $\bm{C}_{1} >0$  such that
$$\left|\lambda_{\g}-\lambda_{0}\right| \leq \bm{C}_{1}\left(\tilde{\eta}(\g)+ \g^{\frac{s}{s+1}}|\log\g|\right), \qquad \g \in (0,\g_{\star}).$$
Introducing the explicit  function $\overline{\eta}(\g)=C_{a}\bm{C}_{1}\left(\tilde{\eta}(\g)+ \g^{\frac{s}{s+1}}|\log\g|\right)+\eta(\g)$, this, together with \eqref{eq:Ca}, proves the result. \end{proof} 
\begin{rem}\label{rem:nonexplicit} Notice that the constants $\bm{C}_{0}$ and $\bm{C}_{1}$ in the above proof depend on upper and lower bounds on $\lambda_{\g}=\left(M_{2}(\Gg)\right)^{-\frac{1}{2}}$ and $M_{2}(\Gg) \leq \frac{1}{2}$. We describe in Section \ref{sec:quantg} a procedure which allows to make the function $\overline{\eta}( \gamma)$ \emph{completely} explicit.
\end{rem}
\subsection{Uniqueness}\label{sec:uniqueness}  
 {We now establish some stability result for $\mathscr{L}_{0}$.}
 \begin{lem}\label{lem:bareta} Let $2 < a < 3$. There exist
     $\g_\star\in(0,1)$ and a mapping $\tilde{\eta}\;:\;[0,\g_\star] \to \R^{+}$ with 
$$\lim_{\g\to0^{+}}\tilde{\eta}(\g)=0$$
and such that, for any $\g\in(0,\g_\star)$, any $\Gg^{1},\Gg^{2} \in \mathscr{E}_{\g}$,
\begin{equation}\label{eq:L0dif}
\|\mathscr{L}_{0}\left(\Gg^{1}-\Gg^{2}\right)\|_{\X_{a}} \leq \tilde{\eta}(\g)\left\|\Gg^{1}-\Gg^{2}\right\|_{\X_{a}}.\end{equation}
\end{lem}
\begin{proof} Let $\delta$ and $\g_\star$ be defined as in the proof of Theorem \ref{theo:sta}. For such $\delta$ and $\g_\star$, the results of Theorem \ref{theo:Unique} and Corollary \ref{L2-weighted} hold and $a+\g_\star<3-\delta$. Let $\g\in(0,\g_\star)$. Let us consider $\Gg^{1},\Gg^{2} \in \mathscr{E}_{\g}$. We introduce the difference 
$$g_{\g}=\Gg^{2}-\Gg^{1}$$
which satisfies \eqref{eq:diffe}. We write this last identity in an equivalent way:
\begin{multline*}
\frac{1}{4}\partial_{x}\left(xg_{\g}(x)\right)=\left[\Q_{\g}(g_{\g},\Gg^{2}-\bm{G}_{0})+\Q_{\g}(\Gg^{1}-\bm{G}_{0},g_{\g})\right]\\
+\left[\Q_{\g}(g_{\g},\bm{G}_{0})-\Q_{0}(g_{\g},\bm{G}_{0})\right]
+\left[\Q_{\g}(\bm{G}_{0},g_{\g})-\Q_{0}(\bm{G}_{0},g_{\g})\right]\\
+\Q_{0}(g_{\g},\bm{G}_{0})+\Q_{0}(\bm{G}_{0},g_{\g})
\end{multline*}
which can be written as
$$-\mathscr{L}_{0}(g_{\g})=\mathcal{A}_{\g}+\mathcal{B}_{\g}+\mathcal{C}_{\g}$$
where
$$\mathcal{A}_{\g}=\left[\Q_{\g}(g_{\g},\Gg^{2}-\bm{G}_{0})+\Q_{\g}(\Gg^{1}-\bm{G}_{0},g_{\g})\right], \qquad \mathcal{B}_{\g}=\left[\Q_{\g}(g_{\g},\bm{G}_{0})-\Q_{0}(g_{\g},\bm{G}_{0})\right]$$
and
$$\mathcal{C}_{\g}=\left[\Q_{\g}(\bm{G}_{0},g_{\g})-\Q_{0}(\bm{G}_{0},g_{\g})\right].$$

Therefore,
$$\|\mathscr{L}_{0}(g_{\g})\|_{L^{1}(\w_{a})}\leq \|\mathcal{A}_{\g}\|_{L^{1}(\w_{a})}+\|\mathcal{B}_{\g}\|_{L^{1}(\w_{a})}
+\|\mathcal{C}_{\g}\|_{L^{1}(\w_{a})}.$$
One estimates separately the norms $\|\mathcal{A}_{\g}\|_{L^{1}(\w_{a})}$, $\|\mathcal{B}_{\g}\|_{L^{1}(\w_{a})}$ and $\|\mathcal{C}_{\g}\|_{L^{1}(\w_{a})}.$ Clearly
\begin{equation*}\begin{split}
\|\mathcal{A}_{\g}\|_{L^{1}(\w_{a})} &\leq C_{0}\|g_{\g}\|_{L^{1}(\w_{a+\g})}\,\left(\|\Gg^{1}-\bm{G}_{0}\|_{L^{1}(\w_{a+\g})}+\|\Gg^{2}-\bm{G}_{0}\|_{L^{1}(\w_{a+\g})}\right)\\
&\leq \eta_{1}(\g)\|g_{\g}\|_{L^{1}(\w_{a+\g})}
\end{split}\end{equation*}
with 
$$\eta_{1}(\g)=C_{0}\left(\|\Gg^{1}-\bm{G}_{0}\|_{L^{1}(\w_{a+\g})}+\|\Gg^{2}-\bm{G}_{0}\|_{L^{1}(\w_{a+\g})}\right).$$
According to {Theorem \ref{theo:sta}}, the mapping  $\eta_{1}:[0,\g_\star] \to \R^{+}$ is such that 
$$\lim_{\g\to0}\eta_{1}(\g)=0.$$
One deduces then from Proposition \ref{diff_Q}, with $s >0$  such that $a+\g_\star+s<3-\delta$ and $p=2$, that
\begin{multline*}
\|\mathcal{B}_{\g}\|_{L^{1}(\w_{a})} \leq C_{s,2}\g^{\frac{s}{s+1}}|\log \g| \|g_{\g}\|_{L^1(\w_{a})}\|\bm{G}_{0}\|_{L^1(\w_{a})}\\
    +  {24}\,\g^{\frac{s}{s+1}} \left(\|\bm{G}_{0}\|_{L^1(\w_{s+\gamma+a})}\|g_{\g}\|_{L^1(\w_{a})}+\|g_{\g}\|_{L^1(\w_{s+\gamma+a})}\|\bm{G}_{0}\|_{L^1(\w_{a})}+\|\bm{G}_{0}\|_{ L^{2}(\w_{a})}\|g_{\g}\|_{L^1(\w_{a})}\right).
\end{multline*}
Using the known bounds on $\bm{G}_{0}$ (in particular Theorem \ref{theo:Unique} and Corollary \ref{L2-weighted}), one deduces that there exists $C_{s} >0$ (independent of $\g$) such that
$$\|\mathcal{B}_{\g}\|_{L^{1}(\w_{a})} \leq C_{s}\g^{\frac{s}{s+1}}{(1+|\log\g|)}\,\|g_{\g}\|_{L^{1}(\w_{a+\g+s})}, \qquad \forall \g \in (0,\g_{\star}).$$
In the same way
$$\|\mathcal{C}_{\g}\|_{L^{1}(\w_{a})} \leq C_{s}\g^{\frac{s}{s+1}} {(1+|\log\g|)}\,\|g_{\g}\|_{L^{1}(\w_{a+\g+s})}, \qquad \forall \g \in (0,\g_{\star}).$$
Gathering all these estimates, we obtain
\begin{multline*}
\|\mathscr{L}_{0}(g_{\g})\|_{L^{1}(\w_{a})}\leq \left(\eta_{1}(\g)+2C_{s}\g^{\frac{s}{s+1}} {(1+|\log\g|)}\right)\,\|g_{\g}\|_{L^{1}(\w_{a+s+\g})}\\
\leq C_{a,s}\left(\eta_{1}(\g)+2C_{s}\g^{\frac{s}{s+1}} {(1+|\log\g|)}\right)\,\|g_{\g}\|_{L^{1}(\w_{a})}
\end{multline*}
for some suitable choice of $s$ (small enough) where the estimate $\|g_{\g}\|_{L^{1}(\w_{a+s+\g})} \leq C_{a,s}\|g_{\g}\|_{L^{1}(\w_{a})}$ is a consequence of Lemma \ref{lem:Diffmom} (for $\g \in (0,{\frac{a}{3}})$). This gives the result.
\end{proof}

Combining the above result with Proposition \ref{restrict0} allows to show directly that two solutions to \eqref{eq:steadyg} \emph{with same energy} coincide as already explained in the introduction. 

In order to extend this line of reasoning to general solutions to \eqref{eq:steadyg} with different energy, one somehow follows the same approach but needs a way to compensate the discrepancy of energies to apply a variant of \eqref{eq:invert0}. Typically, let us now consider two solutions $\Gg^{1},\Gg^{2} \in \mathscr{E}_{\g}$ and let $g_{\g}=\Gg^{1}-\Gg^{2}.$
If one is able to construct $\tilde{g}_{\g} \in \mathbb{Y}_{a}$ such that
\begin{equation}\label{eq:kernL0g}
\mathscr{L}_{0}(\tilde{g}_{\g})=\mathscr{L}_{0}(g_{\g}) \qquad \text{ and } \quad M_{2}(\tilde{g}_{\g})=0 \quad  (\text{i.e. } {\tilde{g}_\g} \in \mathbb{Y}_{a}^{0})\end{equation}
then, as before, one would have
\begin{equation}\label{eq:tildeGg}
\frac{\nu}{C(\nu)} \|\tilde{g}_{\g}\|_{\X_{a}} \leq \|\mathscr{L}_{0}(\tilde{g}_{\g})\|_{\X_{a}}=\|\mathscr{L}_{0}({g}_{\g})\|_{\X_{a}}  \leq \tilde{\eta}(\g)\left\|g_{\g}\right\|_{\X_{a}}\,.\end{equation}
To conclude as before, we also need to check that there is $C >0$ (\emph{independent of $\g$}) such that
\begin{equation}\label{eq:tildeg}
\|g_{\g}\|_{\X_{a}} \leq C \|\tilde{g}_{\g}\|_{\X_{a}}\end{equation}
from which the identity $g_{\g}=0$ would follow  easily, as in the introduction (see end of Section~\ref{Sec:strategy}) for solutions with same energy.

Of course, constructing $\tilde{g}_{\g}$ satisfying \eqref{eq:kernL0g} is easy since $\mathscr{L}_{0}$ is invertible on $\mathbb{Y}_{a}^{0}$. The difficulty is to check \eqref{eq:tildeg}. The main tool to achieve this scope is the \emph{``linearised dissipation of energy''} functional 
$$\mathscr{I}_{0}(f,\bm{G}_{0})=\int_{\R^{2}}f(x)\bm{G}_{0}(y)|x-y|^{2}\log|x-y|\dx\dy, \qquad f \in L^{1}(\w_{s}), \qquad s>2.$$
First, one has the following observations
\begin{lem}\label{rmk:kern} The function defined by
$$g_{0}(x)=\frac{2}{\pi}\frac{1-3x^{2}}{(1+x^{2})^{3}}, \quad x \in \R$$
belongs to $\mathbb{Y}_{a}$ and 
is such that
$$\mathscr{L}(g_{0})=0 \qquad \text{ and }\qquad  M_{2}(g_{0})=-2.$$
Moreover,
\begin{equation}\label{eq:g0H}
\mathscr{I}_{0}(g_{0},\bm{H})= -2\log 2-2.\end{equation}
Finally, it holds 
$$\mathscr{I}_{0}(\bm{H},\bm{H})=2\log 2+1.$$
\end{lem}
\begin{proof} Let $g \in L^{1}(\w_{a})$ be such that $\mathscr{L}(g)=0$ and $\int_{\R} g(x)\d x=0.$ Setting
$$\psi(\xi)=\int_{\R}e^{-i\xi x}g(x)\d x$$
one checks without too many difficulties that {(see also \eqref{eq:linearised-fourier})}
$$-\frac{1}{4}\xi \frac{\d}{\d \xi}\psi(\xi)=2\psi\left(\frac{\xi}{2}\right)\bm{\Phi}\left(\frac{\xi}{2}\right)-\psi(\xi).$$
Direct inspection shows that
$$\psi_{0}(\xi)=|\xi|^{2}e^{-|\xi|}$$
is a solution to the above equation, with 
\begin{equation}\label{eq:psi0}
\psi_{0}(0)=\psi'_{0}(0)=0, \quad \psi''_{0}(0)=2 \neq 0.\end{equation} Since moreover
$e^{-|\xi|}$ is the Fourier transform of $G(x)=\frac{1}{\pi(1+x^{2})}$, one deduces that $\psi_{0}$ is the Fourier transform of 
$$g_{0}(x)=-\dfrac{\d^{2}}{\d x^{2}}G(x)= {\frac{2}{\pi}}\frac{1-3x^{2}}{(1+x^{2})^{3}}.$$
Notice that $g_{0} \in L^{1}(\w_{a})$ for any $2<a<3$ and \eqref{eq:psi0} shows that $g_{0} \in \mathbb{Y}_{a}$ with $M_{2}(g_{0})=-2.$ Let us now prove \eqref{eq:g0H}. Observe that, if $g$ is an eigenfunction of $\mathscr{L}$ with zero mass, then using the weak form of the linearised operator $\mathscr{L}$, $$ \frac14 \int_\R g(x) x \p_x \phi\,\dx
  + 2 \int_\R \int_\R g(x) \bm{H}(y) \left(
    \phi\Big(\frac{x-y}{2}\Big) - \frac12\,\phi(x) - \frac12\,{\phi(-y)}
  \right) \dy \d x=0$$
where we used also that $\bm{H}$ is even. Taking $\phi(x)=x^{2}\log|x|=\frac{1}{2}x^{2}\log x^{2}$ as a test-function we get
\begin{multline*}
\frac{1}{8}\int_{\R}g(x)x\p_x(x^{2}\log x^{2})\,\dx+2\int_{\R^{2}}g(x)\bm{H}(y)\frac{|x-y|^{2}}{4}\log\frac{|x-y|}{2}\d x\d y\\
-\int_{\R}g(x)x^{2}\log |x|\d x=0\end{multline*}
where we used that $\int_{\R}g(x)\d x=0$ while $\int_{\R}\bm{H}(y)\d y=1.$ Thus one obtains that any eigenfunction of $\mathscr{L}$ with zero mass is such that
\begin{equation}\label{eq:IoGG}\begin{split}
\mathscr{I}_{0}(g,\bm{H})&:=\int_{\R^{2}}g(x)\bm{H}(y)|x-y|^{2}\log|x-y|\dx\dy\\
&=\left(\log 2-\frac{1}{2}\right)\int_{\R}g(x)x^{2}\d x  { + } \int_{\R}g(x)x^{2}\log |x|\d x.\end{split}\end{equation} 
In particular, for $g=g_{0}=-\frac{\d^{2}}{\d x^{2}}G$ as defined previously, it holds that
\begin{equation*}\begin{split}
\int_{\R}g_{0}(x)x^{2}\log|x|\d x&=-\int_{\R}G(x)\dfrac{\d^{2}}{\d x^{2}}\left[x^{2}\log|x|\right]\dx=-\frac{1}{2}\int_{\R}G(x)\dfrac{\d^{2}}{\d x^{2}}\left[x^{2}\log x^{2}\right]\dx\\
&=-2\int_{\R}G(x)\log|x|\dx-3\int_{\R}G(x)\dx=-3\end{split}\end{equation*}
using $\int_{\R}G(x)\dx=1$ and
$$\int_{\R}\frac{\log|x|}{1+x^{2}}\dx=2\int_{0}^{\infty}\frac{\log x}{1+x^{2}}\dx=0.$$
Therefore, recalling that $M_{2}(g_{0})= {-2}$, we deduce \eqref{eq:g0H}. The same idea gives also the expression of $\mathscr{I}_{0}(\bm{H},\bm{H})$. Indeed, by definition
\begin{equation*}\begin{split}
-\frac{1}{4}\int_{\R}x\bm{H}(x)\p_{x}\phi(x)\d x&=\int_{\R}\Q_{0}(\bm{H},\bm{H})\phi \dx\\
&=\int_{\R^{2}}\bm{H}(x)\bm{H}(y)\left[\phi\left(\frac{x+y}{2}\right)-\phi(x)\right]\dx \dy.\end{split}\end{equation*}
With $\phi(x)=|x|^{2}\log |x|$, this gives, since $\int_{\R}\bm{H}(x)\dx=\int_{\R}\bm{H}x^{2}\dx=1$,
\begin{equation*}\begin{split}
-\frac{1}{2}\int_{\R}x^{2}\bm{H}(x)\log|x|\d x-\frac{1}{4}&=\frac{1}{4}\int_{\R^{2}}\bm{H}(x)\bm{H}(y)|x+y|^{2}\log|x+y|\dx\dy\\
&\phantom{+++++}-\frac{\log 2}{4}\int_{\R^{2}}\bm{H}(x)\bm{H}(y)|x+y|^{2}\dx\dy-\int_{\R}\bm{H}(x)x^{2}\log|x|\dx\\
&=\frac{1}{4}\mathscr{I}_{0}(\bm{H},\bm{H})-\frac{\log2}{2}-\int_{\R}\bm{H}(x)x^{2}\log|x|\dx
\end{split}\end{equation*}
i.e.
$$\mathscr{I}_{0}(\bm{H},\bm{H})=2\log 2-1+2\int_{\R}\bm{H}(x)x^{2}\log |x|\dx.$$
Using that
$$\int_{\R}\bm{H}(x)x^{2}\log|x|\dx=1$$
we deduce the result.
\end{proof}
Thanks to the above observations, we deduce the following
\begin{lem}\label{rmk:varphi0} Let $2 < a < 3$. There exists $\varphi_{0} \in \mathrm{Ker}(\mathscr{L}_{0}) \cap \mathbb{Y}_{a}$ such that
$$M_{2}(\varphi_{0}) \neq 0 \qquad \text{ and } \quad \mathscr{I}_{0}(\varphi_{0},\bm{G}_{0}) \neq 0.$$
\end{lem}
\begin{proof} Since the function $g_{0}$ defined in Lemma \ref{rmk:kern} belongs to the kernel of $\mathscr{L}$, one has
$$\varphi_{0}(x)=g_{0}(\lambda_{0}x) \in \mathbb{Y}_{a} \cap \mathrm{Ker}(\mathscr{L}_{0}).$$
Moreover, recalling the definition of $\mathscr{I}_{0}$ in \eqref{eq:mathIo} and since $\bm{G}_{0}(x)=\lambda_{0}\bm{H}(\lambda_{0}x)$, one checks easily that
\begin{multline*}
\mathscr{I}_{0}(\varphi_{0},\bm{G}_{0})=\frac{1}{\lambda_{0}^{3}}\left(\mathscr{I}_{0}(g_{0},\bm{H})-\log \lambda_{0}\,\int_{\R^{2}}g_{0}(x)\bm{H}(y)|x-y|^{2}\d x\d y\right)\\
=\frac{1}{\lambda_{0}^{3}}\left(\mathscr{I}_{0}(g_{0},\bm{H})-\log \lambda_{0}M_{2}(g_{0})\right)\end{multline*}
where we used that $g_{0} \in \mathbb{Y}_{a}.$ In particular, since $M_{2}(g_{0})= {-2}$, we deduce that
$$\mathscr{I}_{0}(\varphi_{0},\bm{G}_{0})=\frac{1}{\lambda_{0}^{3}}\mathscr{I}_{0}\left( {g_{0}+\bm{H}},\bm{H}\right)=-\frac{4}{\lambda^{3}_{0}} \neq 0$$
where we used that $\mathscr{I}_{0}(g_{0},\bm{H})= -2\log 2-5$ and $\mathscr{I}_{0}(\bm{H},\bm{H}) = 2\log 2+1$.\end{proof}
The existence of the above function $\varphi_{0}$ implies the following fundamental property of the linearised dissipation of energy 
\begin{lem}\label{lem:roleI0}
Let $2 < a < 3$. If $\varphi \in \mathrm{Ker}(\mathscr{L}_{0}) \cap \mathbb{Y}_{a}$ then 
$$\mathscr{I}_{0}(\varphi,\bm{G}_{0})=0 \implies  M_{2}(\varphi)=0.$$
In particular, in such a case, $\varphi=0.$
\end{lem}
\begin{proof} Let $\varphi \in \mathrm{Ker}(\mathscr{L}_{0}) \cap \mathbb{Y}_{a}$ be such that $\mathscr{I}_{0}(\varphi,\bm{G}_{0})=0$. Let 
$$\varphi^{\perp}=\varphi-\frac{M_{2}(\varphi)}{M_{2}(\varphi_{0})}\varphi_{0}.$$
One has of course $M_{2}(\varphi^{\perp})=0$ (i.e. $\varphi^{\perp} \in \mathbb{Y}_{a}^{0}$) and $\mathscr{L}_{0}(\varphi^{\perp})=0$ since both $\varphi$ and $\varphi_{0}$ belong to $\mathrm{Ker}(\mathscr{L}_{0})$. According to Proposition \ref{restrict0}, one has $\varphi^{\perp}=0$. Therefore,
$\varphi=\frac{M_{2}(\varphi)}{M_{2}(\varphi_{0})}\varphi_{0}$,
so that 
$$\mathscr{I}_{0}(\varphi,\bm{G}_{0})=\frac{M_{2}(\varphi)}{M_{2}(\varphi_{0})}\mathscr{I}_{0}(\varphi_{0},\bm{G}_{0}).$$
Since, by assumption $\mathscr{I}_{0}(\varphi,\bm{G}_{0})=0$ while $\mathscr{I}_{0}(\varphi_{0},\bm{G}_{0}) \neq 0$, it must hold that $M_{2}(\varphi)=0.$ In particular, $\varphi \in \mathbb{Y}_{a}^{0}$ and, using Proposition \ref{restrict0} again, we deduce that $\varphi=0.$ 
\end{proof}
A final technical Lemma regards the smallness of the linearised energy dissipation functional for differences of solutions to \eqref{eq:steadyg}
\begin{lem}\label{lem:lastI0} Let $2 < a <3$. There exist $\g_\star\in(0,1)$ and $\bar{\eta}_{0}(\g)$ with 
$$\lim_{\g\to0}\bar{\eta}_{0}(\g)=0$$
 such that, for any $\g\in(0,\g_\star)$, any $\Gg^{1},\Gg^{2} \in \mathscr{E}_{\g}$,
\begin{equation}\label{eq:diffGg1-2}
\left|\mathscr{I}_{0}\left(\Gg^{1}-\Gg^{2},\bm{G}_{0}\right)\right| \leq \bar{\eta}_{0}(\g)\,\|\Gg^{1}-\Gg^{2}\|_{\X_{a}}.\end{equation}
\end{lem}

\begin{proof}  Let $\delta$ and $\g_\star$ be defined as in the proof of Theorem \ref{theo:sta}. For such $\delta$ and $\g_\star$, the results of Theorem \ref{theo:Unique} and Corollary \ref{L2-weighted} hold and $a+\g_\star<3-\delta$. For $\g\in(0,\g_\star)$,  $\Gg^{1},\Gg^{2} \in \mathscr{E}_{\g}$, let $g_{\g}=\Gg^{1}-\Gg^{2}.$ One notices that
\begin{multline*}
2\mathscr{I}_{0}({g_{\g}},\bm{G}_{0})=\mathscr{I}_{0}(g_{\g},\bm{G}_{0}-\Gg^{1})+\mathscr{I}_{0}(g_{\g},\bm{G}_{0}-\Gg^{2})+
\mathscr{I}_{0}(g_{\g},\Gg^{1}+\Gg^{2})\\
=\mathscr{I}_{0}(g_{\g},\bm{G}_{0}-\Gg^{1})+\mathscr{I}_{0}(g_{\g},\bm{G}_{0}-\Gg^{2})\\
+\left(\mathscr{I}_{0}(g_{\g},\Gg^{1}+\Gg^{2})-\mathscr{I}_{\g}(g_{\g},\Gg^{1}+\Gg^{2})\right)
\end{multline*}
since 
$$\mathscr{I}_{\g}(\Gg^{1}-\Gg^{2},\Gg^{1}+\Gg^{2})= {\mathscr{I}_{\g}(\Gg^{1},\Gg^{1})-\mathscr{I}_{\g}(\Gg^{2},\Gg^{2})}=0$$
for $\Gg^{1},\Gg^{2} \in \mathscr{E}_{\g}.$ One invokes then Lemma \ref{lem:I0fg} and \ref{lem:IgfgI0} to deduce that, for any $s >0$ such that  $s+\g_\star+a<3-\delta$, there are $C_{a} >0,C_{a,s,2} >0$ such that
\begin{multline*}
\left|\mathscr{I}_{0}( {g_{\g}},\bm{G}_{0})\right| \leq C_{a}\left(\left\|\bm{G}_{0}-\Gg^{1}\right\|_{\X_{a}}+\left\|\bm{G}_{0}-\Gg^{2}\right\|_{\X_{a}}\right)\|g_{\g}\|_{\X_{a}}\\
+C_{a,s,2}\g^{\frac{s}{s+1}}|\log \g|\,\|\Gg^{1}+\Gg^{2}\|_{\X_{a}}\|g_{\g}\|_{\X_{a}}\\
+  {12}\,\g^{\frac{s}{s+1}} \bigg(2\|g_{\g}\|_{\X_{s+\gamma+a}}\|\Gg^{1}+\Gg^{2}\|_{\X_{a+s+\g}}
     +\|\Gg^{1}+\Gg^{2}\|_{ L^2(\w_{a})}\|g_{\g}\|_{\X_{a}}\bigg).
\end{multline*} 
Using Lemma \ref{lem:Diffmom} again, for $s>0$ small enough  and $\g$ small enough so that $\g+\frac23(a+s)\le a$ (that is $\g+\frac{2s}3\le \frac{a}3$), one has $\|g_{\g}\|_{\X_{a+s+\g}} \leq C_{a,s}\|g_{\g}\|_{\X_{a}}$ and, thanks to the uniform bounds on $\|\Gg^{i}\|_{L^{2}(\w_{a})}$ and $\|\Gg^{i}\|_{\X_{a+s+\g}}$ $(i=1,2)$  given by Theorem \ref{theo:Unique} and Corollary \ref{L2-weighted}  together with Theorem \ref{theo:sta}, we deduce the result.
\end{proof}
 
We are in position to prove our main result regarding the steady solution to \eqref{eq:steadyg} following the strategy described before.
\begin{proof}[Proof of Theorem \ref{theo:mainUnique}]  Let $\delta$ and $\g_\star$ be defined as in the proof of Lemma \ref{lem:lastI0}. For $\g\in(0,\g_\star)$,  $\Gg^{1},\Gg^{2} \in \mathscr{E}_{\g}$, let $g_{\g}=\Gg^{1}-\Gg^{2}.$ 
Since $\mathscr{L}_{0}$ is invertible on $\mathbb{Y}_{a}^{0}$, there exists a \emph{unique} $\tilde{g}_{\g} \in \mathbb{Y}^{0}_{a}$ such that
$$\mathscr{L}_{0}(\tilde{g}_{\g})=\mathscr{L}_{0}(g_{\g}).$$
It remains to prove the estimate \eqref{eq:tildeg} 
between $\|\tilde{g}_{\g}\|_{\X_{a}}$ and $\|g_{\g}\|_{\X_{a}}$. To do so, we actually prove that $\tilde{g}_{\g}-g_{\g} \in \mathrm{Span}(\varphi_{0})$, more precisely
\begin{equation}\label{eq:exp}
g_{\g}=\tilde{g}_{\g}+z_{0}\varphi_{0}, \qquad z_{0}=\frac{1}{p_{0}}\mathscr{I}_{0}(g_{\g}-\tilde{g}_{\g},\bm{G}_{0})\end{equation}
where $p_{0}=\mathscr{I}_{0}(\varphi_{0},\bm{G}_{0}).$ Indeed, writing $\bar{g}_{\g}=\tilde{g}_{\g}+z_{0}\varphi_{0}$ one sees that, since $\varphi_{0} \in \mathrm{Ker}(\mathscr{L}_{0})$
$$\mathscr{L}_{0}(\bar{g}_{\g})=\mathscr{L}_{0}(\tilde{g}_{\g})=\mathscr{L}_{0}(g_{\g})$$
while, obviously, the choice of $z_{0}$ implies that
$$\mathscr{I}_{0}(\bar{g}_{\g},\bm{G}_{0})=\mathscr{I}_{0}(g_{\g},\bm{G}_{0}).$$
From Lemma \ref{lem:roleI0}, this implies that $M_{2}(\bar{g}_{\g}-g_{\g})=0$ and $\bar{g}_{\g}-g_{\g}=0$. This proves \eqref{eq:exp}. Consequently,
\begin{equation*}\begin{split}
\|g_{\g}\|_{\X_{a}} \leq \|\tilde{g}_{\g}\|_{\X_{a}} + |z_{0}|\,\|\varphi_{0}\|_{\X_{a}} &\leq \|\tilde{g}_{\g}\|_{\X_{a}}+ \frac{\|\varphi_{0}\|_{\X_{a}}}{|p_{0}|}\left|\mathscr{I}_{0}(g_{\g}-\tilde{g}_{\g},\bm{G}_{0})\right|\\
&\leq \|\tilde{g}_{\g}\|_{\X_{a}}+ \frac{\|\varphi_{0}\|_{\X_{a}}}{|p_{0}|}\left(|\mathscr{I}_{0}(g_{\g},\bm{G}_{0})|+|\mathscr{I}_{0}(\tilde{g}_{\g},\bm{G}_{0})|\right)\end{split}\end{equation*}
by definition of $z_{0}.$ According to Lemma  \ref{lem:I0fg}, there is $C_{0} >0$ such that 
$$|\mathscr{I}_{0}(\tilde{g}_{\g},\bm{G}_{0})| \leq C_{0}\|\tilde{g}_{\g}\|_{\X_{a}}.$$
Therefore, there are $C_{1}, C_{2}>0$ (independent of $\g$) such that
\begin{equation}\label{eq:ggI0}
\|g_{\g}\|_{\X_{a}} \leq C_{1}\|\tilde{g}_{\g}\|_{\X_{a}}+ C_{2} \left|\mathscr{I}_{0}(g_{\g},\bm{G}_{0})\right|.\end{equation}
Using now \eqref{eq:diffGg1-2}, we deduce that
$$\|g_{\g}\|_{\X_{a}} \leq C_{1}\|\tilde{g}_{\g}\|_{\X_{a}}+C_{2}\bar{\eta}_{0}(\g)\|g_{\g}\|_{\X_{a}}$$
and, since $\lim_{\g\to0}\bar{\eta}_{0}(\g)=0$, we can choose $\g^{\star} >0$ small enough so that $C_{2}\bar{\eta}_{0}(\g) \leq \frac{1}{2}$ for any $\g \in (0,\g^{\star})$ so that
$$\frac{1}{2}\|g_{\g}\|_{\X_{a}} \leq C_{1}\|\tilde{g}_{\g}\|_{\X_{a}}, \qquad \forall \g \in (0,\g^{\star}).$$
With the strategy described before, we deduce that the function $\tilde{g}_{\g}$ and $g_{\g}$ satisfies \eqref{eq:kernL0g}--\eqref{eq:tildeGg} and \eqref{eq:tildeg} with $C=2C_{1}$. In particular, we deduce from \eqref{eq:tildeGg} that
$$\frac{\nu}{C(\nu)}\,\frac{1}{2C_{1}}\|g_{\g}\|_{\X_{a}} \leq \tilde{\eta}(\g)\,\|g_{\g}\|_{\X_{a}}$$
and, since $\lim_{\g\to0}\tilde{\eta}(\g)=0$, there exists $\g^{\dagger} >0$ small enough so that
$$\|g_{\g}\|_{\X_{a}} < \|g_{\g}\|_{\X_{a}}$$
which implies that $g_{\g}=0$ and proves the result. 
 \end{proof}
 
 \subsection{Quantitative estimate on $\g^{\dagger}$}\label{sec:quantg}
In order to make Theorem \ref{theo:mainUnique} fully exploitable, we need to be able to quantitatively estimate the threshold parameter $\g^{\dagger}$. From the above proof, this amounts to some quantitative estimate on the mapping $\tilde{\eta}(\g)$. As already observed in Remark \ref{rem:nonexplicit}, the only non fully quantitative estimate in the definition of $\tilde{\eta}(\g)$ comes from the mapping $\overline{\eta}(\g)$ in Theorem \ref{theo:sta}. In this subsection, we briefly explain how it is possible to derive such a quantitative estimate. We keep the presentation slighlty informal here just to stress out the main steps of the estimates. The crucial point is then to estimate the rate of convergence of 
$$\|\Gg-\bm{G}_{0}\|_{L^{1}(\w_{a})}$$
to zero as $\g \to 0.$
To do so, we briefly resume the main steps in our proof of uniqueness and introduce, for $\Gg \in \mathscr{E}_{\g}$,
$$h_{\g}=\bm{G}_{0}-\Gg.$$
One sees easily that
$$\mathscr{L}_{0}(h_{\g})=\Q_{0}(h_{\g},h_{\g})+\left[\Q_{\g}(\Gg,\Gg)-\Q_{0}(\Gg,\Gg)\right]$$
which results in
\begin{equation*}\label{eq:L0h}
\|\mathscr{L}_{0}(h_{\g})\|_{\X_{a}} \leq C_{0}\|h_{\g}\|_{\X_{a}}^{2}+C_{0}\g^{\frac{s}{s+1}}\left(1+|\log\g|\right)\end{equation*}
for some positive $C_{0}$ independent of $\g$ (see Lemma \ref{N0_Ggamma} for a similar reasoning). Now, as before, there exists $\tilde{h}_{\g} \in \mathbb{Y}^{0}_{a}$ such that
$$\frac{\nu}{C(\nu)}\|\tilde{h}_{\g}\|_{\X_{a}} \leq \|\mathscr{L}_{0}(\tilde{h}_{\g})\|_{\X_{a}}=\|\mathscr{L}_{0}h_{\g}\|_{\X_{a}}.$$
Therefore, there is $C >0$ independent of $\g$ such that
\begin{equation}\label{eq:tildeh}
\|\tilde{h}_{\g}\|_{\X_{a}} \leq C\|h_{\g}\|_{\X_{a}}^{2}+C\g^{\frac{s}{s+1}}\left(1+|\log \g|\right)\end{equation}
and we need to compare again $\|\tilde{h}_{\g}\|_{\X_{a}}$ to $\|h_{\g}\|_{\X_{a}}.$ As in Eq. \eqref{eq:ggI0}
\begin{equation}\label{eq:hgtildeh}
\|h_{\g}\|_{\X_{a}} \leq C_{1}\|\tilde{h}_{\g}\|_{\X_{a}}+C_{2}\left|\mathscr{I}_{0}(h_{\g},\bm{G}_{0})\right|\end{equation}
for $C_{1},C_{2}$ independent of $\g.$ Now, one checks without major difficulty that
$$2\mathscr{I}_{0}(h_{\g},\bm{G}_{0})=\mathscr{I}_{0}(h_{\g},h_{\g})+\left[\mathscr{I}_{\g}(\Gg,\Gg)-\mathscr{I}_{0}(\Gg,\Gg)\right]$$
where we used that $\mathscr{I}_{0}(\bm{G}_{0},\bm{G}_{0})=\mathscr{I}_{\g}(\Gg,\Gg)=0.$ Thus,{ with Lemmas \ref{lem:I0fg}, \ref{lem:IgfgI0}, Theorem \ref{theo:Unique} and Corollary \ref{L2-weighted}}, we deduce that
$$\left|\mathscr{I}_{0}(h_{\g},\bm{G}_{0})\right| \leq C_{3}\|h_{\g}\|_{\X_{a}}^{2}+C_{3}\g^{\frac{s}{s+1}}{(1+|\log\g|)}$$
for some $C_{3} >0$ independent of $\g$. Summing up this estimate with \eqref{eq:tildeh} and \eqref{eq:hgtildeh} one sees that there exists a positive constant $\bm{c}_{0} >0$ independent of $\g$ such that
$$\|h_{\g}\|_{\X_{a}} \leq \bm{c}_{0}\|h_{\g}\|_{\X_{a}}^{2} +\bm{c}_{0}\g^{\frac{s}{s+1}}\left(1+|\log\g|\right).$$
Now, since we know that $\lim_{\g\to0}\|h_{\g}\|_{\X_{a}}=0$ (without an explicit rate at this stage), there exists $\g_{0} >0$ (non explicit) such that
$$\bm{c}_{0}\|h_{\g}\|_{\X_{a}} \leq \frac{1}{2} \qquad \forall \g \in (0,\g_{0})$$
and therefore
$$\|h_{\g}\|_{\X_{a}} \leq 2\bm{c}_{0}\g^{\frac{s}{s+1}}\left(1+|\log\g|\right) \qquad \forall \g \in (0,\g_{0}).$$
Such an estimate provides actually an explicit estimate for $\g_{0}$ since the optimal parameter becomes clearly the one for which the two last estimates are identity yielding
$$\g_{0}^{\frac{s}{s+1}}\left(1+{|\log\g_{0}|}\right)=\frac{1}{4\bm{c}_{0}^{2}}.$$
This provides then an explicit rate of convergence of $\Gg$ to $\bm{G}_{0}$ as 
$$\|\Gg-\bm{G}_{0}\|_{\X_{a}} \leq 2\bm{c}_{0}\g^{\frac{s}{s+1}}\left(1+|\log\g|\right) \qquad \forall \g \in (0,\g_{0})$$
for some explicit $\g_{0}.$ This makes explicit the mapping $\overline{\eta}(\g)$ in Theorem \ref{theo:sta} and, in turns, provides some quantitative estimates on the parameter $\g^{\dagger}$ in Theorem \ref{theo:mainUnique}.

\section{The case of Maxwell molecules revisited}
\label{sec:exp}

This whole Section is devoted to the special case of Maxwell molecules, corresponding to $\g=0$, which as already observed, is the pivot case around which our analysis revolves for our perturbation analysis. We collect here several results, some of them of broader interest than the mere use we make of them in the previous part of the paper. We begin with revisiting the exponential convergence to equilibrium obtained in  \cite{MR2355628}. Let us recall that, generally speaking, the analysis of  Boltzmann-like models with Maxwellian interaction essentially renders explicit formulas that allow for a very precise analysis (we refer to \cite{boby} for an extensive study).

More precisely, we consider the following equation already in self-similar variables  \begin{equation}
  \label{eq:IB-selfsim}
  \p_t g = -\frac14 \p_x (xg) + \Q_{0}(g,g),
\end{equation}
with initial condition $g(0,x)=f_{0}(x)$ which, using Galilean invariance, we will always assume to be such that
\begin{equation}
  \label{eq:normalisation-f0}
  \int_{\R} f_0(x) \dx = 1,
  \qquad
  \int_{\R} xf_0(x) \dx = 0,
  \qquad
  \int_{\R} x^2 f_0(x) \dx = 1.
\end{equation}

Notice that, as said in the introduction, we chose in \eqref{eq:IB-selfsim} the parameter $c=\frac{1}{4}$ which is, in the special case of Maxwell molecules, the only one which provides energy conservation and, as such, it holds at least formally that\begin{equation}
  \label{eq:normalisation-g}
  \int_{\R} g(t,x) \dx = 1,
  \qquad
  \int_{\R} xg(t,x) \dx = 0,
  \qquad
  \int_{\R} x^2g(t,x) \dx = 1.
\end{equation}
The collision operator for Maxwell molecules is given by \begin{align*}
  \Q_{0}(f,g)(x)
  &= \int_\R f\left(x + \frac{y}{2}\right) g\left(x - \frac{y}{2}\right) \dy
  - \frac12 f(x) \int_\R g(y) \dy
  - \frac12 g(x) \int_\R f(y) \dy
  \\
  &=: \Q_{0}^+(f,g) - \Q_{0}^-(f,g)
\end{align*}
Notice that $\Q_{0}^+$ can be written as
\begin{equation*}\begin{split}
  \Q_{0}^+(f,g)(x) &= \int_\R f\left(x + \frac{y}{2}\right) g\left(x - \frac{y}{2}\right) \dy
  = 2 \int_\R f(x + y) g(x - y) \dy
  \\
 & = 2 \int_\R f(y) g(2x - y) \dy
  = 2 (f\ast g)(2x).
\end{split}\end{equation*}
Alternatively, in weak form we have
\begin{equation}
  \label{eq:weak}
  \int_\R \Q_{0}(f,g)(x) \phi(x) \dx
  =
  \int_\R \int_\R f(x) g(y) \left(
    \phi\Big(\frac{x+y}{2}\Big) - \frac12\phi(x) - \frac12\phi(y)
  \right) \dx \dy.
\end{equation}

We will refer to equation \eqref{eq:IB-selfsim} as the
\emph{self-similar equation} for Maxwell molecules. If we define the Fourier transform of
$g$ as
\begin{equation*}
  \varphi(t, \xi) := \int_{\R} g(t,x) e^{-ix \xi} \dx,
\end{equation*}
then $\varphi(t,\xi)$ satisfies
\begin{equation}
  \label{eq:selfsim-fourier}
  \p_t \varphi(t,\xi) = \frac14 \,\xi  \,\p_\xi \varphi(t,\xi)
  + \varphi\Big(t, \frac{\xi}{2} \Big)^2 - \varphi(t,\xi), 
\end{equation}
with the initial condition $\varphi(0,\cdot)= \int_{\R} f_0(x) e^{-ix \xi} \dx=:\varphi_0$.
Due to \eqref{eq:normalisation-g}, $\varphi$ satisfies for all
$t \geq 0$ that
\begin{equation}
  \label{eq:normalisation-varphi}
  \varphi(t,0) = 1,
  \qquad
  \p_\xi \varphi(t,0) = 0,
  \qquad
  \p_\xi^2 \varphi(t,0) = -1.
\end{equation}
In particular, 
\begin{equation}\label{eq:Phi}
  \bm{\Phi}(\xi) = (1 + |\xi|) e^{-|\xi|}
\end{equation}
is a steady solution to \eqref{eq:selfsim-fourier} and this is exactly the Fourier transform of the steady solution $\bm{H}$ defined in Theorem \ref{theo:bob}.

\subsection{Exponential convergence to equilibrium}\label{Sec:conv:eq:fourier} We investigate here the convergence to equilibrium for solutions to \eqref{eq:IB-selfsim} and show the following
 \begin{theo}\phantomsection\label{k-norm-cvgce}
  Assume that $g = g(t,x)$ is a nonnegative solution to
  \eqref{eq:IB-selfsim} with the normalisation
  \eqref{eq:normalisation-g}, and call $\varphi = \varphi(t,\xi)$ its
  Fourier transform in the $x$ variable. Then, for $0 \leq k < 3$, and
  for all $t \geq 0$,
  \begin{equation*}
   \vertiii{\varphi(t) - \bm{\Phi}}_{k} \leq e^{-\sigma_{k} t} \vertiii{\varphi_0 - \bm{\Phi}}_{k}
    \qquad
    \text{with}  \qquad \sigma_{k} := 1 - \frac14 k - 2^{1-k}\,.
  \end{equation*}
  In particular, $g(t)$ converges exponentially to $\bm{H}$ in the $k$-Fourier
  norm for any $2 < k < 3$.

More generally, for  $p\ge1$,  $\frac{1}{p} < k < 3+\frac1p$, and
  for all $t \geq 0$,
  \begin{equation*}
    \vertiii{\varphi(t) - \bm{\Phi}}_{k,p} \leq e^{-\sigma_{k}(p) t} \vertiii{\varphi_0 - \bm{\Phi}}_{k,p}
    \qquad
    \text{with} \qquad \sigma_{k}(p):= 1 - \frac14 k + \frac{1}{4p} - 2^{1 + \frac{1}{p} - k}\,.
  \end{equation*}
  In particular, $g$ converges exponentially to $\bm{H}$ in the $k$-Fourier
  norm for any $(k,p)$ such that $\sigma_{k}(p)>0$.
\end{theo}
 \begin{proof} We begin with the first part of the proof, corresponding to the special case $p=\infty.$\\

\noindent \textit{$\bullet$ The case $p=\infty$.} Assume that $g=g(t,x)$ is a solution to \eqref{eq:IB-selfsim} with the normalisation
\eqref{eq:normalisation-g}, and call $\varphi = \varphi(t,\xi)$ its
Fourier transform as before. Then $\varphi(t,\cdot)$ is a solution to
\eqref{eq:selfsim-fourier} with the normalisation
\eqref{eq:normalisation-varphi}, and we may take the difference with $\bm{\Phi}$ given in \eqref{eq:Phi}
$$\psi(t,\xi)
:= \varphi(t,\xi) - \bm{\Phi}(\xi)$$ to find that
\begin{equation}\label{nlinear-self-similar}
  \p_t \psi(t,\xi) =
  \frac14 \xi  \p_\xi \psi(t,\xi)
  + \psi\Big(t, \frac{\xi}{2} \Big)
  \left( \varphi\Big(t, \frac{\xi}{2} \Big)
    + \bm{\Phi}\Big(\frac{\xi}{2}
    \Big)  \right)
  - \psi(t,\xi).
\end{equation}
If we call $\left(T(t)\right)_{t\geq0}$ the semigroup associated to the operator $\psi \mapsto \frac14 \xi
\p_\xi \psi - \psi$, given by
\begin{equation*}
  T(t) \phi(\xi) := e^{-t} \phi( \xi e^{\frac{1}{4}t}),
\end{equation*}
then by Duhamel's formula we can write
\begin{equation}
  \label{eq:decay-Duhamel}
  \psi(t) = T(t) \psi_0 + \int_0^t T(t-s) A(s) \d s,\,
\end{equation}
where we denote $\psi(t)=\psi(t,\xi)$ and
\begin{equation*}
  A(s)=A(s,\xi) := \psi\Big(s, \frac{\xi}{2} \Big)
  \left( \varphi\Big(s, \frac{\xi}{2} \Big)
    + \bm{\Phi}\Big(\frac{\xi}{2}
    \Big)  \right).
\end{equation*}
Now we notice that, for any $h$ such that $\vertiii{h}_{k}$ is finite,
\begin{equation}
  \label{eq:decay1}
  \vertiii{T(t) h}_{k}
  = e^{-t} \sup_{\xi \neq 0} \frac{|h(\xi e^{\frac14 t})|}{|\xi|^k}
  = e^{-(1 - \frac14 k)t}
  \sup_{\xi \neq 0} \frac{|h(\xi e^{\frac14 t})|}{|\xi e^{\frac14 t}|^k}
  = e^{-(1 - \frac14 k)t} \vertiii{h}_k.
\end{equation}
On the other hand,
\begin{equation*}
  |A(s,\xi)| \leq 2 \left|\psi\left(s, \frac{\xi}{2}\right)\right|,
\end{equation*}
since $\|\varphi\|_\infty =
\|\bm{\Phi}\|_\infty = 1$ (recall that both $g$ and $\bm{H}$ have unit mass). This implies
\begin{equation}
  \label{eq:decay2}
   \vertiii{A(s)}_{k}
  \leq
  2 \sup_{\xi \neq 0} \frac{|\psi(s, \frac{\xi}{2})|}{|\xi|^k}
  =
  2^{1-k} \sup_{\xi \neq 0} \frac{|\psi(s, \frac{\xi}{2})|}{|\frac{\xi}{2}|^k}
  =
  2^{1-k}  \vertiii{\psi(s)}_{k}.
\end{equation}
Notice that $ \vertiii{\psi(t,\cdot)}_{k} < +\infty$ for all $0 \leq k < 3$,
since $\psi$ is a $\mathcal{C}^2$ function in $\xi$ with
$\psi(t,0) = \p_\xi \psi(t,0) = \p_\xi^2 \psi(t,0) = 0$. Using
\eqref{eq:decay1} and \eqref{eq:decay2} in \eqref{eq:decay-Duhamel} we
see that
\begin{multline*}
   \vertiii{\psi(t)}_k \leq
\vertiii{T(t) \psi_0}_{k} + \int_0^t \vertiii{T(t-s) A(s)}_{k} \d s
  \\
  \leq
  e^{-(1 - \frac14 k) t}\vertiii{\psi_0}_{k}
  +
  \int_0^t e^{-(1 - \frac14 k) (t-s)} \vertiii{A(s)}_k \d s
  \\
  \leq
  e^{-(1 - \frac14 k) t}\vertiii{\psi_0}_{k}
  +
  2^{1-k} \int_0^t e^{-(1 - \frac14 k) (t-s)} \vertiii{\psi(s)}_{k} \d s\,,
\end{multline*}
which immediately gives by Gronwall's lemma that
\begin{equation*}
 \vertiii{\psi(t)}_{k} \leq e^{-\sigma t} \vertiii{\psi_0}_{k}
  \qquad
  \text{with $\sigma := 1 - \frac14 k - 2^{1-k}$.}
\end{equation*}
We deduce the exponential convergence in Theorem \ref{k-norm-cvgce} with rate $\sigma > 0$ for $2 < k < 3$.

\noindent \textit{$\bullet$ The general case $p\geq1.$} For $1 \leq p < \infty$, we recall that the norms $\vertiii{\cdot}_{k,p}$, defined in \eqref{def:normkp}, are given by
\begin{equation*}
 \vertiii{\psi}_{k,p}^p := \ir \frac{|\psi(\xi)|^p}{|\xi|^{kp}} \d \xi,
\end{equation*}
and are well-defined if $|\psi(\xi)| \leq \min\{1 ,C|\xi|^3\}$ for some $C>0$  and $\frac{1}{p} < k < 3 + \frac{1}{p}$.

  With a similar calculation as before,
\begin{equation*}
  \vertiii{T(t) \psi}_{k,p}
  = e^{-\alpha_p t} \vertiii{\psi}_{k,p}
\end{equation*}
with
\begin{equation*}
  \alpha_p := 1 - \frac14 k + \frac{1}{4p}.
\end{equation*}
Also,
\begin{equation*}
  \vertiii{A(s)}_{k,p}
  \leq
  2^{1 + \frac{1}{p} - k} \vertiii{\psi(s)}_{k,p},
\end{equation*}
so we can repeat the same argument to obtain
\begin{multline*}
  \vertiii{\psi(t)}_{k,p} \leq
  \vertiii{T(t) \psi_0}_{k,p} + \int_0^t \vertiii{T(t-s) A(s)}_{k,p} \d s
  \\
  \leq
  e^{-\alpha_p t}\vertiii{\psi_0}_{k,p}
  +
  \int_0^t e^{-\alpha_p (t-s)} \vertiii{A(s)}_{k,p} \d s
  \\
  \leq
  e^{-\alpha_p t}\vertiii{\psi_0}_{k,p}
  +
  2^{1 + \frac{1}{p} - k} \int_0^t e^{-\alpha_p (t-s)} \vertiii{\psi(s)}_{k,p} \d s.
\end{multline*}
 Then one concludes as previously using Gronwall's lemma. \end{proof}
 
\begin{rem}[\textit{\textbf{Invariance by scaling}}]\phantomsection\label{rem:scal}  
The above result holds for solutions $g$ to \eqref{eq:IB-selfsim} satisfying the normalisation \eqref{eq:normalisation-g}. Recall that \eqref{eq:normalisation-g} is preserved by the nonlinear dynamics \eqref{eq:IB-selfsim}. We explain briefly how it applies to solutions of \eqref{eq:IB-selfsim} with positive energy (not necessarily unitary). Namely, assume that $\tilde{g}_{0}$ is an initial datum such that
$$\int_{\R}\tilde{g}_0(x)\dx=1, \qquad \int_{\R}\tilde{g}_{0}(x)\,x\dx=0, \qquad \int_{\R}\tilde{g}_{0}(x)x^2\dx=E >0$$
and let $\tilde{g}(t,x)$ be the associated solution to \eqref{eq:IB-selfsim}. Notice that $\tilde{g}(t,x)$ share the same mass, momentum and energy of $\tilde{g}_0$ for any $t\geq0.$ Setting
$$g_{0}(x)=\lambda\,\tilde{g}_0(\lambda x), \qquad \lambda=\sqrt{E},$$
one sees that $g_0$ satisfies \eqref{eq:normalisation-g}. Denoting by $g(t,x)$ the associated solution to \eqref{eq:IB-selfsim},  the scaling invariance property of $\Q_{0}$ implies that
$$g(t,x)=\lambda\,\tilde{g}(t,\lambda x), \qquad \lambda=\sqrt{E}$$
while Theorem \ref{k-norm-cvgce} asserts that 
 \begin{equation*}
   \vertiii{\varphi(t) - \bm{\Phi}}_{k} \leq e^{-\sigma_{k} t} \vertiii{\varphi_0 - \bm{\Phi}}_{k}
    \qquad
    \text{with}  \qquad \sigma_{k} := 1 - \frac14 k - 2^{1-k}\,,
  \end{equation*} 
 where $\varphi(t)$ is the Fourier transform of $g$ and $\bm{\Phi}$ that of $\bm{H}$. Denoting by  $\widetilde{\varphi}(t,\cdot)$ the Fourier transform of $\tilde{g}(t,\cdot)$, we have 
$$\widetilde{\varphi}(t,\xi)=\varphi(t,\lambda\,\xi) \qquad \text{ and } \qquad \widehat{H}_\lambda(\xi)=\bm{\Phi}(\lambda\,\xi),$$
where $\widehat{H}_\lambda$ is the Fourier transform of the steady solution  
$$H_\lambda(x)=\lambda \bm{H}(\lambda x), \qquad \lambda >0$$
of \eqref{eq:IB-selfsim} with unit mass, zero momentum and energy $E$. Since
$$\vertiii{\widetilde{\varphi}(t)-\widehat{H}_\lambda}_{k}=\lambda^{k}\vertiii{\varphi(t)-\bm{\Phi}}_{k} \qquad \qquad \forall t\geq0$$
one sees that
\begin{equation*}
\vertiii{\widetilde{\varphi}(t) - \widehat{H}_\lambda}_{k} \leq e^{-\sigma_{k} t} \vertiii{\widetilde{\varphi}_0 - \widehat{H}_\lambda}_{k}
    \qquad
    \text{with}  \qquad \sigma_{k} := 1 - \frac14 k - 2^{1-k}\,.
  \end{equation*}
In other words, for any choice of the initial energy $E >0,$ solutions to \eqref{eq:IB-selfsim} relax exponentially fast -- in the $\vertiii{\cdot}_{k}$ norm -- towards the unique steady solution with the prescribed energy $E$.
\end{rem}
 
\subsection{Baseline regularity}\label{sec:baseline}

Let us concentrate the discussion in proving the propagation of baseline regularity of solutions, which in Fourier space follows by showing uniform propagation of decay at infinity.  The argument presented here is an alternative to the one in \cite{FPTT} where propagation of uniform regularity for the equation \eqref{eq:selfsim-fourier} has been proved.  Here the strategy is direct (no iteration/approximation step required) and based purely on comparison.   To this end we present a series of lemmas with the main purpose of proving a comparison principle and showing a proper upper barrier for solutions of the rescaled Boltzmann model. 

The key argument consists in proving that estimates for \emph{low frequencies} transfer to \emph{large frequencies}. We start adopting the following notation:
\begin{equation}\label{ft:e3}
\mathcal{D} = \xi\,\partial_{\xi}\,,\quad \text{thus} \quad e^{\mathcal{D}\,t}u(\xi) = u(e^{t}\xi)\,,\quad t\in\mathbb{R}\,.  
\end{equation}
Also, introduce the operators
\begin{equation}\label{ft:e4}
\Gamma[u](\xi) = u\left(\frac{\xi}{2}\right)\,u\left(\frac{\xi}{2}\right)\,,\quad \text{and}\quad L\,u(\xi) = \,u\left(\frac{\xi}{2}\right)\,.
\end{equation}

\begin{lem}\phantomsection\label{lem:evol-family}
For a given bounded function $\sigma_{0}(t,\cdot) \in L^{\infty}(\R)$ $(t\geq0)$, the unique solution to 
\begin{equation}\label{eq:LS}
\partial_{t}u-\sigma_{0}(t,\cdot)Lu=0, \qquad u(s,s,\xi)=u_{0}, \qquad t \geq s \geq0\end{equation}
is given by the following evolution family
$$u(s,t,\xi)=\mathcal{V}(s,t)u_{0}=\sum_{j=0}^{\infty}\mu_{j}(s,t,\xi)L^{j}u_{0}(\xi)$$
where $\mu_{0}(s,t,\xi)=1$ for any $s,t,\xi$ and
$$\mu_{j}(s,t,\xi)=\int_{\Delta_{t}^{j}(s)}\prod_{k=0}^{j-1}L^{k}\left(\sigma_{0}(s_{k},\cdot)\right)\d \bm{s}_{j}=\int_{\Delta_{t}^{j}(s)}\prod_{k=0}^{j-1}\sigma_{0}\left(s_{k},\frac{\xi}{2^{k}}\right)\d \bm{s}_{j}, \qquad j\geq1$$
with $\Delta_{t}^{j}(s)$ the simplex
$$\Delta_{t}^{j}(s)=\left\{\bm{s}_{j}=\left(s_{0},\ldots,s_{j-1}\right),\;\;\;s \leq s_{j-1} \leq s_{j-2}\leq \ldots \leq s_{1}\leq s_{0} \leq t\right\}$$
and 
$$\int_{\Delta_{t}^{j}(s)}\left(\mathrm{Expression}\right)\d\bm{s}_{j}=\int_{s}^{t}\d s_{0}\int_{s}^{s_{0}}\d s_{1}\ldots \int_{s}^{s_{j-2}}\left(\mathrm{Expression}\right)\d s_{j-1}.$$
\end{lem}
\begin{proof}The proof is by direct inspection. Write
$$v(s,t,\xi)=\sum_{j=0}^{\infty}\mu_{j}(s,t,\xi)L^{j}u_{0}(\xi).$$
Observe that $\mu_{0}(s,s,\xi)=1$, $\mu_{j}(s,s,\xi)=0$ for all $j \geq 1$ so that $v(s,s,\cdot)=u_{0}.$ On the one hand,
$$\partial_{t}v(s,t,\xi)=\sum_{j=1}^{\infty}\partial_{t}\mu_{j}(s,t,\xi)L^{j}u_{0}(\xi)=\sum_{j=0}^{\infty}\partial_{t}\mu_{j+1}(s,t,\xi)L^{j+1}u_{0}(\xi)$$
since we assumed $\mu_{0}$ to be constant. On the other hand,
$$Lv(s,t,\xi)=\sum_{j=0}^{\infty}L\left(\mu_{j}(s,t,\xi)L^{j}u_{0}\right)=\sum_{j=0}^{\infty}L(\mu_{j}(s,t,\xi))L^{j+1}u_{0}(\xi)$$
since $L( w_{1}\,w_{2})=L(w_{1})L(w_{2})$ (if one of the  {$w_{i}$} is bounded at least for the product to make sense). Therefore, if 
$$\partial_{t}\mu_{j+1}(s,t,\cdot)=\sigma_{0}(t,\xi)L\mu_{j}(s,t,\cdot) \qquad \mu_{j+1}(s,s,\cdot)=0\qquad \qquad j\geq 0$$
one gets that $v(s,t,\xi)$ solves \eqref{eq:LS}. By induction, since $\mu_{0}\equiv 1$, one gets the desired expression for $\mu_{j}$, $j\geq1.$
\end{proof}
\begin{rem} If $\sigma_{0}$ is constant, say $\sigma_{0}(t,\xi)=\alpha$ and $s=0$, because the volume of the simplex $\Delta_{t}^{j}=\Delta_{t}^{j}(0)$ is equal to $\frac{t^{j}}{j!}$ one gets 
$$u(t,\xi)=\sum_{j=0}^{\infty}\frac{\left(\alpha t\right)^{j}}{j!}L^{j}u_{0}$$ which is exactly the expression $e^{\alpha t L}u_{0}$ of the semigroup generated by the bounded operator $\alpha L$. 
\end{rem}
\begin{lem}[\textit{\textbf{Comparison lemma}}]\phantomsection \label{ft:l1}
Assume continuous functions $u,v\in[0,1]$ satisfying
\begin{subequations}
\begin{align}
\partial_{t}u + \big(-\tfrac{1}{4}\mathcal{D} + 1\big)u&\geq\Gamma[u]\,,\label{comp:e1}\\
\partial_{t}v + \big(-\tfrac{1}{4}\mathcal{D} + 1\big)v&\leq\Gamma[v]\,,\label{comp:e2}
\end{align}
\end{subequations}
and $u(0,\cdot) \geq v(0,\cdot)$.  Then $u(t,\cdot)\geq v(t,\cdot)$ for any $t \geq 0$.
\end{lem}
\begin{proof}
For such two functions $u$ and $v$ define $S(t,\xi):=\big(u(t,\frac{\xi}{2}) + v(t,\frac{\xi}{2})\big)\in[0,2]$.  Then, one concludes for the difference $d=d(t,\xi):=u(t,\xi)-v(t,\xi)$ the relation
\begin{equation*}
\partial_{t}d + \big(-\tfrac{1}{4}\mathcal{D} - S(t,\xi)\,L + 1\big)d = \mathcal{R}(t,\xi)\,,
\end{equation*}
where $\mathcal{R}(t,\xi)$ is a nonnegative remainder.  One can verify by direct computation that
\begin{align*}
e^{-\frac{1}{4}\mathcal{D}t}\big(S(t,\xi)\,L\big) &= \big(e^{-\frac{1}{4}\mathcal{D}t}\,S(t,\xi) \big) \big(e^{-\frac{1}{4}\mathcal{D}t}\,L\big)\\
&= \big(e^{-\frac{1}{4}\mathcal{D}t}\,S(t,\xi) \big) \big(L\,e^{-\frac{1}{4}\mathcal{D}t}\big) =: S_{0}(t,\xi)\,\big(L\,e^{-\frac{1}{4}\mathcal{D}t}\big) \,.
\end{align*}
Then, for $h=e^{-\frac{1}{4}\mathcal{D}t}d$ it follows that
\begin{equation*}
\partial_{t}h + \big(-S_{0}(t,\xi)\,L + 1 )h = e^{-\frac{1}{4}\mathcal{D}t}\,\mathcal{R}(t,\xi)\,.
\end{equation*}
Using the previous Lemma (with $\sigma_{0}=S_{0}$) and the evolution family $\{\mathcal{V}(s,t)\}_{t\geq s}$, one gets after integrating in time
\begin{equation}\label{comp-ft:e6}
h(t)= e^{-t}\mathcal{V}(0,t)h_{0} + \int^{t}_{0}e^{-(t-s)}\mathcal{V}(s,t)\,e^{-\frac{1}{4}\mathcal{D}s}\,\mathcal{R}(s,\xi)\d s\,.
\end{equation}
It is clear from the expression of $\mathcal{V}(s,t)$ that, since $L$ preserves the positivity and $\sigma_{0}\geq0$, $\mathcal{V}(s,t)$ is a nonnegative operator for any $0\leq s\leq t$, therefore, the second term in \eqref{comp-ft:e6} is nonnegative.  Furthermore, note that $h\geq0$ if and only if $d\geq0$.   In particular, the first term in \eqref{comp-ft:e6} is also nonnegative since $d_{0}\geq0$.  In this way $h$ and hence $d$ are nonnegative.
\end{proof}

\begin{prp}[\textit{\textbf{Propagation of strong smoothness}}]\phantomsection\label{strong-smoothness}
Take $\varphi(t,\xi)$ a solution of nonlinear equation \eqref{eq:selfsim-fourier}  with $|\varphi(t,\xi)|\leq 1$ for all $t\geq 0$ and $\xi\in\R$ and assume $|\varphi(0,\xi)| \leq \bm{\Phi}(a\,\xi)$ for some $a>0$.  Then,
\begin{equation*}
|\varphi(t,\xi)|\leq\bm{\Phi}(a\,\xi) \quad \text{for all} \quad t\geq0\,.
\end{equation*}
\end{prp}
\begin{proof}
Set $v(t,\xi)=\left|\varphi\left(t,\frac{\xi}{a}\right)\right|$ for $a>0$ and $u(t,\xi)=\bm{\Phi}(\xi)$.  Since $v_{0}(\xi) = |\varphi(0,\frac{\xi}{a})| \leq \bm{\Phi}(\xi)$ and 
\begin{align*}
      \partial_t v(t,\xi)& =\frac{1}{2\left|\varphi\left(t,\frac{\xi}{a}\right)\right|} \left(\partial_t\varphi\left(t,\frac{\xi}{a}\right)\overline{\varphi\left(t,\frac{\xi}{a}\right)} + \varphi\left(t,\frac{\xi}{a}\right)\partial_t\overline{\varphi\left(t,\frac{\xi}{a}\right)}\right)\\
      & = \frac14 \xi \partial_\xi v(t,\xi) + \frac12 \left(\varphi\left(t,\frac{\xi}{2a}\right)+\overline{\varphi\left(t,\frac{\xi}{2a}\right)}\right) v\left(t,\frac{\xi}{2}\right) -v(t,\xi)
    \end{align*}
    with 
$$\frac12 \left(\varphi\left(t,\frac{\xi}{2a}\right)+\overline{\varphi\left(t,\frac{\xi}{2a}\right)}\right) \le  {v\left(t,\frac{\xi}{2}\right)},$$
all conditions (inequalities \eqref{comp:e1} and \eqref{comp:e2} and initial condition) of Lemma \ref{ft:l1} are satisfied.  Therefore, $v(t,\xi)\leq u(t,\xi)$ or, equivalently, $|\varphi(t,\xi)|\leq\bm{\Phi}(a\,\xi)$ for all $t\geq0$.
\end{proof}

\begin{rem}
Compare this result with  \cite[Theorem 4]{FPTT}.  Interestingly, the result here is not associated to a physical counterpart $g(t,x)$ since the inverse Fourier transform of $\varphi$ may not be positive. 
\end{rem}

\medskip
Now, we present two lemmas to relax the strong decaying condition on the initial data. For any $\beta \geq 0$, we set
$$\Psi_{\beta}(r)=\langle r\rangle^{-\beta}=\left(1+r^{2}\right)^{-\frac{\beta}{2}}, \qquad r >0.$$
We will use repeatedly that $\Psi_{\beta}(\cdot)$ is non increasing with moreover
$$\Psi_{\beta}(r) \leq \min\left(1,r^{-\beta}\right) \qquad \forall r >0.$$
\begin{lem}[\textit{\textbf{Short time estimate}}]\phantomsection\label{short-time-lemma}
Fix $\beta>0$.  Assume $u(t,\xi)\in[0,1]$ satisfies the inequality
\begin{equation}\label{ft:e12}
\partial_{t}u + \big(-\tfrac{1}{4}\mathcal{D} + 1\big)u \leq \Gamma[u] \,
\end{equation}
together with
$$0\leq u(0,\xi)=u_{0}(\xi) \leq \Psi_{\beta}(|\xi|) \quad \forall \xi \in \R.$$
Assume there is $\delta >0$ such that
$$u(t,\xi)\leq\Psi_{\beta}(|\xi|) \qquad \quad \text{ for } \quad |\xi|\leq\delta, \qquad t \geq 0.$$
Then, for any $\beta' \in \left(0,\frac{\beta}{2}\right]$, there exists $\tau(\delta, \beta, \beta')>0$ such that 
$$0\leq u(t,\xi)\leq\Psi_{\beta'}(|\xi|) \qquad \text{ for any } t\in[0,\tau(\delta,\beta,\beta')], \qquad \xi \in \R.$$  The time $\tau(\delta,\beta,\beta')$ satisfies $\lim_{\beta'\rightarrow0}\tau(\delta,\beta,\beta')=+\infty$ for any fixed $\delta>0$ and $\beta>0$.
\end{lem}
\begin{proof} Let $\mathcal{U}(t)$ be the semigroup associated to the generator $-\frac{1}{4}\mathcal{D}$, i.e. $\mathcal{U}(t)f(\xi)=f(\xi e^{-\frac{1}{4}t})$. Setting $w(t,\xi)=e^{t}\,\mathcal{U}(t)u(t,\xi)$ we write \eqref{ft:e12} as 
$$\partial_{t}u + \big(-\tfrac{1}{4}\D +1\big)u \leq u\Bigl(t,\frac{\xi}{2}\Bigr)Lu(t,\xi)$$ or equivalently
$$\partial_{t}w\leq \mathcal{U}(t)u\Bigl(t,\frac{\xi}{2}\Bigr)Lw(t,\xi)$$ 
and denote by $\left(\mathcal{W}(s,t)\right)_{s,t}$ the evolution family constructed in Lemma \ref{lem:evol-family} with
$$\sigma_{0}(t,\xi)=\left[\mathcal{U}(t)u\Bigl(t,\frac{\xi}{2}\Bigr)\right]=u\left(t,\frac{\xi}{2}e^{-\frac{1}{4}t}\right)\,.$$
We have
$$0 \leq w(t,\xi) \leq \mathcal{W}(0,t)u_{0}(\xi)$$
with
$$\mathcal{W}(0,t)u_{0}(\xi)=\sum_{j=0}^{\infty}\nu_{j}(t,\xi)L^{j}u_{0}(\xi)$$
where $\nu_0(t,\xi)=1$ for any $t,\xi$ and 
$$\nu_{j}(t,\xi)=\int_{0}^{t}u\left(s_{0},\frac{\xi}{2}e^{-\frac{1}{4}s_{0}}\right)\nu_{j-1}\left(s_{0},\frac{\xi}{2}\right)\d s_{0}.$$
Then
\begin{equation}\label{eq:semi}
0 \leq u(t,\xi) \leq e^{-t}\sum_{j=0}^{\infty}\nu_{j}(t,\xi\,e^{\frac{1}{4}t})L^{j}u_0\left(\xi\,e^{\frac{1}{4}t}\right).\end{equation}
Since $u_{0}(\xi) \leq \Psi_{\beta}(|\xi|)$, 
$$0 \leq u(t,\xi) \leq e^{-t}\sum_{j=0}^{\infty}\nu_{j}(t,\xi\,e^{\frac{1}{4}t})L^{j}\Psi_{\beta}\left(|\xi|\,e^{\frac{1}{4}t}\right)=e^{-t}\sum_{j=0}^{\infty}\nu_{j}(t,\xi\,e^{\frac{1}{4}t})\Psi_{\beta}\left(2^{-j}|\xi|\,e^{\frac{1}{4}t}\right).$$
In addition, since $\Psi_{\beta}$ is non increasing and $2^{-j}|\xi|\,e^{\frac{1}{4}t}\geq 2^{-j}|\xi|$, it holds that
$$
0 \leq u(t,\xi) \leq e^{-t}\sum_{j=0}^{\infty}\nu_{j}(t,\xi\,e^{\frac{1}{4}t})\Psi_{\beta}\left(2^{-j}|\xi|\right).$$ 
By assumption $u(t,\xi)\in [0,1]$, therefore
\begin{equation}\label{eq:nuj}
\nu_{j}(t,\xi\,e^{\frac{1}{4}t}) \leq  \frac{t^{j}}{j!},
\end{equation}
and
\begin{equation}\label{ft:e13}
0 \leq u(t,\xi) \leq e^{-t}\sum_{j=0}^{\infty}\frac{t^{j}}{j!}\Psi_{\beta}\left(2^{-j}|\xi|\right).
\end{equation}
Observe that $\langle r \rangle^{a} \geq \langle \sqrt{a}\,r \rangle$ for any $a\geq1$ so that, for any $\beta\geq 2\beta'$, 
$$\Psi_{\beta}(|\xi|)\leq \Psi_{2\beta'}\left(\sqrt{\frac{\beta}{2\beta'}}\,|\xi|\right), \qquad \forall \xi \in\R.$$  Consequently,
\begin{align*}
\Psi_{\beta}\big(2^{-j}\,|\xi|\big)\leq \Psi_{2\beta'}\left(2^{-j}\sqrt{\frac{\beta}{2\beta'}}\,|\xi|\right)&\leq 2^{2\beta'j}\Psi_{2\beta'}\left(\sqrt{\frac{\beta}{2\beta'}}\,|\xi|\right)\\
&\leq 2^{2\beta'j}\Psi_{\beta'}\left(\sqrt{\frac{\beta}{2\beta'}}\,\delta\right)\Psi_{\beta'}(|\xi|), \qquad |\xi|\geq\delta\,.
\end{align*}
Using this estimate in inequality \eqref{ft:e13}, it holds
\begin{multline}\label{ft:e14}
u(t,\xi)\leq  { e^{-t}}\Psi_{\beta'}(|\xi|)\Psi_{\beta'}\left(\sqrt{\frac{\beta}{2\beta'}}\,\delta\right)\sum_{j=0}^{\infty}\frac{t^{j}}{j!}2^{2\beta'\,j}\\
=\Psi_{\beta'}(|\xi|)\Psi_{\beta'}\left(\sqrt{\frac{\beta}{2\beta'}}\,\delta\right)e^{(2^{2\beta'}-1)t}\,,\qquad |\xi|\geq\delta\,.
\end{multline}
Thus, choosing
\begin{equation*}
\tau(\delta,\beta,\beta')=\frac{ \beta' \ln\big(\langle \sqrt{\frac{\beta}{2\beta'}}\delta \rangle\big)}{2^{2\beta'}-1}\,,
\end{equation*}
we have 
$$u(t,\xi)\leq\Psi_{\beta'}(|\xi|) \qquad \text{ for } |\xi|\geq\delta \quad \text{ and } \quad t\in[0,\tau(\delta,\beta,\beta')].$$ Since, by assumption, for $|\xi| \leq \delta$ it holds $u(t,\xi)\leq\Psi_{\beta}(\xi)\leq \Psi_{\beta'}(|\xi|)$ we deduce that 
$$u(t,\xi)\leq\Psi_{\beta'}(|\xi|)$$
holds true for \emph{any} $\xi \in \R$ and $t \in [0,\tau(\delta,\beta,\beta')]$. From the definition of $\tau$, it  is clear that $\lim_{ { \beta'\rightarrow 0}} \tau(\delta,\beta,\beta')=+\infty$. 
\end{proof}

\begin{lem}[\textit{\textbf{Global-in-time estimates}}]\phantomsection\label{Long-time} Assume $u(t,\xi)\in[0,1]$ satisfies the inequality \eqref{ft:e12} for any $t\geq0$ with
$u(0,\xi)=u_0(\xi)\leq \Psi_{\beta}(|\xi|)$ for any $\xi \in \R$.  If 
$$u(t,\xi) \leq \Psi_{\beta}(|\xi|) \qquad \text{ for } |\xi| \leq 4, \qquad t \geq0\,,$$ 
for some $\beta>0$, then $u(t,\xi) \leq \Psi_{\beta}(|\xi|)$ for all $\xi\in\mathbb{R}$.
\end{lem}
\begin{proof} Inequality \eqref{ft:e12} together with
Duhamel's formula gives that
\begin{equation*}
u(t,\xi)\leq  u_{0}\left(\xi\,e^{\frac{1}{4}t}\right)e^{-t} +\int^{t}_{0}e^{-(t-s)}\left[u\left(s,\frac{\xi}{2}e^{\frac{1}{4}(t-s)}\right)\right]^{2}\d s, \qquad t \geq0.
\end{equation*}
For a given $t \geq0$, recall that $u_{0}\left(\xi\,e^{\frac{1}{4}t}\right) \leq \Psi_{\beta}\left(|\xi|\,e^{\frac{1}{4}t}\right)\leq \Psi_{\beta}(|\xi|)$ whereas, if
$|\xi| \leq 8e^{-\frac{t}{4}}$ then $\frac{|\xi|}{2}e^{\frac{1}{4}(t-s)} \leq 4$ for all $s \in [0,t]$ which by assumption gives 
$$ u\left(s,\frac{\xi}{2}e^{\frac{1}{4}(t-s)}\right)  \leq \Psi_{\beta}\left(\frac{|\xi|}{2}e^{\frac{1}{4}(t-s)}\right) \leq \Psi_{\beta}\left(\frac{|\xi|}{2}\right) \qquad \forall s \in [0,t]$$
where we used that $\Psi_{\beta}(\cdot)$ is non increasing. Consequently
$$u(t,\xi) \leq \Psi_{\beta}(|\xi|)e^{-t} + \Psi_{\beta}\left(\frac{|\xi|}{2}\right)^{2}(1-e^{-t})\,,\quad 0\leq |\xi|\leq 8e^{-t/4}\,.$$
In particular, setting 
$$t_{0}:=4\log\frac{4}{3}$$
so that $|\xi| \leq 6 \Longrightarrow |\xi| \leq 8e^{-\frac{t}{4}}$ for $t \in [0,t_{0}]$, one deduces that
\begin{equation}\label{eq:upsib}
u(t,\xi) \leq \Psi_{\beta}(|\xi|)e^{-t} + \Psi_{\beta}\left(\frac{|\xi|}{2}\right)^{2}(1-e^{-t})\,,\quad 0\leq |\xi|\leq 6, \qquad t \in [0,t_{0}].\end{equation}
Since $\Psi_{\beta}\left(\frac{|\xi|}{2}\right)^{2}\leq \Psi_{\beta}(|\xi|)$ for $|\xi|\geq {\sqrt{8}}$, one deduces that,
 \begin{equation*}
u(t,\xi) \leq  \Psi_{\beta}(|\xi|)\,,\qquad \qquad  {\sqrt{8}} \leq |\xi|\leq 6\,\qquad t \in [0,t_{0}]
\end{equation*}
which, by assumption, yields
$$u(t,\xi) \leq \Psi_{\beta}(|\xi|) \qquad \quad \text{ for all } 0 \leq |\xi| \leq 6, \qquad t \in [0,t_{0}].$$
Iterating this process $k$-times one gets 
\begin{equation*}
u(t,\xi) \leq \Psi_{\beta}(|\xi|)\,,\qquad 0\leq |\xi|\leq 4\cdot \left(\frac{3}{2}\right)^{k}\,,\quad t\in[0,t_{0}]\,.
\end{equation*}
Since $k$ is arbitrary, we get
$$u(t,\xi) \leq \Psi_{\beta}(|\xi|), \qquad \text{ for all } \xi \in \R\,,\quad t \in [0,t_{0}].$$ Since then, for any $t \geq t_{0}$
$$u(t,\xi) \leq e^{-(t-t_{0})}u\left(t_0,\xi\,e^{\frac{1}{4}(t-t_{0})}\right)+\int_{0}^{t-t_{0}}e^{-(t-t_{0}-s)}\left[u\left(s+t_0,\frac{\xi}{2}e^{\frac{1}{4}(t-t_{0}-s)}\right)\right]^{2}\d s$$
one can reproduce the above argument to show that the bound $u(t,\xi) \leq \Psi_{\beta}(|\xi|)$ holds also on the interval $[t_{0},2t_{0}]$. Iterating the procedure, the bound holds for any time $t\geq0$ and any $\xi \in \R.$
\end{proof}
We are in conditions to prove the main result of the section.
\begin{theo}\label{theo-baseline}
Let $\varphi(t,\xi)$ be a solution of the self-similar problem \eqref{eq:selfsim-fourier} satisfying $|\varphi(t,\xi)|\leq 1$ and with initial condition $\varphi_0$ enjoying the regularity
\begin{equation*}
\vertiii{\varphi_0 - \bm{\Phi}}_{k} <\infty \qquad \text{and} \qquad |\varphi_{0}(\xi)|\leq \Psi_{\alpha}(c|\xi|)
\end{equation*}
for some $k\in(2,3)$, $c\in(0,1]$, and $\alpha>0$.  Then,
\begin{equation*}
\sup_{t\geq0}|\varphi(t,\xi)| \leq \Psi_{\alpha}(c_0\,\xi)
\end{equation*}
for some positive constant $c_{0} >0$ depending only on $\alpha$, $c$, and $\vertiii{\varphi_0 - \bm{\Phi}}_{k}$.
\end{theo}
\begin{proof}
Note that $\Psi_{\alpha}(c\,|\xi|)\leq\Psi_{\beta}\big(\sqrt{\alpha/\beta}\,c \,|\xi|\big)$ for $\beta \in(0, c^2\,\alpha]$.  Hence, choosing $\beta=\min\{\frac{1}{2}, c^2\,\alpha\}$ it holds $|\varphi_{0}(\xi)|\leq \Psi_{\beta}(|\xi|)$.  Now, Theorem \ref{k-norm-cvgce} states that
\begin{equation*}
\vertiii{\varphi(t) - \bm{\Phi}}_{k} \leq e^{-\sigma t} \vertiii{\varphi_0 - \bm{\Phi}}_{k}
\qquad\text{with}  \qquad \sigma = 1 - \frac14 k - 2^{1-k}>0\,.
\end{equation*}
Therefore, with $C_{k}:=\vertiii{\varphi_{0}-\bm{\Phi}}_{k}$, 
\begin{equation}\begin{split}\label{eq:varPhiCk} 
|\varphi(t,\xi)| &\leq \bm{\Phi}(\xi) + C_{k}|\xi|^{k}e^{-\sigma t} \\
&\leq (1+|\xi|)e^{-|\xi|}+C_{k}|\xi|^{k}
\,\qquad \forall \xi \in \R\;;\;t \geq0.
\end{split}\end{equation}
For any $\beta \in (0,1),$ the mapping $F(r)=(1+r)e^{-r}+C_{k}r^{k}-(1+r^{2})^{-\frac{\beta}{2}}$ is such that
$$F(0)=F'(0)=0, \qquad F''(0)=-1+\beta < 0$$
from which one sees that there is $\delta >0$ (depending on $\beta$ and $C_{k}$) such that $F(r) \leq 0$ for $r \in (0,\delta)$, i.e. 
\begin{equation}\label{small-freq}
|\varphi(t,\xi)| \leq \Psi_{\beta}(|\xi|) \qquad \forall |\xi| \leq \delta\,,\,\quad t\geq0.\end{equation}
For large time, we  introduce, for $\beta \in (0,1)$,
$$G_{t}(r)=(1+r)e^{-r}+C_{k}r^{k}e^{-\sigma t}-(1+r^{2})^{-\frac{\beta}{2}}, \qquad r >0.$$
One first observes that
\begin{equation}\label{pou}
G_{t}(r) \leq (1+r)e^{-r}-1+\frac{\beta}{2}r^{2} + C_{k}r^{k}=(1+r) e^{-r} -1+ \beta r^2 + C_k r^k - \frac{\beta}2 r^2,
\end{equation}
with  
$$(1+r) e^{-r} -1+ \beta r^2 \le 0 \qquad  \mbox{ for any } 0 \le r \le 4,$$
 when $\beta < \frac{e^{-4}}2$ and  $$ C_k r^k - \frac{\beta}2 r^2 \le 0 \qquad \mbox{ for any }  0 \le r \le \left(\frac{\beta}{2C_k}\right)^{\frac{1}{k-2}}.$$ 
 Therefore,  if $\beta < \frac{e^{-4}}2$, then
$$ G_{t}(r) \le 0 \qquad \mbox{ for any } t\ge0 \mbox{ and } 0 \le r \le r_{\beta,k},$$ 
where  $r_{\beta,k}:=\min\left\{ \left(\frac{\beta}{2C_k}\right)^{\frac{1}{k-2}},4\right\}>0$. Now, for $ r_{\beta,k}\le r\le 4$, we have, again with \eqref{pou} 
$$G_t(r)\le h_{\beta} (r_{\beta,k}) + C_{k} 4^{k}e^{-\sigma t},$$
since $h_{\beta}(r):=(1+r)e^{-r}-1+\frac{\beta}{2}r^{2}$ is decreasing on $[r_{\beta,k},4]$ when $\beta < \frac{e^{-4}}2  < e^{-4}$. Note that $h_{\beta}(r_{\beta,k})<0$.  Choosing $t_*\ge \frac{-1}{\sigma} \log\left(-\frac{1}{C_k 4^k} h_{\beta}(r_{\beta,k})\right)$, we obtain that 
$$\max_{0 \leq r \leq 4}G_{t}(r) \leq 0, \qquad \forall t \geq t_{*}.$$
From this we  conclude that
\begin{align}\label{large-freq}
|\varphi(t,\xi)|\leq\Psi_{\beta}(|\xi|)\,,\quad\text{for}\quad |\xi| \leq 4\,,\quad t \geq t_{*}\,.
\end{align}
Given the estimate \eqref{small-freq}, we invoke Lemma \ref{short-time-lemma} with $u(t,\xi)=|\varphi(t,\xi)|$, $\beta \in (0,1)$, and $\beta'\in(0,\beta/2]$ sufficiently small such that $\tau(\delta,\beta,\beta')\geq t_{*}$ to obtain that
\begin{equation*}
|\varphi(t,\xi)| \leq \Psi_{\beta'}(|\xi|)\,,\qquad \xi\in\mathbb{R}\,,\quad t\in[0,t_{*}]\,.
\end{equation*}
With this and the estimate \eqref{large-freq} we use Lemma \ref{Long-time} in the interval $[t_{*},\infty)$, with $u(t,\xi)=|\varphi(t,\xi)|$ and $\beta=\beta'$, to conclude that 
\begin{equation}\label{eq:Psib'}
|\varphi(t,\xi)|\leq \Psi_{\beta'}(\xi) \qquad \text{ for all } \xi \in \R\,,\quad t\geq0\,.
\end{equation}

\medskip
\noindent
In order to upgrade the decay rate up to $\alpha$, we can bootstrap the previous estimate after noticing that, thanks to \eqref{eq:Psib'}, 
\begin{equation*}
\Big| \varphi\Big(t, \frac\xi2 \Big)^{2} \Big| \leq \Psi_{2\beta'}\left(\frac{|\xi|}{2}\right)\,
\end{equation*}
so that, $u(t,\xi)=|\varphi(t,\xi)|$ satisfies $\partial_{t}u + \big(-\tfrac{1}{4}\mathcal{D} + 1\big)u \leq \Psi_{2\beta'}\left(\frac{|\xi|}{2}\right)\,.$ Using Duhamel's formula, it holds that
\begin{equation*}
|\varphi(t,\xi)| \leq \max\left\{\Psi_{\alpha}(c\,|\xi|),\Psi_{2\beta'}\left(\frac{|\xi|}{2}\right)\right\}\,,\qquad \qquad t \geq0.
\end{equation*} 
Iterating this process, we see that, for any $j \in \N$, $j \geq 1$,
$$|\varphi(t,\xi)| \leq \max\left\{\Psi_{\alpha}(c\,|\xi|),\Psi_{2\alpha}\left(\frac{c\,|\xi|}{2}\right),\ldots,\Psi_{2^{j-1}\alpha}\left(\frac{c\,|\xi|}{2^{j-1}}\right),\Psi_{2^{j}\beta'}\left(\frac{|\xi|}{2^{j}}\right)\right\}\,$$
holds for any $\xi\in \R$ and $t \geq 0.$ Notice that
\begin{multline*}
  \max\left\{\Psi_{\alpha}(c\,|\xi|),\Psi_{2\alpha}\left(\frac{c\,|\xi|}{2}\right),\ldots,\Psi_{2^{j-1}\alpha}\left(\frac{c\,|\xi|}{2^{j-1}}\right),\Psi_{2^{j}\beta'}\left(\frac{|\xi|}{2^{j}}\right)\right\} \\
  \leq \max\left\{ \Psi_{\alpha}\left(\frac{c\,|\xi|}{2^{j-1}}\right), \Psi_{2^{j}\beta'}\left(\frac{|\xi|}{2^{j}}\right)\right\} \leq \Psi_{\alpha}\left(\frac{c\,|\xi|}{2^{j}}\right),
  \end{multline*}
as soon as $2^{j}\beta' \geq \alpha$. Setting
$$c_{0}=c \,2^{-j}\quad\text{with}\quad j=\Big\lfloor\max\bigg(\frac{\log\big(\alpha/\beta'\big)}{\log 2}\bigg)\Big\rfloor+1$$
the above condition is satisfied  and the result proved.
\end{proof}

\begin{rem}
Compare this result with Theorem 2 in \cite{FPTT}.  Again, the result here is not associated to a physical counterpart $g(t,x)$, yet it requires boundedness $|\varphi|\leq1$ linked to the mass of $g(t,x)$.

\smallskip
\noindent
It is pointed out in \cite[Lemma 14]{FPTT}, if a function $0 \leq h\in L^{1}$ with unitary norm satisfies that $\sqrt{h}\in \dot{H}^{\alpha}$ then $|\widehat{h}(\xi)| \leq \Psi_{\alpha}(c\,|\xi|)$ with $1/c^{\alpha}=\max\{2,2^{\alpha}\}\|\sqrt{h}\|_{\dot{H}^{\alpha}}$.
\end{rem}
\subsection{Higher regularity norms} Let $g$ be a solution to the Boltzmann problem \eqref{eq:IB-selfsim}-\eqref{eq:normalisation-g} with initial condition $g_{0}$. Then, its Fourier transform $\varphi$ is a solution of the self-similar problem \eqref{eq:selfsim-fourier} with initial condition $\varphi_0=\widehat{g_0}$.
Let us start the discussion by assuming that the initial datum $\varphi_{0}$ satisfies the baseline regularity 
$$\vertiii{\varphi_0 - \bm{\Phi}}_{k} <\infty \qquad \text{and} \qquad |\varphi_{0}(\xi)|\leq \Psi_{\beta}(c|\xi|)$$ for some $k\in(2,3)$, $c\in(0,1]$ and $\beta>0$.  Then, by Theorem \ref{theo-baseline} it holds that
\begin{equation}\label{under}
\sup_{t\geq0}|\varphi(t,\xi)| \leq \Psi_{\beta}(c_{0}|\xi|)\,,\,\qquad \quad \xi \in \R
\end{equation}
for some constant $c_{0}:=c_{0}(\beta,c,\vertiii{\varphi_0 - \bm{\Phi}}_{k})>0$.  With this estimate at hand we can propagate higher regularity norms:
\begin{theo}[\textit{\textbf{Sobolev norm propagation and relaxation}}]\phantomsection\label{theo:Sobolev}
Let $g(t)=g(t,x)$ be a solution to the Boltzmann problem \eqref{eq:IB-selfsim}-\eqref{eq:normalisation-g} with initial condition $g_{0}(x)=g(0,x)$ satisfying 
$$|\widehat{g_0}(\xi) | \leq \Psi_{\beta}(c\,|\xi|), \qquad \xi \in \R$$ for some $\beta,\,c>0$ and $g_0 \in H^{\ell}(\mathbb{R})$ for $\ell\geq0$.  Then, for $\frac{5}{2} < k < {3}$ and any $0<\sigma < \frac98 - \frac14 k - 2^{\frac{3}{2} - k}$ one has
\begin{equation}\label{H-relax}
  \| g(t) - \bm{H}\|_{H^{\ell}} \leq e^{-\sigma t}\,\Big( \| g_0 - \bm{H}\|_{H^{\ell}} +  C(\sigma,\ell,k)\,\| g_0 - \bm{H}\|_{ L^{1}(\w_{k}) } \Big)\,.
\end{equation}
\end{theo}
\begin{rem}
Compare this result with  \cite[Theorem 5]{FPTT}.  Theorem \ref{theo:Sobolev} is new, it proves propagation and convergence in Sobolev norms at the same time with detailed rates.  Furthermore, using the interpolation
\begin{equation*}
\| g(t) - \bm{H} \|_{L^1} \leq C\| (1+|x|)^{2(1+\kappa)}(g(t) - \bm{H}) \|^{\frac15}_{L^1}\| g(t) - \bm{H}\|^{\frac45}_{L^2}\,,
\end{equation*}
with $\kappa>0$ and Theorem \ref{theo:Sobolev} with $\ell=0$ show the exponential relaxation in the $L^{1}$ topology with $\frac45\sigma$ rate assuming the finiteness of the initial datum $L^{2}$ norm.  Similarly, the exponential convergence in $L^{\infty}$ with rate $\sigma$ is shown by taking $\ell>\frac12$ and using Sobolev embedding.
\end{rem}
\begin{proof} As before, we call $\varphi = \varphi(t,\xi)$
    the Fourier transform of $g$. Then $\varphi(t,\cdot)$ is a solution to
\eqref{eq:selfsim-fourier} with the normalisation
\eqref{eq:normalisation-varphi} and the difference $\psi(t,\xi)
:= \varphi(t,\xi) - \bm{\Phi}(\xi)$ with $\bm{\Phi}$ given in \eqref{eq:Phi} satisfies \eqref{nlinear-self-similar}. We introduce the notation 
$$\phi_{m}:=|\xi|^{m}\phi$$ for any $m>0$ and any mapping $\phi=\phi(\xi)$.  Multiplying the self-similar equation \eqref{nlinear-self-similar} by $|\xi|^{m}$ we obtain that the mapping $\psi_{m}(t,\xi)=|\xi|^{m}\psi(t,\xi)$ satisfies
\begin{equation*}
  \p_t \psi_m =
  \frac14 \,\xi\,  \p_\xi \psi_m
  + 2^{m}\psi_m\Big(t, \frac{\xi}{2} \Big)\varphi\Big(t, \frac{\xi}{2}\Big)
    + 2^{m}\psi\Big(t, \frac{\xi}{2} \Big)\bm{\Phi}_{m}\Big(\frac{\xi}{2}
    \Big) - \Big(1+\frac{m}{4}\Big)\psi_m.
\end{equation*}
We define $\left(T_{m}(t)\right)_{t\geq0}$ the semigroup associated to $\frac14 \xi  \p_\xi - \big(1+\frac{m}{4}\big)$, i.e.
$$T_{m}(t)g(\xi)=e^{-\left(1+\frac{m}{4}\right)t}g\left(\xi\,e^{\frac{1}{4}t}\right), \qquad t \geq0$$
and
\begin{equation*}
A_{m}(t,\xi) := 2^{m}\psi_m\Big(t, \frac{\xi}{2} \Big)\varphi\Big(t, \frac{\xi}{2}\Big)\,,\qquad B_{m}(t,\xi):=2^{m}\psi\Big(t, \frac{\xi}{2} \Big)\bm{\Phi}_{m}\Big( \frac{\xi}{2}
    \Big)
\end{equation*}
so that
\begin{equation} \label{eq:decay-Duhamel-m}
  \psi_{m}(t) = T_{m}(t) \psi_{m}(0) + \int_0^t T_{m}(t-s)\big( A_{m}(s) + B_{m}(s) \big)  \d s\,.
\end{equation}
Note that with a similar calculation as before, for any suitable $h$,
\begin{equation*}
  \vertiii{T_{m}(t) h }_{k,p}
  = e^{-\alpha_{m,p} t} \vertiii{h}_{k,p} \quad \text{ with } \quad  \alpha_{m,p} := 1 + \frac{m-k}{4} + \frac{1}{4p}.
\end{equation*}
Also,
\begin{align*}
  &\vertiii{A_{m}(s)}_{k,p}
  \leq
  2^{m - k + \frac{1}{p}} \vertiii{\psi_{m}(s)\,\varphi(s)}_{k,p}\leq 2^{m - k + \frac{1}{p}}\vertiii{\psi_{m-\beta}(s)}_{k,p}\|\varphi_{\beta}(s)\|_{L^\infty},\\
 &\qquad \text{and}\qquad \vertiii{B_{m}(s)}_{k,p}
  \leq
  2^{m - k + \frac{1}{p}}\| \bm{\Phi}_m\|_{L^\infty} \vertiii{\psi(s)}_{k,p}\,.
\end{align*}
Observe that H\"older's inequality implies that
\begin{equation*}
\vertiii{\psi_{m-\beta}(s)}_{k,p} \leq \vertiii{\psi_{m}(s)}^{1-\frac{\beta}{m}}_{k,p}\vertiii{\psi(s)}^{\frac{\beta}{m}}_{k,p}\,,\qquad m\geq\beta, 
\end{equation*}
and \eqref{under} leads to $\|\varphi_{\beta}(s)\|_{L^\infty}\leq c^{-\beta}_{0}$.  Consequently, using Young's inequality we are led to
\begin{multline*}
  \vertiii{\psi_{m}(t)}_{k,p} \leq
  \vertiii{T_{m}(t)\psi_{m}(0)}_{k,p} + \int_0^t \vertiii{T_{m}(t-s)\big(A_{m}(s) + B_{m}(s)\big)}_{k,p} \d s
  \\
  \leq
  e^{-\alpha_{m,p} t}\vertiii{\psi_{m}(0)}_{k,p}
  +
  \int_0^t e^{-\alpha_{m,p} (t-s)} \big( \vertiii{ A_{m}(s) }_{k,p} + \vertiii{ B_{m}(s) }_{k,p}\big) \d s
  \\
  \leq
  e^{-\alpha_{m,p} t}\vertiii{\psi_{m}(0)}_{k,p}
  +
   \int_0^t e^{-\alpha_{m,p} (t-s)} \Big(\varepsilon\,\vertiii{\psi_{m}(s)}_{k,p} 
  + \frac{C}{\varepsilon^{\frac{m}{\beta}-1}} \vertiii{\psi(s) }_{k,p}\Big)\d s,
\end{multline*}
for a constant that can be taken as $C:=c^{-m}_{0}\,2^{(m-k+\frac1p)\frac m\beta} + 2^{m - k + \frac{1}{p}} \, \| \bm{\Phi}_m\|_{L^\infty}$.  Note that 
$$\alpha_{m,p}>\sigma_{k}(p)>0,$$  
where we recall that $\sigma_{k}(p) = 1 - \frac14 k + \frac{1}{4p} - 2^{1 + \frac{1}{p} - k}$.  Therefore, thanks to Theorem \ref{k-norm-cvgce} it follows that
\begin{equation*}
\int_0^t e^{-\alpha_{m,p} (t-s)}\vertiii{\psi(s) }_{k,p}\d s \leq \frac{e^{-\sigma_{k}(p) t} }{\alpha_{m,p}-\sigma_{k}(p)}\, \vertiii{\psi(0)}_{k,p} \,\qquad \qquad t\geq0.
\end{equation*} 
As a consequence, calling $u(t) := e^{\sigma_{k}(p) t}\,\vertiii{ \psi_{m}(t)}_{k,p}$ we see that
\begin{equation*}
u(t) \leq \vertiii{\psi_{m}(0)}_{k,p} +  \frac{C\, \vertiii{\psi(0) }_{k,p}}{\varepsilon^{\frac{m}{\beta}-1}(\alpha_{m,p}-\sigma_{k}(p))} + \varepsilon \int_0^t u(s) \d s,
\end{equation*}
which, by Gronwall's lemma, immediately gives  that
\begin{equation}\label{higher-norm:relax}
  \vertiii{\psi_{m}(t)}_{k,p} \leq e^{-(\sigma_{k}(p) - \varepsilon) t}\,\bigg(\vertiii{\psi_{m}(0)}_{k,p} +  \frac{C\, {\vertiii{\psi(0)}_{k,p}}}{\varepsilon^{\frac{m}{\beta}-1}(\alpha_{m,p}-\sigma_{k}(p))}\bigg)\,.
\end{equation}
One chooses  $p=2$ and $m = k$ so that 
\begin{equation*}
\vertiii{\psi_{m}(t)}_{k,p} = \| \psi(t) \|_{L^2} = \| g(t) - \bm{H} \|_{L^2}\,
\end{equation*} 
thanks to Parseval identity. Moreover, one has,  {for $2<k<3$ (see Lemma \ref{lem:mukvertkp})}
\begin{equation*}
 {\vertiii{\psi(0)}_{k,2}} \leq C\| g_0 - \bm{H}\|_{L^1(\w_{k})}\,.
\end{equation*}
Consequently, from \eqref{higher-norm:relax} one obtains the exponential relaxation in $L^{2}(\mathbb{R})$ as
\begin{equation}\label{L2:relax}
  \|g(t)- \bm{H}\|_{L^2} \leq e^{-(\sigma_{k}(2) - \varepsilon) t}\,\bigg( \| g_0 - \bm{H}\|_{L^2} +  \frac{C\,\| g_0 - \bm{H} \|_{ L^{1}(\w_k) } }{\varepsilon^{\frac{k}{\beta} - 1}(\alpha_{k,2}-\sigma_k(2)})\bigg)\,,\qquad \frac52 < k < {3}.
\end{equation}
More generally, for any $\ell>0$ one can choose $m=\ell+k$, $p=2$ and use the fact that 
\begin{equation*}
\vertiii{\psi_{\ell+k}(t)}_{k,2} = \| |\xi|^{\ell}\psi(t) \|_{L^2} = \| g(t) - \bm{H} \|_{\dot{H}^{\ell}}\,.
\end{equation*} 
Consequently, \eqref{higher-norm:relax} implies that
\begin{equation}\label{Hom-H-relax}
  \| g(t) - \bm{H}\|_{\dot{H}^{\ell}} \leq e^{-(\sigma_{k}(2) - \varepsilon) t}\,\bigg( \| g_0 - \bm{H}\|_{\dot{H}^{\ell}} +  \frac{C\,\| g_0 - \bm{H} \|_{ L^{1}(\w_{k}) } }{\varepsilon^{\frac{\ell+k}{\beta} - 1}(\alpha_{\ell+k,2}-\sigma_k(2)})\bigg)\,,\qquad \frac52 < k < {3}.
\end{equation}
Estimates \eqref{L2:relax}-\eqref{Hom-H-relax} gives the theorem.\end{proof}
\subsection{Spectral gap in Fourier norms}
\label{sec:sg}

 {We prove in this section, that the linearised operator $\mathscr{L}$ has a spectral gap with respect to the norms $\vertiii{\cdot}_{k}$ and $\vertiii{\cdot}_{k,p}$. The proof follows exactly the lines of the proof of the above Theorem \ref{k-norm-cvgce} and turns out to be simpler so we just describe the main steps of it. We begin with the following 
 \begin{defi}\label{def:L}
   We define the linearised operator 
 $$\mathscr{L}:\mathscr{D}(\mathscr{L}) \subset X_0\to X_0$$
 with the Banach space $X_0$ defined in \eqref{eq:X0:measure} by
 \begin{equation*}
 \mathscr{L}h (x) := -\frac14 \p_x( x h) + 2\Q_{0}(h,\bm{H}),
\end{equation*}
and $\mathscr{D}(\mathscr{L})=\{h \in X_{0}\;;\;\p_{x}(x h) \in X_{0}\}.$
\end{defi}
The linearised operator $\mathscr{L}$ corresponds of course to  first order expansion, for $h$ small, of the quantity
$-\frac14 \p_x (xg) + \Q_{0}(g,g)$ for $g = \bm{H} + h$. The existence of a spectral gap is useful for the study of the linearised equation
\begin{equation}
  \label{eq:linearised}
  \p_t h = \mathscr{L}h = -\frac14 \p_x( x h) + 2 \Q_{0}(h, \bm{H}),
\end{equation}
with initial datum $h(0,x)=h_{0}(x)=f_{0}(x)-\bm{H}$ such that 
\begin{equation}
   \label{eq:normalisation-lin}
   \int_\R h_{0}(x) \dx =
   \int_\R x h_{0}(x) \dx  = \int_\R x^2 h_{0}(x) \dx = 0,
 \end{equation}
 Our main result is then the following spectral gap estimate in the Fourier norms $\vertiii{\cdot}_{k}$ and $\vertiii{\cdot}_{k,p}$:
\begin{theo}\label{specgap}
  Assume that $h = h(t,x)$ is a solution to
  \eqref{eq:linearised} with the normalisation
  \eqref{eq:normalisation-lin} and call $\psi = \psi(t,\xi)$ its
  Fourier transform in the $x$ variable. Then, for $0 \leq k < 3$
  \begin{equation*}
    \vertiii{\psi(t)}_k \leq \exp\left(-\sigma_{k} t\right) \vertiii{\psi_0}_k 
    \qquad \forall t \geq0\,,\end{equation*}
    where $\psi_{0}(\xi)=\psi(0,\xi)$ and
$\sigma_{k}:= 1 - \frac14 k - 2^{1-k}.$  
 In particular, $\psi(t)$ converges exponentially to $0$ in the $k$-Fourier norm for any $2 < k < 3$. Moreover, for any $1 \leq p < \infty$, 
\begin{equation*}
\vertiii{\psi(t)}_{k,p} \leq \exp\left(-\sigma_{k}(p) t\right)\vertiii{\psi_{0}}_{k,p} \qquad \forall t \geq0,\end{equation*}
where $\sigma_{k}(p)=1 - \frac14 k + \frac{1}{4p} + 2^{1 + \frac{1}{p} - k}.$ 
\end{theo}

 \begin{proof}   
One directly sees that,  {under the normalisation \eqref{eq:normalisation-lin}}, the equation \eqref{eq:linearised}
preserves mass, momentum and energy. Notice also that, for $h$ satisfying \eqref{eq:normalisation-lin}, 
$$2\Q_{0}(h,\bm{H})(x)=4\left(h\ast \bm{H}\right)(2x)-h(x), \qquad x \in \R.$$

If $h$ is a solution to
\eqref{eq:linearised} and $\psi=\psi(t,\xi)$ is the Fourier transform of $h(t,x)$ in
the $x$ variable, $\psi(t,\xi)$ satisfies the equation
\begin{equation}
  \label{eq:linearised-fourier}
  \p_t \psi(t,\xi) = \frac14 \xi  \p_\xi \psi(t,\xi)
  + 2 \psi\Big(t, \frac{\xi}{2} \Big) \bm{\Phi}\Big( \frac{\xi}{2} \Big)
  - \psi(t,\xi),
\end{equation}
which corresponds of course to the linearisation of
\eqref{eq:selfsim-fourier} around $\bm{\Phi}$. In much the same way we did
for the nonlinear equation, we can show equation
\eqref{eq:linearised-fourier} converges to equilibrium exponentially
fast: by Duhamel's formula,
\begin{equation}
  \label{eq:lin-Duhamel}
  \psi(t) = T(t) \psi_0 + \int_0^t T(t-s) B(s) \d s,
\end{equation}
where now 
\begin{equation*}
  B(s)=B(s,\xi) := 2 \psi\Big(s, \frac{\xi}{2} \Big)
  \bm{\Phi}\Big(\frac{\xi}{2} \Big).
\end{equation*}
Similarly to our calculation in Section \ref{Sec:conv:eq:fourier} we have
\begin{equation*}
  |B(s,\xi)| \leq 2 \left|\psi\left(s, \frac{\xi}{2}\right)\right|,
  \qquad
  \text{so}
  \qquad
  \vertiii{B(s)}_{k}
  \leq
  2^{1-k} \vertiii{\psi(s)}_k,
\end{equation*}
and we can use again \eqref{eq:decay1} to obtain that
\begin{multline*}
  \vertiii{\psi(t)}_{k} \leq
  \vertiii{T(t) \psi_0}_k + \int_0^t \vertiii{T(t-s) B(s)}_k \d s
  \\
  \leq
  e^{-(1 - \frac14 k) t}\vertiii{\psi_0}_{k}
  +
  \int_0^t e^{-(1 - \frac14 k) (t-s)} \vertiii{B(s)}_{k} \d s
  \\
  \leq
  e^{-(1 - \frac14 k) t}\vertiii{\psi_0}_k
  +
  2^{1-k} \int_0^t e^{-(1 - \frac14 k) (t-s)} \vertiii{ \psi(s) }_k \d s.
\end{multline*}
With the same argument as  in Theorem \ref{k-norm-cvgce} we derive the first result. The obtention of the second result also follows the same lines as in Theorem \ref{k-norm-cvgce}. \end{proof}
 
\subsection{Spectral gap in smaller spaces}
\label{sec:sg-small}

As explained in the Introduction, it is important to obtain an equivalent of the above Theorem \ref{specgap} in the more tractable space (see Definition \ref{def:spaces})  
$${\mathbb{Y}_{a}^{0}}=\left\{ f\in L^1(\w_{a}) \mid \int_{\R} f(x) \d x = \int_{\R} x\, f(x) \d x = \int_{\R} x^2 \,f(x) \d x = 0  \right\},$$ 
for some $a>0$ to be determined.  Recalling that $\mathbb{Y}_{a}^{0} \subset X_{0}$ for any $a \geq k$ and since,  by Theorem \ref{specgap}, the linearised operator 
$\mathscr{L}$, with domain 
$$\mathscr{D}(\mathscr{L})=\{f \in \mathbb{Y}^{0}_{a}\;;\;\partial_{x}(xf(x)) \in L^{1}(\w_{a})\},$$ has a spectral gap in $X_0$ for $2<k<3$. Our scope here is to prove that $\mathscr{L}$ still has a spectral gap (of comparable size) in the space $\mathbb{Y}_{a}^{0}$, namely

\begin{theo}\label{restrict}
  Let $2<a<3$. The operator $(\mathscr{L},\mathscr{D}(\mathscr{L}))$ generates a strongly continuous semigroup $\left(S_{0}(t)\right)_{t\geq0}$ on $\mathbb{Y}_{a}^{0}$ and for any $ {\nu \in(0,1-\frac{a}{4} -2^{1-a})}$, there exists $C(\nu)>0$ such that
  $$ \|S_{0}(t)h\|_{L^1(\w_{a})}\le C(\nu) e^{-\nu t} \|h\|_{L^1(\w_{a})}$$
  for any $h\in \mathbb{Y}_{a}^{0}$ and any $t\ge 0$. Moreover, one has
  $$\|\mathscr{L}h\|_{L^1(\w_a)}\ge \frac{\nu}{C(\nu)} \|h\|_{L^1(\w_a)}, \qquad \text{for any } h \in \mathscr{D}(\mathscr{L}).$$
\end{theo}

To prove such a result, as explained already in the Introduction, we resort to results from \cite{CanizoThrom,GMM} and split the linearised operator as 
$$\mathscr{L}=A+B,$$ with 
$$A: X_0 \to \mathbb{Y}_{a}^{0} \qquad \text{ bounded }$$
and $B$ enjoying some dissipative properties.} To this end we introduce some truncation function and some projection from $L^1(\w_{a})$ to $\mathbb{Y}_{a}^{0}$. 

For $R>1$ we consider nonnegative functions $\rho_R$  {and} $\theta_R\in {\mathcal C}^\infty(\R)$ which are bounded by $1$ and satisfy 
$$\theta_R(x)=\rho_{R}(x)=1 \qquad \quad \text{ for $|x|\leq \frac{R}{2}$}$$ 
and
$$\theta_R(x)=0 \quad \text{ for $|x|\geq \frac{R}{2}+1$}, \qquad  \rho_R(x)=0  \quad \text{ for $|x|\geq \frac{2}{3}R$.}$$ Let us now introduce the normalised Maxwellian 
$$\M(x)=\dfrac{e^{-x^2}}{\sqrt{\pi}}, \qquad x\in \R$$ and
$$
\zeta_1(x)=\left(\frac{3}{2} -x^2\right) \M(x), \qquad \zeta_2(x)=2x\,\M(x), \qquad
\zeta_3(x)=(-1+2x^2)\,\M(x).$$ We then define a bounded operator $\P :L^1(\w_{a})\to L^1(\w_{a})$ by 
\begin{equation}\label{projection}
\P h(x)= \zeta_{1}(x)\int_\R h(y)\dy \; +\zeta_{2}(x)\int_\R h(y) \, y \, \dy \;+\zeta_{3}(x)\int_\R h(y)\, y^2\, \dy\;,\qquad \quad (x \in \R). 
\end{equation}
For any $f\in L^1(\w_{a})$, one easily checks that 
$$f-\P(f)\in \mathbb{Y}_{a}^{0}.$$ 
Let us split $\mathscr{L}$ as $\mathscr{L}=A+B$ with 
$$A=A_1+A_2, \qquad \text{ and  } \qquad B=B_1+B_{2}+B_3,$$ 
where 
\begin{equation*}\begin{split}
A_1h(x)&=  4\theta_R(x)\,((h\rho_{R})\ast \bm{H})(2x),   \qquad  A_2h=-\P(A_1h), \\
  B_1h(x)&= -\frac{1}{4}\partial_x(xh)-h, \qquad  B_{3}h=\P(A_{1}h)
\end{split}\end{equation*}
and $B_{2}=B_{2,1}+B_{2,2}$ with
$$B_{2,1}h(x)=  4(1-\theta_R(x))\,((h\rho_{R})\ast \bm{H})(2x), \qquad B_{2,2}h= 4((h(1-\rho_{R}))\ast \bm{H})(2x).$$
Recalling that
$$\mathscr{L}h(x)=-\frac{1}{4}\partial_{x}(xh(x))-h(x)+4\left(h \ast \bm{H}\right)(2x)$$
for any $h$ satisfying \eqref{eq:normalisation-lin}, one sees that, indeed, $A+B=A_{1}+A_{2}+B_{1}+B_{2,1}+B_{2,2}+B_{3}=\mathscr{L}.$
The main property of $B=B_{1}+B_{2}+B_{3}$ is established in the following

\begin{prp}\phantomsection\label{prp:dissi}
Let $a>0$ satisfying $1-\dfrac{a}{4} -2^{1-a}>0$. Then, for any $0\leq \nu <1-\dfrac{a}{4} -  {2^{1-a}}$, the operator
$B+\nu$ is dissipative in $L^{1}(\w_{a})$, i.e.
$$\int_{\R}Bh(x)\mathrm{sign}(h(x))\w_{a}(x)\dx \leq -\nu \int_{\R}|h(x)|\w_{a}(x)\d x, \qquad \forall h \in \mathscr{D}(\mathscr{L}) \subset L^{1}(\w_{a}).$$
\end{prp}

This proposition is a direct consequence of the following three lemmas. Let us note that   $1-\dfrac{a}{4} -2^{1-a}>0$ for any $2<a<3$.

\begin{lem}\phantomsection
For any $h \in \mathscr{D}(\mathscr{L}) \subset L^{1}(\w_{a})$, 
\begin{equation}\label{B1}
\int_{\R} B_1h(x) \, \mathrm{sign}(h(x))\w_{a}(x) \d x\leq \frac{a-4}{4}\int_{\R} |h(x)| \w_{a}(x) \d x.
\end{equation}
\end{lem}

\begin{proof}
Since $B_1h=-\dfrac{1}{4}\;x \partial_xh-\dfrac{5}{4} \;h$, an integration by parts leads to  

\begin{eqnarray*}
& & \hspace{-1cm}\int_{\R} B_1h(x) \ \mbox{sign}(h(x))\,\w_{a}(x) \d x\\
& & \hspace{2cm} = \frac{1}{4} \int_{\R}|h(x)|\, (\w_a(x)+a|x|\w_{a-1}(x)) \d x- \frac{5}{4} \int_{\R}|h(x)|\, \w_{a}(x) \d x\\
&& \hspace{2cm} \leq -\int_{\R} |h(x)|\, \w_{a}(x) \d x+\frac{a}{4} \int_{\R} |h(x)| \,\w_{a}(x) \d x,
\end{eqnarray*}
since $|x|\w_{a-1}(x)\leq \w_{a}(x)$
and \eqref{B1} follows.
\end{proof}

\begin{lem}\phantomsection
For any $a\in(2,3)$ and any $\varepsilon>0$, there exists $R>1$ such that for any $h\in L^{1}(\w_{a})$, 
\begin{equation}\label{B2}
\begin{split}
 \int_{\R} |B_{2,1}h(x)|\, \w_{a}(x)\d x &\leq \varepsilon  \int_{\R} |h(x)|\,\w_{a}(x) \d x\\
   \int_{\R} |B_{2,2}h(x)|\, \w_{a}(x)\d x &\leq (2^{1-a} +\varepsilon)  \int_{\R} |h(x)|\,\w_{a}(x) \d x\,.
\end{split}
\end{equation}
\end{lem}

\begin{proof}

We start with $B_{2,2}$ and a change of variables leads to 
\begin{multline}\label{eq:B2:1}
\int_{\R} |B_{2,2}h(x)| \,\w_{a}(x)\d x   = 2^{1-a}\int_{\R} |((h(1-\rho_R))\ast \bm{H})(x)| (2+|x|)^{a} \d x \\*
= {2^{1-a}} \int_{\R} |((h(1-\rho_R))\ast \bm{H})(x)| |x|^a \d x + 2^{1-a}\int_{\R} |((h(1-\rho_R))\ast \bm{H})(x)| \bigl((2+|x|)^{a}-|x|^a\bigr) \d x\,\\*
\leq {2^{1-a}} \int_{\R} |((h(1-\rho_R))\ast \bm{H})(x)| |x|^a \d x + 2^{2-a}a\int_{\R} |((h(1-\rho_R))\ast \bm{H})(x)|(2+|x|)^{a-1} \d x\,.
\end{multline}
On the one hand, since $\rho_R \in [0,1]$, we deduce that  
$$\int_{\R} |((h(1-\rho_R))\ast \bm{H})(x)|\, |x|^a \d x \leq \int_{\R}\int_{|y|\geq \frac{R}{2}}|h(y)|\,  \bm{H}(x-y) \, |x|^a \d y\d x\,.$$
Now, writing $a=p \alpha $ with $\alpha \in (0,1)$ and $p\in\N$, we have 
\begin{multline*}
|x|^a=|x-y+y|^{p\alpha} \leq  (|x-y|+|y| )^{p\alpha} =  \left(\sum_{j=0}^p \left(
\begin{array}{c}
p\\j
\end{array}\right) |x-y|^j\, |y|^{p-j}\right)^\alpha \\
\leq \sum_{j=0}^p \left(
\begin{array}{c}
p\\j
\end{array}\right)^\alpha |x-y|^{j\alpha} \, |y|^{(p-j)\alpha}=  |y|^{p\alpha} + \sum_{j=1}^p \left(
\begin{array}{c}
p\\j
\end{array}\right)^\alpha |x-y|^{j\alpha}\,|y|^{(p-j)\alpha}.
\end{multline*}
Consequently, recalling $a=p\alpha$,
\begin{multline*}
\int_{\R} |((h(1-\rho_R))\ast \bm{H})(x)|\, |x|^a \d x \leq \int_{\R}\int_{|y|\geq \frac{R}{2}}|h(y)|\, \bm{H}(x-y) \, |y|^a \d y\d x  \\+ \sum_{j=1}^p \left(
\begin{array}{c}
p\\j
\end{array}\right)^\alpha  \int_{\R}\int_{|y|\geq \frac{R}{2}} |h(y)|\,\bm{H}(x-y) \,|x-y|^{j\alpha}\,|y|^{(p-j)\alpha}\d y\d x\,. 
\end{multline*}
Since $\bm{H}$ has mass $1$ and $R>2$, we obtain 
\begin{multline*}
\int_{\R} |((h(1-\rho_R))\ast \bm{H})(x)|\, |x|^a \d x \\
\leq \int_{\R} |h(y)| \, |y|^a \d y  + \sum_{j=1}^p \left(
\begin{array}{c}
p\\j
\end{array}\right)^\alpha  2^{j\alpha}\int_{\R}\int_{\R} \frac{|y|^a}{R^{j\alpha}} |h(y)|\,  \bm{H}(x-y) \,|x-y|^{j\alpha}\d y\d x \\
\leq \int_{\R} |h(y)| \, |y|^a \d y  + 2^{a}R^{-\alpha}\sum_{j=1}^p \left(
\begin{array}{c}
p\\j
\end{array}\right)^\alpha  \int_{\R}|y|^{a}|h(y)|\dy \int_{\R}\bm{H}(x-y) \,|x-y|^{j\alpha}\d x\,.
\end{multline*}
Since $j \leq p$, $j\alpha \leq a$ and $|x-y|^{j\alpha} \leq 1+|x-y|^{a}$ and for $a<3$, $\bm{H}\in L^1(\w_{a})$, we conclude that 
\begin{equation}\label{eq:B2:2}
\int_{\R} |((h(1-\rho_R))\ast \bm{H})(x)|\, |x|^a \d x\leq \left(1+ \frac{C}{R^\alpha}\right) \int_{\R} |h(y)| \, |y|^a \d y,  
\end{equation}
for some constant $C>0$ depending on $a$ and $\|\bm{H}\|_{L^1(\w_{a})}$. Similarly, one has
$$ 2^{2-a}a\int_{\R} |((h(1-\rho_R))\ast \bm{H})(x)|(2+|x|)^{a-1} \d x\leq 2^{2-a}a\int_{|y|\geq \frac{R}{2}}\int_{\R} |h(y)|\bm{H}(x) (2+|x+y|)^{a-1} \d x$$
and, using that 
$$(2+|x+y|)^{a-1}\leq 2^{a-1}\frac{(1+|x|+|y|)^{a}}{1+|y|}\leq 2^{a-1}\frac{\w_a(x)\w_a(y)}{1+|y|},$$ we have
\begin{multline}\label{eq:B2:3}
 2^{2-a}a\int_{\R} |((h(1-\rho_R))\ast \bm{H})(x)|(2+|x|)^{a-1} \d x \\
 \leq \frac{4a}{2+R}\int_{\R}|h(y)|\w_{a}(y)\dy\int_{\R}\bm{H}(x)\w_{a}(x)\dx
 \leq \frac{C_{a}}{2+R}\int_{\R}|h(y)|\w_{a}(y)\dy.
\end{multline}
With this at hands, the second estimate in \eqref{B2} is a consequence of \eqref{eq:B2:1} together with \eqref{eq:B2:2} and \eqref{eq:B2:3} if we choose $R$ large enough.
 
For the first bound in \eqref{B2} we proceed similarly and first change variables and use the properties of the cutoff functions to get
\begin{equation*}\begin{split}
 \int_{\R}|B_{2,1}h(x)|\w_a(x)\dx &= 2\int_{\R}|((h\rho_R)\ast \bm{H})(x)|\left(1-\theta_R\left(\frac{x}{2}\right)\right)\w_a\left(\frac{x}{2}\right)\dx\\*
& \leq 2\int_{|x|\geq R}\int_{\R}|h(y)|\rho_R(y)\bm{H}(x-y)\left(1+\abs{\frac{x}{2}}\right)^{a}\dy\dx\\*
& \leq 2\int_{|x|\geq R}\int_{|y|\leq \frac{2}{3}R}|h(y)|\bm{H}(x-y)\left(1+|x-y|+|y|\right)^{a}\dy\dx\\
& \leq 2\int_{|x|\geq R}\int_{|y|\leq \frac{2}{3}R}|h(y)|\bm{H}(x-y)\w_{a}(y)\w_{a}(x-y)\dy\dx\,.
\end{split}\end{equation*}
We next exploit that $\bm{H}\in L^1(\w_{\frac{3+a}{2}})$ for $a<3$ and $|x-y|\geq |x|-|y|\geq \frac{R}{3}$ for $|x|\geq R$ and $|y|\leq \frac{2R}{3}$ to deduce
\begin{multline*}
 \int_{\R}|B_{2,1}h(x)|\w_a(x)\dx \leq 2 \int_{|x|\geq R}\int_{|y|\leq \frac{2}{3}R}|h(y)|\bm{H}(x-y)\w_{a}(y)\frac{\w_{\frac{a+3}{2}}(x-y)}{(1+|x-y|)^{\frac{3-a}{2}}}\dy\dx\\*
 \leq C\Bigl(1+\frac{R}{3}\Bigr)^{-\frac{3-a}{2}}\int_{\R}|h(y)|\w_a(y)\dy.
\end{multline*}
Since $2<a<3$, the first estimate in \eqref{B2} follows if we choose $R$ sufficiently large.
\end{proof}

\begin{lem}\phantomsection
For any $a\in (2,3)$ and any $\varepsilon>0$, there exists $R>1$ such that 
\begin{equation}\label{B3}
 \int_{\R} |B_3h(x)|\, \w_{a}(x)\d x \leq \varepsilon  \int_{\R} |h(x)|\, \w_{a}(x) \d x \qquad \forall h \in \mathscr{D}(\mathscr{L}) \subset \mathbb{Y}_{a}^{0}.
\end{equation}
\end{lem}

\begin{proof} Recall that $B_3h\,= \P(A_1h)$. Let us compute the first moments of $A_1h$. Using that $h\in \mathbb{Y}_{a}^{0}$ and that $\bm{H}$ has mass $1$,  momentum $0$ and energy $1$, one obtains 
\begin{eqnarray*}
\int_\R A_1h(x)\dx & = &  {2\int_\R h(x-y)\rho_{R}(x-y) \,\int_\R \bm{H}(y)\theta_R\left(\frac{x}{2}\right)\dx\dy}\\
& =&  - {2}\int_{\R^2} h(x)\bm{H}(y) \left[1-\rho_R(x)\theta_R\left(\frac{x+y}{2}\right)\right]\dy\dx,
 \end{eqnarray*}
\begin{eqnarray*}
\int_\R A_1h(x)\ x \dx & =&   {2} \int_\R h(x-y)\rho_R(x-y)\, \int_\R \frac{x}{2}\theta_R\left(\frac{x}{2}\right) \;\bm{H}(y) \dx\dy\\
&=& -  \int_{\R^2} h(x)\bm{H}(y)\left[1-\rho_R(x)\theta_R\left(\frac{x+y}{2}\right)\right]\, (x+y) \dy\dx,
 \end{eqnarray*}
and 
\begin{eqnarray*}
\int_\R A_1h(x)\ x^2 \dx & =&   {2}\int_\R h(x) \rho_R(x)\, \int_{\R} \left(\frac{x+y}{2}\right)^2\theta_R\left(\frac{x+y}{2}\right) \;\bm{H}(y)\dx\dy \\
& = &  -  {\frac{1}{2}}\int_{\R^2} h(x)\bm{H}(y)\left[1-\rho_R(x)\theta_R\left(\frac{x+y}{2}\right)\right](x+y)^2  \dy\dx.
 \end{eqnarray*}
Consequently, one easily gets that
$$\left|B_{3}h\right| \leq   2\left(\int_{\R^2}|h(x)|\bm{H}(y)\left|1-\rho_R(x)\theta_{R}\left(\frac{x+y}{2}\right)\right|\w_{2}(x+y)\dy\dx\right)\sum_{i=1}^{3}\left|\zeta_{i}(\cdot)\right|$$
since $\max\left(1,|z|,(1+|z|^2)\right)\leq \w_{2}(z)$ and thus
\begin{equation*}
 \left|B_{3}h \right| \leq   2\left(\int_{\R^2}|h(x-y)|\bm{H}(y)\left|1-\rho_R(x-y)\theta_{R}\left(\frac{x}{2}\right)\right|\w_{2}(x)\dy\dx\right)\sum_{i=1}^{3}\left|\zeta_{i}(\cdot)\right|.
\end{equation*}
We next use the properties of the cutoff functions $\theta_R$ and $\rho_R$ together with $\w_s(x) \leq \w_s(y)\w_s(x-y)$ for $s\in\{a,2\}$ to deduce that
{\begin{equation*}\begin{split}
      \abs{1-\rho_R(x-y)\theta_R\Bigl(\frac{x}{2}\Bigr)}\w_2(x)
      &\leq \bm{1}_{\{|x-y|\geq \frac{R}{2}\}}\frac{\w_a(x-y)\w_a(y)}{\w_{a-2}(x-y)\w_{a-2}(y)} + \bm{1}_{\{|x|\geq R\}} \frac{\w_a(x)}{\w_{a-2}(x)} \\*      
 &\leq \frac{2^{a-2}}{(2+R)^{a-2}}\w_{a}(x-y)\w_{a}(y) + \frac{1}{(1+R)^{a-2}}\w_{a}(x)\\*      
 &\leq \frac{C}{(1+R)^{a-2}}\w_{a}(x-y)\w_{a}(y).
\end{split}\end{equation*}}
This yields
\begin{multline*}
\|B_3h\|_{L^{1}(\w_{a})}
 \leq \frac{C}{ {(1+R)}^{a-2}}\left(\int_{\R}\bm{H}(y)\w_a(y)\int_{\R} |h(x-y)| {\w_a}(x-y)\dx\dy\right)\sum_{i=1}^{3} \|\zeta_i\|_{L^1(\w_{a})} \\
 \leq  \frac{C}{ {(1+R)}^{a-2}} \|h\|_{L^1(\w_{a})},
\end{multline*}
for some contant $C>0$ where we also used $\bm{H}\in L^{1}(\w_{a})$. We then deduce that \eqref{B3} holds provided $R$ is large enough.
\end{proof}
\begin{proof}[Proof of Proposition \ref{prp:dissi}]
The proof follows directly from the combination of \eqref{B1}--\eqref{B2}--\eqref{B3} since it implies that, for any $\varepsilon >0$, one can choose $R >1$ large enough so that
$$\int_{\R}Bh(x)\mathrm{sign}(h(x))\w_{a}(x)\dx \leq -\left(1-\frac{a}{4}-2^{1-a}-3\varepsilon\right)\|h\|_{L^{1}(\w_{a})} \qquad \forall h \in \mathscr{D}(\mathscr{L}) \subset \mathbb{Y}_{a}^{0}$$
which gives the result choosing $\varepsilon >0$ small enough so that $\nu=1-\frac{a}{4}-2^{1-a}-2\varepsilon \geq0.$
\end{proof}
We establish now the regularising effect of $A$:
\begin{prp}\label{prp:bounded}
Let $2<a<3$. The operator $A:X_0\to \mathbb{Y}_a^{0}$ is bounded. 
\end{prp}

This proposition follows directly from the following two lemmas. 

\begin{lem}\phantomsection \label{bound_A1}
Let $2<a<3$. There exists some constant $C>0$ such that, for any $h\in X_0$
$$\|A_1h\|_{L^1(\w_{a})} \leq  C \vertiii{h}_{k}$$
for any $k>2$.
\end{lem}

\begin{proof}
First, one observes as before that
\begin{multline*}
\|A_1h\|_{L^1(\w_{a})} \leq  2\int_{\R} |((h\rho_R)\ast \bm{H})(x)|\theta_R\left(\frac{x}{2}\right) \,\w_{a}\left(\frac{x}{2}\right) \d x\\
\leq  2\int_{\R} |((h\rho_R)\ast \bm{H})(x)|\theta_R\left(\frac{x}{2}\right) \,\w_{a}(x) \d x\end{multline*}
where we used that $\w_{a}\left(\frac{x}{2}\right) \leq \w_{a}(x)$. We then deduce from the Cauchy-Schwarz inequality that 
$$
\|A_1h\|_{L^1(\w_{a})}
\leq 2\left(\int_{\R} |((h\rho_R)\ast \bm{H})(x)|^2\,\theta^2_R\left(\frac{x}{2}\right) \w_{a}(x)^2\,(1+|x|)^{2\chi} \d x\right)^{\frac{1}{2}}\left( \int_{\R} \frac{\d x}{(1+|x|)^{2\chi}} \right)^{\frac{1}{2}}
$$
with $\chi>\frac{1}{2}$. Thus, it holds
\begin{eqnarray}
\|A_1h\|_{L^1(\w_{a})}  & \leq &2\|\w_{-\chi}\|_{L^{2}} \,\|((h\rho_R)\ast \bm{H})\theta_R\left(\frac{\cdot}{2}\right)\,\w_{a+\chi}\|_{L^2} \nonumber \\
& \leq & C_{a,\chi} \,\|(h\rho_R)\ast \bm{H} \|_{L^2} + C_{a,\chi} \,\left\|((h\rho_R)\ast \bm{H})\theta_R\left(\frac{\cdot}{2}\right)\, |\cdot|^{a+\chi}\right\|_{L^2} \label{sepA}
\end{eqnarray}
where we used that $\w_{a+\chi}\leq C_{a,\chi}\left(1+|\cdot|^{a+\chi}\right)$ for some $C_{a,\chi} >0$ and we also used that $|\theta_R|\leq 1$. Let us first consider the first term in the right-hand side of \eqref{sepA}. We deduce from the properties of the Fourier transform that  
$$\|(h\rho_R)\ast \bm{H} \|_{L^2}= \frac{1}{\sqrt {2\pi}}\; \|\widehat{(h\rho_R)\ast \bm{H}} \|_{L^2}  = \frac{1}{\sqrt {2\pi}}\;\|\widehat{h\rho_R}\cdot \widehat{\bm{H}} \|_{L^2}= \frac{1}{(2\pi)^{\frac{3}{2}}} \; \|(\widehat{h}\ast \widehat{\rho_R})\, \widehat{\bm{H}} \|_{L^2}.$$
{We have $|\eta|^k\leq (|\xi-\eta|+|\xi|)^k\leq \w_{k}(\xi-\eta)\w_{k}(\xi)$.} Thus, 
\begin{equation}\label{eq:convolution:fourier}
|(\widehat{h}\ast \widehat{\rho_R})(\xi) |
\leq   \vertiii{h}_{k} \int_{\R} |\eta|^k \, |\widehat{\rho_R}(\xi-\eta)| \d \eta \leq  \vertiii{h}_{k}\, \w_{k}(\xi)\,  \|\widehat{\rho_R}\|_{L^1(\w_{k})}. 
\end{equation}
Hence, 
$$\|(h\rho_R)\ast \bm{H} \|_{L^2} \leq  \frac{ {1}}{(2\pi)^{\frac{3}{2}}}\; \vertiii{h}_{k} \, 
\| \widehat{\rho_R} \|_{L^1(\w_k)} \, \| \w_{k}(\cdot)\widehat{\bm{H}}\|_{L^2}\,.$$
Since $\rho_R \in {\mathcal C}^\infty(\R)$ is compactly supported, $\widehat{\rho_R}\in L^1(\w_{k})$ for any {$k>2$}. Furthermore, $\w_{k}(\xi)\widehat{\bm{H}}(\xi)= (1+|\xi|)^{k+1}e^{-|\xi|} \in L^2(\R)$  for any  {$k>2$}. Consequently, there exists some constant $C_1(k,R)>0$ such that 
\begin{equation}\label{A1}
\|(h\rho_R)\ast \bm{H} \|_{L^2} \leq  C_1(k,R) \vertiii{h}_{k}.
\end{equation}
Let us now consider the last term in the right-hand side of \eqref{sepA}. Set
$$F(x)=((h\rho_R)\ast \bm{H})(x)\theta_R\left(\frac{x}{2}\right)\, {|x|^{a+\chi}}= {|x|^{a+\chi}}F_{0}(x).$$
Notice that, as previously,
$$\left\|((h\rho_R)\ast \bm{H})\theta_R\left(\frac{\cdot}{2}\right)\,|\cdot|^{a+\chi}\right\|_{L^{2}}=\|F\|_{L^{2}}=\frac{1}{\sqrt{2\pi}}\|\widehat{F}\|_{L^{2}}.$$
For $\ell\in\N$ and $g\in L^1(\w_\ell)$, we have $\widehat{x^\ell g}=i^\ell \widehat{{g}}^{(\ell)}$. Thus, if $a+\chi=2p$ with $p\in\N$, we have 
$$\left|\widehat{F}(\xi)\right|= \left| \frac{\d^{2p}}{\d \xi^{2p}}\, \widehat{F}_{0}(\xi)\right|.$$
As previously, $\widehat{F}_{0}= \dfrac{1}{(2\pi)^2}  ((\widehat{h}\ast \widehat{\rho_R}) \, \widehat{\bm{H}})\ast \widehat{\theta_R\left(\frac{\cdot}{2}\right)}$ and we deduce that 
$$\frac{\d^{2p}}{\d \xi^{2p}}\, \widehat{F}_{0} =  {\frac{1}{(2\pi)^{2}}}\left[((\widehat{h}\ast \widehat{\rho_R}) \, \widehat{\bm{H}})\ast \widehat{\theta_R\left(\frac{\cdot}{2}\right)}\right]^{(2p)} $$
Hence, 
$$\left\|((h\rho_R)\ast \bm{H})\theta_R\left(\frac{\cdot}{2}\right)\, |\cdot|^{a+\chi}\right\|_{L^2}\leq \frac{1}{(2\pi)^{\frac{5}{2}}}
\left\|  \left[((\widehat{h}\ast \widehat{\rho_R}) \, \widehat{\bm{H}})\ast \widehat{\theta_R\left(\frac{\cdot}{2}\right)}\right]^{(2p)}\right\|_{L^2}\,.$$
Young's convolution inequality then implies
\begin{equation*}
 \begin{split}
 \|((h\rho_R)\ast \bm{H})\theta_R\left(\frac{\cdot}{2}\right)\, |\cdot|^{a+\chi}\|_{L^2}&\leq \frac{1}{(2\pi)^{\frac{5}{2}}}
\left\| (\widehat{h}\ast \widehat{\rho_R}) \, \widehat{\bm{H}}\right\|_{L^{1}}\left\|\widehat{\theta_R\left(\frac{\cdot}{2}\right)}^{(2p)}\right\|_{L^2}\\
& =\frac{1}{(2\pi)^2}
\left\|  (\widehat{h}\ast \widehat{\rho_R}) \, \widehat{\bm{H}}\right\|_{L^{1}}\left\||\cdot|^{2p}\theta_R\left(\frac{\cdot}{2}\right)\right\|_{L^2}\,.
 \end{split}
\end{equation*}
Recalling \eqref{eq:convolution:fourier} we get
$$
\left\|((h\rho_R)\ast \bm{H})\theta_R\left(\frac{\cdot}{2}\right)\, |\cdot|^{a+\chi}\right\|_{L^2} \leq   \vertiii{h}_{k} \; \frac{1}{(2\pi)^{2}}\;
\, \|\w_{k}\, \widehat{\bm{H}}\|_{L^1}\| \widehat{\rho_R} \|_{L^1(\w_k)} \,\left\||\cdot|^{2p}\theta_R\left(\frac{\cdot}{2}\right)\right\|_{L^2}.
$$
For any $2<a<3$, we may choose $\chi$ such that $a+\chi=2p=4$. Since $\theta_R,\rho_R \in {\mathcal C}^\infty(\R)$ are compactly supported, $\widehat{\rho_R}$ belongs to $L^1(\w_{k})$ for any {$k>2$} and $|\cdot|^{4}\theta_R\left(\frac{\cdot}{2}\right)\in L^2(\R)$. Finally, $\widehat{\bm{H}}=(1+|\xi|)e^{-|\xi|}$ and thus $\widehat{\bm{H}}\in L^1(\w_k)$ for all $k\in\N$ and there exists some constant $C_2(k,R)>0$ such that 
\begin{equation}\label{A2}
\left\|((h\theta_R)\ast \bm{H})\theta_R\left(\frac{\cdot}{2}\right) \, |\cdot|^{a+\chi}\right\|_{L^2} \leq  C_2(k,R) \vertiii{h}_{k}.
\end{equation}
Gathering \eqref{sepA}, \eqref{A1} and \eqref{A2} completes the proof. 
\end{proof}

\begin{lem}\phantomsection
Let $2<a<3$. There exists some constant $C>0$ such that, for any $h\in X_0$
$$\|A_2h\|_{L^1(\w_{a})} \leq  C \vertiii{h}_{k}$$
for any $k>2$.
\end{lem}

\begin{proof}
It follows from the definition \eqref{projection} of $\P$ that 
\begin{eqnarray*}
\|A_2h\|_{L^1(\w_{a})} &\leq  & 2 \int_\R |A_1h(x)| \, (1+x^2)\dx \max_{i\in\{1,2,3\}}  \|\zeta_i\|_{L^1(\w_{a})} \\
& \leq & C \|A_1h\|_{L^1(\w_{a})}, 
\end{eqnarray*}
and the result follows from Lemma \ref{bound_A1}.
\end{proof}

 \begin{proof}[Proof of Theorem~\ref{restrict}]
 The existence of a spectral gap for $\mathscr{L}$ in $\mathbb{Y}_{a}^{0}$ is now a direct consequence of Propositions~\ref{prp:dissi} and \ref{prp:bounded} together with \cite[Theorem 5.2]{CanizoThrom}.
\end{proof}
 
 We explain now how we can deduce the spectral gap estimate in Proposition \ref{restrict0} in the Introduction from Theorem \ref{restrict}.
 \begin{proof}[Proof of Proposition \ref{restrict0}] We consider $a \geq k$ and the spaces $X_{0}$ and $\X_{a}$ defined previously so that $\mathbb{Y}_{a}^{0} \subset X_{0}.$ Notice that $\mathscr{D}(\mathscr{L}_{0})=\mathscr{D} {(\mathscr{L})} \cap \X_{a}$ and, since $\bm{G}_{0}(\cdot)=\lambda_{0}\bm{H}(\lambda_{0}\cdot)$, one checks easily that, for any test function $\phi$
$$\int_{\R}\mathscr{L}_{0}(f)(x)\phi(x)\d x=\frac{1}{\lambda_{0}}\int_{\R}\mathscr{L}(\tau_{0}{f})(x)\phi\left(\lambda_{0}^{-1}x\right)\d x=\int_{\R}\mathscr{L}(\tau_{0}{f})(\lambda_{0} y)\phi(y)\d y$$
where 
$$\tau_{0}{f}(x)=f\left(\frac{x}{\lambda_{0}}\right), \qquad x \in \R.$$
This shows that
$$\mathscr{L}_{0}f=\tau_{0}^{-1}\mathscr{L}\left(\tau_{0}f\right), \qquad \forall f \in \mathscr{D}(\mathscr{L}_{0}).$$
In particular, since $\mathbb{Y}_{a},\mathbb{Y}_{a}^{0}$ are invariant under the action of the bijective transformation $\tau_{0}$ and of course
$$\mathrm{Range}(\mathscr{L}_{0})=\mathrm{Range}(\mathscr{L})=\mathbb{Y}_{a}^{0}$$
one sees that $\mathbb{Y}_{a}^{0}$
is a closed linear subspace of $\X_{a}$ stable under $\mathscr{L}_{0}.$ 

 This allows to define in a standard way  the restriction $\widetilde{\mathscr{L}}_{0}:=\mathscr{L}_{0}\vert_{\mathbb{Y}_{a}^{0}}$ of $\mathscr{L}_{0}$ to the space $\mathbb{Y}_{a}^{0}$ 
$$\widetilde{\mathscr{L}}_{0}=\mathscr{L}_{0}\vert_{\mathbb{Y}_{a}^{0}} \;\;:\;\;\mathscr{D}(\mathscr{L}_{0}) \cap \mathbb{Y}_{a}^{0} \to \mathbb{Y}_{a}^{0}$$
and one can deduce then from  Theorem \ref{restrict} the result.\end{proof}
   
\appendix

\numberwithin{equation}{section}                                                     

\section{Properties of the Fourier norm and interpolation estimates}\label{app:fourier}

The following lemma is a consequence of \cite[Lemma 2.5]{MR2355628}.

\begin{lem}\label{lem:mukvertk}
 Let $2<k<3$. There exists a constant $C>0$ depending only on $k$ such that
  $$\vertiii{\hat{\mu}}_{k}\leq C \int_{\R}  (1+|x|)^{k} \,|\mu|(dx),$$
 for any $\mu\in \mathcal{M}_k(\R)$ satisfying
  \begin{equation}\label{mom}\int_{\R} \mu(\d x) = \int_{\R} x\, \mu(\d x) = \int_{\R} x^2 \,\mu(\d x) = 0.
    \end{equation}
\end{lem}

\begin{proof}
  Since $\mu$ satisfies \eqref{mom}, we have $\hat{\mu}(0)=0$, $\hat{\mu}'(0)=0$ and $\hat{\mu}''(0)=0$. Hence, Taylor formula implies that
  $$|\hat{\mu}(\xi)|\le |\xi|^2 \int_0^1 | \hat{\mu}''(t\xi)| \d t .$$
We set $s=k-2\in(0,1)$. Then, for ${\phi}(r)=r^s$, we have
  $$M:=\int_\R (1+x^2)\phi(|x|)|\mu|(\d x)<\infty.$$
  Moreover, $\phi$ is a strictly increasing function with $\frac{\phi(r)}{r}$ nonincreasing. It follows from \cite[Lemma 2.5]{MR2355628} that
  $$|\hat{\mu}''(t\xi)|\le 2 M \psi(|t\xi|),$$
  where $\psi(y)=[\phi(y^{-1})]^{-1}=y^s$. Hence,
   \begin{equation}\label{bound}
|\hat{\mu}(\xi)|\le  2M\, |\xi|^{2+s}   \int_0^1t^s\d t \le  \frac{2 M}{s+1} |\xi|^{2+s}. 
\end{equation}
This proves the result {since $s+2=k$ and $M\le\int_{\R} (1+|x|)^k \,|\mu|(\dx)$.}
\end{proof}

 { A similar Lemma holds for the more general Fourier norms $\vertiii{\cdot}_{k,p}$ defined by \eqref{def:normkp}.  Namely, one has the following.   
\begin{lem}\label{lem:mukvertkp}
 Let $1\le p <\infty$ and $2<k<3$. There exists a constant $C>0$ depending only on $k$ and $p$ such that
  $$\vertiii{\hat{\mu}}_{k,p}\leq C \int_{\R}  (1+|x|)^{k} \,|\mu|(dx),$$
 for any $\mu\in \mathcal{M}_k(\R)$ satisfying
  \begin{equation*}\int_{\R} \mu(\d x) = \int_{\R} x\, \mu(\d x) = \int_{\R} x^2 \,\mu(\d x) = 0.
    \end{equation*}
\end{lem}
\begin{proof}
First, we have 
$$\vertiii{\hat{\mu}}_{k,p}^p = \int_{|\xi|\le 1} \frac{|\hat{\mu}(\xi)|^p}{|\xi|^{kp}} \; \d\xi + \int_{|\xi|> 1} \frac{|\hat{\mu}(\xi)|^p}{|\xi|^{kp}} \; \d\xi.   $$
Next, for $|\xi|> 1$, we simply use the bound $|\hat{\mu}(\xi)| \le \int_{\R} |\mu|(\d x)$  whereas, for $|\xi|\le 1$, we use the bound \eqref{bound}. This leads to 
  $$\vertiii{\hat{\mu}}_{k,p}^p  \le \frac{2^{p+1}}{(k-1)^p}  \left(\int_{\R}  (1+|x|)^{k} \,|\mu|(dx)\right)^p + \left(\int_{\R} |\mu|(\d x)\right)^p \int_{|\xi|>1} \frac{\d \xi}{|\xi|^{pk}}.$$
The result then follows since $ \int_{|\xi|>1} \frac{\d \xi}{|\xi|^{pk}}<\infty$. 
\end{proof}}

The following lemma is a consequence of \cite[Theorem 4.1]{CGT}.
\begin{lem}\phantomsection \label{L2-knorm}
For $k> 2$, $\beta >0$ and $0<r<1$,  one has 
$$\|f\|^{2}_{L^2}\le C(r,\beta) \vertiii{\hat{f}}_{k}^{2(1-r)}\left(\|f\|_{H^M}^{2r} +\|f\|_{H^N}^{2r}\right), $$
with $\widehat{f}(\xi)=\displaystyle\int_\R f(x) e^{-ix\xi}\d x$,
$$M=\frac{k(1-r)}{r}, \qquad N=M+\frac{(1-r)(\beta+1)}{2r}, \qquad 
C(r,\beta)=\left(2 \left(1+\frac{1}{\beta}\right)\right)^{1-r}. $$
\end{lem}

\begin{lem}\phantomsection\label{L1-L2}
Let $a_*>a$ and $\alpha \in \left( 0 , \frac{2(a_*-a)}{2a_*+1}\right)$ be given. There exists a constant $C>0$ depending only on $\alpha$, $a$ and $a_*$ such that, for every $f\in L^1(\w_{a_*})\cap L^2$, 
 $$\|f\|_{L^1(\w_{a})}\le C \|f\|^\alpha_{L^2} \|f\|^{1-\alpha}_{L^1(\w_{a_*})} .$$
\end{lem}

\begin{proof}
The H\"older inequality with the three exponents $p_1=\frac{2}{\alpha}$, $p_2=\frac{1}{1-\alpha}$ and $p_3=\frac{2}{\alpha}$ leads to 
\begin{align*}
\|f\|_{L^1(\w_{a})} &= \int_\R |f(x)|^\alpha |f(x)|^{1-\alpha} {(1+|x|)^{a_*(1-\alpha)} (1+|x|)^{-a_*(1-\alpha)}} \w_{a}(x)  \d x  \\
& \le \|f\|^\alpha_{L^2} \|f\|^{1-\alpha}_{L^1(\w_{a_*})}  \left(\int_\R { (1+|x|)^{-\frac{2a_*(1-\alpha)}{\alpha}}(1+|x|)^{\frac{2a}{\alpha}} }\d x\right)^{\frac{\alpha}{2}}.
\end{align*}
The last integral converges since $\alpha<\frac{2(a_*-a)}{2a_*+1}$.
\end{proof}

\section{Nonlinear estimates for {$\Q_{\g}$ and $\mathscr{I}_{\g}$}}\label{app:QgQ0}
  
We gather here nonlinear estimates involving {integrals of the collision operator $\Q_{\g}$ for $\g \geq 0$ in the spirit of \cite{ACG}. The same kind of computations also enables to get nonlinear estimates of the functional $\mathscr{I}_{\g}$ introduced in {Section \ref{sec:upgrade}}.} We begin with the following easy result for $\Q^{+}_{0}$.
\begin{lem}\phantomsection\label{lem:Q+0} For any measurable $f,g,h$, one has
$$\int_{\R}\Q^{+}_{0}(f,g)\,h\dx \leq \sqrt{2}\|h\|_{L^2}\min\left(\|f\|_{L^1}\|g\|_{L^2},\|g\|_{L^1}\|f\|_{L^2}\right).$$\end{lem}
\begin{proof} There is no loss of generality in assuming $f,g,h$ nonnegative. One has then
$$\int_{\R}\Q^{+}_{0}(f,g)\,h\dx=\int_{\R^{2}}f(x)g(y)h\left(\frac{x+y}{2}\right)\dx\dy.$$
Given $x \in \R$, one deduces from Cauchy-Schwarz inequality that
$$\int_{\R}g(y)h\left(\frac{x+y}{2}\right)\dy \leq \|g\|_{L^2}\left\|h\left(\frac{x+\cdot}{2}\right)\right\|_{L^2}=\sqrt{2}\|g\|_{L^2}\|h\|_{L^2}$$
from which we get that
$$\int_{\R}\Q^{+}_{0}(f,g)\,h\dx \leq \sqrt{2}\|f\|_{L^1}\|g\|_{L^2}\|h\|_{L^2}.$$
Exchanging the role of $g$ and $f$, one deduces the result.\end{proof}
We now turn to some $L^{2}$ estimate for $\Q^{+}_{\g}$ for $\g >0$:
\begin{prp}\phantomsection\label{prop:QgL2} For any  {$k \ge0$}, there is $C=C_{k,\g}>0$ such that
 {\begin{equation*}\label{eq:Qgffminus}
\|\Q_{\g}^{-}(f,f)\w_{k}\|_{L^2} \leq \|f\w_{\g}\|_{L^1}\,\|f\w_{k+\g}\|_{L^2}
  \end{equation*}
  and 
\begin{equation*}\label{eq:Qgffplus}
\|\Q_{\g}^{+}(f,f)\w_{k}\|_{L^2} \leq C\,{\min(\|f\w_{k}\|_{L^1}\,\|f\w_{k+\g}\|_{L^2},\|f\w_{k}\|_{L^2}\,\|f\w_{k+\g}\|_{L^1})}.
\end{equation*}}
\end{prp}
\ \begin{proof}
  One has
$$\Q_{\g}^{+}(f,f)=\int_\R f\left(x + \frac{y}{2}\right) f\left(x - \frac{y}{2}\right) |y|^\gamma \dy, \qquad \qquad \Q_{\g}^{-}(f,f)=f(x) \int_\R f(y) |x-y|^\gamma \dy.$$
and, in particular, 
$$\|\Q_{\g}^{-}(f,f)\w_{k}\|_{L^2}^{2}\leq \int_{\R}|f(x)|^{2}\w^{2}_{k}(x)\left[\int_{\R}|f(y)|\,|x-y|^{\g}\dy\right]^{2}\dx \leq \|f\w_{k+\g}\|_{L^2}^{2}\|f\w_{\g}\|_{L^1}^{2}$$
where we used that $|x-y|^{\g}\leq \w_{\g}(x)\w_{\g}(y)$.
For the $\Q^{+}_{\g}(f,f)$ term one can for instance use that
$$\|\Q_{\g}^{+}(f,f)\w_{k}\|_{L^2}=\sup_{\|\varphi\|_{L^2}\leq 1}I^{+}(\varphi)$$
where
$$I^{+}(\varphi)=\int_{\R}\Q_{\g}^{+}(f,f)\w_{k}(x)\varphi(x)\dx{=\int_{\R^2}}f(x)f(y)|x-y|^{\g}\w_{k}\left(\frac{x+y}{2}\right)\varphi\left(\frac{x+y}{2}\right)\dx\dy.$$
Since $\w_{k}\left(\frac{x+y}{2}\right) \leq \w_{k}(x)\w_{k}(y)$ one has, with $F_{k}=|f|\w_{k}$
$$I^{+}(\varphi) \leq \int_{\R^{2}}|x-y|^{\g}F_{k}(x)F_{k}(y)\varphi\left(\frac{x+y}{2}\right)\dx\dy$$
and, since $\g\in (0,1)$, $|x-y|^{\g}\leq |x|^{\g}+|y|^{\g}$ so that, with a symmetry argument
$$I^{+}(\varphi) \leq 2\int_{\R^{2}}F_{k+\g}(x)F_{k}(y)\varphi\left(\frac{x+y}{2}\right)\dx\dy.$$
One deduces easily by Cauchy-Schwarz inequality that, 
$$I^{+}(\varphi) \leq 2\|F_{k+\g}\|_{L^2}{\|F_{k}\|_{L^1}}\sup_{y}\left\|\varphi\left(\frac{\cdot+y}{2}\right)\right\|_{L^2}=2\sqrt{2}\|F_{k+\g}\|_{L^2} {\|F_{k}\|_{L^1}}\|\varphi\|_{L^2}.$$
{We also have
   $$I^{+}(\varphi) \leq 2\|F_{k+\g}\|_{L^1}\|F_{k}\|_{L^2}\sup_{y}\left\|\varphi\left(\frac{\cdot+x}{2}\right)\right\|_{L^2}=2\sqrt{2}\|F_{k+\g}\|_{L^1}\|F_{k}\|_{L^2}\|\varphi\|_{L^2}$$}
which gives 
$$\|\Q_{\g}^{+}(f,f)\w_{k}\|_{L^2} \leq 2\sqrt{2}\,{\min(\|f\w_{k+\g}\|_{L^2}\|f\w_{k}\|_{L^1},\|f\w_{k+\g}\|_{L^1}\|f\w_{k}\|_{L^2})}$$
and ends the proof.
\end{proof} 

We establish now some comparison estimates between $\Q_{\g}$ and $\Q_{0}$ in the limit $\g \to 0$:
\begin{prp}\phantomsection\label{diff_Q}
  Let $2<a<3$,  $p>1$ and $\gamma,s>0$ satisfying  $s+\gamma+a<3$. There exist some positive constant $C=C_{s,p}>0$ depending only on $s,p$ such that, for any $f \in L^{1}(\w_{s+\g+a})$ and any $g\in L^p(\w_{a})\cap L^1(\w_{s+\gamma+a})$, it holds  
  \begin{multline*}
\|\Q_{0}(g,f)-\Q_{\g}(g,f)\|_{L^{1}(\w_{a})}+\|\Q_{0}(f,g)-\Q_\gamma(f,g)\|_{L^1(\w_{a})} \\
\le C_{s,p}\g^{\frac{s}{s+1}}|\log \g| \|f\|_{L^1(\w_{a})}\|g\|_{L^1(\w_{a})}
    +  {24}\,\g^{\frac{s}{s+1}} \bigg(\|g\|_{L^1(\w_{s+\gamma+a})}\|f\|_{L^1(\w_{a})}\\ 
    +\|f\|_{L^1(\w_{s+\gamma+a})}\|g\|_{L^1(\w_{a})}
     +\|g\|_{ L^p(\w_{a})}\|f\|_{L^1(\w_{a})} +\|f\|_{ L^p(\w_{a})}\|g\|_{L^1(\w_{a})}\bigg).
    \end{multline*}
\end{prp}

\begin{proof} We prove the result for $\|\Q_{0}(f,g)-\Q_{\g}(f,g)\|_{L^{1}(\w_{a})}$ only. First a change of variables leads to
  \begin{align*}
    \|\Q_{0}(f,g)-\Q_\gamma(f,g)\|_{L^1(\w_{a})}& \le\int_\R\int_\R |f(x)| |g(y)| \left|1-|x-y|^\gamma\right|  {\w_{a}\left(\frac{x+y}{2}\right)} \dx\dy \\
    & +\int_\R\int_\R|f(x)| |g(y)|  \left|1-|x-y|^\gamma\right|  {\left(\frac{\w_{a}(x)}{2}+\frac{\w_{a}(y)}{2}\right)}\d x\d y \\
    & \leq 2 \int_\R\int_\R|f(x)| |g(y)|  \left|1-|x-y|^\gamma\right| \w_{a}(x)\w_{a}(y)\d x\d y, 
  \end{align*}
  since {$\w_{a}(\frac{x+y}{2})\le \w_{a}(x)\w_{a}(y)$ and $\frac{1}{2}\left(\w_{a}(x)+\w_{a}(y)\right)\le \w_{a}(x)\w_{a}(y)$}.
  
  Let $0<\delta<1$ and $R>1$. Splitting the above integral according to $|x-y|\le \delta$, $\delta<|x-y|<R$ and $|x-y|\ge R$, we get
  $$\|\Q_{0}(f,g)-\Q_\gamma(f,g)\|_{L^1(\w_{a})} \le I_1+I_2+I_3,$$
  with
  \begin{align*}
    I_1 &= 2\int_\R\int_{|x-y|\le\delta} |f(x)| |g(y)| \left|1-|x-y|^\gamma\right| \w_{a}(x)\w_{a}(y)\dx\dy\\
    I_2 &= 2\int_\R\int_{\delta<|x-y|<R} |f(x)| |g(y)| \left|1-|x-y|^\gamma\right| \w_{a}(x)\w_{a}(y)\dx\dy \\
    I_3 &= 2\int_\R\int_{|x-y|\ge R} |f(x)| |g(y)| \left|1-|x-y|^\gamma\right|\w_{a}(x)\w_{a}(y)\dx\dy.  \end{align*}
  Since $\delta<1$, for  $|x-y|\le \delta$, we have $\left|1-|x-y|^\gamma\right|\le1$. We then deduce from H\"older's inequality that
  \begin{align*}
    I_1& \le 2 \int_\R \left(\int_{|x-y|\le \delta}\dy \right)^{\frac{p-1}{p}}\left(\int_\R |g(y)|^p \w_{a}(y)^p \dy \right)^\frac{1}{p} |f(x)| \w_{a}(x)\dx \\
       & \le 2  (2\delta)^{\frac{p-1}{p}} \|g\w_{a}\|_{ L^p}\|f\|_{L^1(\w_{a})},
  \end{align*}
  for $p>1$.

  Now,  since $R>1$, for $|x-y|\ge R$, we have $\left|1-|x-y|^\gamma\right|\le |x-y|^\gamma$. Moreover, $|x-y|\ge R$ implies that either $|x|\ge \frac{R}{3}$ or $|y|\ge \frac{R}{3}$. If $|x|\ge |y|$, this means that, necessarily, $|x|\ge\frac{R}{3}$ and $|x-y|\le 2|x|$. Now, if $|y|\ge |x|$, this means that $|y|\ge\frac{R}{3}$ and $|x-y|\le 2|y|$. We thus deduce that
  \begin{align*}
    I_2 & \le 2^{1+\gamma} \int_\R \int_{|x|\ge \frac{R}{3}} |f(x)| |g(y)| |x|^{\gamma}\w_{a}(x)\w_{a}(y) \d x\d y\\
    & + 2^{1+\gamma} \int_\R \int_{|y|\ge \frac{R}{3}} |f(x)| |g(y)|\w_{a}(x) |y|^{\gamma}\w_{a}(y) \d x\d y\\
    & \le \frac{2^{1+\gamma} 3^s}{R^s} \int_\R \int_{|x|\ge \frac{R}{3}} |f(x)| |g(y)| |x|^{\gamma+s }\w_{a}(x)\w_{a}(y) \d x\d y\\
    & + \frac{2^{1+\gamma} 3^s}{R^s} \int_\R \int_{|y|\ge \frac{R}{3}} |f(x)| |g(y)| \w_{a}(x) |y|^{\gamma+s}\w_{a}(y) \d x\d y\\
    & \le \frac{2^{1+\gamma} 3^s}{R^s}\left( \|g\|_{L^1(\w_{s+\gamma+a})}\|f\|_{L^1(\w_{a})}+\|f\|_{L^1(\w_{s+\gamma+a})}\|g\|_{L^1(\w_{a})}\right).
  \end{align*}
  Finally, for $\delta<|x-y|<R$, we have $\left|1-|x-y|^\gamma\right|\le \gamma\max\{|\log\delta|, R\log R\}$. Hence,
  \begin{align*}
    I_3& \le 2 \gamma\max\{|\log\delta|, R\log R\}  \int_\R\int_\R |f(x)| |g(y)| \w_{a}(x)\w_{a}(y) \d x\d y\\
    & \le 2 \gamma\max\{|\log\delta|, R\log R\}\|f\|_{L^1(\w_{a})}\|g\|_{L^1(\w_{a})}.
  \end{align*}
We deduce that
 \begin{multline*}
    \|\Q_{0}(f,g)-\Q_\gamma(f,g)\|_{L^1(\w_{a})}\le  C_{0}\gamma \|f\|_{L^1(\w_{a})}\|g\|_{L^1(\w_{a})}\\
    + C_{1}\left(\|g\|_{L^1(\w_{s+\gamma+a})}\|f\|_{L^1(\w_{a})}+\|f\|_{L^1(\w_{s+\gamma+a})}\|g\|_{L^1(\w_{a})}+\|g\|_{ L^p(\w_{a})}\|f\|_{L^1(\w_{a})}\right)
    \end{multline*}
with $C_{0}=C_{0}(\delta,R)=2\max\left\{|\log \delta|\,,\,R\log R\right\},$ and
$$C_{1}=C_{1}(\delta,R,p,s)=\max\left\{2(2\delta)^{\frac{p-1}{p}}\,,\,2^{1+\g}3^{s}R^{-s}\right\}.$$
We choose then $R,\delta$ such that {$R=\g^{-\frac{1}{s+1}}$} and $2\delta=\g^{\frac{ps}{(p-1)(s+1)}}$
from which
$$C_{1}=\g^{\frac{s}{s+1}}\,\max\left\{2,2^{1+\g}3^{s}\right\} \leq {12}\g^{\frac{s}{s+1}}$$
since $0 < s,\g <1$. Now, with such a choice, one sees easily that 
$$C_{0} \leq C_{s,p}\g^{-\frac{1}{s+1}}|\log \g|$$
for some positive constant $C_{s,p}$ depending only on $s$ and $p$. The conclusion follows.
\end{proof}

We recall here some notations introduced in {Section \ref{sec:upgrade}}. Namely, set
$$\mathscr{I}_{0}(f,g)=\int_{\R^{2}}f(x)g(y)|x-y|^{2}\Lambda_{0}\left(|x-y|\right)\dx\dy, \qquad f,g \in L^{1}(\w_{s}), \qquad s>2,$$
and
$$\mathscr{I}_{\g}(f,g)=\int_{\R^{2}}f(x)g(y)|x-y|^{2}\Lambda_{\g}\left(|x-y|\right)\dx\dy, \qquad \qquad f,g \in L^{1}(\w_{2+\g})$$ 
where
$$\Lambda_{0}(r)=\log r,\qquad \qquad \Lambda_{\g}(r)=\frac{r^{\g}-1}{\g}, \qquad \g>0, r >0.$$
One has then the following first basic observation
\begin{lem}\phantomsection\label{lem:I0fg}
For $a >2,$ $f,g\in L^{1}(\w_{a})$, one has
$$\left|\mathscr{I}_{0}(f,g)\right| \leq C_{a}\|f\|_{L^{1}(\w_{a})}\|g\|_{L^{1}(\w_{a})}$$
for some positive constant $C_{a} >0$ depending only on $a.$
\end{lem}
\begin{proof} Up to replacing $f$ with $|f|$ and $g$ with $|g|$, we may assume without loss of generality that both $f$ and $g$ are nonnegative. One has then
\begin{equation*}\begin{split}
{\left|\mathscr{I}_{0}(f,g)\right|}&=\int_{|x-y|>1}f(x)g(y)|x-y|^{2}\log|x-y|\dx\dy\\
&\phantom{++++} +\int_{|x-y|<1}f(x)g(y)|x-y|^{2}\bigl|\log|x-y|\bigr|\dx\dy\\
&\leq c_{a}\int_{|x-y|>1}f(x)g(y)|x-y|^{a}\dx\dy+{\frac{1}{2e}}\|f\|_{L^1}\|g\|_{L^1}\end{split}\end{equation*}
since there is $c_{a} >0$ such that $\log r \leq c_{a}r^{a-2}$ for any $r >1$ and $a >2$ and 
{$r^2\bigl|\log r\bigr|\leq \frac{1}{2e}$ for $r \in (0,1)$}. This gives the result since {$|x-y|^{a} \leq \w_{a}(x)\w_{a}(y)$} and $\|\cdot\|_{L^1}\leq \|\cdot\|_{L^{1}(\w_a)}$.
\end{proof}
Recalling that $\lim_{\g\to0^{+}}\Lambda_{\g}(r)=\Lambda_{0}(r)$ for any $r >0$, one has the following estimate for the difference between $\mathscr{I}_{0}$ and $\mathscr{I}_{\g}$:
\begin{lem}\phantomsection\label{lem:IgfgI0}
Let $2<a<3$,  $p>1$ and $\gamma,s>0$ satisfying  $s+\gamma+a<3$. There exist some positive constant $C=C_{a,s,p}>0$ depending only on $s,p$ such that, for any $f \in L^{1}(\w_{s+\g+a})$ and any $g\in L^p(\w_{a})\cap L^1(\w_{s+\gamma+a})$, it holds  
\begin{multline*}
\left|\mathscr{I}_{\g}(f,g)-\mathscr{I}_{0}(f,g)\right| \leq C_{a,s,p}\g^{\frac{s}{s+1}}|\log \g| \|f\|_{L^1(\w_{a})}\|g\|_{L^1(\w_{a})}		\\
    +  {12}\,\g^{\frac{s}{s+1}} \bigg(\|g\|_{L^1(\w_{s+\gamma+a})}\|f\|_{L^1(\w_{a})}
    +\|f\|_{L^1(\w_{s+\gamma+a})}\|g\|_{L^1(\w_{a})}
     +\|g\|_{ L^p(\w_{a})}\|f\|_{L^1(\w_{a})}\bigg).\end{multline*}
\end{lem}
\begin{proof} As before, we assume without loss of generality that $f,g$ are nonnegative and observe that
$$\left|\mathscr{I}_{\g}(f,g)-\mathscr{I}_{0}(f,g)\right| \leq \int_{\R^{2}}f(x)g(y)|x-y|^{2}\left|\Lambda_{\g}(|x-y|)-\Lambda_{0}(|x-y|)\right|\dx\dy.$$
Observe that, given $r >0$, 
$$\dfrac{\d}{\d\beta}\Lambda_{\beta}(r)=\frac{r^{\beta}\log r^{\beta}-r^{\beta}+1}{\beta^{2}}$$
so that
$$0 \leq \dfrac{\d}{\d\beta}\Lambda_{\beta}(r) \leq \Lambda_{\beta}(r)^{2}$$
since for any $x >0$, $0 \leq x\log x-x+1 \leq (x-1)^{2}$. Integrating this inequality  over $\beta \in (0,\g)$, yields
 {
$$-\frac{1}{\Lambda_{\g}(r)}+\frac{1}{\Lambda_{0}(r)} \leq \g$$
which also reads}
$$0\leq \Lambda_{\g}(r)-\Lambda_{0}(r) \leq \gamma\Lambda_{\g}(r)\Lambda_{0}(r)=\left(r^{\g}-1\right)\Lambda_{0}(r).$$
Therefore
\begin{multline*}
\left|\mathscr{I}_{\g}(f,g)-\mathscr{I}_{0}(f,g)\right| \leq \int_{\R^{2}}f(x)g(y)|x-y|^{2}\left||x-y|^{\g}-1\right|\left|\log(|x-y|)\right|\dx\dy\\
\leq C_{a}\int_{\R^{2}}f(x)g(y)\left(1+|x-y|^{a}\right)|\left|1-|x-y|^{\g}\right|\dx\dy
\end{multline*}
where we used that there exists $C_{a} >0$ such that $r^{2}\log r \leq C_{a}\left(1+r^{a}\right)$ for any $a >2, r >0.$ Therefore, there exists $\tilde{C}_{a} >0$ such that
$$\left|\mathscr{I}_{\g}(f,g)-\mathscr{I}_{0}(f,g)\right| \leq \tilde{C}_{a}\int_{\R^{2}}f(x)g(y)\left|1-|x-y|^{\g}\right|\w_{a}(x)\w_{a}(y)\dx\dy.$$
The computations performed in Proposition \ref{diff_Q} give then the result.
\end{proof}

\section{Rigorous justifications of $L^{2}$ and Sobolev estimates}\label{app:rigor} 									

We provide in this Appendix the rigorous justifications of some of the formal estimates derived in Sections \ref{sec:L2est}, \ref{Sec:limit:temp} and \ref{sec:weighted} about the regularity of the profile $\Gg.$ We begin with the rigorous proof of Lemma \ref{lem:boundL2L}.

\subsection{Justification of the $L^2$ estimates}\label{sec:justif-L2}

To justify rigorously the estimates in Section \ref{sec:L2est} we begin with the following lower bound on the collision frequency 
$$\Sigma_{\g}(y)=\int_{\R}\Gg(x)|x-y|\d x,\qquad y \in \R.$$ We point out that, even though such a lower bound is not uniform with respect to $\g$, it will allows subsequently to derive $L^{2}$-estimates which are uniform with respect to $\g$:
\begin{lem}\label{Lem:coll:freq:non:uniform}
 Let $\g\in(0,1)$ and $\Gg\in\mathscr{E}_{\g}$. There exists a constant $c_\g$ such that
 \begin{equation*}
  \Sigma_{\g}(y)\geq c_\g \qquad \text{for all } y\in\R.
 \end{equation*}
\end{lem}

\begin{proof}
 Clearly from the triangle inequality we have
 \begin{equation*}
  \Sigma_\g(y)=\int_{\R}\Gg(x)|x-y|^\g \dx\geq M_\g(\Gg)-|y|^\g.
 \end{equation*}
 Consequently,  
 \begin{equation}\label{eq:coll:freq:nu:1}
  \Sigma_\g(y)\geq  \frac{M_\g(\Gg)}{2} \qquad  \text{ if }  |y|\leq \left(\frac{M_\g(\Gg)}{2}\right)^{\frac{1}{\g}}.
 \end{equation}
On the other hand, for $\tilde{\delta}<1$, as in the proof of Lemma~\ref{lem:Sigmag}, we have
\begin{equation*}\begin{split}
\Sigma_{\g}(y)&= \int_{\R}\left(|x-y|^{\g}+\ind_{|x-y|<\tilde{\delta}}\right)\,\Gg (x)\,\dx - \int_{\R}\ind_{|x-y|<\tilde{\delta}}\,\Gg (x)\,\dx\\
&\geq \tilde{\delta}^{\gamma}-\int_{-\tilde{\delta}}^{\tilde{\delta}}\Gg(x+y)\dx\,.
\end{split}\end{equation*}
Thus, if $|y|\geq \left(\frac{M_\g(\Gg)}{2}\right)^{\frac{1}{\g}}$ we deduce from the pointwise upper bound \eqref{eq:pointX} that
\begin{equation*}
\Sigma_{\g}(y)\geq \tilde{\delta}^{\gamma}-C\int_{-\tilde{\delta}}^{\tilde{\delta}}\frac{1}{|x+y|}\dx\geq \tilde{\delta}^{\g}-\frac{2C\tilde{\delta}}{|y|-\tilde{\delta}}\geq \tilde{\delta}^{\g}- 4C\tilde{\delta} \left(\frac{M_\g(\Gg)}{2}\right)^{-\frac{1}{\g}}\end{equation*}
as soon as  $\tilde{\delta} <\frac{1}{2}\left(\frac{M_\g(\Gg)}{2}\right)^{\frac{1}{\g}}$. Since $\g<1$, there exists $\tilde{\delta}=\tilde{\delta}_*<\frac{1}{2}\left(\frac{M_\g(\Gg)}{2}\right)^{\frac{1}{\g}}$ (depending on $\g$) for which the above left-hand-side is positive. Recalling \eqref{eq:coll:freq:nu:1} and introducing
\begin{equation*}
 c_{\g}=\min\left\{\frac{M_\g(\Gg)}{2}\,,\,  \tilde{\delta}_{*}^{\g}- 4C\tilde{\delta}_{*} \left(\frac{M_\g(\Gg)}{2}\right)^{-\frac{1}{\g}}\right\}\,,
\end{equation*}
the claim follows.
\end{proof}

\begin{proof}[Rigorous proof of  Lemma \ref{lem:boundL2L}:] We introduce the following regularization of $\Gg$: 
$$\Psi_{\varepsilon}=\Gg \ast \varrho_{\varepsilon}, \qquad \varepsilon >0$$
where  $\left(\varrho_{\varepsilon}\right)_{\varepsilon >0}$ is a family of mollifiers
$$\varrho_{\varepsilon}(x)=\varepsilon^{-1}\varrho\left(\frac{x}{\varepsilon}\right), \qquad x \in \R,\qquad \varepsilon >0$$
where $\varrho \in \mathscr{C}^{\infty}(\R)$ is nonnegative, compactly supported in the interval $[-1,1]$, with unit mass and such that $x\varrho'(x)\le 0$ for any $x\in\R$ (one possible choice being the classical function $\varrho(x)=\exp\left(\frac{1}{x^{2}-1}\right)\ind_{(-1,1)}(x)$).

It is not difficult to check then that $\Psi_{\varepsilon}$ satisfies 
\begin{equation}\label{eq:L2:regular:1}
\frac{1}{4}\dfrac{\d }{\d x} \left(x\Psi_{\varepsilon}\right)=\frac{1}{4}\dfrac{\d }{\d x}\left[\Gg \ast (x\varrho_{\varepsilon})\right] + \Q_{\g}(\Gg,\Gg) \ast \varrho_{\varepsilon}. 
\end{equation}
Now, as for the proof of \eqref{eq:QgQ0}, one sees hat for any \emph{nonnegative} $\varphi$
$$\int_{\R}\Q^{+}_{\g}(\Gg,\Gg)\varphi\d x \leq 2\int_{\R}\Q^{+}_{0}(|\cdot|^{\g}\Gg,\Gg)\varphi\d x$$
which, in turns, shows that, since $\varrho_{\varepsilon}\geq0$
$$\left[\Q^{+}_{\g}(\Gg,\Gg)\ast \varrho_{\varepsilon}\right] \leq 2 \Q^{+}_{0}(|\cdot|^{\g}\Gg,\Gg) \ast \varrho_{\varepsilon}.$$
Now, since 
$$\Q^{+}_{0}(f,g)(x)=2(f\ast g)(2x)$$
one has
$$\Q^{+}_{0}(f,g) \ast \varrho_{\varepsilon}(x) = \Q_{0}^{+}(f,g \ast \tilde{\varrho}_{\varepsilon}), \qquad \widetilde{\varrho}_{\varepsilon}(x)=\frac{1}{2}\varrho_{\varepsilon}\left(\frac{x}{2}\right).$$
Setting then
$$\widetilde{\Psi}_{\varepsilon}:=\Gg \ast \widetilde{\varrho}_{\varepsilon}$$
we further deduce that
\begin{equation}\label{eq:L2:gain:term}
\left[\Q^{+}_{\g}(\Gg,\Gg)\ast \varrho_{\varepsilon}\right] \leq 2 \Q^{+}_{0}\left(|\cdot|^{\g}\Gg,\widetilde{\Psi}_{\varepsilon}\right). 
\end{equation}
Notice that $\widetilde{\varrho}_{\varepsilon}=\varrho_{2\varepsilon}$ and
$$\widetilde{\Psi}_{\varepsilon}=\Psi_{2\varepsilon}.$$
On the other hand, using  Lemma~\ref{Lem:coll:freq:non:uniform} we have
\begin{equation}\label{eq:L2:loss:lower}
 [\Q_\g^{-}(\Gg,\Gg)\ast\varrho_{\varepsilon}](x)=\int_{\R}\Gg(z)\Sigma_\g(z)\varrho_{\varepsilon}(x-z)\dz\geq c_\g\Psi_{\varepsilon}(x)\qquad x \in \R.
\end{equation} {In the remainder we follow ideas from \cite[Proposition~2.1]{MM09} and introduce for $A>0$ the cut-off function}
\begin{equation}\label{eq:Lambda}
 \Lambda(x):=\Lambda_A(x):=\frac{x^2}{2}\ind_{x\leq A}+\Bigl(A x-\frac{A^2}{2}\Bigr)\ind_{x>A}
\end{equation}
which satisfies $\Lambda'(x)=\min\{x,A\}$, $\Lambda(x)\leq x\Lambda'(x)$ as well as $x\Lambda'(y)\leq \Lambda(x)+\Lambda(y)$.

We test now \eqref{eq:L2:regular:1} with $\Lambda'(\Psi_\varepsilon)$ to get
\begin{multline*}
 \int_{\R}[\Q^{-}_{\g}(\Gg,\Gg)\ast \varrho_{\varepsilon}]\Lambda'(\Psi_\varepsilon)\dx+\frac{1}{8}\int_{\R}\Psi_\varepsilon^2(x)\ind_{\Psi_\varepsilon\leq A}+A^2\ind_{\Psi_\varepsilon>A}\dx\\*
 \leq \frac{1}{4}\int_{\R}\dfrac{\d }{\d x}\left[\Gg \ast (x\varrho_{\varepsilon})\right]\Lambda'(\Psi_{\varepsilon}(x))\d x+2\int_{\R}\Q^{+}_{0}\left(|\cdot|^{\g}\Gg,\widetilde{\Psi}_{\varepsilon}\right)\Lambda'(\Psi_{\varepsilon})\d x.
\end{multline*}
Notice that
$\dfrac{\d }{\d x}\left[\Gg \ast (x\varrho_{\varepsilon})\right]=\Gg \ast \varrho_{\varepsilon}+ \Gg \ast \left(x \varrho'_{\varepsilon}\right)$
i.e.
$$\dfrac{\d}{\d x}\left[\Gg \ast (x\varrho_{\varepsilon})\right]=\Psi_{\varepsilon} - \Gg \ast h_{\varepsilon}, \qquad h_{\varepsilon}(x)=-\frac{x}{\varepsilon^{2}}\varrho'\left(\frac{x}{\varepsilon}\right).$$
Using \eqref{eq:L2:loss:lower} and $\Psi_\varepsilon^2(x)\ind_{\Psi_\varepsilon\leq A}+A^2\ind_{\Psi_\varepsilon>A}=\min\{\Psi_\varepsilon^2(x),A^2\}$ we can further deduce that
\begin{multline*}
 c_{\g}\int_{\R}\Psi_\varepsilon(x)\Lambda'(\Psi_\varepsilon)\dx+\frac{1}{8}\int_{\R}\min\{\Psi_\varepsilon^2(x),A^2\}\dx\\*
 \leq \frac{1}{4}\int_{\R}\left[\Psi_{\varepsilon}(x)-\Gg \ast h_{\varepsilon}\right]\Lambda'(\Psi_{\varepsilon}(x))\d x+2\int_{\R}\Q^{+}_{0}\left(|\cdot|^{\g}\Gg,\widetilde{\Psi}_{\varepsilon}\right)\Lambda'(\Psi_{\varepsilon})\d x.
\end{multline*}
Recalling $\Lambda(x)\leq x\Lambda'(x)$ we finally get 
\begin{multline}\label{eq:L2:reg:1}
 c_{\g}\int_{\R}\Lambda(\Psi_\varepsilon)\dx+\frac{1}{8}\int_{\R}\min\{\Psi_\varepsilon^2(x),A^2\}\dx\\*
 \leq \frac{1}{4}\int_{\R}\left[\Psi_{\varepsilon}(x)-\Gg \ast h_{\varepsilon}\right]\Lambda'(\Psi_{\varepsilon}(x))\d x+2\int_{\R}\Q^{+}_{0}\left(|\cdot|^{\g}\Gg,\widetilde{\Psi}_{\varepsilon}\right)\Lambda'(\Psi_{\varepsilon})\d x.
\end{multline}
Next, we note that
\begin{multline*}
2\int_{\R}\Q^{+}_{0}\left(|\cdot|^{\g}\Gg,\widetilde{\Psi}_{\varepsilon}\right) {\Lambda'(\Psi_{\varepsilon})}\d x \leq 
2\int_{-\ell}^{\ell}|x|^{\g}\Gg(x)\dx\int_{\R}\widetilde{\Psi}_{\varepsilon}(y)\Lambda'\left(\Psi_{\varepsilon}\left(\frac{x+y}{2}\right)\right)\dy\\
\phantom{++++} + 2C\ell^{\g-1}\int_{\R}\dx\int_{\R}\widetilde{\Psi}_{\varepsilon}(y)\Lambda'\left(\Psi_{\varepsilon}\left(\frac{x+y}{2}\right)\right)\dy\,.
\end{multline*}
Using $\Lambda'(x)\leq x$ to get $$\int_{\R}\dx\int_{\R}\widetilde{\Psi}_{\varepsilon}(y)\Lambda'\left(\Psi_{\varepsilon}\left(\frac{x+y}{2}\right)\right)\dy\leq \int_{\R}\dx\int_{\R}\widetilde{\Psi}_{\varepsilon}(y)\Psi_{\varepsilon}\left(\frac{x+y}{2}\right)\dy \leq2\|\Gg\|_{L^1}^2=2$$
we obtain together with $x\Lambda'(y)\leq \Lambda(x)+\Lambda(y)$ that
\begin{multline}\label{eq:L2:reg:2}
2\int_{\R}\Q^{+}_{0}\left(|\cdot|^{\g}\Gg,\widetilde{\Psi}_{\varepsilon}\right) {\Lambda'(\Psi_{\varepsilon})}\d x\leq 2\int_{-\ell}^{\ell}|x|^{\g}\Gg(x)\dx\int_{\R}\Lambda\left(\widetilde{\Psi}_{\varepsilon}(y)\right)+\Lambda\left(\Psi_{\varepsilon}\left(\frac{x+y}{2}\right)\right)\dy\\
\phantom{++++} + 4C\ell^{\g-1}\\*
\leq  {2\int_{-\ell}^{\ell}|x|^{\g}\Gg(x)\dx\int_{\R}\Lambda\left(\widetilde{\Psi}_{\varepsilon}(y)\right)+2\Lambda\left({\Psi}_{\varepsilon}(y)\right)\dy}+ 4C\ell^{\g-1}.
\end{multline}
With $\Lambda'(\Psi_\varepsilon)\leq A$, we have furthermore that 
\begin{equation}\label{eq:L2:regular:conv}
 \frac{1}{4}\left|\int_{\R}\left[\Psi_{\varepsilon}(x)-\Gg \ast h_{\varepsilon}\right]\Lambda'(\Psi_{\varepsilon}(x))\dx\right|\leq \frac{A}{4}\|\Psi_{\varepsilon}-\Gg \ast h_{\varepsilon}\|_{L^1}.\end{equation}
 Notice that the family $\left(h_{\varepsilon}\right)_{\varepsilon >0}$ is also a family of approximation of identity in particular since {each $h_{\varepsilon}$ is nonnegative and}  
 \begin{equation}\label{eq:h:eps}
 \int_{\R}h_{\varepsilon}(x)\d x=-\int_{\R}\frac{x}{\varepsilon^{2}}\varrho_{\varepsilon}'(x)\d x=-\int_{\R}y\varrho'(y)\d y=\int_{\R}\varrho(y)\d y=1, \qquad \forall \varepsilon >0.  
 \end{equation}
Thus
 $$\lim_{\varepsilon\to0}\left\|\Gg \ast h_{\varepsilon}-\Gg\right\|_{L^{1}}=0.$$
 Since $\Psi_{\varepsilon}$ also converges to $\Gg$ {in $L^1$ } as $\varepsilon \to 0$, we deduce from \eqref{eq:L2:regular:conv} that
$$\lim_{\varepsilon\to0}\frac{1}{4}\left|\int_{\R}\left[\Psi_{\varepsilon}(x)-\Gg \ast h_{\varepsilon}\right]\Lambda'(\Psi_{\varepsilon}(x))\dx\right|=0.$$
Combining this with \eqref{eq:L2:reg:1} and \eqref{eq:L2:reg:2}, there exists $\nu_A(\varepsilon)\to 0$ for $\varepsilon\to 0$ and $A>0$ fixed such that
\begin{multline}\label{eq:L2:reg:3}
 c_{\g}\int_{\R}\Lambda(\Psi_\varepsilon)\dx+\frac{1}{8}\int_{\R}\min\{\Psi_\varepsilon^2(x),A^2\}\dx\\*
 \leq \nu_{A}(\varepsilon)+ {2\int_{-\ell}^{\ell}|x|^{\g}\Gg(x)\dx\int_{\R}\Lambda\left(\widetilde{\Psi}_{\varepsilon}(y)\right)+2\Lambda\left({\Psi}_{\varepsilon}(y)\right)}\dy+ 4C\ell^{\g-1}.
\end{multline}
Since $\Psi_{\varepsilon}\to \Gg$ in $L^1$ as $\varepsilon\to 0$, there exists a sequence $\left(\varepsilon_k\right)_{k \in \N}$ converging to $0$ as $k\to\infty$ such that $\Psi_{\varepsilon_k}\to \Gg$ and $\tilde{\Psi}_{\varepsilon_k}\to \Gg$ for a.e.  $x\in\R$ as $k\to\infty$. Moreover, since $0\leq \Lambda(x)\leq Ax$, we have $\Lambda(\Psi_{\varepsilon_k})\leq A\Psi_{\varepsilon_k}$ and $\Lambda(\tilde{\Psi}_{\varepsilon_k})\leq A\tilde{\Psi}_{\varepsilon_k}$ as well as $\Lambda(\Psi_{\varepsilon_k}),\Lambda(\tilde{\Psi}_{\varepsilon_k})\to \Lambda(\Gg)$ a.e. on $\R$ as $k\to\infty$. By a generalised version of Lebesgue's dominated convergence theorem it then also follows that $\Lambda(\Psi_{\varepsilon_k}),\Lambda(\tilde{\Psi}_{\varepsilon_k})\to \Lambda(\Gg)$ in $L^1$ as $k\to\infty$. In the same way, using $\min\{\Psi_\varepsilon^2(x),A^2\}\leq A\Psi_\varepsilon(x)$, we get $\min\{\Psi_{\varepsilon_k}^2,A^2\}\to \min\{\Gg^2,A^2\}$ in $L^1$. Thus, restricting to the sequence $\left(\varepsilon_k\right)_{k\in \N}$ in \eqref{eq:L2:reg:3} and passing to the limit $k\to\infty$ we get for fixed $A>0$ that
\begin{multline}\label{eq:L2:reg:4}
 c_{\g}\int_{\R}\Lambda(\Gg)\dx+\frac{1}{8}\int_{\R}\min\{\Gg^2(x),A^2\}\dx\\*
 \leq 6\int_{-\ell}^{\ell}|x|^{\g}\Gg(x)\dx\int_{\R}\Lambda\left(\Gg(y)\right)\dy+ 4C\ell^{\g-1}.
\end{multline}
We can choose $\ell>0$ sufficiently small such that $6\int_{-\ell}^{\ell}|x|^{\g}\Gg(x)\dx<\frac{c_\g}{2}$ which implies
\begin{equation*}
 \frac{c_{\g}}{2}\int_{\R}\Lambda(\Gg)\dx+\frac{1}{8}\int_{\R}\min\{\Gg^2(x),A^2\}\dx\leq 4C\ell^{\g-1}.
\end{equation*}
Thus, for $A\to\infty$ by Fatou we get $\|\Gg\|_{L^2}<\infty$ (with of course a non uniform estimate with respect to $\g$). 
Since $\min\{\Gg^2,A^2\}\leq \Gg^2$ and $\Lambda(\Gg)\leq \Gg^2$ we get for $A\to\infty$ by means of Lebesgue's dominated convergence theorem from \eqref{eq:L2:reg:4} that
\begin{equation*}
  \Bigl( {\frac{c_{\g}}{4}}+\frac{1}{8}\Bigr)\int_{\R}\Gg^2(x)\dx\leq 3\int_{-\ell}^{\ell}|x|^{\g}\Gg(x)\dx\int_{\R}\Gg^2(y)\dy+ 4C\ell^{\g-1}.
\end{equation*}
This fully justifies the estimates in Lemma \ref{lem:boundL2L}.\end{proof}

Following the same lines of proof, we can rigorously prove the weighted $L^{2}$-estimates in Corollary \ref{L2-weighted}

\begin{proof}[Justification of Corollary \ref{L2-weighted}:] We proceed as in the above proof and introduce, for $\varepsilon >0$,
$$\Psi_{\varepsilon}=\Gg \ast \varrho_{\varepsilon}$$ and recall that $\Psi_{\varepsilon}$ satisfies \eqref{eq:L2:regular:1}.

We next introduce $\varphi_{\ell}\in \mathcal{C}^{1}_{b}(\R)$ such that
\begin{equation}\label{eq:Cor:regular:2}
 \varphi_{\ell}(x)=\begin{cases}
                    |x|^{k} & \text{for } |x|\leq \ell\\
                    \ell^{k}+\frac{k\ell^{k-1}}{2} &\text{for } |x|\geq \ell+1
                   \end{cases},
                   \quad 0\leq\varphi_{\ell}(x)\leq |x|^{k}\quad \text{and}\quad
                   x\varphi'_{\ell}(x)\leq k\varphi_{\ell}(x).
\end{equation}
We use $\varphi_{\ell}^2\Psi_\varepsilon$ as test function in \eqref{eq:L2:regular:1} to get
\begin{multline}\label{eq:Cor:regular:weak}
 \frac{1}{8}\int_{\R}\Psi^2_\varepsilon(x)[\varphi_{\ell}^{2}(x)-2x\varphi_{\ell}(x)\varphi_{\ell}'(x)]\dx+\int_{\R}\left[\Q_\g^{-}(\Gg,\Gg)\ast \varrho_\varepsilon\right](x)\varphi_{\ell}^{2}(x)\Psi_{\varepsilon}(x)\dx\\*
 \leq \frac{1}{4}\int_{\R}\Bigl[\Psi_{\varepsilon}(x)-\Gg\ast h_{\varepsilon}\Bigr]\varphi_{\ell}^{2}(x)\Psi_{\varepsilon}(x)\dx\\
 +\int_{\R}\left[\Q_\g^{+}(\Gg,\Gg)\ast \varrho_\varepsilon\right](x)\varphi_{\ell}^{2}(x)\Psi_{\varepsilon}(x)\dx.
\end{multline}
We first note that by means of Cauchy-Schwarz and Young's inequality we have
\begin{multline*}
 \frac{1}{4}\int_{\R}\Bigl[\Psi_{\varepsilon}(x)-\Gg\ast h_{\varepsilon}\Bigr]\varphi_{\ell}^{2}(x)\Psi_{\varepsilon}(x)\dx\leq \frac{1}{4}\Bigl\|\Psi_{\varepsilon}-\Gg\ast h_{\varepsilon}\Bigr\|_{L^2}\|\varphi_\ell^{2}\Psi_{\varepsilon}\|_{L^2}\\*
 \leq \frac{\|\varphi_\ell\|_{L^\infty}^2}{4}\Bigl\|\Psi_{\varepsilon}-\Gg\ast h_{\varepsilon}\Bigr\|_{L^2}\|\Gg\|_{L^2}\|\varrho_\varepsilon\|_{L^1}\end{multline*}
 so that
\begin{equation}\label{eq:Cor:regular:3}
 \frac{1}{4}\int_{\R}\Bigl[\Psi_{\varepsilon}(x)-\Gg\ast h_{\varepsilon}\Bigr]\varphi_{\ell}^{2}(x)\Psi_{\varepsilon}(x)\dx \leq C_{\ell}\Bigl\|\Psi_{\varepsilon}-\Gg\ast h_{\varepsilon}\Bigr\|_{L^2}=: \nu_{\ell}(\varepsilon)
\end{equation}
where we also used that $\|\Gg\|_{L^2}$ is uniformly bounded according to Theorem \ref{theo:Unique} and that $\varphi_{\ell}$ is bounded by a constant depending on $\ell$. Thus, as in \eqref{eq:L2:regular:conv} we see that for fixed $\ell>0$, we have $\nu_{\ell}(\varepsilon)\to 0$ as $\varepsilon\to0$.

Next, using Lemma \ref{lem:Sigmag} together with $\w_\g\geq 1$, we can bound $\Q_\g^{-}(\Gg,\Gg)$ from below as
\begin{multline*}
 \int_{\R}\left[\Q_\g^{-}(\Gg,\Gg)\ast \varrho_\varepsilon\right](x)\varphi_{\ell}^{2}(x)\Psi_{\varepsilon}(x)\dx=\int_{\R}[(\Gg\Sigma_\g)\ast \varrho_{\varepsilon}]\varphi_{\ell}^{2}(x)\Psi_{\varepsilon}(x)\dx\\*
 \geq \kappa_{\g}\int_{\R}[(\Gg\w_\g)\ast \varrho_{\varepsilon}]\varphi_{\ell}^{2}\Psi_{\varepsilon}\dx-[(1-\tilde{\delta}^{\g})+\sqrt{2\tilde{\delta}}\|\Gg\|_{L^2}]\int_{\R}\Psi_{\varepsilon}^{2}(x)\varphi_{\ell}^{2}(x)\dx\\*
 \geq \Bigl(\kappa_{\g}-(1-\tilde{\delta}^{\g})-\sqrt{2\tilde{\delta}}\|\Gg\|_{L^2}\Bigr)\int_{\R}\Psi_{\varepsilon}^{2}(x)\varphi_{\ell}^{2}(x)\dx.
\end{multline*}
Choosing $\tilde{\delta}=\g^2$, we get, as in the formal proof of Corollary \eqref{L2-weighted}, that $-(1-\tilde{\delta}^{\g}) \simeq 2\g\log\g$ and $\sqrt{2\tilde{\delta}}=\sqrt{2}\g$ which yields together with the uniform bound on $\|\Gg\|_{L^2}$ from Theorem \ref{theo:Unique} that
\begin{equation}\label{eq:Cor:regular:4}
 \int_{\R}\left[\Q_\g^{-}(\Gg,\Gg)\ast \varrho_\varepsilon\right](x)\varphi_{\ell}^{2}(x)\Psi_{\varepsilon}(x)\dx\geq \Bigl(\kappa_{\g}-C{\g|\log\g|}\Bigr)\int_{\R}\Psi_{\varepsilon}^{2}(x)\varphi_{\ell}^{2}(x)\dx.
\end{equation}
From \eqref{eq:Cor:regular:2} we obtain
\begin{equation}\label{eq:Cor:regular:5}
 \frac{1}{8}\int_{\R}\Psi^2_\varepsilon(x)[\varphi_{\ell}^{2}(x)-2x\varphi_{\ell}(x)\varphi_{\ell}'(x)]\dx\geq \frac{1-2k}{8}\int_{\R}\Psi^2_\varepsilon(x)\varphi_{\ell}^{2}(x)\dx.
\end{equation}
Finally, to estimate $\Q_\g^{+}(\Gg,\Gg)$ we use \eqref{eq:L2:gain:term} to deduce that
\begin{multline*}
 \int_{\R}\left[\Q_\g^{+}(\Gg,\Gg)\ast \varrho_\varepsilon\right](x)\varphi_{\ell}^{2}(x)\Psi_{\varepsilon}(x)\dx\\
 \leq 2\int_{\R}\int_{\R}|x|^{\g}\Gg(x)\widetilde{\Psi}_{\varepsilon}(y)\varphi_{\ell}^{2}\Bigl(\frac{x+y}{2}\Bigr)\Psi_{\varepsilon}\Bigl(\frac{x+y}{2}\Bigr)\dy\dx.
\end{multline*}
Next, with \eqref{eq:Cor:regular:2} we have $\varphi_{\ell}(\frac{x+y}{2})\leq |\frac{x+y}{2}|^{k}\leq \frac{1}{2}(|x|^{k}+|y|^{k})$ which yields
\begin{multline*}
 \int_{\R}\left[\Q_\g^{+}(\Gg,\Gg)\ast \varrho_\varepsilon\right](x)\varphi_{\ell}^{2}(x)\Psi_{\varepsilon}(x)\dx\\*
 \leq \int_{\R}\int_{\R}|x|^{k+\g}\Gg(x)\widetilde{\Psi}_{\varepsilon}(y)\varphi_{\ell}\Bigl(\frac{x+y}{2}\Bigr)\Psi_{\varepsilon}\Bigl(\frac{x+y}{2}\Bigr)\dy\dx\\*
 +\int_{\R}\int_{\R}|x|^{\g}|y|^{k}\Gg(x)\widetilde{\Psi}_{\varepsilon}(y)\varphi_{\ell}\Bigl(\frac{x+y}{2}\Bigr)\Psi_{\varepsilon}\Bigl(\frac{x+y}{2}\Bigr)\dy\dx.
\end{multline*}
Young's inequality implies $|x|^{\g}|y|^{k}\leq \frac{\g}{k+\g}|x|^{\g+k}+\frac{k}{k+\g}|y|^{k+\g}\leq |x|^{k+\g}+|y|^{k+\g}$ and thus
\begin{multline*}
 \int_{\R}\left[\Q_\g^{+}(\Gg,\Gg)\ast \varrho_\varepsilon\right](x)\varphi_{\ell}^{2}(x)\Psi_{\varepsilon}(x)\dx\\*
 \leq 2\int_{\R}\int_{\R}|x|^{k+\g}\Gg(x)\widetilde{\Psi}_{\varepsilon}(y)\varphi_{\ell}\Bigl(\frac{x+y}{2}\Bigr)\Psi_{\varepsilon}\Bigl(\frac{x+y}{2}\Bigr)\dy\dx\\*
 +\int_{\R}\int_{\R}\Gg(x)|y|^{k+\g}\widetilde{\Psi}_{\varepsilon}(y)\varphi_{\ell}\Bigl(\frac{x+y}{2}\Bigr)\Psi_{\varepsilon}\Bigl(\frac{x+y}{2}\Bigr)\dy\dx.
\end{multline*}
Using Cauchy Schwarz in the $y$ integral for the first term and in the $x$ integral for the second term on the right-hand side, we finally conclude
\begin{multline}\label{eq:Cor:regular:6}
 \int_{\R}\left[\Q_\g^{+}(\Gg,\Gg)\ast \varrho_\varepsilon\right](x)\varphi_{\ell}^{2}(x)\Psi_{\varepsilon}(x)\dx\\*
 \leq \Bigl(2 M_{k+\g}(\Gg)\|\widetilde{\Psi}_{\varepsilon}\|_{L^2}+M_{k+\g}(\widetilde{\Psi}_{\varepsilon})\|\Gg\|_{L^2}\Bigr)\biggl(\int_{\R}\Psi_{\varepsilon}^{2}(x)\varphi_{\ell}^{2}(x)\dx\biggr)^{\frac{1}{2}}.
\end{multline}
From Young's inequality we get $\|\widetilde{\Psi}_{\varepsilon}\|_{L^2}\leq \|\Gg\|_{L^2}$ and one has
\begin{multline}\label{eq:Cor:regular:7}
 M_{k+\g}(\widetilde{\Psi}_{\varepsilon})=\int_{\R}\int_{\R}|x|^{k+\g}\Gg(x-y)\widetilde{\varrho}_{\varepsilon}(y)\dy\dx=\int_{\R}\int_{\R}|x+y|^{k+\g}\Gg(x)\widetilde{\varrho}_{\varepsilon}(y)\dx\dy\\*
 \leq 2^{k+\g-1}\Bigl(M_{k+\g}(\Gg)+\int_{\R}|y|^{k+\g}\widetilde{\varrho}_{\varepsilon}(y)\dy\Bigr)\leq 4 (M_{k+\g}(\Gg)+c)\,,
\end{multline}
where we observe that $\varrho$  {has been chosen in such a way that $\sup_{\varepsilon\in(0,1)}\int_{\R}|y|^{k+\g}\widetilde{\varrho}_{\varepsilon}(y)\d y < \infty.$}
Gathering \eqref{eq:Cor:regular:weak}, \eqref{eq:Cor:regular:3}, \eqref{eq:Cor:regular:4}, \eqref{eq:Cor:regular:5}, \eqref{eq:Cor:regular:6} and \eqref{eq:Cor:regular:7}, we get
\begin{equation}\label{eq:Cor:regular:8}
 \Bigl(\kappa_{\g}-C\g|\log \g|+\frac{1}{8}-\frac{k}{4}\Bigr)\|\Psi_{\varepsilon}\varphi_{\ell}\|_{L^2}^{2}\leq \nu_{\ell}(\varepsilon)+4 (M_{k+\g}(\Gg)+c)\|\Psi_{\varepsilon} \varphi_{\ell}\|_{L^2}.
\end{equation}
For fixed $\ell>0$, we have $\varphi_{\ell}\leq C_\ell$ which together with $\Psi_\varepsilon \to \Gg$ in $L^2$ yields $\|\Psi_\varepsilon \varphi_\ell\|_{L^2}\to \|\Gg \varphi_\ell\|_{L^2}$ as $\varepsilon\to 0$. Thus, passing to the limit $\varepsilon\to 0$ in \eqref{eq:Cor:regular:8} yields
\begin{equation*}
 \Bigl(\kappa_{\g}-C\g|\log \g|+\frac{1}{8}-\frac{k}{4}\Bigr)\|\Gg\varphi_{\ell}\|_{L^2}^{2}\leq 4 (M_{k+\g}(\Gg)+c)\|\Gg \varphi_{\ell}\|_{L^2}.
\end{equation*}
Since $\kappa_\gamma\rightarrow1$ as $\gamma \to 0^{+}$, one easily concludes that for some explicit $\gamma_{\star} >0$, it holds $\kappa_\gamma  + \frac{1}{8} - \frac{k}{4} - \gamma\,{|\log \g|}\,C \geq \frac{1}{8}$ for any $\gamma\in[0,\gamma_{\star})$ which implies
\begin{equation*}
 \|\Gg\varphi_{\ell}\|_{L^2}\leq  32 (M_{k+\g}(\Gg)+c)\qquad \text{for } \g<\gamma_{\star}.
\end{equation*}
The claim then follows by Fatou's lemma upon passing to the limit $\ell\to\infty$.\end{proof} 

\subsection{Rigorous justification of the Sobolev estimates}\label{sec:justification:sobolev}

We now fully justify the regularity estimates in Theorem \ref{theo:gradient}. For this, we proceed by a series of lemmas. 
\begin{lem}\label{lem:Q+:Linfty:L1}
 For each $\Gg\in \mathscr{E}_{\g}$ we have $\Q_{\g}^{+}(\Gg,\Gg)\in L^{\infty}(\R)\cap L^{1}(\w_{k})$ for any $k\in[0,3-\gamma)$
 \begin{equation*}
  \begin{split}
     \|\Q_{\g}^{+}(\Gg,\Gg)\|_{L^{\infty}}&\leq  {4} \|\Gg\|_{L^2(\w_{\g})}\|\Gg\|_{L^2}\\
    \|\Q_{\g}^{\pm}(\Gg,\Gg)\|_{L^{1}(\w_k)}&\leq 2\|\Gg\|_{L^1(\w_{k+\g})}\|\Gg\|_{L^1(\w_{k})}.
  \end{split}
 \end{equation*}
\end{lem}

\begin{proof}
 Using $|y|^{\g}=|x+\frac{y}{2}+\frac{y}{2}-x|^\g\leq |x+\frac{y}{2}|^{\g}+|x-\frac{y}{2}|^{\g}$  {since $\gamma\in(0,1)$} together with the Cauchy Schwarz inequality, we find
 \begin{multline*}
  \Q_{\g}^{+}(\Gg,\Gg)(x)=\int_{\R}\Gg\left(x+\frac{y}{2}\right)\Gg\left(x-\frac{y}{2}\right)|y|^{\g}\dy\\*
  \leq \int_{\R}\left|x+\frac{y}{2}\right|^{\g}\Gg\left(x+\frac{y}{2}\right)\Gg\left(x-\frac{y}{2}\right)\dy+\int_{\R}\Gg\left(x+\frac{y}{2}\right)\left|x-\frac{y}{2}\right|^{\g}\Gg\left(x-\frac{y}{2}\right)\dy\\*
  \leq   {4} \|\Gg\|_{L^2(\w_{\g})}\|\Gg\|_{L^2}.
 \end{multline*}
This proves the first claim. For the second claim we use $|x-y|^{\g}\leq |x|^{\g}+|y|^{\g}$  {for $\gamma\in(0,1)$} to deduce
 \begin{multline*}
  \int_{\R}\Q_{\g}^{+}(\Gg,\Gg)(x)\w_k(x)\dx=\int_{\R^2}\Gg(x)\Gg(y)|x-y|^{\g} \w_k\left(\frac{x+y}{2}\right)\dy\dx\\*
  \leq \int_{\R^2}\Gg(x)\Gg(y) (|x|^{\g}+|y|^{\g}) \w_k(x)\w_k(y) \dy\dx \\*
    \leq  2\|\Gg\|_{L^1(\w_{k+\g})}\|\Gg\|_{L^1(\w_k)}.
 \end{multline*}
and
\begin{multline*}
  \int_{\R}\Q_{\g}^{-}(\Gg,\Gg)(x)\w_k(x)\dx=\int_{\R^2}\Gg(x)\Gg(y)|x-y|^{\g} \w_k(x)\dy\dx\\*
  \leq \int_{\R^2} |x|^{\g} \w_k(x) \Gg(x)\Gg(y)\dy\dx+\int_{\R^2}\Gg(x)\w_k(x) |y|^{\g}\Gg(y)\dy\dx\\*
  \leq  2\|\Gg\|_{L^1(\w_{k+\g})}\|\Gg\|_{L^1(\w_k)}.
 \end{multline*}
 This concludes the proof.
\end{proof}

\begin{lem}\label{lem:L1:sing}
 For each $\alpha\in(0,1)$ and $\Gg\in\mathscr{E}_{\g}$ we have
 \begin{equation*}
  \frac{\Gg}{|\cdot|^\alpha}\in L^1(\w_{1}) \quad \text{and}\qquad |\cdot|^{\alpha}\Gg' \in L^1(\w_{1}).
 \end{equation*}
 In particular, we have $\partial_{x}(|\cdot|^{\alpha}\Gg)\in L^1(\w_{1})$.
\end{lem}

\begin{proof}
 To prove the first claim it suffices to show that $\Gg/|\cdot|^{\alpha}\in L^1$ since $\Gg\in\mathscr{E}_{\g}$ already implies $\Gg\in L^1(\w_{1})$.
 
 We take a differentiable  {nonnegative} approximation $\nu_{\varepsilon}(x)$ of $|x|^{-\alpha}$ such that $\nu_{\varepsilon}(x)\to |x|^{-\alpha}$ and $\nu_{\varepsilon}'(x)\to -\alpha \sgn(x)|x|^{-\alpha-1}$ pointwise as $\varepsilon\to 0$ and $\nu_{\varepsilon}(x)\leq |x|^{-\alpha}$ for all $x\in\R$. Multiplying \eqref{eq:steadyg} by $\nu_\varepsilon$, we obtain
 \begin{equation}\label{eq:L1:sing:1}
  \frac{1}{4}\partial_{x}(x\Gg \nu_{\varepsilon})-\frac{1}{4}x\Gg \nu_\varepsilon' =\nu_{\varepsilon}\Q_{\g}(\Gg,\Gg).
 \end{equation}
Since $\Gg$ satisfies $x\partial_{x}\Gg = 4\Q_{\g}(\Gg ,\Gg ) - \Gg$ in the sense of distributions and the right-hand side is in $L^1$, we have $x\Gg'\in L^1$. Thus, there exists a sequence $R_n\to \infty$ as $n\to\infty$ such that upon integrating \eqref{eq:L1:sing:1} we get
\begin{equation}\label{eq:L1:sing:2}
 \frac{1}{4}x\Gg(x)\nu_{\varepsilon}(x)\big|_{-R_n}^{R_n}-\frac{1}{4}\int_{-R_n}^{R_n}x\Gg(x)\nu_{\varepsilon}'(x)\dx=\int_{-R_n}^{R_n}\nu_{\varepsilon}(x)\Q_{\g}(\Gg,\Gg)(x)\dx.
\end{equation}
According to \eqref{eq:pointX} and $\nu_{\varepsilon}(x)\leq |x|^{-\alpha}$ we have $|\pm R_n\Gg(\pm R_n)\nu_{\varepsilon}(\pm R_n)|\leq R_{n}^{-\alpha}\to 0$ as $n\to\infty$.
Moreover, recalling Lemma~\ref{lem:Q+:Linfty:L1}, $\nu_{\varepsilon}(x)\leq |x|^{-\alpha}$ and choosing $n$ large enough such that $R_n>1$ we have
\begin{multline*}
 \int_{-R_n}^{R_n}\nu_{\varepsilon}(x)\Q_{\g}(\Gg,\Gg)(x)\dx\leq C \|\Gg\|_{L^2(\w_{\g})}\|\Gg\|_{L^2}\int_{-1}^{1}|x|^{-\alpha}\dx+\int_{-R_n}^{R_{n}}\Q_{\g}^{+}(\Gg,\Gg)\dx\\*
 \leq \frac{C}{1-\alpha}\|\Gg\|_{L^2(\w_{\g})}\|\Gg\|_{L^2}+C\|\Gg\|_{L^1(\w_{\g})}\|\Gg\|_{L^1}.
\end{multline*}
Thus, passing first to the limit $n\to\infty$ and then to $\varepsilon\to 0$ in \eqref{eq:L1:sing:2} we get
\begin{equation*}
 \frac{\alpha}{4}\int_{\R}\frac{\Gg(x)}{|x|^{\alpha}}\dx\leq \frac{C}{1-\alpha}\|\Gg\|_{L^2(\w_{\g})}\|\Gg\|_{L^2}+C\|\Gg\|_{L^1(\w_{\g})}\|\Gg\|_{L^1}.
\end{equation*}
The first part of the claim then follows from Theorem~\ref{theo:Unique} and Corollary~\ref{L2-weighted}. The second part is an immediate consequence of \eqref{eq:L1:sing:1} which can be rewritten as
\begin{equation*}
 \frac{1}{4}x\nu_{\varepsilon}(x)\Gg'(x)=-\frac{1}{4}\Gg(x)\nu_{\varepsilon}(x)+\nu_{\varepsilon}\Q_{\g}(\Gg,\Gg).
\end{equation*}
From the arguments above together with Theorem~\ref{theo:Unique}, the right-hand side is in $L^1(\w_{1})$ uniformly in $\varepsilon$. Thus, we can pass to the limit $\varepsilon\to 0$ and deduce that $\frac{x}{|x|^{\alpha}}\Gg'(x)\in L^1(\w_{1})$ for each $\alpha\in (0,1)$ which immediately implies the claim.
\end{proof}

\begin{lem}\label{Lem:gain:der:L1:L2}
 Let $\gamma\in(0,1)$. For each $\alpha\in[0,1]$, for every $\Gg\in\mathscr{E}_{\g}$,  we have  $\frac{\d}{\dx}\Q_{\g}^{+}(\Gg,\Gg)\in L^1(\w_{\alpha})\cap L^2(\w_{\alpha})$. 
\end{lem}

\begin{proof}
 We argue by density and assume first $\Gg\in C_{c}^{\infty}(\R)$. We can thus write
 \begin{multline}\label{eq:grad:gain:1}
  \frac{\d}{\dx}\Q_{\g}^{+}(\Gg,\Gg)(x)=\int_{\R}\left(\Gg'\left(x+\frac{y}{2}\right)\Gg\left(x-\frac{y}{2}\right)+\Gg\left(x+\frac{y}{2}\right)\Gg'\left(x-\frac{y}{2}\right)\right)|y|^{\g}\dy\\*
  =2\int_{\R}\Gg'\left(x+\frac{y}{2}\right)\Gg\left(x-\frac{y}{2}\right)|y|^{\g}\dy=2\int_{\R}\Gg'\left(x+\frac{y}{2}\right)\Gg\left(x-\frac{y}{2}\right)\Bigl|x+\frac{y}{2}-\left(x-\frac{y}{2}\right)\Bigr|^{\g}\dy.
 \end{multline}
The inequality $||1-z|^\g-(1+|z|^\g)|\leq 2|z|^{\theta}$ for all $z\in\R$ and $0\leq \theta\leq \g\leq 1$ implies for $\theta=\frac{\g}{2}$ and $z=X/Y$ that
\begin{equation}\label{eq:weight:tr}
 K(X,Y):=\frac{|X-Y|^{\g}-|X|^\g-|Y|^{\g}}{|X|^{\frac{\g}{2}}|Y|^{\frac{\g}{2}}}\leq 2 \qquad \text{for all } X,Y\in\R.
\end{equation}
We thus rewrite \eqref{eq:grad:gain:1} to get
\begin{multline*}
  \frac{1}{2}\frac{\d}{\dx}\Q_{\g}^{+}(\Gg,\Gg)(x)=\int_{\R}\Gg'\left(x+\frac{y}{2}\right)\Gg\left(x-\frac{y}{2}\right)\Bigl|x+\frac{y}{2}\Bigr|^{\frac{\g}{2}}\Bigl|x-\frac{y}{2}\Bigr|^{\frac{\g}{2}}K\Bigl(x+\frac{y}{2},x-\frac{y}{2}\Bigr)\dy\\*
  +\int_{\R}\Gg'\left(x+\frac{y}{2}\right)\Gg\left(x-\frac{y}{2}\right)\Bigl|x+\frac{y}{2}\Bigr|^{\g}\dy+\int_{\R}\Gg'\left(x+\frac{y}{2}\right)\Gg\left(x-\frac{y}{2}\right)\Bigl|x-\frac{y}{2}\Bigr|^{\g}\dy.
 \end{multline*}
 Integrating by parts in the last term on the right-hand side, we obtain
 \begin{multline}\label{eq:der:gain:0}
  \frac{1}{2}\frac{\d}{\dx}\Q_{\g}^{+}(\Gg,\Gg)(x)=\int_{\R}\Gg'\left(x+\frac{y}{2}\right)\Gg\left(x-\frac{y}{2}\right)\Bigl|x+\frac{y}{2}\Bigr|^{\frac{\g}{2}}\Bigl|x-\frac{y}{2}\Bigr|^{\frac{\g}{2}}K\Bigl(x+\frac{y}{2},x-\frac{y}{2}\Bigr)\dy\\*
  +\int_{\R}\Gg'\left(x+\frac{y}{2}\right)\Gg\left(x-\frac{y}{2}\right)\Bigl|x+\frac{y}{2}\Bigr|^{\g}\dy+\int_{\R}\Gg\left(x+\frac{y}{2}\right)\bigl(\Gg|\cdot|^{\g}\bigr)'\left(x-\frac{y}{2}\right)\dy.
 \end{multline}
 Together with \eqref{eq:weight:tr} we thus obtain on the one hand
 \begin{multline*}
  \frac{1}{2}\left\|\frac{\d}{\dx}\Q_{\g}^{+}(\Gg,\Gg)\right\|_{L^{1}(\w_{\alpha})}\leq 2\int_{\R^2}|y|^{\frac{\g}{2}}|\Gg'(y)|\Gg(x)|x|^{\frac{\g}{2}}\w_{\alpha}\left(\frac{x+y}{2}\right)\dx\dy\\*
  +\int_{\R^2}|y|^{\g}|\Gg'(y)|\Gg(x)\w_{\alpha}\left(\frac{x+y}{2}\right)\dx\dy+\int_{\R^2}\Gg(y)|(\Gg|\cdot|^{\g})'(x)|\w_{\alpha}\left(\frac{x+y}{2}\right)\dx\dy.
 \end{multline*}
Since $\w_{\alpha}\bigl(\frac{x+y}{2}\bigr)\leq \w_{\alpha}(x)\w_{\alpha}(y)$, it follows immediately
 \begin{multline}\label{eq:der:gain:L1:1}
  \frac{1}{2}\left\|\frac{\d}{\dx}\Q_{\g}^{+}(\Gg,\Gg)\right\|_{L^{1}(\w_{\alpha})}\leq 2\|\Gg'|\cdot|^{\frac{\g}{2}}\|_{L^{1}(\w_{\alpha})}\|\Gg|\cdot|^{\frac{\g}{2}}\|_{L^{1}(\w_{\alpha})}\\*
  +\|\Gg'|\cdot|^{\g}\|_{L^{1}(\w_{\alpha})}\|\Gg\|_{L^{1}(\w_{\alpha})}+\|\Gg\|_{L^{1}(\w_{\alpha})}\|(\Gg|\cdot|^{\g})'\|_{L^{1}(\w_{\alpha})}.
 \end{multline}
 According to Lemma \ref{lem:L1:sing} the right-hand side is uniformly bounded with respect to $\alpha$ and thus, by density, we have that $\frac{\d}{\dx}\Q_{\g}^{+}(\Gg,\Gg)\in L^{1}(\w_{\alpha})$.
 
 To get the $L^2$ bound, we proceed similarly. Taking the $L^2$-norm of \eqref{eq:der:gain:0}, changing variables in the integrals in the right-hand side and taking also \eqref{eq:weight:tr} into account, we obtain
 \begin{multline*}
  \frac{1}{2}\bigl\|\frac{\d}{\dx}\Q_{\g}^{+}(\Gg,\Gg)\bigr\|_{L^{2}(\w_{\alpha})}\\*
  \leq 4\biggl(\int_{\R}\w_{2\alpha}(x)\biggl(\int_{\R}|\Gg'(y)|\Gg(2x-y)|y|^{\frac{\g}{2}}|2x-y|^{\frac{\g}{2}}\dy\biggr)^2\dx\biggr)^{\frac{1}{2}}\\*
  +2\biggl(\int_{\R}\w_{2\alpha}(x)\biggl(\int_{\R}|\Gg'(y)|\Gg(2x-y)|y|^{\g}\dy\biggr)^{2}\dx\biggr)^{\frac{1}{2}}\\*
  +2\biggl(\int_{\R}\w_{2\alpha}(x)\biggl(\int_{\R}\Gg(2x-y)\bigl|\bigl(\Gg|\cdot|^{\g}\bigr)'(y)\bigr|\dy\biggr)^{2}\dx\biggr)^{\frac{1}{2}}.
 \end{multline*}
 By means of Minkowski's inequality, the change $x\mapsto \frac{x+y}{2}$ and $\w_{2\alpha}\bigl(\frac{x+y}{2}\bigr)\leq \w_{2\alpha}(x)\w_{2\alpha}(y)$, we deduce
 \begin{multline}\label{eq:der:gain:L2:1}
  \frac{1}{2}\bigl\|\frac{\d}{\dx}\Q_{\g}^{+}(\Gg,\Gg)\bigr\|_{L^{2}(\w_{\alpha})}\leq 2\sqrt{2}\|\Gg'|\cdot|^{\frac{\g}{2}}\|_{L^{1}(\w_{\alpha})}\|\Gg|\cdot|^{\frac{\g}{2}}\|_{L^{2}(\w_{\alpha})}\\*
  +\sqrt{2}\|\Gg'|\cdot|^{\g}\|_{L^{1}(\w_{\alpha})}\|\Gg\|_{L^{2}(\w_{\alpha})}+\sqrt{2}\|\Gg\|_{L^{2}(\w_{\alpha})}\|(\Gg|\cdot|^{\g})'\|_{L^{1}(\w_{\alpha})}.
 \end{multline}
 Again the right-hand side is uniformly bounded with respect to $\alpha$ due to Corollary~\ref{L2-weighted} and Lemma~\ref{lem:L1:sing} from which the claimed $L^2$ bound follows by density.
\end{proof}

\begin{lem}\label{Lem:loss:der:L1:L2}
 Let $\gamma\in(0,1).$ For each $\alpha\in[0,1]$, for every $\Gg\in\mathscr{E}_{\g}$, we have $\Gg(\Gg\ast|\cdot|^{\g})'\in L^1(\w_{\alpha})\cap L^2(\w_{\alpha})$.
\end{lem}

\begin{proof}
 We proceed analogously to the proof of Lemma \ref{Lem:gain:der:L1:L2} and first assume $\Gg\in C_{c}^{\infty}(\R)$. We have together with \eqref{eq:weight:tr} that
 \begin{multline*}
  (\Gg\ast|\cdot|^{\g})'(x)=\int_{\R}\Gg'(y)|x-y|^{\g}\dy=\int_{\R}\Gg'(y)|y|^{\frac{\g}{2}}|x|^{\frac{\g}{2}}K(x,y)\dy\\*
  +\int_{\R}\Gg'(y)|y|^{\g}\dy+\int_{\R}\Gg'(y)|x|^{\g}\dy\\*
  =\int_{\R}\Gg'(y)|y|^{\frac{\g}{2}}|x|^{\frac{\g}{2}}K(x,y)\dy+\int_{\R}\Gg'(y)|y|^{\g}\dy.
 \end{multline*}
In the last step we used that $\int_{\R}\Gg'(y)\dy=0$ for $\Gg\in C_{c}^{\infty}(\R)$. Consequently, we get together with \eqref{eq:weight:tr} that
\begin{equation}\label{eq:der:loss:L1:1}
 \|\Gg(\Gg\ast|\cdot|^{\g})'\|_{L^{1}(\w_{\alpha})}\leq 2\|\Gg'|\cdot|^{\frac{\g}{2}}\|_{L^{1}(\w_{\alpha})}\|\Gg|\cdot|^{\frac{\g}{2}}\|_{L^{1}(\w_{\alpha})}+\|\Gg'|\cdot|^{\g}\|_{L^{1}(\w_{\alpha})}\|\Gg\|_{L^{1}(\w_{\alpha})}.
\end{equation}
The first claim thus follows from Lemma \ref{lem:L1:sing} by density. Similarly we get
\begin{equation}\label{eq:der:loss:L2:1}
 \|\Gg(\Gg\ast|\cdot|^{\g})'\|_{L^{2}(\w_{\alpha})}\leq 2\|\Gg'|\cdot|^{\frac{\g}{2}}\|_{L^{1}(\w_{\alpha})}\|\Gg|\cdot|^{\frac{\g}{2}}\|_{L^{2}(\w_{\alpha})}+\|\Gg'|\cdot|^{\g}\|_{L^{1}(\w_{\alpha})}\|\Gg\|_{L^{2}(\w_{\alpha})}
\end{equation}
from which the second claim follows again by density taking Corollary \ref{L2-weighted} and Lemma \ref{lem:L1:sing} into account.
\end{proof}

\begin{lem}\label{Lem:loss:der:2}
 If $\gamma\in(0,1)$ we have for each $\alpha\in(0,1]$ that $|\cdot|^{\alpha}\Gg'(\Gg\ast|\cdot|^{\g})\in L^1(\R)$ for every $\Gg\in\mathscr{E}_{\g}$.
\end{lem}

\begin{proof}
 The claim is an immediate consequence of Lemma~\ref{lem:L1:sing}. In fact, for $\g\in(0,1)$ we have $|x-y|^{\g}\leq|x|^\g+|y|^\g$ and thus
 \begin{equation*}
  0\leq (\Gg\ast|\cdot|^{\g})\leq |x|^{\g}\int_{\R}\Gg(y)\dy+\int_{\R}|y|^{\g}\Gg(y)\dy\leq |x|^{\g}+\|\Gg\|_{L^{1}(\w_{\g})}.
 \end{equation*}
 Thus, using Lemma~\ref{lem:L1:sing} we have
 \begin{equation*}
  \||\cdot|^{\alpha}\Gg'(\Gg\ast|\cdot|^{\g})\|_{L^{1}}\leq \||\cdot|^{\alpha+\g}\Gg'\|_{L^1}+\||\cdot|^{\alpha}\Gg'\|_{L^1}\|\Gg\|_{L^{1}(\w_{\g})}
 \end{equation*}
 and the right-hand side is bounded according to Lemma~\ref{lem:L1:sing}.
\end{proof}

\begin{lem}\label{Lem:der:L2}
 For each $\Gg\in \mathscr{E}_{\g}$ we have $\Gg'\in L^2(\R)$.
\end{lem}

\begin{proof}
Taking the distributional derivative of the steady equation \eqref{eq:steadyg} we get
 \begin{equation}\label{eq:profile:der}
  \frac{1}{4}\frac{\d}{\d x}(x\Gg')+\frac{1}{4}\Gg'=\frac{\d}{\d x}\Q_{\g}(\Gg,\Gg).
 \end{equation} 
Multiplying with $|x|^\alpha\sgn(\Gg'(x))$ for $\alpha\in(0,1)$ gives
\begin{equation*}
 \frac{1}{4}\frac{\d}{\d x}(x|x|^{\alpha}|\Gg'|)+\frac{1-\alpha}{4}|x|^{\alpha}|\Gg'|=|x|^{\alpha}\sgn(\Gg'(x))\frac{\d}{\d x}\Q_{\g}(\Gg,\Gg).
\end{equation*}
Denoting $F(x):=|x|^{\alpha}|\Gg'(x)|$ and $F_{\varepsilon}=F\ast\varrho_{\varepsilon}$ for a suitable mollifier one can check that $F_\varepsilon$ satisfies
\begin{equation*}
 \frac{1}{4}\frac{\d}{\d x}(xF_{\varepsilon})+\frac{1-\alpha}{4}F_{\varepsilon}=\frac{\d}{\dx}(F\ast (x\varrho_{\varepsilon}))+\Bigl(|x|^{\alpha}\sgn(\Gg'(x))\frac{\d}{\d x}\Q_{\g}(\Gg,\Gg)\Bigr)\ast\rho_{\varepsilon}.
\end{equation*}
We test this equation with $\Lambda'(F_{\varepsilon})$ where $\Lambda=\Lambda_A$ is given in \eqref{eq:Lambda}. Together with $\Lambda(x)\leq x\Lambda'(x)$ this yields after straightforward manipulations (similarly as in Section \ref{sec:justif-L2})
\begin{multline}\label{eq:der:L2:1}
 \frac{1-\alpha}{4}\int_{\R}\Lambda(F_{\varepsilon}(x))\dx\\*
 \leq\int_{\R}\frac{\d}{\dx}(F\ast (x\varrho_{\varepsilon}))\Lambda'(F_{\varepsilon})\dx+\int_{\R}\Bigl(|\cdot|^{\alpha}\sgn(\Gg'(\cdot))\frac{\d}{\d x}\Q_{\g}(\Gg,\Gg)\Bigr)\ast \rho_{\varepsilon}\Lambda'(F_{\varepsilon})\dx.
\end{multline}
By means of Lemma~\ref{lem:L1:sing} we can proceed similarly as in Section~\ref{sec:justif-L2} to control the first term on the right-hand side, i.e.\@ with $h_{\varepsilon}(x)=-\frac{x}{\varepsilon^{2}}\varrho'\left(\frac{x}{\varepsilon}\right)$ we have 
\begin{equation*}
 \biggl|\int_{\R}\frac{\d}{\dx}(F\ast (x\varrho_{\varepsilon}))\Lambda'(F_{\varepsilon})\dx\biggr|\leq A \|F_{\varepsilon}-F\ast h_{\varepsilon}\|_{L^1} \to 0 \qquad \text{as }\varepsilon\to 0 \text{ for any } A>0.
\end{equation*}
As in Section~\ref{sec:justif-L2}, there exists a sequence $(\varepsilon_k)_{k\in\N}$ such that $(\Lambda(F_{\varepsilon_k}(x)))_{k\in\N}$ converges towards $\Lambda(F)$ in $L^1$. 
Thus, taking $\varepsilon= \varepsilon_k$ in \eqref{eq:der:L2:1} and passing to the limit $k\to \infty$  we obtain
\begin{equation}\label{eq:der:L2:3}
 \frac{1-\alpha}{4}\int_{\R}\Lambda(F(x))\dx=\int_{\R}|x|^{\alpha}\sgn(\Gg'(x))\frac{\d}{\d x}\Q_{\g}(\Gg,\Gg)\Lambda'(F)\dx.
\end{equation}
{Using that $\frac{\d}{\d x}\Q_{\g}(\Gg,\Gg)=\frac{\d}{\d x}\Q^{+}_{\g}(\Gg,\Gg)-\Gg(\Gg\ast|\cdot|^{\g})'-\Gg'(\Gg\ast|\cdot|^{\g})$,
we can rewrite and estimate the right-hand side together with $\Lambda'(F)\geq 0$ as
\begin{multline*}
 \frac{1-\alpha}{4}\int_{\R}\Lambda(F(x))\dx=\int_{\R}|x|^{\alpha}\sgn(\Gg'(x))\frac{\d}{\dx}\Q^{+}_{\g}(\Gg,\Gg)\Lambda'(F)\dx\\*
 -\int_{\R}|x|^{\alpha}\sgn(\Gg'(x))\Gg(x)(\Gg\ast|\cdot|^{\g})'(x)\Lambda'(F)\dx-\int_{\R}|x|^{\alpha}|\Gg'(x)|(\Gg\ast|\cdot|^{\g})(x)\Lambda'(F)\dx\\*
 \leq \int_{\R}|x|^{\alpha}\Bigl|\frac{\d}{\dx}\Q^{+}_{\g}(\Gg,\Gg)\Bigr|\Lambda'(F)\dx +\int_{\R}|x|^{\alpha}\bigl|\Gg(x)(\Gg\ast|\cdot|^{\g})'(x)\bigr|\Lambda'(F)\dx.
\end{multline*}
Thus, by means of Cauchy-Schwarz we get 
\begin{equation*}
 \frac{1-\alpha}{4}\int_{\R}\Lambda(F(x))\dx\leq \Bigl(\bigl\||\cdot|^{\alpha}\frac{\d}{\dx}\Q^{+}_{\g}(\Gg,\Gg)\bigr\|_{L^2} +\||\cdot|^{\alpha}\Gg(\Gg\ast|\cdot|^{\g})'\|_{L^2}\Bigr)\|\Lambda'(F)\|_{L^2}.
\end{equation*}
Recalling \eqref{eq:der:gain:L2:1} and \eqref{eq:der:loss:L2:1} (with $\alpha=1$) and taking $|\cdot|^{\alpha}\leq\w_{\alpha}(\cdot)\leq \w_{1}(\cdot)$ into account, the right-hand side can be further estimated 
\begin{multline*}
 \frac{1-\alpha}{4}\int_{\R}\Lambda(F(x))\dx \leq\Bigl( (4\sqrt{2}+2)\|\Gg'|\cdot|^{\frac{\g}{2}}\|_{L^{1}(\w_{1})}\|\Gg|\cdot|^{\frac{\g}{2}}\|_{L^{2}(\w_{1})}\\*
  +(2\sqrt{2}+1)\|\Gg'|\cdot|^{\g}\|_{L^{1}(\w_{1})}\|\Gg\|_{L^{2}(\w_{1})}+2\sqrt{2}\|\Gg\|_{L^{2}(\w_{1})}\|(\Gg|\cdot|^{\g})'\|_{L^{1}(\w_{1})}\Bigr)\|\Lambda'(F)\|_{L^2}.
\end{multline*}
Now $(\Lambda'(F))^2 \le 2\Lambda(F)$ implies that $\|\Lambda'(F)\|_{L^2}\leq \sqrt{2}\|\Lambda(F)\|_{L^1}^{\frac{1}{2}}$ and we get 
\begin{multline*}
 \frac{1-\alpha}{4}\biggl(\int_{\R}\Lambda(F(x))\dx\biggr)^{\frac{1}{2}}\leq (8+2\sqrt{2})\|\Gg'|\cdot|^{\frac{\g}{2}}\|_{L^{1}(\w_{1})}\|\Gg|\cdot|^{\frac{\g}{2}}\|_{L^{2}(\w_{1})}\\*
  +(4+\sqrt{2})\|\Gg'|\cdot|^{\g}\|_{L^{1}(\w_{1})}\|\Gg\|_{L^{2}(\w_{1})}+4\|\Gg\|_{L^{2}(\w_{1})}\|(\Gg|\cdot|^{\g})'\|_{L^{1}(\w_{1})}.
\end{multline*}
According to Corollary \ref{L2-weighted} and Lemma \ref{lem:L1:sing} the right-hand side is bounded (independent of $\alpha$ and $A$). Thus, we can pass to the limit $A\to \infty$ and $\alpha\to 0$ which yields $\|\Gg'\|_{L^2}<\infty$.
}\end{proof}

\subsection{Proof of Proposition \ref{prop:allmoments}}\label{app:C3}

We conclude the paper with the proof of Proposition \ref{prop:allmoments} stated in the Introduction.
\begin{proof}[Proof of Proposition \ref{prop:allmoments}]
{The idea is to apply \cite[Proposition 2.4]{ABCL} to some solution to the evolution equation $\partial_t f=\Q_\g(f,f)$ with some family of approximation of $\Gg$ as initial data. We then translate this result in terms of self-similar variables and pass to the limit. First, we define some variant of the Mehler transform introduced in \cite{LM}, 
$$f_0^n(x)=e^n\int_{\R}M\left(e^n\left(x-y\right)\right)\Gg(y)\dy,\qquad
  \mbox{ where } \qquad
  M(x)=\frac{e^{-\frac{x^2}{2M_2(\Gg)}}}{\sqrt{2\pi M_2(\Gg)}}.$$
We then have 
$$\int_{\R}f_0^n(x) \dx= \int_{\R}\Gg(x) \dx, \qquad
\int_{\R}f_0^n(x) \, x \dx= \int_{\R}\Gg(x) \, x \dx,$$
$$\int_{\R}f_0^n(x)\, x^2\dx = (1+e^{-2n})\int_{\R}\Gg(x) \,x^2\dx, \qquad
\int_{\R}f_0^n(x) \,|x|^3\dx \le C \|\Gg\|_{L^1_3},$$
for some constant depending only on $M_2(\Gg)$. Moreover, for any $\psi\in L^\infty_{-3}(\R)\cap {\mathcal C}(\R)$, 
$$\lim_{n\to \infty}\int_{\R} \psi(x)f_0^n(x)\dx= \int_{\R} \psi(x)\Gg(x) \dx.$$
 For every $n$, we then choose $K_n>n$ such that 
$$\int_{\R}\left(f_0^n(x)-\min\left(f_0^n(x),K_n\right) e^{-\frac{x^2}{K_n}}\right) \langle x \rangle^3 \dx \le \frac{1}{2n},$$
and we set 
$$\tilde{f}_0^n(x)=\min\left(f_0^n(x),K_n\right) e^{-\frac{x^2}{K_n}}.$$
For any $\psi\in L^\infty_{-3}(\R)\cap {\mathcal C}(\R)$, we have 
$$\lim_{n\to \infty}\int_{\R} \psi(x)\tilde{f}_0^n(x)\dx= \int_{\R} \psi(x)\Gg(x) \dx.$$
Since $\tilde{f}_0^n\in \cap_{k\in\N} L^1_k(\R)$, we deduce from \cite[Theorem A.1]{ABCL} that there exists a unique weak solution $f^n\in {\mathcal C}([0,\infty);L^1_3(\R))$  to $\partial_t f=\Q_\g(f,f)$ with initial condition $\tilde{f}_0^n$. Moreover, for every $t\ge0$, $f^n(t)\in \cap_{k\in\N} L^1_k(\R)$ and it satisfies
$$\int_{\R} f^n(t,x) \dx = \int_{\R} \tilde{f}_0^n(x) \dx, \qquad  \int_{\R} f^n(t,x) x \dx = \int_{\R} \tilde{f}_0^n(x) x  \dx,$$
$$\int_{\R} f^n(t,x) |x|^k \dx \le \int_{\R} \tilde{f}_0^n(x) |x|^k\dx\qquad \mbox{ for any } k\ge 2.$$
With the scaling  $$g^n(t,x)=\frac{1}{V(\tau(t))} f^n\left(\tau(t),\frac{x}{V(\tau(t))}\right)=e^{-\frac{t}{4}}f^{n}\left(\tau(t),e^{-\frac{t}{4}}x\right),$$
 where  
$$V(\tau)=V_{\g}(\tau)=\left(1+\frac{\g}{4}\tau\right)^{\frac{1}{\g}} \qquad \mbox{ and } \qquad 
\tau(t)= \frac{4}{\g}\left(e^{\frac{\g t}{4}}-1\right),$$
 we deduce the existence of a unique weak solution $g^n\in {\mathcal C}([0,\infty);L^1_3(\R))$  to 
\begin{equation}\label{evol_Qgamma}
\partial_t g= -\frac14 \partial_x(xg)+\Q_\g(g,g)
\end{equation}
  with initial condition $\tilde{f}_0^n$. Then, 
$$\int_{\R} g^n(t,x) \dx = \int_{\R} \tilde{f}_0^n(x) \dx, \qquad  \int_{\R} g^n(t,x) x \dx = e^{\frac{t}{4}} \int_{\R} \tilde{f}_0^n(x) x  \dx,$$
and, more generally, for any $k\ge0$ 
$$\int_{\R} g^n(t,x) |x|^k \dx =  e^{\frac{kt}{4}}  \int_{\R} f^n\left(\tau(t),x \right) |x|^k\dx <\infty. $$
Now, since $f^n(t)\in \cap_{k\in\N} L^1_k(\R)$ for any $t\ge 0$, we may deduce from \cite[Proposition 2.4]{ABCL} the existence, for any $k>2$, of a constant $C_k$ depending only on $k$, $\g$ and $\| \tilde{f}_0^n\|_{L^1_2}$ such that
 $$\int_{\R} f^n(t,x) |x|^k \dx \le C_k(\g,\| \tilde{f}_0^n\|_{L^1_2}) \min\left(t^{-\frac{k-2}{\g}}, t^{-\frac{k}{\g}}\right), \qquad \forall t>0. $$
Observing that $\lim_{n\to\infty} \| \tilde{f}_0^n\|_{L^1_2} = \|\Gg\|_{L^1_2}$ and  setting $\tilde{C}_k(\g)= \sup_{n\ge 1} C_k(\g,\| \tilde{f}_0^n\|_{L^1_2}) <\infty$, we get 
$$ \sup_{n\ge 1} \int_{\R} f^n(t,x) |x|^k \dx \le \tilde{C}_k(\g) \min\left(t^{-\frac{k-2}{\g}}, t^{-\frac{k}{\g}}\right), \qquad \forall t>0. $$
It then follows that, for every $n\ge 1$,  
$$  \int_{\R} g^n(t,x) |x|^k \dx \le \tilde{C}_k(\g)\, e^{\frac{kt}{4}} \min\left(\tau(t)^{-\frac{k-2}{\g}}, \tau(t)^{-\frac{k}{\g}}\right), \qquad \forall t>0. $$
Our aim is now to pass to the limit $n\to\infty$ in the above inequality. To this end, we fix $T>0$ and we shall prove that $(g^n)_{n\ge1}$ is relatively sequentially compact in ${\mathcal C}([0,T], w-L^1(\R))$, where ${\mathcal C}([0,T], w-L^1(\R))$ denotes the space of continuous function from $[0,T]$ in $L^1(\R)$ endowed with its weak topology. Let us first show that, for any $t\in[0,T]$, the set $\{g^n(t), n\ge 1 \}$ is weakly relatively compact in $L^1(\R)$. Since $\Gg\in L^1(\R)$, a refined version of the de la Vall\'ee Poussin Theorem ensures the existence of nonnegative convex function $\Phi\in{\mathcal C}^2([0,\infty))$ such that  $\Phi(0)=0$, $\Phi'(0)=0$, $\Phi'$ is concave, $\Phi'(r)>0$ if $r>0$,  
$$\lim_{r\to\infty} \frac{\Phi(r)}{r}= \lim_{r\to\infty} \Phi'(r)=\infty \qquad \mbox{ and }  \qquad \int_{\R}\Phi(\Gg(x))\dx < \infty.$$
Let us note that $\Phi$ also satisfies, for $r\ge0$, $s\ge 0$, 
\begin{equation}\label{propPhi}
\Phi(r)\le r\Phi'(r),\qquad  s\Phi'(r)\le \Phi(r)+\Phi(s). 
\end{equation}
Since $\Phi$ is convex, the Jensen inequality implies that 
$$\Phi(f_0^n(x)) \leq \int_{\R} \Phi(\Gg(y))   e^n M\left(e^n\left(x-y\right)\right) \dy,$$
and thus, 
$$\int_{\R} \Phi(f_0^n(x)) \dx  \leq \int_{\R} \Phi(\Gg(x)) \dx.$$ 
Now, since $\Phi$ is nondecreasing and $\tilde{f}_0^n \le f_0^n$, we get 
 $$\int_{\R} \Phi(\tilde{f}_0^n(x)) \dx  \leq \int_{\R} \Phi(f_0^n(x))  \dx \leq \int_{\R} \Phi(\Gg(x)) \dx.$$ 
Let us show that $\sup_{t\in[0,T]}\sup_{n\ge1} \int_{\R} \Phi(g^n(t,x)) \dx <\infty.$ 
Multiplying \eqref{evol_Qgamma} by $\Phi'(g^n(t,x))$ and integrating by parts, we obtain
\begin{eqnarray*}
\frac{\d}{\d t} \int_{\R} \Phi(g^n(t,x))\dx & =& -\frac14 \int_{\R} g^n(t,x) \Phi'(g^n(t,x))\dx +\frac14 \int_{\R}\Phi(g^n(t,x))\dx \\
& &+\int_{\R}\Q_\gamma(g^n(t),g^n(t))(x) \Phi'(g^n(t,x))\dx.
\end{eqnarray*}
Thanks to \eqref{propPhi} and the nonnegativity of $\Phi'$, this leads to 
$$ \frac{\d}{\d t} \int_{\R} \Phi(g^n(t,x))\dx  \leq \int_{\R}\Q^+_\gamma(g^n(t),g^n(t))(x) \Phi'(g^n(t,x))\dx.$$
Now, since $\g\in(0,1)$, we have $|x-y|^\g\le |x|^\g+|y|^\g$ and thus, thanks to symmetry,  
$$ \frac{\d}{\d t} \int_{\R} \Phi(g^n(t,x))\dx  \leq 2 \int_{\R^2} |x|^\g g^n(t,x)g^n(t,y) \Phi'\left(g^n\left(t,\frac{x+y}{2}\right)\right)\dx\dy.$$
Finally, we deduce from  \eqref{propPhi} that 
\begin{eqnarray*}
\frac{\d}{\d t} \int_{\R} \Phi(g^n(t,x))\dx & \le&  2 \int_{\R^2} |x|^\g g^n(t,x) \left( \Phi(g^n(t,y)) +  \Phi\left(g^n\left(t,\frac{x+y}{2}\right)\right)\right)\dx\dy \\ 
&\le&  4 \int_{\R} |x|^\g g^n(t,x) \dx  \int_{\R}\Phi(g^n(t,y)) \dy \\
& \le & 4\, e^{\frac{\g t}{4}} \|\tilde{f}_0^n\|_{L^1_2} \int_{\R}\Phi(g^n(t,y)) \dy. 
\end{eqnarray*}
We thus conclude that 
$$\int_{\R} \Phi(g^n(t,x)) \dx \le  \int_{\R} \Phi(\tilde{f}_0^n(x)) \dx\, e^{4 \|\tilde{f}_0^n\|_{L^1_2}\int_0^t e^{\frac{\g s}{4}} \d s } \le  \int_{\R} \Phi(\Gg(x)) \dx \, e^{\frac{16}{\g} \|\tilde{f}_0^n\|_{L^1_2}(e^{\frac{\g t}{4}}-1)}.$$
Therefore, we have proved that 
$$ \sup_{t\in[0,T]} \sup_{n\ge 1} \left(  \int_{\R} g^n(t,x) (1+x^2) \dx 
+ \int_{\R} \Phi(g^n(t,x)) \dx  \right)<\infty,$$
and we deduce from the Dunford-Pettis Theorem that, for any $t\in[0,T]$, the set $\{g^n(t), n\ge 1 \}$ is weakly relatively compact in $L^1(\R)$.
It now suffices to check that the family $g^n : [0,T]\to L^1(\R)$ is weakly equicontinuous to conclude that $(g^n)_{n\ge1}$ is relatively sequentially compact in ${\mathcal C}([0,T], w-L^1(\R))$. Let $\varphi\in{\mathcal C}^1_{c}(\R)$, $t_1,t_2\in[0,T]$. We infer from \eqref{evol_Qgamma} that 
\begin{multline*}
\left| \int_{\R} \varphi(x) g^n(t_1,x) \dx -  \int_{\R} \varphi(x) g^n(t_2,x) \dx \right|  \le  \frac14 \left| \int_{t_1}^{t_2} \int_{\R} \varphi'(x) x g^n(s,x) \dx \d s \right|  \\
+\frac12 \left|\int_{t_1}^{t_2}\int_{\R^2} |x-y|^\g g^n(s,x)g^n(s,y) \left(2\varphi\left(\frac{x+y}{2}\right) -\varphi(x) -\varphi(y)\right) \dx \dy \d s \right|.
\end{multline*}
Hence, since $|x-y|^\g\le |x|^\g+|y|^\g$ for $\g\in(0,1)$, 
 \begin{multline*}
\left| \int_{\R} \varphi(x) g^n(t_1,x) \dx -  \int_{\R} \varphi(x) g^n(t_2,x) \dx \right|  \le  \frac14 \|\varphi\|_{W^{1,\infty}} \left| \int_{t_1}^{t_2} e^{\frac{s}{4}} \int_{\R}  |x| f^n( \tau(s) ,x) \dx \d s \right|  \\
+ 4 \|\varphi\|_{L^\infty} \left|\int_{t_1}^{t_2}\int_{\R^2} |x|^\g g^n(s,x)g^n(s,y)  \dx \dy \d s \right|.
\end{multline*}
Finally, we have 
\begin{multline*}
\left| \int_{\R} \varphi(x) g^n(t_1,x) \dx -  \int_{\R} \varphi(x) g^n(t_2,x) \dx \right|  \\
 \le  \frac14 \|\varphi\|_{W^{1,\infty}}  \|\tilde{f}_0^n\|_{L^1_2}\left| \int_{t_1}^{t_2} e^{\frac{s}{4}} \d s \right|  
+ 4 \|\varphi\|_{L^\infty} \|\tilde{f}_0^n\|_{L^1_2} \|\tilde{f}_0^n\|_{L^1} \left|\int_{t_1}^{t_2} e^{\frac{s}{4}} \d s \right|,
\end{multline*}
where the right-hand side tends to $0$ as $|t_1-t_2|$ tends to $0$. 
Enlarging this result to $\varphi\in L^\infty(\R)$ is classical and uses the uniform bound for moments of order $2$ of $g^n(t)$ with respect to both $n\ge1$ and $t\in[0,T]$. It enables to conlude that there exists a nonnegative function $g$ and a subsequence of $(g^n)_{n\ge 1}$ (not relabelled) such that  
$$g\in L^\infty( (0,T);L^1_3(\R))  \qquad \mbox{ and } \qquad g^n\to g \mbox{ in } {\mathcal C}([0,T],w-L^1(\R)).$$
Moreover, for any $k>2$, 
\begin{equation}\label{mom_limit}
 \int_{\R} g(t,x) |x|^k \dx \le \tilde{C}_k(\g)\, e^{\frac{k t}{4}} \min\left(\tau(t)^{-\frac{k-2}{\g}}, \tau(t)^{-\frac{k}{\g}}\right), \qquad \forall t\in(0,T). 
\end{equation}
It is easy to check that $g$ is a solution to \eqref{evol_Qgamma} with initial condition $\Gg$. By uniqueness of such a solution (see \cite[Theorem A.1]{ABCL}), we deduce that $g(t,\cdot)=\Gg$ for any $t\in[0,T]$. It follows from \eqref{mom_limit} that $\Gg\in \bigcap_{k\ge 0 } L^1_k(\R)$.
}\end{proof}

\bibliographystyle{plainnat-linked}

\end{document}